\documentclass{article}
\usepackage{amssymb}
\usepackage{amsfonts}
\usepackage{amsmath}
\usepackage{tikz-cd}
\usepackage{faktor}
\usepackage{graphicx,nicefrac}
\usepackage[nottoc]{tocbibind}
\usepackage{microtype}
\usepackage{capt-of}
\usepackage{hyperref}

\newtheorem{theorem}{Theorem}[section]

\newtheorem{corollary}[theorem]{Corollary}

\newtheorem{definition}[theorem]{Definition}
\newtheorem{example}[theorem]{Example}

\newtheorem{lemma}[theorem]{Lemma}

\newtheorem{remark}[theorem]{Remark}

\renewcommand{\theequation}{\thesection.\arabic{equation}}
\newenvironment{proof}[1][Proof]{\noindent\textbf{#1.} }{\ \rule{0.5em}{0.5em}}

\input{tcilatex}
\begin{document}

\title{Smooth loops and loop bundles}
\author{Sergey Grigorian \\
School of Mathematical \& Statistical Sciences\\
University of Texas Rio Grande Valley\\
Edinburg, TX 78539\\
USA}
\maketitle

\begin{abstract}
A loop is a rather general algebraic structure that has an identity element
and division, but is not necessarily associative. Smooth loops are a direct
generalization of Lie groups. A key example of a non-Lie smooth loop is the
loop of unit octonions. In this paper, we study properties of smooth loops
and their associated tangent algebras, including a loop analog of the
Mauer-Cartan equation. Then, given a manifold, we introduce a loop bundle as
an associated bundle to a particular principal bundle. Given a connection on
the principal bundle, we define the torsion of a loop bundle structure and
show how it relates to the curvature, and also consider the critical points
of some related functionals. Throughout, we see how some of the known
properties of $G_{2}$-structures can be seen from this more general setting.
\end{abstract}

\tableofcontents

\section{Introduction}

\setcounter{equation}{0}A major direction in differential geometry is the
study of Riemannian manifolds with exceptional holonomy, i.e. $7$%
-dimensional $G_{2}$-manifolds and $8$-dimensional $\func{Spin}\left(
7\right) $-manifolds, as well as more generally, $G_{2}$-structures and $%
\func{Spin}\left( 7\right) $-structures. As it turns out, both of these
groups are closely related to the octonions \cite{Harvey}, which is the $%
8$-dimensional nonassociative normed division algebra $\mathbb{O}$ over $%
\mathbb{R}.$ A number of properties of $G_{2}$-structures and $\func{Spin}%
\left( 7\right) $-structure are hence artifacts of the octonionic origin of
these groups. In particular, in \cite{GrigorianOctobundle}, the author has
explicitly used an octonion formalism to investigate properties of isometric 
$G_{2}$-structures. In that setting, it emerged that objects such as the
torsion of a $G_{2}$-structure are naturally expressed in terms of sections
of a unit octonion bundle. The set of unit octonions $U\mathbb{O}\cong
S^{7}, $ has the algebraic structure of a \emph{Moufang loop}. Indeed, a
closer look shows that in the context of $G_{2}$-structure, the algebra
structure of $\mathbb{O}$ played a secondary role to the loop structure on $U%
\mathbb{O} $ and the corresponding cross-product structure on the tangent
space at the identity $T_{1}U\mathbb{O\cong }\func{Im}\mathbb{O}$, the pure
imaginary octonions. This suggests that there is room for generalization by
considering bundles of other smooth loops. As far as possible, we will
minimize assumptions made on the loops. Generally, there is a large supply
of smooth loops, because given a Lie group $G,$ a Lie subgroup $H$, and a
smooth section $\sigma :G/H\longrightarrow G$ (i.e. a smooth collection of
coset representatives), we may define a loop structure on $G/H$ if $\sigma $
satisfies certain conditions, such as $\sigma \left( H\right) =1$, and for
any cosets $xH$ and $yH,$ there exists a unique element $z\in \sigma \left(
G/H\right) $ such that $zxH=yH$ \cite{NagyStrambachBook}. Conversely, any
smooth loop can also be described in terms of a section of a quotient of Lie
groups. Special kinds of smooth loops, such as Moufang loops have been
classified \cite{NagyStrambachBook}$,$ however for a broader classes, such
as Bol loops, there exists only a partial classification \cite{FigulaBol}.

In \cite{GrigorianOctobundle}, the octonion bundle is constructed out of the
tangent bundle, and is hence very specific, one could say canonical. However
to understand properties of the bundle, it is helpful to decouple the bundle
structure and the properties of the base manifold. Hence, another direction
for generalization is to consider loop bundles over arbitrary manifolds. In
particular, such an approach will also make it more clear which properties
of the octonion bundle in the $G_{2}$ setting are generic and which are
intrinsic to the $G_{2}$-structure.

The purpose of this paper is two-fold. One is to carefully build up the
theory of loop bundles starting with all the necessary algebraic
preliminaries and properties of smooth loops. The second is to define a
unified framework through which geometric structures based on certain
algebraic structures may be studied. In this sense, this can be considered
as an extension of the normed division algebra approach to various
structures in Riemannian geometry as developed by Leung \cite{LeungDivision}%
. The long-term goal in $G_{2}$-geometry is to obtain some kind of analog of
Yau's celebrated theorem on existence of Calabi-Yau metrics \cite{CalabiYau}%
, and thus a key theme in the study of $G_{2}$-manifolds is to try to
compare and contrast the corresponding theory of K\"{a}hler and Calabi-Yau
manifolds. This requires putting the complex and octonionic geometries into
the same framework. However, a certain amount of generalization allows to
see clearer some aspects of the theory.

In Section \ref{sectLoop} we give an overview of the key algebraic
properties of loops. While many basic properties of loops may be known to
algebraists, they may be new to geometers. Moreover, we adopt a point of
view where we emphasize the pseudoautomorphism group of a loop, which is a
generalization of the automorphism group, and properties of modified
products defined on loops. These are the key objects that are required to
define loop bundles, however in algebraic literature they typically take the
backstage. In particular, we show how the pseudoautomorphism group, the
automorphism group, the nucleus of a loop are related and how these
relationships manifest themselves in the octonion case as well-known
relationships between the groups $\func{Spin}\left( 7\right) ,$ $SO\left(
7\right) $, and $G_{2}$.

In Section \ref{sectSmooth}, we then restrict attention to smooth loops,
which are the not necessarily associative analogs of Lie groups. We also
make the assumption that the pseudoautomorphism group acts on the smooth
loop via diffeomorphisms and is hence itself a Lie group. This is an
important assumption and it is not known whether this is always true. The
key example of a non-associative smooth loop is precisely the loop of unit
octonions. We first define the concept of an exponential function, which is
similar to that on Lie groups. This is certainly not a new concept - it
first defined by Malcev in 1955 \cite{Malcev1955}, but here we show that in
fact, generally, there may be different exponential maps, based on the
initial conditions of the flow equation. This then relates to the concept of
the modified product as defined in Section \ref{sectLoop}. Then, in Section %
\ref{secTangent}, we define an algebra structure on tangent spaces of the
loop. The key difference with Lie algebras is that in the non-associative
case, there is a bracket defined at each point of the loop. Indeed, as shown
in Section \ref{sectStruct}, the differential of the bracket depends on the
associator, which of course vanishes on Lie algebras, but is non-trivial on
tangent algebras of non-associative loops. Moreover, in Section \ref%
{sectStruct}, we prove a loop version of the Maurer-Cartan structural
equation. Namely, for any point $p$ in the loop, the right Maurer-Cartan
form satisfies the following equation:%
\begin{equation}
\left( d\theta \right) _{p}-\frac{1}{2}\left[ \theta ,\theta \right]
^{\left( p\right) }=0,
\end{equation}%
where $\left[ \cdot ,\cdot \right] ^{\left( p\right) }$ is the bracket at
point $p$. In Lie theory, the Jacobi identity is the integrability condition
for the Maurer-Cartan equation, however in the non-associative case, the
corresponding equation is known as the Akivis identity \cite%
{HofmannStrambach}, and involves the associator.

In Section \ref{sectStruct} we define another key component in the theory of
smooth loops. As discussed above, each element $s$ of the loop $\mathbb{L}$
defines a bracket $b_{s}$ on the tangent algebra $\mathfrak{l}$. Moreover,
we also define a map $\varphi _{s}$ that maps the Lie algebra $\mathfrak{p}$
of the pseudoautomorphism group to the loop tangent algebra. The kernel of
this map is precisely the Lie algebra $\mathfrak{h}_{s}$ of the stabilizer
of $s$ in the pseudoautomorphism group. In the case of unit octonions, we
know $\mathfrak{p\cong so}\left( 7\right) \cong \Lambda ^{2}\left( \mathbb{R}%
^{7}\right) ^{\ast }$ and $\mathfrak{l=}\func{Im}\mathbb{O}\cong \mathbb{R}%
^{7}$ , so $\varphi _{s}$ can be regarded as an element of $\mathbb{R}%
^{7}\otimes $ $\Lambda ^{2}\mathbb{R}^{7},$ and this is (up to a constant
factor) a dualized version of the $G_{2}$-invariant $3$-form $\varphi $, as
used to project from $\Lambda ^{2}\left( \mathbb{R}^{7}\right) ^{\ast }$ to $%
\mathbb{R}^{7}.$ The kernel of this map is then the Lie algebra $\mathfrak{g}%
_{2}.$ The $3$-form $\varphi $ also defines the bracket on $\func{Im}\mathbb{%
O}$, so in this case, both $b_{s}$ and $\varphi _{s}$ are determined by the
same object, but in general they have different roles. By considering the
action of $U\left( n\right) $ on $U\left( 1\right) $ (i.e. the unit complex
numbers) and $Sp\left( n\right) Sp\left( 1\right) $ on $Sp\left( 1\right) $
(i.e. the unit quaternions), we find that Hermitian and hyperHermitian
structures fit into the same framework. Namely, a complex Hermitian form, a
quaternionic triple of Hermitian forms, and the $G_{2}$-invariant $3$-form
have the same origin as $2$-forms with values in imaginary complex numbers,
quaternions, and octonions, respectively.

In Section \ref{sectKilling} we define an analog of the Killing form on $%
\mathfrak{l}$ and give conditions for it to be invariant under both the
action of $\mathfrak{p}$ and the bracket on $\mathfrak{l}.$ In particular,
using the Killing form, we define the adjoint $\varphi _{s}^{t}$ of $\varphi
_{s}$. This allows to use the Lie bracket on $\mathfrak{p}$ to define
another bracket on $\mathfrak{l}.$ In the case of octonions, it's
proportional to the standard bracket on $\mathfrak{l},$ but in general could
be a distinct object.

In Section \ref{sectDarboux}, we consider maps from some smooth manifold $M$
to a smooth loop. Given a fixed map $s$, we can then define the
corresponding products of loop-valued maps and correspondingly a bracket of $%
\mathfrak{l}$-valued maps. Similarly as for maps to Lie groups, we define
the Darboux derivative \cite{SharpeBook} of $s$ - this is just $s^{\ast
}\theta $ - the pullback of the Maurer-Cartan form on $\mathbb{L}.$ This now
satisfies a structural equation, which is just the pullback of the loop
Maurer-Cartan equation, as derived in Section \ref{sectStruct}, with respect
to the bracket defined by $s$. For maps to Lie groups, there holds a
non-abelian \textquotedblleft Fundamental Theorem of
Calculus\textquotedblright\ \cite[Theorem 7.14]{SharpeBook}, namely that if
a Lie algebra-valued $1$-form on $M$ satisfies the structural equation, then
it is the Darboux derivative of some Lie group-valued function. Here, we
prove an analog for $\mathfrak{l}$-valued $1$-forms (Theorem \ref%
{thmLoopCartan}). However, since in the non-associative case, the bracket in
the structural equation depends on $s,$ Theorem \ref{thmLoopCartan} requires
that such a map already exists and some additional conditions are also
needed, so as expected, it's not as powerful as for Lie groups. However, in
the case the loop is associative, it does reduce to the theorem for Lie
groups.

Finally, in Section \ref{sectBundle}, we turn our attention to loop bundles
over a smooth manifold $M$. In fact, since it's not a single bundle, it's
best to refer to a \emph{loop structure} over a manifold. The key component
is $\Psi $-principal bundle $\mathcal{P}$ where $\Psi $ is a group that acts
via pseudoautomorphisms on the loop $\mathbb{L}.$ Then, several bundles
associated to $\mathcal{P}$ are defined: two bundles $\mathcal{Q}$ and $%
\mathcal{\mathring{Q}}$ with fibers diffeomorphic to $\mathbb{L}$, but with
the bundle structure with respect to different actions of $\Psi $; the
vector bundle $\mathcal{A}$ with fibers isomorphic to $\mathfrak{l},$ as
well as some others. Crucially, a section $s$ of the bundle $\mathcal{%
\mathring{Q}}$ then defines a fiberwise product structure on sections of $%
\mathcal{Q}$, a fiberwise bracket structure, and a map $\varphi _{s}$ from
sections of the adjoint bundle $\mathfrak{p}_{\mathcal{P}}$ to sections of $%
\mathcal{A}.$ In the key example of a $G_{2}$-structure on a $7$-manifold $M$%
, the bundle $\mathcal{P}$ is then the $Spin\left( 7\right) $-bundle that is
the lifting of the orthonormal frame bundle. The bundles $\mathcal{Q}$ and $%
\mathcal{\mathring{Q}}$ are unit octonion bundles, similarly as defined in 
\cite{GrigorianOctobundle}, but $\mathcal{Q}$ transforms under $SO\left(
7\right) ,$ and hence corresponds the the unit subbundle of $\mathbb{R}%
\oplus TM,$ while $\mathcal{\mathring{Q}}$ transforms under $Spin\left(
7\right) $, and hence corresponds to the unit subbundle of the spinor
bundle. The section $s$ then defines a global unit spinor, and hence defines
a reduction of the $Spin\left( 7\right) $ structure group to $G_{2}$, and
thus defines a $G_{2}$-structure. In the complex and quaternionic examples,
the corresponding bundle $\mathcal{P}$ then has $U\left( n\right) $ and $%
Sp\left( n\right) Sp\left( 1\right) $ structure group, respectively, and the
section $s$ defines a reduction to $SU\left( n\right) $ and $Sp\left(
n\right) ,$ respectively. Thus, as noted in \cite{LeungDivision}, indeed the
octonionic analog of a reduction from K\"{a}hler structure to Calabi-Yau
structure and from quaternionic K\"{a}hler to HyperK\"{a}hler, is the
reduction from $Spin\left( 7\right) $ to $G_{2}.$

Using the equivalence between sections of bundles associated to $\mathcal{P}$
and corresponding equivariant maps, we generally work with equivariant maps.
Indeed, in that case, $s:\mathcal{P}\longrightarrow \mathbb{L}$ is an
equivariant map, and given a connection $\omega $ on $\mathcal{P}$, we find
that the Darboux derivative of $s$ decomposes as 
\begin{equation}
s^{\ast }\theta =T^{\left( s,\omega \right) }-\hat{\omega}^{\left( s\right) }%
\text{,}  \label{sTom}
\end{equation}%
where $\hat{\omega}^{\left( s\right) }=\varphi _{s}\left( \omega \right) $
and $T^{\left( s,\omega \right) }$ is the \emph{torsion of }$s$ \emph{with
respect to the connection }$\omega $, which is defined as the horizontal
part of $s^{\ast }\theta .$ The quantity $T^{\left( s,\omega \right) }$ is
called the torsion because in the case of $G_{2}$-structures on a $7$%
-manifold, if we take $\mathcal{P}$ to be the spin bundle and $\omega $ the
Levi-Civita connection for a fixed metric, then $T^{\left( s,\omega \right)
} $ is precisely (up to the chosen sign convention) the torsion of the $%
G_{2} $-structure defined by the section $s$. Moreover, vanishing of $%
T^{\left( s,\omega \right) }\ $implies a reduction of the holonomy group of $%
\omega $. As shown in \cite{GrigorianOctobundle}, the torsion of a $G_{2}$%
-structure may be considered as a $1$-form with values in the bundle of
imaginary octonions. Indeed, in general, $T^{\left( s,\omega \right) }$ is a
basic (i.e. horizontal and equivariant) $\mathfrak{l}$-valued $1$-form on $%
\mathcal{P}$, so it corresponds to an $\mathcal{A}$-valued $1$-form on $M$.
It also enters expressions for covariant derivatives of products of sections
of $\mathcal{Q}$ and the bracket on $\mathcal{A}.$

The relation (\ref{sTom}) is significant because it shows that the torsion
vanishes if and only if $-\hat{\omega}^{\left( s\right) }$ is equal to the $%
\mathfrak{l}$-valued Darboux derivative $s^{\ast }\theta $. In particular, a
necessary condition is then that $-\hat{\omega}^{\left( s\right) }$
satisfies the loop structural equation. In Theorem \ref{thmTNucl}, we give a
partial converse under certain assumptions on $\mathbb{L}.$

In Section \ref{sectCurv}, we then also consider the projection of the
curvature $F$ of $\omega $ to $\mathfrak{l}.$ We define $\hat{F}=\varphi
_{s}\left( F\right) $, which is then equal to the horizontal part of $d\hat{%
\omega},$ and show in Theorem \ref{thmFTstruct} that $\hat{F}$ and $T$ are
related via a structural equation:%
\begin{equation}
\hat{F}=d^{\mathcal{H}}T-\frac{1}{2}\left[ T,T\right] ^{\left( s\right) },
\end{equation}%
where $\left[ \cdot ,\cdot \right] ^{\left( s\right) }$ is the bracket
defined by $s$. Again, such a relationship is recognizable from $G_{2}$%
-geometry, where the projection $\pi _{7}\func{Riem}$ of the Riemann
curvature to the $7$-dimensional representation of $G_{2}$ satisfy the
\textquotedblleft $G_{2}$ Bianchi identity\textquotedblright\ \cite%
{GrigorianOctobundle,karigiannis-2007}. We also consider gauge
transformations. In this setting, we have two quantities - the connection
and the section $s.$ We show that under a simultaneous gauge transformation
of the pair $\left( s,\omega \right) ,$ $\hat{F}$ and $T$ transform
equivariantly.

Finally, in Section \ref{sectVar}, we consider several functionals and the
corresponding critical points, at least under some assumptions on the loop $%
\mathbb{L}.$ Indeed, if we consider the loop bundle structure over a $3$%
-dimensional manifold, then we can write down an analog of the Chern-Simons
functional. The critical points over the space of connections, but with a
fixed section $s$, are connections for which $\hat{F}=0$, i.e. the curvature
lies in $\mathfrak{h}_{s}$ everywhere. If we moreover consider the critical
points over pairs $\left( s,\omega \right) $, then we get an additional
condition on the torsion, namely that $\left[ T,T,T\right] ^{\left( s\right)
}=0$, where $\left[ \cdot ,\cdot ,\cdot \right] ^{\left( s\right) }$ is the
associator defined by $s$ and wedge products of $1$-forms are implied.

Another functional that we consider is the $L^{2}$-norm squared of the
torsion $\int_{M}\left\vert T\right\vert ^{2}$. In this case, we fix the
connection, and consider critical points over the space of sections $s$, or
equivalently, equivariant loop-valued maps from $\mathcal{P}.$ In the $G_{2}$
setting, similar functionals have been considered in \cite%
{Bagaglini2,DGKisoflow,GrigorianOctobundle, GrigorianIsoflow,GrigorianIsoFlowSurvey,SaEarpLoubeau}.
This is then closely related to the Dirichlet energy functional, but
restricted to equivariant maps. The critical points then are maps $s$, for
which the torsion is divergence-free.

\subsection*{Acknowledgements}

This work was supported by the National Science Foundation [DMS-1811754].
The author also thanks Henrique S\'{a} Earp and Jonathan D.H. Smith for some
helpful suggestions.

\section{Loops}

\setcounter{equation}{0}\label{sectLoop}

\subsection{Definitions}

The main object of study in this paper is a \emph{loop}. Roughly, this can
be thought of as a non-associative analog of a group, but with a few
caveats. According to \cite{PflugHistorical}, this term was coined by the
group of Abraham Albert in Chicago in 1940's, as rhyming with \emph{group }%
and also referring to the Chicago Loop. Unfortunately however, for
non-algebraists, and especially in geometry and topology, this term may
cause confusion. A less ambiguous term would be something like a \emph{%
unital quasigroup }or \emph{quasigroup with identity}, however this would be
nonstandard terminology and also much longer than a loop. In general,
non-associative algebra requires a large number of definitions and concepts
that become unnecessary in the more standard associative setting. In this
section we go over some of the terminology and notation that we will be
using. The reader can also refer to \cite%
{HofmannStrambach,KiechleKloops,NagyStrambachBook,SabininBook,SmithJDHQuasiReps}
for the various concepts, although, as far as the author knows, much of the
notation in this setting is not standardized.

\begin{definition}
A \emph{quasigroup }$\mathbb{L}$ is a set together with the following
operations $\mathbb{L}\times \mathbb{L}\longrightarrow \mathbb{L}$

\begin{enumerate}
\item Product $\left( p,q\right) \mapsto pq$

\item Right quotient $\left( p,q\right) \mapsto p\backslash q$

\item Left quotient $\left( p,q\right) \mapsto q\backslash p$,
\end{enumerate}

that satisfy the following properties

\begin{enumerate}
\item $\left( p\backslash q\right) q=p$

\item $q\left( q\backslash p\right) =p$

\item $\faktor{pq}{q}=p$

\item $\scalebox{-1}[1]{\nicefrac{\scalebox{-1}[1]{$pq$}}{%
\scalebox{-1}[1]{$p$}}} =q.$
\end{enumerate}
\end{definition}

We will interchangeably denote the product operation by $p\cdot q.$ To avoid
multiple parentheses, at times we will use the convention $a\cdot bc=a\left(
bc\right) $ and $ab/c=\left( ab\right) /c$. If the same underlying set $%
\mathbb{L}$ is equipped with a different product operation $\circ _{r}$(to
be defined later), then the corresponding quasigroup will be denoted by $%
\left( \mathbb{L},\circ _{r}\right) $ and the corresponding quotient
operation by $\backslash _{r}$.

\begin{definition}
Let $\mathbb{L}$ be a quasigroup. The \emph{right nucleus }$\mathcal{N}%
^{R}\left( \mathbb{L}\right) $ \emph{of }$\mathbb{L}$ is the set of all $%
r\in \mathbb{L},$ such that for any $p,q\in \mathbb{L}$, 
\begin{equation}
pq\cdot r=p\cdot qr.  \label{assoc}
\end{equation}%
Similarly, define the left nucleus $\mathcal{N}^{L}\left( \mathbb{L}\right) $
and the middle nucleus $\mathcal{N}^{M}\left( \mathbb{L}\right) $.
\end{definition}

Elements of $\mathcal{N}^{R}\left( \mathbb{L}\right) $ satisfy several other
useful properties.

\begin{lemma}
\label{LemAssoc} If $r\in \mathcal{N}^{R}\left( \mathbb{L}\right) $, then
for any $p,q\in \mathbb{L}$,

\begin{enumerate}
\item $\faktor{pr}{qr}=p/q$

\item $p\cdot q\slash r=\faktor{pq}{r}$

\item $\scalebox{-1}[1]{\nicefrac{\scalebox{-1}[1]{$qr$}}{%
\scalebox{-1}[1]{$p$}}}=p\backslash q\cdot r.$
\end{enumerate}
\end{lemma}

\begin{lemma}
The first property follows from (\ref{assoc}) using 
\begin{equation*}
p/q\cdot qr=\left( p/q\cdot q\right) r.
\end{equation*}%
The second property follows similarly using 
\begin{equation*}
p\left( q/r\cdot r\right) =\left( p\cdot q/r\right) r.
\end{equation*}%
The third property follows using 
\begin{equation*}
\left( p\cdot p\backslash q\right) r=p\left( p\backslash q\cdot r\right) .
\end{equation*}
\end{lemma}

In group theory the only reasonable morphism between groups is a group
homomorphism, however for quasigroups there is significantly more
flexibility.

\begin{definition}
Suppose $\mathbb{L}_{1},\mathbb{L}_{2}$ are quasigroups. Then a triple $%
\left( \alpha ,\beta ,\gamma \right) $ of maps from $\mathbb{L}_{1}$ to $%
\mathbb{L}_{2}$ is a \emph{homotopy} from $\mathbb{L}_{1}$ to $\mathbb{L}%
_{2} $ if for any $p,q\in \mathbb{L}_{1}$, 
\begin{equation}
\alpha \left( p\right) \beta \left( q\right) =\gamma \left( pq\right) .
\label{Qhom}
\end{equation}%
If $\left( \alpha ,\alpha ,\alpha \right) $ is a homotopy, then $\alpha $ is
a \emph{quasigroup homomorphism}. If each of the maps $\alpha ,\beta ,\gamma 
$ is a bijection, then $\left( \alpha ,\beta ,\gamma \right) $ is an \emph{%
isotopy}. An isotopy from a quasigroup to itself is an \emph{autotopy}. The
set of all autotopies of a quasigroup $\mathbb{L}$ is clearly a group under
composition. If $\left( \alpha ,\alpha ,\alpha \right) \ $is an autotopy,
then $\alpha $ is an automorphism of $\mathbb{L}$, and the group of
automorphisms is denoted by $\func{Aut}\left( \mathbb{L}\right) $.
\end{definition}

We will only be concerned with quasigroups that have an identity element,
i.e. loops.

\begin{definition}
A \emph{loop} $\mathbb{L}$ is a quasigroup that has a unique identity
element $1\in \mathbb{L}$ such that for any $q\in \mathbb{L}$, 
\begin{equation}
1\cdot q=q\cdot 1=q.  \label{idelem}
\end{equation}
\end{definition}

\begin{definition}
Let $\mathbb{L}$ be a loop. Then, for any $q\in \mathbb{L}$ define

\begin{enumerate}
\item The \emph{right inverse }$q^{\rho }=q\backslash 1.$

\item The \emph{left inverse }$q^{\lambda }=1/q.$

In particular, they satisfy 
\begin{equation}
qq^{\rho }=q^{\lambda }q=1.
\end{equation}
\end{enumerate}
\end{definition}

For a general quasigroup, the nuclei may be empty, however if $\mathbb{L}$
is a loop, the identity element $1$ associates with any other element, so
the nuclei are non-empty. Moreover, it is easy to show that $\mathcal{N}%
^{R}\left( \mathbb{L}\right) $ (and similarly, $\mathcal{N}^{L}\left( 
\mathbb{L}\right) $ and $\mathcal{N}^{M}\left( \mathbb{L}\right) $) is a
group.

\begin{theorem}
Let $\mathbb{L}$ be a loop, then the right nucleus $\mathcal{N}^{R}\left( 
\mathbb{L}\right) $ is a group.
\end{theorem}

\begin{proof}
Clearly, $1\in \mathcal{N}^{R}\left( \mathbb{L}\right) $. Also, suppose $%
a,b\in \mathcal{N}^{R}\left( \mathbb{L}\right) .$ Then, for any $p,q\in 
\mathbb{L}$, 
\begin{eqnarray*}
pq\cdot ab &=&\left( pq\cdot a\right) b=\left( p\cdot qa\right) b \\
&=&p\left( qa\cdot b\right) =p\left( q\cdot ab\right)
\end{eqnarray*}%
and hence, $ab\in \mathcal{N}^{R}\left( \mathbb{L}\right) $. Moreover, it is
clear that the product on $\mathbb{L}$ restricted to $\mathcal{N}^{R}\left( 
\mathbb{L}\right) $ is associative.

If $a\in \mathcal{N}^{R}\left( \mathbb{L}\right) $, then 
\begin{equation*}
a=a\cdot a^{\lambda }a=aa^{\lambda }\cdot a
\end{equation*}%
and thus, $\alpha ^{\lambda }=a^{\rho }$, so $a$ has a well-defined inverse $%
a^{-1}=a^{\lambda }=a^{\rho }.$ Moreover, since for any $p\in \mathbb{L}$, $%
\left( pa^{-1}\right) a=p$, we see that $pa^{-1}=p/a$. Now, for $p,q\in 
\mathbb{L}$ we have 
\begin{equation*}
\left( p\cdot qa^{-1}\right) a=p\left( qa^{-1}\cdot a\right) =pq
\end{equation*}%
and hence 
\begin{equation*}
p\cdot qa^{-1}=\left( pq\right) /a=pq\cdot a^{-1}.
\end{equation*}%
Thus, $a^{-1}\in \mathcal{N}^{R}\left( \mathbb{L}\right) .$
\end{proof}

Loops may be endowed with additional properties that bestow various weaker
forms of associativity and inverse properties.

\begin{enumerate}
\item \emph{Two-sided inverse}: for any $p\in \mathbb{L}$, $p^{\rho
}=p^{\lambda }.$ Then we can define a unique two-sided inverse $p^{-1}.$

\item \emph{Right inverse property}: for any $p,q\in \mathbb{L}$, $pq\cdot
q^{\rho }=p.$ In particular, this implies that the inverses are two-sided,
so we can set $p^{-1}=p^{\rho }=p^{\lambda }$, and moreover $p/q=pq^{-1}$.
The \emph{left} inverse property is defined similarly. A loop with both the
left and right inverse properties is said to be an \emph{inverse loop}.

\item \emph{Power-associativity }(or\emph{\ monoassociativity}): any element 
$p\in \mathbb{L}$ generates a subgroup of $\mathbb{L}.$ In particular, this
implies that $\mathbb{L}$ has two-sided inverses. Power-associativity allows
to unambiguously define integer powers $p^{n}$ of elements. Note that some
authors use monoassociativity as a more restrictive property, namely only
that $pp\cdot p=p\cdot pp$.

\item \emph{(Left)-alternative}: for any $p,q\in \mathbb{L}$, $p\cdot
pq=pp\cdot q.$ Similarly we can define the right-alternative property (i.e. $%
q\cdot pp=qp\cdot p$). In each of these cases, $\mathbb{L}$ has two-sided
inverses. If $\mathbb{L}$ is both left-alternative and right-alternative,
then it is said to be \emph{alternative. }A loop with a similar property
that $p\cdot qp=pq\cdot p$ is known as a \emph{flexible loop}.

\item \emph{Diassociative: }any two elements $p,q\in \mathbb{L}$ generate a
subgroup of $\mathbb{L}$. Clearly, a diassociative loop has the inverse
property, is power-associative, alternative, and flexible.

\item \emph{(Left) Bol loop}: for any $p,q,r\in \mathbb{L}$, 
\begin{equation}
p\left( q\cdot pr\right) =\left( p\cdot qp\right) r.  \label{leftBol}
\end{equation}%
It is easy to see that a left Bol loop has the left inverse property and is
left-alternative and flexible \cite{RobinsonBol}. It is also
power-associative. Similarly, define a right Bol loop: for any $p,q,r\in 
\mathbb{L}$%
\begin{equation}
\left( pq\cdot r\right) q=p\left( qr\cdot q\right) .  \label{rightBol}
\end{equation}

\item \emph{Moufang loop: }a loop is a Moufang loop if it satisfies both the
left and right Bol identities. In particular, Moufang loops are
diassociative.

\item \emph{Group}: clearly any associative loop is a group.
\end{enumerate}

\begin{example}
The best-known example of a non-associative loop is the Moufang loop of unit
octonions.
\end{example}

\subsection{Pseudoautomorphisms}

Suppose now $\mathbb{L}$ is a loop and $\left( \alpha ,\beta ,\gamma \right)
\ $is an autotopy of $\mathbb{L}.$ Let $B=\alpha \left( 1\right) ,$ $A=\beta
\left( 1\right) $, $C=\gamma \left( 1\right) $. It is clear that $BA=C$.
Moreover, from (\ref{Qhom}) we see that 
\begin{eqnarray*}
\alpha \left( p\right) &=&\gamma \left( p\right) /A \\
\beta \left( p\right) &=&B\backslash \gamma \left( p\right) .
\end{eqnarray*}%
We can rewrite (\ref{Qhom}) as 
\begin{equation*}
\alpha \left( p\right) \cdot \scalebox{-1}[1]{\nicefrac{\scalebox{-1}[1]{$a%
\left( q\right) A$}}{\scalebox{-1}[1]{$B$}}} =\alpha \left( pq\right) A
\end{equation*}%
If $B=1,$ then, we obtain a \emph{right pseudoautomorphism }$\alpha $\emph{\
of }$\mathbb{L}$\emph{\ with companion }$A$, which we'll denote by the pair $%
\left( \alpha ,A\right) ,$ and which satisfies 
\begin{equation}
\alpha \left( p\right) \cdot \alpha \left( q\right) A=\alpha \left(
pq\right) A.  \label{PsAutoPair}
\end{equation}%
We have the following useful relations for quotients: 
\begin{subequations}%
\label{PsAutquot} 
\begin{eqnarray}
\alpha \left( q\backslash p\right) A &=& \scalebox{-1}[1]{\nicefrac{%
\scalebox{-1}[1]{$\alpha \left( p\right) A$}}{\scalebox{-1}[1]{$\alpha
\left( q\right)$}}} \\
\alpha \left( p/q\right) \cdot \alpha \left( q\right) A &=&\alpha \left(
p\right) A
\end{eqnarray}%
\end{subequations}%
There are several equivalent ways of characterizing \emph{right
pseudoautomorphisms}$.$

\begin{theorem}
Let $\mathbb{L}$ be a loop and suppose $\alpha :\mathbb{L}\longrightarrow 
\mathbb{L}$. Also, let $A\in \mathbb{L}$ and $\gamma =R_{A}\circ \alpha $.
Then the following are equivalent:

\begin{enumerate}
\item $\left( \alpha ,A\right) $ is a \emph{right pseudoautomorphism of }$%
\mathbb{L}$\emph{\ with companion }$A$.

\item $\left( \alpha ,\beta ,\gamma \right) $ is an autotopy of $\mathbb{L}$
with $\alpha \left( 1\right) =1$ and $\beta \left( 1\right) =\gamma \left(
1\right) =A$.

\item $\gamma \left( 1\right) =A$ and $\gamma $ satisfies 
\begin{equation}
\gamma \left( p\right) \gamma \left( q\gamma ^{-1}\left( 1\right) \right)
=\gamma \left( pq\right) .  \label{PsAutosingle}
\end{equation}
\end{enumerate}
\end{theorem}

\begin{remark}
Similarly, if $A=1,$ then we can rewrite (\ref{Qhom}) as%
\begin{equation*}
B\beta \left( p\right) \cdot \beta \left( q\right) =B\beta \left( pq\right)
\end{equation*}%
and in this case, $\beta $ is a \emph{left pseudoautomorphism} with
companion $B$. Finally, suppose $C=1,$ so that then $A=B^{\rho },$ and we
can rewrite (\ref{Qhom})%
\begin{equation*}
\gamma \left( p\right) /B^{\rho }\cdot B\backslash \gamma \left( q\right)
=\gamma \left( pq\right)
\end{equation*}%
so that in this case, $\gamma $ is a \emph{middle pseudoautomorphism} with
companion $B$.
\end{remark}

\begin{example}
In a Moufang loop, consider the map $\func{Ad}_{q},$ given by $p\longmapsto
qpq^{-1}.$ Note that this can be written unambiguously due to
diassociativity. Then, this is a right pseudoautomorphism with companion $%
q^{3}$ \cite[Lemma 1.2]{NagyStrambachBook}. Indeed, using diassociativity
for $\left\{ q,xy\right\} $, we have 
\begin{equation*}
q\left( xy\right) q^{-1}\cdot q^{3}=q\left( xy\right) q^{2}.
\end{equation*}%
On the other hand, 
\begin{eqnarray*}
qxq^{-1}\cdot qyq^{2} &=&q\left( xq^{-1}\right) \cdot \left( qyq\right) q \\
&=&\left( q\left( xq^{-1}\cdot qyq\right) \right) q \\
&=&\left( q\left( xy\cdot q\right) \right) q \\
&=&q\left( xy\right) q^{2},
\end{eqnarray*}%
where we have use appropriate Moufang identities. Hence, indeed, 
\begin{equation*}
q\left( xy\right) q^{-1}\cdot q^{3}=\left( qxq^{-1}\right) \left(
qyq^{-1}\cdot q^{3}\right) .
\end{equation*}%
In general, the adjoint map on a loop is \emph{not} a pseudoautomorphism or
a loop homomorphism. For each $q\in \mathbb{L}$, $\func{Ad}_{q}$ is just a
bijection that preserves $1\in \mathbb{L}$. However, as we see above, it is
a pseudoautomorphism if the loop is Moufang. Keeping the same terminology as
for groups, we'll say that $\func{Ad}$ defines an adjoint action of $\mathbb{%
L}$ on itself, although for a non-associative loop, this is not an action in
the usual sense of a group action.
\end{example}

We can easily see that the right pseudoautomorphisms of $\mathbb{L}$ form a
group under composition. Denote this group by $\func{PsAut}^{R}\left( 
\mathbb{L}\right) $. Clearly, $\func{Aut}\left( \mathbb{L}\right) \subset 
\func{PsAut}^{R}\left( \mathbb{L}\right) $. Similarly for left and middle
pseudoautomorphisms. More precisely, $\alpha \in $ $\func{PsAut}^{R}\left( 
\mathbb{L}\right) $ if there exists $A\in \mathbb{L}$ such that (\ref%
{PsAutoPair}) holds. Here we are not fixing the companion. On the other
hand, consider the set $\Psi ^{R}\left( \mathbb{L}\right) $ of all pairs $%
\left( \alpha ,A\right) $ of\emph{\ right pseudoautomorphisms with fixed
companions}. This then also forms a group.

\begin{lemma}
The set $\Psi ^{R}\left( \mathbb{L}\right) $ of all pairs $\left( \alpha
,A\right) $, where $\alpha \in \func{PsAut}^{R}\left( \mathbb{L}\right) $
and $A\in \mathbb{L}$ is its companion, is a group with identity element $%
\left( \func{id},1\right) $ and the following group operations:%
\begin{subequations}%
\begin{eqnarray}
\text{product}\text{:\ } &&\left( \alpha _{1},A_{1}\right) \left( \alpha
_{2},A_{2}\right) =\left( \alpha _{1}\circ \alpha _{2},\alpha _{1}\left(
A_{2}\right) A_{1}\right)  \label{PsAutprod} \\
\text{inverse}\text{: } &&\left( \alpha ,A\right) ^{-1}=\left( \alpha
^{-1},\alpha ^{-1}\left( A^{\lambda }\right) \right) =\left( \alpha
^{-1},\left( \alpha ^{-1}\left( A\right) \right) ^{\rho }\right) .
\label{PsAutInv}
\end{eqnarray}%
\end{subequations}%
\end{lemma}

\begin{proof}
Indeed, it is easy to see that $\alpha _{1}\left( A_{2}\right) A_{1}$ is a
companion of $\alpha _{1}\circ \alpha _{2}$, that (\ref{PsAutprod}) is
associative, and that $\left( \func{id},1\right) $ is the identity element
with respect to it. Also, it is easy to see that 
\begin{equation*}
\left( \alpha ,A\right) \left( \alpha ^{-1},\alpha ^{-1}\left( A^{\lambda
}\right) \right) =\left( \func{id},1\right) .
\end{equation*}%
On the other hand, setting $B=\alpha ^{-1}\left( A^{\lambda }\right) $, we
have 
\begin{eqnarray*}
B &=&\alpha ^{-1}\left( 1\right) B=\alpha ^{-1}\left( A^{\lambda }A\right) B
\\
&=&\alpha ^{-1}\left( A^{\lambda }\right) \cdot \alpha ^{-1}\left( A\right) B
\\
&=&B\cdot \alpha ^{-1}\left( A\right) B.
\end{eqnarray*}%
Cancelling $A$ on both sides on the left, we see that $B=\left( \alpha
^{-1}\left( A\right) \right) ^{\rho }.$
\end{proof}

Let $\mathcal{C}^{R}\left( \mathbb{L}\right) $ be the set of elements of $%
\mathbb{L}$ that are a companion for a right pseudoautomorphism. Then, (\ref%
{PsAutprod}) shows that there is a left action of $\Psi ^{R}\left( \mathbb{L}%
\right) $ on $\mathcal{C}^{R}\left( \mathbb{L}\right) $ given by:%
\begin{subequations}%
\label{PsiLeftaction}%
\begin{eqnarray}
\Psi ^{R}\left( \mathbb{L}\right) \times \mathcal{C}^{R}\left( \mathbb{L}%
\right) &\longrightarrow &\mathcal{C}^{R}\left( \mathbb{L}\right) \\
\left( \left( \alpha ,A\right) ,B\right) &\mapsto &\left( \alpha ,A\right)
B=\alpha \left( B\right) A.
\end{eqnarray}%
\end{subequations}%
This action is transitive, because if $A,B\in \mathcal{C}^{R}\left( \mathbb{L%
}\right) $, then exist $\alpha ,\beta \in \func{PsAut}^{R}\left( \mathbb{L}%
\right) $, such that $\left( \alpha ,A\right) ,\left( \beta ,B\right) \in
\Psi ^{R}\left( \mathbb{L}\right) $, and hence $\left( \left( \beta
,B\right) \left( \alpha ,A\right) ^{-1}\right) A=B.$ Similarly, $\Psi
^{R}\left( \mathbb{L}\right) $ also acts on all of $\mathbb{L}$. Let $%
h=\left( \alpha ,A\right) \in \Psi ^{R}\left( \mathbb{L}\right) $, then for
any $p\in \mathbb{L}$, $h\left( p\right) =\alpha \left( p\right) A.$ This is
in general non-transitive, but a faithful action (assuming $\mathbb{L}$ is
non-trivial). Using this, the definition of (\ref{PsAutoPair}) can be
rewritten as 
\begin{equation}
h\left( pq\right) =\alpha \left( p\right) h\left( q\right)
\label{PsAutProd2}
\end{equation}%
and hence the quotient relations (\ref{PsAutquot}) may be rewritten as 
\begin{subequations}%
\label{PsAutquot2} 
\begin{eqnarray}
h\left( q\backslash p\right) &=&\alpha \left( q\right) \backslash h\left(
p\right)  \label{PsAutquot2b} \\
\alpha \left( p/q\right) &=&h\left( p\right) /h\left( q\right) .
\label{PsAutquot2a}
\end{eqnarray}%
\end{subequations}
If $\Psi ^{R}\left( \mathbb{L}\right) $ acts transitively on $\mathbb{L},$
then $\mathcal{C}^{R}\left( \mathbb{L}\right) \cong \mathbb{L}$, since every
element of $\mathbb{L}$ will be a companion for some right
pseudoautomorphism. In that case, $\mathbb{L}$ is known as a (\emph{right)} $%
\emph{G}$\emph{-loop}. Note that usually a loop is known as a $G$-loop is
every element of $\mathbb{L}$ is a companion for a right pseudoautomorphism
and for a left pseudoautomorphism \cite{KunenGloops}. However, in this paper
we will only be concerned with right pseudoautomorphisms, so for brevity we
will say $\mathbb{L}$ is a $G$-loop if $\Psi ^{R}\left( \mathbb{L}\right) $
acts transitively on it.

There is another action of $\Psi ^{R}\left( \mathbb{L}\right) $ on $\mathbb{L%
}$ - which is the action by the pseudoautomorphism. This is a non-faithful
action of $\Psi ^{R}\left( \mathbb{L}\right) $, but corresponds to a
faithful action of $\func{PsAut}^{R}\left( \mathbb{L}\right) $. Namely, let $%
h=\left( \alpha ,A\right) \in \Psi ^{R}\left( \mathbb{L}\right) $, then $h\ $%
acts on $p\in \mathbb{L}$ by $p\mapsto \alpha \left( p\right) $. To
distinguish these two actions, we make the following definitions.

\begin{definition}
Let $\mathbb{L}$ be a loop and let $\Psi ^{R}\left( \mathbb{L}\right) $ the
group of right pseudoautomorphism pairs. $\mathbb{L}$ admits two left
actions of $\Psi ^{R}\left( \mathbb{L}\right) $ on itself. Let $h=\left(
\alpha ,A\right) \in \Psi ^{R}\left( \mathbb{L}\right) $ and $p\in \mathbb{L}%
.$

\begin{enumerate}
\item The \emph{full} action is given by $\left( h,p\right) \mapsto h\left(
p\right) =\alpha \left( p\right) A.$ The set $\mathbb{L}$ together with this
action of $\mathbb{\Psi }^{R}\left( \mathbb{L}\right) $ will be denoted by $%
\mathbb{\mathring{L}}.$

\item The \emph{partial} action, given by $\left( h,p\right) \mapsto
h^{\prime }\left( p\right) =\alpha \left( p\right) .$ The set $\mathbb{L}$
together with this action of $\mathbb{\Psi }^{R}\left( \mathbb{L}\right) $
will be denoted by $\mathbb{L}$ again.
\end{enumerate}
\end{definition}

\begin{remark}
From (\ref{PsAutProd2}), these definitions suggest that the loop product on $%
\mathbb{L}$ can be regarded as a map $\cdot :\mathbb{L\times \mathring{L}}%
\longrightarrow \mathbb{\mathring{L}}$. This bears some similarity to
Clifford product structure on spinors, however without the linear structure,
but instead with the constraint that $\mathbb{L}$ and $\mathbb{\mathring{L}}$
are identical as sets. This however allows to define left and right
division. '
\end{remark}

Now let us consider several relationships between the different groups
associated to $\mathbb{L}.$ First of all define the following maps:%
\begin{eqnarray}
\iota _{1} &:&\func{Aut}\left( \mathbb{L}\right) \hookrightarrow \Psi
^{R}\left( \mathbb{L}\right)  \label{i1map} \\
\gamma &\mapsto &\left( \gamma ,1\right)  \notag
\end{eqnarray}%
and 
\begin{eqnarray}
\iota _{2} &:&\mathcal{N}^{R}\left( \mathbb{L}\right) \hookrightarrow \Psi
^{R}\left( \mathbb{L}\right)  \notag \\
C &\mapsto &\left( \func{id},C\right) ,  \label{i2map}
\end{eqnarray}%
The map $\iota _{1}$ is clearly injective and is a group homomorphism, so $%
\iota _{1}\left( \func{Aut}\left( \mathbb{L}\right) \right) $ is a subgroup
of $\Psi ^{R}\left( \mathbb{L}\right) .$ On the other hand, if $A,B\in 
\mathcal{N}^{R}\left( \mathbb{L}\right) $, then in $\Psi ^{R}\left( \mathbb{L%
}\right) $, $\left( \func{id},A\right) \left( \func{id},B\right) =\left( 
\func{id},BA\right) ,$ so $\iota _{2}$ is an antihomomorphism from $\mathcal{%
N}^{R}\left( \mathbb{L}\right) $ to $\Psi ^{R}\left( \mathbb{L}\right) $ and
thus a homomorphism from the opposite group $\mathcal{N}^{R}\left( \mathbb{L}%
\right) ^{\func{op}}.$ So, $\iota _{2}\left( \mathcal{N}^{R}\left( \mathbb{L}%
\right) \right) $ is a subgroup of $\Psi ^{R}\left( \mathbb{L}\right) $ that
is isomorphic to $\mathcal{N}^{R}\left( \mathbb{L}\right) ^{\func{op}}.$

Using (\ref{i1map}) let us define a right action of $\func{Aut}\left( 
\mathbb{L}\right) $ on $\Psi ^{R}\left( \mathbb{L}\right) .$ Given $\gamma
\in \func{Aut}\left( \mathbb{L}\right) $ and $\left( \alpha ,A\right) \in
\Psi ^{R}\left( \mathbb{L}\right) $, we define 
\begin{equation}
\left( \alpha ,A\right) \cdot \gamma =\left( \alpha ,A\right) \iota
_{1}\left( \gamma \right) =\left( \alpha \circ \gamma ,A\right) .
\label{AutRAct}
\end{equation}%
Similarly, (\ref{i2map}) allows to define a left action of $\mathcal{N}%
^{R}\left( \mathbb{L}\right) ^{\func{op}}$,and hence a right action of $%
\mathcal{N}^{R}\left( \mathbb{L}\right) $, on $\Psi ^{R}\left( \mathbb{L}%
\right) $:%
\begin{equation}
C\cdot \left( \alpha ,A\right) =\iota _{2}\left( C\right) \left( \alpha
,A\right) =\left( \alpha ,AC\right) .  \label{NLAct}
\end{equation}%
The actions (\ref{AutRAct}) and (\ref{NLAct}) commute, so we can combine
them to define a left action of $\func{Aut}\left( \mathbb{L}\right) \times 
\mathcal{N}^{R}\left( \mathbb{L}\right) ^{\func{op}}.$ Indeed, given $\gamma
\in \func{Aut}\left( \mathbb{L}\right) $ and $C\in \mathcal{N}^{R}\left( 
\mathbb{L}\right) $, 
\begin{equation}
\left( \alpha ,A\right) \cdot \left( \gamma ,C\right) =\iota _{2}\left(
C\right) \left( \alpha ,A\right) \iota _{1}\left( \gamma \right) =\left(
\alpha \circ \gamma ,AC\right) .  \label{AutNAct}
\end{equation}

\begin{remark}
Since any element of $\mathcal{N}^{R}\left( \mathbb{L}\right) $ is a right
companion for any automorphism, we can also define the semi-direct product
subgroup $\iota _{1}\left( \func{Aut}\left( \mathbb{L}\right) \right)
\ltimes \iota _{2}\left( \mathcal{N}^{R}\left( \mathbb{L}\right) \right)
\subset \Psi ^{R}\left( \mathbb{L}\right) $. Suppose $\beta ,\gamma \in 
\func{Aut}\left( \mathbb{L}\right) $ and $B,C\in \mathcal{N}^{R}\left( 
\mathbb{L}\right) ,$ then in this semi-direct product, 
\begin{equation*}
\left( \beta ,B\right) \left( \gamma ,C\right) =\left( \beta \circ \gamma
,\beta \left( C\right) B\right) .
\end{equation*}
\end{remark}

\begin{lemma}
Given the actions of $\func{Aut}\left( \mathbb{L}\right) $ and $\mathcal{N}%
^{R}\left( \mathbb{L}\right) $ on $\Psi ^{R}\left( \mathbb{L}\right) $ as in
(\ref{AutRAct}) and (\ref{NLAct}), respectively, we have the following
properties.

\begin{enumerate}
\item $%
\faktor{\Psi ^{R}\left( \mathbb{L}\right)}{\func{Aut}\left(
\mathbb{L}\right)} \cong \mathcal{C}^{R}\left( \mathbb{L}\right) $ as $\Psi
^{R}\left( \mathbb{L}\right) $-sets.

\item The image $\iota _{2}\left( \mathcal{N}^{R}\left( \mathbb{L}\right)
\right) $ is a normal subgroup of $\Psi ^{R}\left( \mathbb{L}\right) $ and
hence 
\begin{equation*}
\faktor{\Psi ^{R}\left( \mathbb{L}\right)}{\mathcal{N}^{R}\left(
\mathbb{L}\right)} \cong \func{PsAut}^{R}\left( \mathbb{L}\right) .
\end{equation*}

\item Moreover, 
\begin{equation*}
\faktor{\Psi ^{R}\left( \mathbb{L}\right)}{\func{Aut}\left(
\mathbb{L}\right) \times \mathcal{N}^{R}\left( \mathbb{L}\right)} \cong %
\faktor{\func{PsAut}^{R}\left( \mathbb{L}\right)}{\func{Aut}\left(
\mathbb{L}\right)} \cong \faktor{\mathcal{C}^{R}\left(
\mathbb{L}\right)}{\mathcal{N}^{R}\left( \mathbb{L}\right)}
\end{equation*}%
where equivalence is as $\func{Aut}\left( \mathbb{L}\right) \times \mathcal{N%
}^{R}\left( \mathbb{L}\right) $-sets.
\end{enumerate}
\end{lemma}

\begin{proof}
Suppose $\mathbb{L}$ is a loop.

\begin{enumerate}
\item Consider the projection on the second component $\func{prj}_{2}:\Psi
^{R}\left( \mathbb{L}\right) \longrightarrow \mathcal{C}^{R}\left( \mathbb{L}%
\right) $ under which $\left( \alpha ,A\right) \mapsto A.$ Both $\Psi
^{R}\left( \mathbb{L}\right) \ $and $\mathcal{C}^{R}\left( \mathbb{L}\right) 
$ are left $\Psi ^{R}\left( \mathbb{L}\right) $-sets, since both admit a
left $\Psi ^{R}\left( \mathbb{L}\right) $ action - $\Psi ^{R}\left( \mathbb{L%
}\right) $ acts on itself by left multiplication and acts on $\mathcal{C}%
^{R}\left( \mathbb{L}\right) $ via the action (\ref{PsiLeftaction}). Hence, $%
\func{prj}_{2}$ is a $\Psi ^{R}\left( \mathbb{L}\right) $-equivariant map
(i.e. a $G$-set homomorphism). On the other hand, given the action (\ref%
{AutRAct}) of $\func{Aut}\left( \mathbb{L}\right) $ on $\Psi ^{R}\left( 
\mathbb{L}\right) ,$ we easily see that two pseudoautomorphisms have the
same companion if and only if they lie in the same orbit of $\func{Aut}%
\left( \mathbb{L}\right) $. Thus, $\func{prj}_{2}$ descends to a $\Psi
^{R}\left( \mathbb{L}\right) $-equivariant bijection $\Psi ^{R}\left( 
\mathbb{L}\right) /\func{Aut}\left( \mathbb{L}\right) \longrightarrow 
\mathcal{C}^{R}\left( \mathbb{L}\right) $, so that $\Psi ^{R}\left( \mathbb{L%
}\right) /\func{Aut}\left( \mathbb{L}\right) \cong \mathcal{C}^{R}\left( 
\mathbb{L}\right) $ as $\Psi ^{R}\left( \mathbb{L}\right) $-sets.

\item It is clear that $C\in $ $\mathcal{C}^{R}\left( \mathbb{L}\right) $ is
a right companion of the identity map $\func{id}$ if and only if $C\in $ $%
\mathcal{N}^{R}\left( \mathbb{L}\right) $. Now, let $\nu =\left( \func{id}%
,C\right) \in $ $\iota _{2}\left( \mathcal{N}^{R}\left( \mathbb{L}\right)
\right) $ and $g=\left( \alpha ,A\right) \in $ $\Psi ^{R}\left( \mathbb{L}%
\right) .$ Then, 
\begin{equation}
g\nu g^{-1}=\left( \alpha ,A\right) \left( \func{id},C\right) \left( \alpha
^{-1},\alpha ^{-1}\left( A^{\lambda }\right) \right) =\left( \func{id}%
,A^{\lambda }\cdot \alpha \left( C\right) A\right) .  \label{PsiAdjN}
\end{equation}%
In particular, this shows that $g\nu g^{-1}\in \iota _{2}\left( \mathcal{N}%
^{R}\left( \mathbb{L}\right) \right) $ since $A^{\lambda }\cdot \alpha
\left( C\right) A$ is the right companion of $\func{id}$. Thus indeed, $%
\iota _{2}\left( \mathcal{N}^{R}\left( \mathbb{L}\right) \right) $ is a
normal subgroup of $\Psi ^{R}\left( \mathbb{L}\right) .$ Now consider the
projection on the first component $\func{prj}_{1}:\Psi ^{R}\left( \mathbb{L}%
\right) \longrightarrow \func{PsAut}^{R}\left( \mathbb{L}\right) $ under
which $\left( \alpha ,A\right) \mapsto \alpha .$ This is clearly a group
homomorphism with kernel $\iota _{2}\left( \mathcal{N}^{R}\left( \mathbb{L}%
\right) \right) .$ Thus, $^{R}\left( \mathbb{L}\right) ^{\func{op}%
}\backslash \Psi ^{R}\left( \mathbb{L}\right) \cong \Psi ^{R}\left( \mathbb{L%
}\right) /\mathcal{N}^{R}\left( \mathbb{L}\right) \cong \func{PsAut}%
^{R}\left( \mathbb{L}\right) $.

\item Since the actions of $\mathcal{N}^{R}\left( \mathbb{L}\right) $ and $%
\func{Aut}\left( \mathbb{L}\right) $ on $\Psi ^{R}\left( \mathbb{L}\right) $
commute, the action of $\func{Aut}\left( \mathbb{L}\right) $ descends to $%
\mathcal{N}^{R}\left( \mathbb{L}\right) ^{\func{op}}\backslash \Psi
^{R}\left( \mathbb{L}\right) \cong \func{PsAut}^{R}\left( \mathbb{L}\right) $
and the action of $\mathcal{N}^{R}\left( \mathbb{L}\right) ^{\func{op}}$
descends to $\Psi ^{R}\left( \mathbb{L}\right) /\func{Aut}\left( \mathbb{L}%
\right) \cong \mathcal{C}^{R}\left( \mathbb{L}\right) .$ Since the left
action of $\mathcal{N}^{R}\left( \mathbb{L}\right) ^{\func{op}}$ on $\Psi
^{R}\left( \mathbb{L}\right) $ corresponds to an action by right
multiplication on $\mathcal{C}^{R}\left( \mathbb{L}\right) $, we find that
there is a bijection $\func{PsAut}^{R}\left( \mathbb{L}\right) /\func{Aut}%
\left( \mathbb{L}\right) \longrightarrow \mathcal{C}^{R}\left( \mathbb{L}%
\right) /\mathcal{N}^{R}\left( \mathbb{L}\right) .$

Suppose $\left( \alpha ,A\right) \in \Psi ^{R}\left( \mathbb{L}\right) $ and
let $\left[ \alpha \right] _{\func{Aut}\left( \mathbb{L}\right) }\in $ $%
\func{PsAut}^{R}\left( \mathbb{L}\right) /\func{Aut}\left( \mathbb{L}\right) 
$ be the orbit of $\alpha $ under the action of $\func{Aut}\left( \mathbb{L}%
\right) $ and let $\left[ A\right] _{\mathcal{N}^{R}\left( \mathbb{L}\right)
}\in \mathcal{C}^{R}\left( \mathbb{L}\right) /\mathcal{N}^{R}\left( \mathbb{L%
}\right) $ be the orbit of $A$ under the action of $\mathcal{N}^{R}\left( 
\mathbb{L}\right) $. Then the bijection is given by $\left[ \alpha \right] _{%
\func{Aut}\left( \mathbb{L}\right) }\mapsto \left[ A\right] _{\mathcal{N}%
^{R}\left( \mathbb{L}\right) }$. Moreover, each of these orbits also
corresponds to the orbit of $\left( \alpha ,A\right) $ under the right
action of $\func{Aut}\left( \mathbb{L}\right) \times \mathcal{N}^{R}\left( 
\mathbb{L}\right) $ on $\Psi ^{R}\left( \mathbb{L}\right) .$ These quotients
preserve actions of $\func{Aut}\left( \mathbb{L}\right) \times \mathcal{N}%
^{R}\left( \mathbb{L}\right) $ on corresponding sets and thus these coset
spaces are equivalent as $\func{Aut}\left( \mathbb{L}\right) \times \mathcal{%
N}^{R}\left( \mathbb{L}\right) $-sets.
\end{enumerate}
\end{proof}

The above relationships between the different groups are summarized in
Figure \ref{tikGroup}.

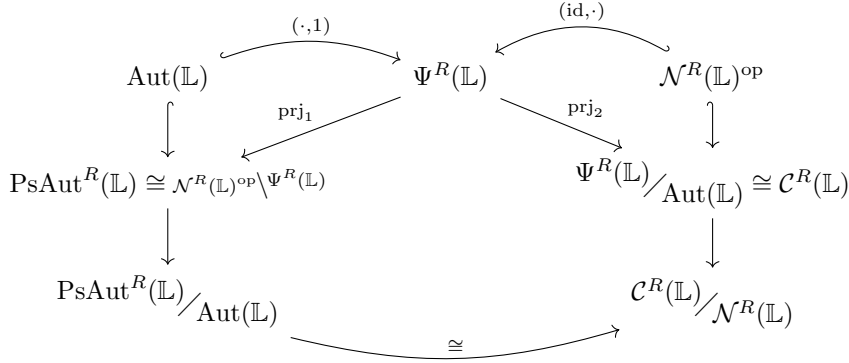
\begin{figure}[tbp]
\begin{tikzcd} \func{Aut}(\mathbb{L}) \arrow[r,hookrightarrow,bend
left=20,"{(\cdot,1)}"] \arrow[d,hook] & \Psi ^ {R} (\mathbb{L})
\arrow[dl,swap,"\func{prj}_1"] \arrow[dr,"\func{prj}_2"] &
\mathcal{N}^{R}(\mathbb{L})^{\func{op}}\arrow[d,hook'] \arrow[l,hook', bend
right,swap,"{(\func{id},\cdot)}"]\\ \func{PsAut} ^ {R} (\mathbb{L}) \cong
\scalebox{-1}[1]{\nicefrac{\scalebox{-1}[1]{$\Psi ^
{R}(\mathbb{L})$}}{\scalebox{-1}[1]{$\mathcal{N}^{R}(\mathbb{L})^{%
\func{op}}$}}} \arrow[d]& &\faktor{ \Psi ^
{R}(\mathbb{L})}{\func{Aut}(\mathbb{L})} \cong \mathcal{C}^{R}(\mathbb{L})
\arrow[d] \\ \faktor{\func{PsAut} ^ {R}
(\mathbb{L})}{\func{Aut}(\mathbb{L})} \arrow[rr,bend right=15,"\cong"] &
&\faktor{\mathcal{C}^{R}(\mathbb{L})}{\mathcal{N}^{R}(\mathbb{L})}
\end{tikzcd}
\caption{Groups related to the loop $\mathbb{L}$}
\label{tikGroup}
\end{figure}

\begin{example}
\label{ExPsQuat}Suppose $\mathbb{L=}U\mathbb{H\cong }S^{3}$ - the group of
unit quaternions. Then, since this is associative, $\mathcal{N}^{R}\left( U%
\mathbb{H}\right) =U\mathbb{H\cong }Sp\left( 1\right) .$ We also know that $%
\func{Aut}\left( U\mathbb{H}\right) \cong SO\left( 3\right) .$ Now however, $%
\Psi ^{R}\left( U\mathbb{H}\right) $ consists of all pairs $\left( \alpha
,A\right) \in SO\left( 3\right) \times U\mathbb{H}$ with the group structure
defined by (\ref{PsAutprod}),which is the semi-direct product 
\begin{equation}
\Psi ^{R}\left( U\mathbb{H}\right) \cong SO\left( 3\right) \ltimes Sp\left(
1\right) \cong Sp\left( 1\right) Sp\left( 1\right) \cong SO\left( 4\right) .
\end{equation}%
In this case, $\func{PsAut}^{R}\left( U\mathbb{H}\right) \cong \func{Aut}%
\left( U\mathbb{H}\right) \cong SO\left( 3\right) .$ Here $\left( p,q\right)
\sim \left( -p,-q\right) $ acts on $U\mathbb{H}$ via $r\mapsto prq^{-1}$.
\end{example}

\begin{example}
\label{exGroup}More generally, suppose $\mathbb{L=}G$ is a group. Then, $%
\func{PsAut}^{R}\left( G\right) \cong \func{Aut}\left( G\right) $ and $\Psi
^{R}\left( G\right) \cong \func{Aut}\left( G\right) \ltimes G^{\func{op}}$,
with $h=\left( \alpha ,A\right) \in \Psi ^{R}\left( G\right) $ acting on $G$
by 
\begin{equation}
h\left( g\right) =\alpha \left( g\right) A  \label{hG}
\end{equation}%
Note that the group $\func{Aut}\left( G\right) \ltimes G$ is known as the 
\emph{holomorph} of $G.$
\end{example}

\begin{example}
\label{ExPsOcto}Suppose $\mathbb{L=}U\mathbb{O}$ - the Moufang loop of unit
octonions, which is homeomorphic to the $7$-sphere $S^{7}.$ From \cite[Lemma
14.61]{Harvey} we know that $g\in O\left( \mathbb{O}\right) $ belongs to $%
Spin\left( 7\right) $ if and only if%
\begin{equation}
g\left( uv\right) =\chi _{g}\left( u\right) g\left( v\right)  \label{spin7}
\end{equation}%
for all $u,v\in \mathbb{O}$ where $\chi _{g}\left( u\right) =g\left(
ug^{-1}\left( 1\right) \right) $ gives the vector representation of $%
Spin\left( 7\right) $ on $\func{Im}\mathbb{O}$. We may as well restrict
everything to the non-zero octonions $\mathbb{O}^{\ast }$ or the unit
octonions $U\mathbb{O}$, so that we have a loop. Now, 
\begin{eqnarray*}
g\left( u\right) &=&g\left( u\cdot 1\right) =\chi _{g}\left( u\right)
g\left( 1\right) \\
g\left( uv\right) &=&g\left( uv\cdot 1\right) =\chi _{g}\left( uv\right)
g\left( 1\right)
\end{eqnarray*}%
Hence, we find that (\ref{spin7}) implies 
\begin{equation*}
\chi _{g}\left( uv\right) g\left( 1\right) =\chi _{g}\left( u\right) \cdot
\chi _{g}\left( v\right) g\left( 1\right) .
\end{equation*}%
Thus, $\left( \chi _{g},g\left( 1\right) \right) $ is a right
pseudoautomorphism of $U\mathbb{O}$ with companion $g\left( 1\right) $.
Thus, in this case we find that $\Psi ^{R}\left( U\mathbb{O}\right) \cong
Spin\left( 7\right) $. We also know that $\mathcal{N}^{R}\left( U\mathbb{O}%
\right) =\left\{ \pm 1\right\} \cong \mathbb{Z}_{2}$ and thus the projection 
$\left( \chi ,A\right) \mapsto \chi $ corresponds to the double cover $%
Spin\left( 7\right) \longrightarrow SO\left( 7\right) $. Hence, $\func{PsAut}%
^{R}\left( U\mathbb{O}\right) \cong SO\left( 7\right) $ and as we know, $%
\func{Aut}\left( U\mathbb{O}\right) \cong G_{2}$. Since $U\mathbb{O}$ is a
Moufang loop, and we know that for any $q$, the map $\func{Ad}_{q}$ is a
right pseudoautomorphism with companion $q$, we see that $\mathcal{C}%
^{R}\left( U\mathbb{O}\right) =U\mathbb{O},$ and indeed as we know, $%
Spin\left( 7\right) /G_{2}\cong S^{7}.$
\end{example}

\begin{remark}
We have defined the group $\Psi ^{R}\left( \mathbb{L}\right) $ as the set of 
\emph{all} right pseudoautomorphism pairs $\left( \alpha ,A\right) ,$
however we could consistently truncate $\Psi ^{R}\left( \mathbb{L}\right) $
to a subgroup, or more generally, if $G$ is some group with a homomorphism $%
\rho :G\longrightarrow \Psi ^{R}\left( \mathbb{L}\right) $, we can use this
homomorphism to define a \emph{pseudoautomorphism action} of $G$ on $\mathbb{%
L}.$ For example, if $G=\func{Aut}\left( \mathbb{L}\right) \ltimes \mathcal{N%
}^{R}\left( \mathbb{L}\right) ^{\func{op}},$ then we know that $\iota
_{1}\times \iota _{2}:G\longrightarrow \Psi ^{R}\left( \mathbb{L}\right) $
is a homomorphism. With respect to the action of $G,$ the companions would
be just the elements of $\mathcal{N}^{R}\left( \mathbb{L}\right) .$
\end{remark}

\begin{example}
\label{ExNormedDiv}In \cite{LeungDivision}, Leung developed a general
framework for structures in Riemannian geometry based on division algebras - 
$\mathbb{R},\mathbb{C},\mathbb{H},\mathbb{O}$. As a first step, this
involved representations of unitary groups with values in each of these
algebras on the algebras themselves. The unitary groups, $O\left( n\right) $%
, $U\left( n\right) $, $Sp\left( n\right) Sp\left( 1\right) $, and $%
Spin\left( 7\right) ,$ as well as the corresponding special unitary groups $%
SO\left( n\right) ,\ SU\left( n\right) ,$ $Sp\left( n\right) $, and $G_{2},$
are precisely the possible Riemannian holonomy groups for irreducible, not
locally symmetric smooth manifolds \cite{Berger1955}. By considering the
corresponding loops (groups for the associative cases) we can look at the
pseudoautomorphism actions. The octonionic case is already covered in
Example \ref{ExPsOcto}.

\begin{enumerate}
\item In the case of $\mathbb{R},$ consider instead the group of
\textquotedblleft unit reals\textquotedblright\ $U\mathbb{R}=\left\{ \pm
1\right\} \cong \mathbb{Z}_{2}.$ Then, $\Psi ^{R}\left( U\mathbb{R}\right)
=\left\{ 1\right\} \ltimes \left\{ \pm 1\right\} \mathbb{\cong }\mathbb{Z}%
_{2},$ however consider now for some positive integer $n,$ the homomorphism $%
\det :O\left( n\right) \longrightarrow \mathbb{Z}_{2}.$ Thus, $O\left(
n\right) $ acts on $\mathbb{Z}_{2}$ via this homomorphism: $\left(
g,x\right) \mapsto x\det g$, where $x\in \mathbb{Z}_{2}$ and $g\in O\left(
n\right) .$ The preimage $\func{Aut}\left( \mathbb{Z}_{2}\right) =\left\{
1\right\} $ is then just $\ker \det =SO\left( n\right) .$ Thus, we can now
define the group $\Psi _{n}^{R}\left( U\mathbb{R}\right) =O\left( n\right) .$
The full action of $\Psi _{n}^{R}\left( U\mathbb{R}\right) $ on $U\mathbb{R}$
is transitive, while the partial action is trivial. Similarly, we can also
define $\func{Aut}_{n}\left( U\mathbb{R}\right) =SO\left( n\right) .$

\item In the complex case, the group of unit complex numbers $U\mathbb{C}%
=U\left( 1\right) \cong S^{1}.$ Similarly, as above, $\Psi ^{R}\left( U%
\mathbb{C}\right) =\left\{ 1\right\} \ltimes U\left( 1\right) \mathbb{\cong }%
U\left( 1\right) .$ Now however, we also have the homomorphism $\det_{%
\mathbb{C}}:U\left( n\right) \longrightarrow U\left( 1\right) .$ Then, $%
U\left( n\right) $ acts on $U\left( 1\right) $ via $\left( g,z\right)
\mapsto z\det g$, where $z\in U\left( 1\right) $ and $g\in U\left( n\right)
. $ The preimage of $\func{Aut}\left( U\left( 1\right) \right) =\left\{
1\right\} $ is then just $\ker \det_{\mathbb{C}}=SU\left( n\right) .$ Thus,
similarly as above, we can now define the group $\Psi _{n}^{R}\left( U%
\mathbb{C}\right) =U\left( n\right) .$ The full action of $\Psi
_{n}^{R}\left( U\mathbb{R}\right) $ on $U\mathbb{C}$ is transitive, while
the partial action is trivial. Similarly, we can also define $\func{Aut}%
_{n}\left( U\mathbb{C}\right) =SU\left( n\right) .$

\item In the quaternionic case, we have already seen the case $n=1$ in
Example \ref{ExPsQuat}. The $n$-dimensional quaternionic unitary group is in
general $Sp\left( n\right) Sp\left( 1\right) $, where $Sp\left( n\right) $
is the compact symplectic group or equivalently, the quaternion special
unitary group. The group $Sp\left( n\right) Sp\left( 1\right) $ acts on $%
\mathbb{H}^{n}$ by $Sp\left( n\right) $ on the left, and multiplication by a
unit quaternion on the right, and hence can be represented by pairs $%
h=\left( \alpha ,q\right) \in Sp\left( n\right) \times Sp\left( 1\right) ,$
with the identification $\left( -\alpha ,-q\right) \sim \left( \alpha
,q\right) $. For $n\geq 2$, define the homomorphism $\rho _{\mathbb{H}%
}:Sp\left( n\right) Sp\left( 1\right) \longrightarrow Sp\left( 1\right)
Sp\left( 1\right) $ given by $\left[ \alpha ,q\right] \mapsto \left[ 1,q%
\right] .$ The image of this homomorphism simply corresponds to elements of $%
\Psi ^{R}\left( U\mathbb{H}\right) $ that are of the form $\left( \func{id}%
,q\right) ,$ i.e. act by right multiplication of $U\mathbb{H}$ on itself.
The preimage of $\func{Aut}\left( U\mathbb{H}\right) \cong SO\left( 3\right) 
$ is then $\ker \rho _{\mathbb{H}}\cong Sp\left( n\right) .$ Overall, we may
define the group $\Psi _{n}^{R}\left( U\mathbb{H}\right) =Sp\left( n\right)
Sp\left( 1\right) $ and $\func{Aut}_{n}\left( U\mathbb{H}\right) =Sp\left(
n\right) .$ As in the previous examples, the full action of $\Psi
_{n}^{R}\left( U\mathbb{H}\right) $ on $U\mathbb{H}$ is transitive, whereas
the partial action is again trivial. We will refer to this example later on,
with the assumption that $n\geq 2$.
\end{enumerate}

Thus, in each of the above cases, we may regard $\Psi _{n}^{R}$ ($O\left(
n\right) ,U\left( n\right) ,$ or $Sp\left( n\right) Sp\left( 1\right) $) as 
\emph{a }group of\emph{\ }pseudoautomorphism pairs acting on the unit real
numbers, unit complex numbers, and unit quaternions with a trivial partial
action and will the full action just given by right multiplication. The
corresponding automorphism subgroups are then the \textquotedblleft
special\textquotedblright\ unitary subgroups $SO\left( n\right) ,$ $SU\left(
n\right) ,$ $Sp\left( n\right) .$
\end{example}

\subsection{Modified product}

Let $r\in \mathbb{L}$, and define the modified product $\circ _{r}$ on $%
\mathbb{L}$ via%
\begin{equation}
p\circ _{r}q=\faktor{\left( p\cdot qr\right)}{r}.  \label{rprod}
\end{equation}%
Then, $p\circ _{r}q=p\cdot q$ if and only if $p\cdot qr=pq\cdot r.$ This is
true for all $p,q$ if and only if $r\in \mathcal{N}^{R}\left( \mathbb{L}%
\right) $. However, this will not hold for all $r$ unless $\mathbb{L}$ is
associative (and is thus a group). If $\mathbb{L}$ is a right Bol loop, and $%
a\in \mathbb{L}$, setting $r=q\backslash a$ in the right Bol identity (\ref%
{rightBol}), gives us 
\begin{equation}
pq\cdot q\backslash a=\faktor{\left( p\cdot aq\right)}{q}=p\circ _{q}a.
\label{midprod}
\end{equation}%
On octonions, the left-hand side of (\ref{midprod}) is precisely the
\textquotedblleft modified octonion product\textquotedblright\ defined in 
\cite{GrigorianOctobundle} and also used in \cite{GrigorianOctoSUSY}. Since
unit octonions are in particular a right Bol loop, the two products are
equal on octonions.

The product (\ref{rprod}) gives us a convenient definition of the \emph{loop
associator}.

\begin{definition}
Given $p,q,r\in \mathbb{L}$, the \emph{loop associator }of $p,q,r$ is given
by 
\begin{equation}
\left[ p,q,r\right] =\faktor{\left( p\circ _{r}q\right)}{pq}.
\label{loopassoc}
\end{equation}%
The \emph{loop commutator }of $p$ and $q$ is given by 
\begin{equation}
\left[ p,q\right] =\faktor{\left( pq/p\right)}{q}.  \label{loopcomm}
\end{equation}
\end{definition}

From the definition (\ref{loopassoc}), we see that $\left[ p,q,r\right] =1$
if and only if $p\left( qr\right) =\left( pq\right) r.$ There are several
possible equivalent definitions of the associator, but from our point of
view, (\ref{loopassoc}) will be the most convenient. Similarly, the loop
commutator can be defined in different ways, however (\ref{loopcomm}) has an
advantage, because if we define $\func{Ad}_{p}\left( q\right) =pq/p$, then $%
\left[ p,q\right] =\left( \func{Ad}_{p}\left( q\right) \right) /q,$ which is
a similar relation as for the group commutator.

We can easily see that $\left( \mathbb{L},\circ _{r}\right) $ is a loop.

\begin{lemma}
Consider the pair $\left( \mathbb{L},\circ _{r}\right) $ of the set $\mathbb{%
L}$ equipped with the binary operation $\circ _{r}$.

\begin{enumerate}
\item The right quotient $/_{r}$ and the left quotient $\backslash _{r}$ on $%
\left( \mathbb{L},\circ _{r}\right) $ are given by 
\begin{subequations}%
\label{rprodq} 
\begin{eqnarray}
p/_{r}q &=&\faktor{pr}{qr}  \label{rprodqright} \\
p\backslash _{r}q &=&\faktor{\left( p\backslash qr\right)}{r},
\label{rprodqleft}
\end{eqnarray}%
\end{subequations}%
and hence, $\left( \mathbb{L},\circ _{r}\right) $ is a quasigroup.

\item $1\in \mathbb{L}$ is the identity element for $\left( \mathbb{L},\circ
_{r}\right) ,$ and hence $\left( \mathbb{L},\circ _{r}\right) $ is a loop.

\item Let $q\in \mathbb{L}$, the left and right inverses with respect to $%
\circ _{r}$ are given by 
\begin{subequations}
\begin{eqnarray}
q^{\lambda _{\left( r\right) }} &=&\faktor{r}{qr}  \label{linvr} \\
q^{\rho _{\left( r\right) }} &=&\faktor{\left( q\backslash r\right)}{r}.
\label{rinvr}
\end{eqnarray}%
\end{subequations}%

\item $\left( \mathbb{L},\circ _{r}\right) $ is isomorphic to $\left( 
\mathbb{L},\cdot \right) $ if and only if $r\in \mathcal{C}^{R}\left( 
\mathbb{L}\right) $. In particular, $\alpha :\left( \mathbb{L},\cdot \right)
\longrightarrow \left( \mathbb{L},\circ _{r}\right) $ is an isomorphism,
i.e. for any $p,q\in \mathbb{L},$%
\begin{equation}
\alpha \left( pq\right) =\alpha \left( p\right) \circ _{r}\alpha \left(
q\right) ,  \label{alpharcirc}
\end{equation}%
if and only if $\alpha $ is a right pseudoautomorphism on $\left( \mathbb{L}%
,\cdot \right) $ with companion $r$.
\end{enumerate}
\end{lemma}

\begin{proof}
Let $x,p,q,r\in \mathbb{L}.$

\begin{enumerate}
\item Suppose 
\begin{equation*}
x\circ _{r}q=p.
\end{equation*}%
Using (\ref{rprod}), 
\begin{equation*}
x\cdot qr=pr,
\end{equation*}%
and thus 
\begin{equation*}
x=pr/qr:=p/_{r}q.
\end{equation*}%
Similarly, suppose%
\begin{equation*}
p\circ _{r}x=q,
\end{equation*}%
so that 
\begin{equation*}
p\cdot xr=qr,
\end{equation*}%
and thus 
\begin{equation*}
x=\left( p\backslash \left( qr\right) \right) /r:=p\backslash _{r}q.
\end{equation*}%
Since the left and right quotients are both defined, $\left( \mathbb{L}%
,\circ _{r}\right) $ is a quasigroup.

\item We have 
\begin{eqnarray*}
p\circ _{r}1 &=&\left( p\cdot r\right) /r=p \\
1\circ _{r}p &=&\left( 1\cdot pr\right) /r=p.
\end{eqnarray*}%
Hence, $1$ is indeed the identity element for $\left( \mathbb{L},\circ
_{r}\right) ,$ and thus $\left( \mathbb{L},\circ _{r}\right) $ is a loop.

\item Setting $p=1$ in (\ref{rprodq}) we get the desired expressions.

\item Suppose $\left( \alpha ,r\right) \in \Psi ^{R}\left( \mathbb{L}\right) 
$. Then, by definition, for any $p,q\in \mathbb{L}$, 
\begin{equation*}
\alpha \left( pq\right) =\faktor{\left( \alpha \left( p\right) \cdot \alpha
\left( q\right) r\right)}{r}
\end{equation*}%
Hence, from (\ref{rprod}), 
\begin{equation}
\alpha \left( pq\right) =\alpha \left( p\right) \circ _{r}\alpha \left(
q\right) ,
\end{equation}%
Thus, $\alpha $ is an isomorphism\emph{\ }from $\left( \mathbb{L},\cdot
\right) $ to $\left( \mathbb{L},\circ _{r}\right) $. Clearly the converse is
also true: if $\alpha $ is an isomorphism from $\left( \mathbb{L},\cdot
\right) $ to $\left( \mathbb{L},\circ _{r}\right) $, then $r$ is companion
for $\alpha $. Hence, $\left( \mathbb{L},\cdot \right) $ and $\left( \mathbb{%
L},\circ _{r}\right) $ are isomorphic if and only if $r$ is a companion for
some right pseudoautomorphism.
\end{enumerate}
\end{proof}

Suppose $r,x\in \mathbb{L}$, then the next lemma shows the relationship
between products $\circ _{x}$ and $\circ _{rx}$.

\begin{lemma}
\label{lemxrprod}Let $r,x\in \mathbb{L}$, then 
\begin{equation}
p\circ _{rx}q=\left( p\circ _{x}\left( q\circ _{x}r\right) \right) /_{x}r.
\label{xrprod}
\end{equation}
\end{lemma}

\begin{proof}
Let $r,x\in \mathbb{L}$, and suppose $y=rx.$ Then, by (\ref{rprod}), 
\begin{eqnarray*}
p\cdot qy &=&\left( p\circ _{y}q\right) \cdot y \\
&=&\left( p\circ _{y}q\right) \cdot rx \\
&=&\left( \left( p\circ _{y}q\right) \circ _{x}r\right) \cdot x.
\end{eqnarray*}%
On the other hand, using (\ref{rprod}) in a different way, we get 
\begin{eqnarray*}
p\cdot qy &=&p\cdot q\left( rx\right) \\
&=&p\cdot \left( \left( q\circ _{x}r\right) x\right) \\
&=&\left( p\circ _{x}\left( q\circ _{x}r\right) \right) \cdot x
\end{eqnarray*}%
Hence, 
\begin{equation*}
\left( p\circ _{y}q\right) \circ _{x}r=p\circ _{x}\left( q\circ _{x}r\right)
.
\end{equation*}%
Dividing by $r$ on the right using $/_{x}$ gives (\ref{xrprod}).
\end{proof}

\begin{remark}
Lemma \ref{lemxrprod} shows that the $rx$-product is equivalent to the $r$%
-product, \emph{but defined on }$\left( \mathbb{L},\circ _{x}\right) .$ That
is, if we start with $\circ _{x}$ define the $r$-product using $\circ _{x}$,
then we obtain the $rx$-product \emph{on }$\left( \mathbb{L},\cdot \right) $%
. If $x\in \mathcal{C}^{R}\left( \mathbb{L},\cdot \right) $, then $\left( 
\mathbb{L},\circ _{x}\right) $ is isomorphic to $\left( \mathbb{L},\cdot
\right) $. Similarly, if $r\in \mathcal{C}^{R}\left( \mathbb{L},\circ
_{x}\right) $, then $\left( \mathbb{L},\circ _{rx}\right) $ is isomorphic to 
$\left( \mathbb{L},\circ _{x}\right) .$
\end{remark}

On $\left( \mathbb{L},\circ _{x}\right) $ we can define the associator and
commutator. Given $p,q,r\in \mathbb{L}$, the \emph{loop associator }on\emph{%
\ } $\left( \mathbb{L},\circ _{x}\right) $ is given by 
\begin{equation}
\left[ p,q,r\right] ^{\left( x\right) }=\left( p\circ _{rx}q\right)
/_{x}\left( p\circ _{x}q\right) .  \label{loopassoc2}
\end{equation}%
The \emph{loop commutator }on $\left( \mathbb{L},\circ _{x}\right) $ is
given by 
\begin{equation}
\left[ p,q\right] ^{\left( x\right) }=\left( \left( p\circ _{x}q\right)
/_{x}p\right) /_{x}q.  \label{loopcomm2}
\end{equation}%
For any $x\in \mathbb{L}$, the adjoint map $\func{Ad}^{\left( x\right) }:$ $%
\mathbb{L\times L}\longrightarrow \mathbb{L}$ with respect to $\circ _{x}$
is given by 
\begin{equation}
\func{Ad}_{p}^{\left( x\right) }\left( q\right) =\left( \left( R_{p}^{\left(
x\right) }\right) ^{-1}\circ L_{p}^{\left( x\right) }\right) q=\left( p\circ
_{x}q\right) /_{x}p  \label{Adpx}
\end{equation}%
for any $p,q\in \mathbb{L},$ and its inverse for a fixed $p$ is%
\begin{equation}
\left( \func{Ad}_{p}^{\left( x\right) }\right) ^{-1}\left( q\right) =\left(
\left( L_{p}^{\left( x\right) }\right) ^{-1}\circ R_{p}^{\left( x\right)
}\right) q=p\backslash _{x}\left( q\circ _{x}p\right) .
\end{equation}

Let us now consider how pseudoautomorphisms of $\left( \mathbb{L},\cdot
\right) $ act on $\left( \mathbb{L},\circ _{r}\right) $.

\begin{lemma}
\label{lemPseudoHom}Let $h=\left( \beta ,B\right) \in \Psi ^{R}\left( 
\mathbb{L},\cdot \right) $. Then, for any $p,q,r\in \mathbb{L},$ 
\begin{equation}
\beta \left( p\circ _{r}q\right) =\beta \left( p\right) \circ _{h\left(
r\right) }\beta \left( q\right)  \label{PsiActcircr}
\end{equation}%
and $\beta $ is a right pseudoautomorphism of $\left( \mathbb{L},\circ
_{r}\right) $ with companion $h\left( r\right) /r$. It also follows that 
\begin{equation}
\beta \left( p/_{r}q\right) =\beta \left( p\right) /_{h\left( r\right)
}\beta \left( q\right) .  \label{PsiActQuot}
\end{equation}
\end{lemma}

\begin{proof}
Consider $\beta \left( p\circ _{r}q\right) $. Then, using (\ref{PsAutprod})
and (\ref{PsAutquot2}), 
\begin{eqnarray*}
\beta \left( p\circ _{r}q\right) &=&\beta \left( \left( p\cdot qr\right)
/r\right) \\
&=&h\left( p\cdot qr\right) /h\left( r\right) \\
&=&\left( \beta \left( p\right) \cdot h\left( qr\right) \right) /h\left(
r\right) \\
&=&\left( \beta \left( p\right) \cdot \beta \left( q\right) h\left( r\right)
\right) /h\left( r\right) \\
&=&\beta \left( p\right) \circ _{h\left( r\right) }\beta \left( q\right) ,
\end{eqnarray*}%
and hence we get (\ref{PsiActcircr}). Alternatively, using (\ref{rprodqright}%
), 
\begin{eqnarray*}
\beta \left( p\circ _{r}q\right) &=&\faktor{\left( \beta \left( p\right)
\cdot \beta \left( q\right) h\left( r\right) \right)}{h\left( r\right)} \\
&=&\left(\faktor{\left( \beta \left( p\right) \cdot \beta \left( q\right)
h\left( r\right) \right)}{r}\right) /_{r}\left(\faktor{ h\left( r\right)}{r}%
\right) .
\end{eqnarray*}%
Now, let $C=h\left( r\right) /r$. Thus, 
\begin{eqnarray*}
\beta \left( p\circ _{r}q\right) &=&\left( \faktor{\left( \beta \left(
p\right) \left( \beta \left( q\right) \cdot Cr\right) \right)}{r}\right)
/_{r}C \\
&=&\left( \beta \left( p\right) \circ _{r}\left( \beta \left( q\right) \circ
_{r}C\right) \right) /_{r}C
\end{eqnarray*}%
Thus, indeed, $\beta $ is a right pseudoautomorphism of $\left( \mathbb{L}%
,\circ _{r}\right) $ with companion $C=h\left( r\right) /r$.

Now using (\ref{PsiActcircr}) with $p/_{r}q$ and $q$, we find 
\begin{equation*}
\beta \left( p\right) =\beta \left( p/_{r}q\circ _{r}q\right) =\beta \left(
p/_{r}q\right) \circ _{h\left( r\right) }\beta \left( q\right)
\end{equation*}%
and hence we get (\ref{PsiActQuot}).
\end{proof}

\begin{remark}
We will use the notation $\left( \beta ,C\right) _{r}$ to denote that $%
\left( \beta ,C\right) _{r}$ is considered as a pseudoautomorphism pair on $%
\left( \mathbb{L},\circ _{r}\right) $, i.e. $\left( \beta ,C\right) _{r}\in
\Psi ^{R}\left( \mathbb{L},\circ _{r}\right) $. Of course, the product of $C$
with any element in $\mathcal{N}^{R}\left( \mathbb{L},\circ _{r}\right) $ on
the right will also give a companion of $\beta $ on $\left( \mathbb{L},\circ
_{r}\right) $. Any right pseudoautomorphism of $\left( \mathbb{L},\cdot
\right) $ is also a right pseudoautomorphism of $\left( \mathbb{L},\circ
_{r}\right) $, however their companions may be different. In particular, $%
\func{PsAut}^{R}\left( \mathbb{L},\cdot \right) =\func{PsAut}^{R}\left( 
\mathbb{L},\circ _{r}\right) $. For $\Psi ^{R}\left( \mathbb{L},\cdot
\right) $ and $\Psi ^{R}\left( \mathbb{L},\circ _{r}\right) $ we have a
group isomorphism 
\begin{eqnarray}
\Psi ^{R}\left( \mathbb{L},\cdot \right) &\longrightarrow &\Psi ^{R}\left( 
\mathbb{L},\circ _{r}\right)  \notag \\
h &=&\left( \beta ,B\right) \mapsto h_{r}=\left( \beta ,\faktor{h\left(
r\right)} {r}\right) _{r}.  \label{PsAutoriso}
\end{eqnarray}
Conversely, if we have $h_{r}=\left( \beta ,C\right) _{r}\in \Psi ^{R}\left( 
\mathbb{L},\circ _{r}\right) $, then this corresponds to $h=\left( \beta
,B\right) \in \Psi ^{R}\left( \mathbb{L},\cdot \right) $ where 
\begin{equation}
B=\beta \left( r\right) \backslash \left( Cr\right) .  \label{PsAutorisorev}
\end{equation}
\end{remark}

The group isomorphism (\ref{PsAutoriso}) together with $R_{r}^{-1}$ (right
division by $r$) induces a $G$-set isomorphism between $\left( \mathbb{%
\mathring{L}},\cdot \right) $with the action of $\Psi ^{R}\left( \mathbb{L}%
,\cdot \right) $ and $\left( \mathbb{\mathring{L}},\circ _{r}\right) $ with
the action of $\Psi ^{R}\left( \mathbb{L},\circ _{r}\right) $.

\begin{lemma}
Let $r\in \mathbb{L}$, then the mapping (\ref{PsAutoriso}) $h\mapsto h_{r}$
from $\Psi ^{R}\left( \mathbb{L},\cdot \right) \ $to $\Psi ^{R}\left( 
\mathbb{L},\circ _{r}\right) $ together with the map $R_{r}^{-1}:\left( 
\mathbb{\mathring{L}},\cdot \right) \longrightarrow \left( \mathbb{\mathring{%
L}},\circ _{r}\right) $ gives a $G$-set isomorphism. In particular, for any $%
A\in \mathbb{\mathring{L}}$ and $h\in \Psi ^{R}\left( \mathbb{L},\cdot
\right) ,$ 
\begin{equation}
h\left( A\right) /r=h_{r}\left( A/r\right) .  \label{Gsetiso}
\end{equation}
\end{lemma}

\begin{proof}
Suppose $h=\left( \beta ,B\right) $ and correspondingly, from (\ref%
{PsAutoriso}), $h_{r}=\left( \beta ,\faktor{h\left( r\right)} {r}\right) $.
Then, we have, 
\begin{eqnarray*}
h_{r}\left( A/r\right) &=&\beta \left( A/r\right) \circ _{r}\faktor{h\left(
r\right)} {r} \\
&=&\faktor{\left( h\left( A\right) /h\left( r\right) \cdot h\left( r\right)
\right)}{ r } \\
&=&h\left( A\right) /r,
\end{eqnarray*}%
where we have also used (\ref{PsAutquot2a}).
\end{proof}

Using (\ref{PsAutoriso}), we now have the following characterizations of $%
\mathcal{C}^{R}\left( \mathbb{L},\circ _{r}\right) ,$ $\mathcal{N}^{R}\left( 
\mathbb{L},\circ _{r}\right) $, and $\func{Aut}\left( \mathbb{L},\circ
_{r}\right) $.

\begin{lemma}
Let $r,C\in \mathbb{L}$, then 
\begin{subequations}
\begin{eqnarray}
C &\in &\mathcal{C}^{R}\left( \mathbb{L},\circ _{r}\right) \iff C=A/r\ \text{%
for some }A\in \func{Orb}_{\Psi ^{R}\left( \mathbb{L},\cdot \right) }\left(
r\right)  \label{CRrdef} \\
C &\in &\mathcal{N}^{R}\left( \mathbb{L},\circ _{r}\right) \iff C=\func{Ad}%
_{r}\left( A\right) \ \text{for some }A\in \mathcal{N}^{R}\left( \mathbb{L}%
,\cdot \right)  \label{CRNucl}
\end{eqnarray}%
\end{subequations}%
and 
\begin{equation}
\func{Aut}\left( \mathbb{L},\circ _{r}\right) \cong \func{Stab}_{\Psi
^{R}\left( \mathbb{L},\cdot \right) }\left( r\right) .  \label{AutLr}
\end{equation}%
If $r\in \mathcal{C}^{R}\left( \mathbb{L},\cdot \right) $, so that there
exists a right pseudoautomorphism pair $h=\left( \alpha ,r\right) \in \Psi
^{R}\left( \mathbb{L},\cdot \right) $, then $\func{Aut}\left( \mathbb{L}%
,\circ _{r}\right) \cong h\func{Aut}\left( \mathbb{L},\cdot \right) h^{-1}.$
\end{lemma}

\begin{proof}
From (\ref{PsAutoriso}) we see that any companion in $\left( \mathbb{L}%
,\circ _{r}\right) $ is of the form $h\left( r\right) /r$ for some $h\in
\Psi ^{R}\left( \mathbb{L},\cdot \right) $. Therefore, $C\in \mathbb{L}$ is
a companion in $\left( \mathbb{L},\circ _{r}\right) $ if and only if it is
of the form $C=A/r\ $for some $A\in \func{Orb}_{\Psi ^{R}\left( \mathbb{L}%
,\cdot \right) }\left( r\right) .$

The right nucleus $\mathcal{N}^{R}\left( \mathbb{L},\circ _{r}\right) $
corresponds to the companions of the identity map $\func{id}$ on $\mathbb{L}$%
, hence taking $\beta =\func{id}$ in (\ref{PsAutoriso}), we find that
companions of $\func{id}$ in $\left( \mathbb{L},\circ _{r}\right) $ must be
of the form $C=\left( rA\right) /r=\func{Ad}_{r}\left( A\right) \ $for some $%
A\in \mathcal{N}^{R}\left( \mathbb{L},\cdot \right) $. Conversely, suppose $%
C=\left( rA\right) /r\ $for some $A\in \mathcal{N}^{R}\left( \mathbb{L}%
,\cdot \right) $, then we can explicitly check that for any $p,q\in \mathbb{L%
}$, we have 
\begin{eqnarray*}
\left( p\circ _{r}q\right) \circ _{r}C &=&\left( \left( p\cdot qr\right)
/r\cdot rA\right) /r \\
&=&\left( \left( p\cdot qr\right) \cdot A\right) /r \\
&=&\left( p\cdot \left( qr\cdot A\right) \right) /r=\left( p\cdot \left(
q\cdot rA\right) \right) /r \\
&=&\left( p\cdot \left( q\cdot Cr\right) \right) /r=\left( p\cdot \left(
q\circ _{r}C\right) r\right) /r \\
&=&p\circ _{r}\left( q\circ _{r}C\right)
\end{eqnarray*}%
and hence, $C\in \mathcal{N}^{R}\left( \mathbb{L},\circ _{r}\right) $.

The group $\func{Aut}\left( \mathbb{L},\circ _{r}\right) $ is isomorphic to
the preimage $\func{prj}_{2}^{-1}\left( 1\right) $ with respect to the
projection map $\func{prj}_{2}:$ $\Psi ^{R}\left( \mathbb{L},\circ
_{r}\right) \longrightarrow \mathcal{C}^{R}\left( \mathbb{L},\circ
_{r}\right) $. From (\ref{PsAutoriso}), this corresponds precisely to the
maps $h\in \Psi ^{R}\left( \mathbb{L},\cdot \right) $ for which $h\left(
r\right) =r$. If $r$ is in the $\Psi ^{R}\left( \mathbb{L},\cdot \right) $%
-orbit of $1$, then clearly $\func{Aut}\left( \mathbb{L},\circ _{r}\right) $
is conjugate to $\func{Aut}\left( \mathbb{L},\cdot \right) .$
\end{proof}

\begin{remark}
Suppose $r\in \mathcal{C}^{R}\left( \mathbb{L}\right) $, then from (\ref%
{CRrdef}), we see that if $A\in \mathcal{C}^{R}\left( \mathbb{L},\circ
_{r}\right) $, then $Ar\in \mathcal{C}^{R}\left( \mathbb{L}\right) .$ Also,
using the isomorphism (\ref{PsAutoriso}), we can define the left action of $%
\Psi ^{R}\left( \mathbb{L},\circ _{r}\right) $ on $\Psi ^{R}\left( \mathbb{L}%
,\cdot \right) $ just by composition on the left by the corresponding
element in $\Psi ^{R}\left( \mathbb{L},\cdot \right) $. Now recall that 
\begin{equation*}
\mathcal{C}^{R}\left( \mathbb{L},\circ _{r}\right) \cong \faktor{\Psi
^{R}\left( \mathbb{L},\circ _{r}\right)}{\func{Aut}\left( \mathbb{L},\circ
_{r}\right)}\ \text{and\ }\mathcal{C}^{R}\left( \mathbb{L}\right) \cong %
\faktor{\Psi ^{R}\left( \mathbb{L},\cdot \right)}{\func{Aut}\left(
\mathbb{L},\cdot \right)}.
\end{equation*}%
Then, for any equivalence classes $\left\lfloor \alpha ,A\right\rfloor
_{r}\in 
\faktor{\Psi ^{R}\left( \mathbb{L},\circ _{r}\right)}{\func{Aut}\left(
\mathbb{L},\circ _{r}\right)}$ and $\left\lfloor \beta ,r\right\rfloor \in 
\faktor{\Psi
^{R}\left( \mathbb{L},\cdot \right)}{\func{Aut}\left( \mathbb{L},\cdot
\right)}$, we find that 
\begin{equation}
\left\lfloor \alpha ,A\right\rfloor _{r}\cdot \left\lfloor \beta
,r\right\rfloor =\left\lfloor \alpha \circ \beta ,Ar\right\rfloor .
\label{CCaction}
\end{equation}%
Another way to see this is the following. From (\ref{PsAutorisorev}), the
element in $\Psi ^{R}\left( \mathbb{L},\cdot \right) $ that corresponds to $%
\left( \alpha ,A\right) _{r}\in \Psi ^{R}\left( \mathbb{L},\circ _{r}\right) 
$ is $\left( \alpha ,%
\scalebox{-1}[1]{\nicefrac{\scalebox{-1}[1]{$
Ar$}}{\scalebox{-1}[1]{$\alpha \left( r\right)$}}}\right) .$ The composition
of this with $\left( \beta ,r\right) $ is then $\left( \alpha \circ \beta
,Ar\right) .$ Then, it is easy to see that this reduces to cosets.
\end{remark}

\begin{example}
\label{exMouf}Recall that in a Moufang loop $\mathbb{L}$, the map $\func{Ad}%
_{q}$ is a right pseudoautomorphism with companion $q^{3}$. The relation (%
\ref{CCaction}) then shows that for any $r\in \mathbb{L},$ 
\begin{equation}
\func{Ad}_{q}^{\left( r^{3}\right) }\circ \func{Ad}_{r}=\func{Ad}_{\left(
q^{3}r^{3}\right) ^{\frac{1}{3}}}\circ h
\end{equation}%
where $h\in \func{Aut}\left( \mathbb{L}\right) $. This follows because $%
\func{Ad}_{q}^{\left( r^{3}\right) }$ has companion $q^{3}$ in $\Psi
^{R}\left( \mathbb{L},\circ _{r^{3}}\right) $ and $\func{Ad}_{r}$ has
companion $r^{3}$ in $\Psi ^{R}\left( \mathbb{L}\right) $, thus the
composition has companion $q^{3}r^{3}$, so up to composition with $\func{Aut}%
\left( \mathbb{L}\right) ,$ it is given by $\func{Ad}_{\left(
q^{3}r^{3}\right) ^{\frac{1}{3}}}.$ A similar expression for octonions has
been derived in \cite{GrigorianOctobundle}.
\end{example}

As we have seen, $\Psi ^{R}\left( \mathbb{L}\right) $ acts transitively on $%
\mathcal{C}^{R}\left( \mathbb{L}\right) $ and moreover, for each $r\in 
\mathcal{C}^{R}\left( \mathbb{L}\right) $, the loops $\left( \mathbb{L}%
,\circ _{r}\right) $ are all isomorphic to one another, and related via
elements of $\Psi ^{R}\left( \mathbb{L}\right) $. Concretely, consider $%
\left( \mathbb{L},\circ _{r}\right) $ and suppose $h=\left( \alpha ,A\right)
\in \Psi ^{R}\left( \mathbb{L}\right) $. Then, define the map 
\begin{equation*}
h:\left( \mathbb{L},\circ _{r}\right) \longrightarrow \left( \mathbb{L}%
,\circ _{h\left( r\right) }\right) ,
\end{equation*}%
where $h$ acts on $\mathbb{L}$ via the partial action (i.e. via $\alpha $).
Indeed, from (\ref{alpharcirc}), we have for $p,q\in h\left( \mathbb{L}%
\right) $ 
\begin{equation}
\alpha \left( \alpha ^{-1}\left( p\right) \circ _{r}\alpha ^{-1}\left(
q\right) \right) =p\circ _{h\left( r\right) }q.  \label{alphaprod}
\end{equation}%
Moreover, if we instead consider the action of $\Psi ^{R}\left( \mathbb{L}%
,\circ _{r}\right) ,$ then given $h_{r}=\left( \alpha ,x\right) _{r}\in \Psi
^{R}\left( \mathbb{L},\circ _{r}\right) $, $h_{r}\left( \mathbb{L}\right)
\cong \left( \mathbb{L},\circ _{xr}\right) .$ This is summarized in the
theorem below.

\begin{theorem}
\label{thmLeftProd}Let $\mathbb{L}$ be a loop with the set of right
companions $\mathcal{C}^{R}\left( \mathbb{L}\right) .$ For every $r\in 
\mathcal{C}^{R}\left( \mathbb{L}\right) $ and every $h\in \Psi ^{R}\left( 
\mathbb{L}\right) $, the loop $\left( \mathbb{L},\circ _{r}\right) $ is
isomorphic to $\left( \mathbb{L},\circ _{h\left( r\right) }\right) .$
Moreover, if instead, the action of $\Psi ^{R}\left( \mathbb{L},\circ
_{r}\right) $ is considered, then an element of $\Psi ^{R}\left( \mathbb{L}%
,\circ _{r}\right) $ with companion $x$ induces a loop isomorphism from $%
\left( \mathbb{L},\circ _{r}\right) $ to $\left( \mathbb{L},\circ
_{xr}\right) .$
\end{theorem}

Now again, let $h=\left( \alpha ,A\right) \in \Psi ^{R}\left( \mathbb{L}%
\right) $, and we will consider the action of $h$ on the nucleus. It is easy
to see how the loop associator transforms under this map. Using (\ref%
{loopassoc2}) and (\ref{PsiActQuot}), we have%
\begin{eqnarray}
\alpha \left( \left[ p,q,r\right] ^{\left( x\right) }\right) &=&\alpha
\left( \left( p\circ _{rx}q\right) /_{x}\left( p\circ _{x}q\right) \right) 
\notag \\
&=&\left( \alpha \left( p\right) \circ _{\alpha \left( r\right) h\left(
x\right) }\alpha \left( q\right) \right) /_{h\left( x\right) }\left( \alpha
\left( p\right) \circ _{h\left( x\right) }\alpha \left( q\right) \right) 
\notag \\
&=&\left[ \alpha \left( p\right) ,\alpha \left( q\right) ,\alpha \left(
r\right) \right] ^{\left( h\left( x\right) \right) }.  \label{alphaassoc}
\end{eqnarray}%
So in particular, taking $x=1$, $C\in \mathcal{N}^{R}\left( \mathbb{L}%
\right) $ if and only if $\alpha \left( C\right) \in \mathcal{N}^{R}\left( 
\mathbb{L},\circ _{A}\right) .$ However from (\ref{CRNucl}), we know that $%
C\in \mathcal{N}^{R}\left( \mathbb{L}\right) $ if and only if $\left( \func{%
Ad}_{A}\right) C\in \mathcal{N}^{R}\left( \mathbb{L},\circ _{A}\right) .$ In
particular, this means that $C\in \mathcal{N}^{R}\left( \mathbb{L}\right) $
if and only if $\alpha ^{-1}\left( \func{Ad}_{A}C\right) \in \mathcal{N}%
^{R}\left( \mathbb{L}\right) .$ This defines a left action of $\Psi
^{R}\left( \mathbb{L}\right) $ on $\mathcal{N}^{R}\left( \mathbb{L}\right) $%
: 
\begin{equation}
h^{\prime \prime }\left( C\right) =\func{Ad}_{A}^{-1}\left( \alpha \left(
C\right) \right) =\scalebox{-1}[1]{\nicefrac{\scalebox{-1}[1]{$ h\left(
C\right)$}}{\scalebox{-1}[1]{$A$}}}  \label{nuclearaction}
\end{equation}%
for $h=\left( \alpha ,A\right) \in \Psi ^{R}\left( \mathbb{L}\right) $ and $%
C\in \mathcal{N}^{R}\left( \mathbb{L}\right) .$ The action (\ref%
{nuclearaction}) can be seen from the following considerations. Recall $%
\mathcal{N}^{R}\left( \mathbb{L}\right) ^{\func{op}}$ embeds in $\Psi
^{R}\left( \mathbb{L}\right) $ via the map $C\mapsto \iota _{2}\left(
C\right) =\left( \func{id},C\right) .$ The group $\Psi ^{R}\left( \mathbb{L}%
\right) $ acts on itself via the adjoint action, so let $h=\left( \alpha
,A\right) \in \Psi ^{R}\left( \mathbb{L}\right) $, then from (\ref{PsiAdjN})
recall, 
\begin{equation}
h\left( \iota _{2}\left( C\right) \right) h^{-1}=\left( \alpha ,h\left(
C\right) \right) h^{-1}=\left( \func{id},A^{\lambda }\cdot h\left( C\right)
\right) .
\end{equation}%
On the other hand, suppose 
\begin{equation*}
\left( \alpha ,h\left( C\right) \right) h^{-1}=\left( \func{id},x\right) ,
\end{equation*}%
so that 
\begin{equation*}
\left( \alpha ,h\left( C\right) \right) =\left( \func{id},x\right) \left(
\alpha ,A\right) =\left( \alpha ,Ax\right)
\end{equation*}%
Therefore, $x=A\backslash h\left( C\right) .$ In particular, $A\backslash
h\left( C\right) \in \mathcal{N}^{R}\left( \mathbb{L}\right) .$ Thus the
induced action on $\mathcal{N}^{R}\left( \mathbb{L}\right) $ is precisely $%
C\mapsto A\backslash h\left( C\right) =\func{Ad}_{A}^{-1}\left( \alpha
\left( C\right) \right) $. Moreover, right multiplication of elements in $%
\mathbb{\mathring{L}}$ by elements of $\mathcal{N}^{R}\left( \mathbb{L}%
\right) $ is compatible with the corresponding actions of $\Psi ^{R}\left( 
\mathbb{L}\right) $.

\begin{lemma}
For any $s\in \mathbb{\mathring{L}},C\in \mathcal{N}^{R}\left( \mathbb{L}%
\right) $, and $h\in \Psi ^{R}\left( \mathbb{L}\right) $, we have 
\begin{equation}
h\left( sC\right) =h\left( s\right) h^{\prime \prime }\left( C\right) ,
\label{nuclearaction1}
\end{equation}%
where $h^{\prime \prime }$ is the action (\ref{nuclearaction}).
\end{lemma}

\begin{proof}
Indeed, to show (\ref{nuclearaction1}), we have 
\begin{eqnarray*}
h\left( sC\right) &=&\alpha \left( s\right) h\left( C\right) \\
&=&h\left( s\right) /A\cdot Ah^{\prime \prime }\left( C\right) \\
&=&\left( h\left( s\right) /A\cdot A\right) h^{\prime \prime }\left( C\right)
\\
&=&h\left( s\right) \cdot h^{\prime \prime }\left( C\right) ,
\end{eqnarray*}%
since $h^{\prime \prime }\left( C\right) \in \mathcal{N}^{R}\left( \mathbb{L}%
\right) .$
\end{proof}

\section{Smooth loops}

\setcounter{equation}{0}\label{sectSmooth}Suppose the loop $\mathbb{L}$ is a
smooth finite-dimensional manifold such that the loop multiplication and
division are smooth functions. Define maps

\begin{equation}
\begin{array}{c}
L_{r}:\mathbb{L}\longrightarrow \mathbb{L} \\ 
q\longmapsto rq%
\end{array}
\label{lprod}
\end{equation}%
and 
\begin{equation}
\begin{array}{c}
R_{r}:\mathbb{L}\longrightarrow \mathbb{L} \\ 
q\longmapsto qr.%
\end{array}
\label{rprod0}
\end{equation}%
These are diffeomorphisms of $\mathbb{L}$ with smooth inverses $L_{r}^{-1}$
and $R_{r}^{-1}$ that correspond to left division and right division by $r$,
respectively. Also, assume that $\Psi ^{R}\left( \mathbb{L}\right) $ acts
smoothly on $\mathbb{L}$ (as before, $\mathbb{L}$ together with the full
action of $\Psi ^{R}\left( \mathbb{L}\right) $ will be denoted by $\mathbb{%
\mathring{L}}$). Thus, the action of $\Psi ^{R}\left( \mathbb{L}\right) $ is
a group homomorphism from $\Psi ^{R}\left( \mathbb{L}\right) $ to $\func{Diff%
}\left( \mathbb{L}\right) .$ In particular, this allows to induce a Lie
group structure on $\Psi ^{R}\left( \mathbb{L}\right) .$ Similarly, $\func{%
PsAut}^{R}\left( \mathbb{L}\right) $ is then also a Lie group, and for any $%
s\in \mathbb{\mathring{L}}$, $\func{Aut}\left( \mathbb{L},\circ _{s}\right)
\cong \func{Stab}_{\Psi ^{R}\left( \mathbb{L}\right) }\left( s\right) $ is
then a Lie subgroup of $\Psi ^{R}\left( \mathbb{L}\right) $, and indeed of $%
\func{PsAut}^{R}\left( \mathbb{L}\right) $ as well. The assumption that
pseudoautomorphisms acts smoothly on $\mathbb{L}$ may be nontrivial. To the
best of the author's knowledge, it is an open question whether this is
always true. However, for the loop $U\mathbb{O}$ of unit octonions, this is
indeed true, as can be seen from Example \ref{ExPsOcto}.

Define $X$ to be a \emph{right fundamental vector field}\textbf{\ }if for
any $q\in \mathbb{L},$ it is determined by a tangent vector at $1$ via right
translations. That is, given a tangent vector $\xi \in T_{1}\mathbb{L}$, we
define a corresponding right fundamental vector field $\rho \left( \xi
\right) $ given by 
\begin{equation}
\rho \left( \xi \right) _{q}=\left( R_{q}\right) _{\ast }\xi
\end{equation}%
at any $p\in \mathbb{L}$. If $\mathbb{L}$ is a Lie group, then this
definition is equivalent to the standard definition of a right-invariant
vector field $X$ such that $\left( R_{q}\right) _{\ast }X_{p}=X_{pq}$,
however in the non-associative case, $R_{q}\circ R_{p}\neq R_{pq},$ so the
standard definition wouldn't work, so a right fundamental vector field is
not actually right-invariant in the usual sense. We can still say that the
vector space of right fundamental vector fields has dimension $\dim \mathbb{L%
}$, and at any point, they still form a basis for the tangent space. In
particular, any smooth loop is parallelizable. However this vector space is
now in general not a Lie algebra under the Lie bracket of vector fields,
which is to be expected, since $T_{1}\mathbb{L}$ doesn't necessarily have
the Lie algebra structure either.

Instead of right invariance, we see that given a right fundamental vector
field $X=\rho \left( \xi \right) $, 
\begin{eqnarray}
\left( R_{p}^{-1}\right) _{\ast }X_{q} &=&\left( R_{p}^{-1}\circ
R_{q}\right) _{\ast }\xi  \notag \\
&=&\left( R_{q/p}^{\left( p\right) }\right) _{\ast }\xi  \label{rightvect}
\end{eqnarray}%
where $R^{\left( p\right) }$ is the right product with respect to the
operation $\circ _{p}.$ This is because 
\begin{eqnarray}
\left( R_{p}^{-1}\circ R_{q}\right) r &=&\left( rq\right) /p  \notag \\
&=&\left( r\cdot \left( q/p\cdot p\right) \right) /p  \notag \\
&=&r\circ _{p}\left( q/p\right) =R_{q/p}^{\left( p\right) }r,  \label{RinvR}
\end{eqnarray}%
where we have used (\ref{rprod}).

\subsection{Exponential map}

\label{secExpMap}For some $\xi \in T_{1}\mathbb{L},$ define a flow $p_{\xi }$
on $\mathbb{L}$ given by 
\begin{equation}
\left\{ 
\begin{array}{c}
\frac{dp_{\xi }\left( t\right) }{dt}=\left( R_{p_{\xi }\left( t\right)
}\right) _{\ast }\xi \\ 
p_{\xi }\left( 0\right) =1%
\end{array}%
\right.  \label{floweq}
\end{equation}%
This generally has a solution for some sufficiently small time interval $%
\left( -\varepsilon ,\varepsilon \right) $, and is only a local 1-parameter
subgroup. However it is shown in \cite{Kuzmin1971,Malcev1955} that if $%
\mathbb{L}$ is at least power-associative, then $p_{\xi }\left( t+s\right)
=p_{\xi }\left( t\right) p_{\xi }\left( s\right) $ for all $t,s$, and hence
the solution can extended for all $t$. The weakest power-associativity
assumption is required in order to be able to define $p_{\xi }\left(
nh\right) =p_{\xi }\left( h\right) ^{n}$ unambiguously.

The solutions of (\ref{floweq}) define the (local) exponential map: $\exp
\left( t\xi \right) :=p_{\xi }\left( t\right) $. The corresponding
diffeomorphisms are then the right translations $R_{\exp \left( t\xi \right)
}$. We will generally only need this locally, so the power-associativity
assumption will not be necessary. Now consider a similar flow but with a
different initial condition: 
\begin{equation}
\left\{ 
\begin{array}{c}
\frac{dp_{\xi ,q}\left( t\right) }{dt}=\left( R_{p_{\xi ,q}\left( t\right)
}\right) _{\ast }\xi \\ 
p_{\xi ,q}\left( 0\right) =q%
\end{array}%
\right.  \label{floweq2}
\end{equation}%
where $q\in \mathbb{L}$. Applying $R_{q}^{-1}$, and setting $\tilde{p}\left(
t\right) =\faktor{p_{\xi ,q}\left( t\right)}{q}$, we obtain 
\begin{equation}
\left\{ 
\begin{array}{c}
\frac{d\tilde{p}\left( t\right) }{dt}=\left( R_{q}^{-1}\circ R_{p_{\xi
,q}\left( t\right) }\right) _{\ast }\xi \\ 
\tilde{p}\left( 0\right) =1%
\end{array}%
\right. .  \label{floweq2a}
\end{equation}%
If $\mathbb{L}$ is associative, then $R_{q}^{-1}\circ R_{p_{\xi ,q}\left(
t\right) }=R_{\left( p_{\xi ,q}\left( t\right) \right) /q},$ and thus $%
\tilde{p}\left( t\right) $ would satisfy (\ref{floweq}), and we could
conclude that $p_{\xi ,q}\left( t\right) =\exp \left( t\xi \right) q.$
However, in the general case, we have (\ref{RinvR}) and hence, $\tilde{p}%
\left( t\right) $ satisfies the following equation%
\begin{equation}
\left\{ 
\begin{array}{c}
\frac{d\tilde{p}\left( t\right) }{dt}=\left( R_{\tilde{p}\left( t\right)
}^{\left( q\right) }\right) _{\ast }\xi \\ 
\tilde{p}\left( 0\right) =1%
\end{array}%
\right. .  \label{floweq3}
\end{equation}%
This is now an integral curve equation for $\xi $ on $\left( \mathbb{L}%
,\circ _{q}\right) $, and hence for sufficiently small $t$ we can define a
local exponential map $\exp _{q}$ for $\left( \mathbb{L},\circ _{q}\right) $:%
\begin{equation}
\tilde{p}\left( t\right) =\exp _{q}\left( t\xi \right) ,  \label{ptildesol}
\end{equation}%
so, that 
\begin{equation}
p_{\xi ,q}\left( t\right) =\exp _{q}\left( t\xi \right) q.  \label{pxiqsol}
\end{equation}%
If $q\in \mathcal{C}^{R}\left( \mathbb{L}\right) $, then $\left( \mathbb{L}%
,\circ _{q}\right) $ is isomorphic to $\mathbb{L}$, so if $\mathbb{L}$ is
power-associative, then so is $\left( \mathbb{L},\circ _{q}\right) $, and
hence, the solutions (\ref{ptildesol}) are defined for all $t.$

Suppose $h=\left( \alpha ,q\right) \in \Psi ^{R}\left( \mathbb{L}\right) ,$
then let $\hat{p}\left( t\right) =\alpha ^{-1}\left( \tilde{p}\left(
t\right) \right) .$ This then satisfies $\hat{p}\left( 0\right) =1$ and 
\begin{equation}
\frac{d\hat{p}\left( t\right) }{dt}=\left( \alpha ^{-1}\right) _{\ast
}\left( R_{\tilde{p}\left( t\right) }^{\left( q\right) }\right) _{\ast }\xi .
\label{dphat1}
\end{equation}%
However, let $r\in \mathbb{L}$ and consider $R_{p}^{\left( q\right) }$:%
\begin{eqnarray*}
R_{p}^{\left( q\right) }r &=&r\circ _{q}p=\alpha \left( \alpha ^{-1}\left(
r\right) \cdot \alpha ^{-1}\left( p\right) \right) \\
&=&\left( \alpha \circ R_{\alpha ^{-1}\left( p\right) }\circ \alpha
^{-1}\right) \left( r\right) .
\end{eqnarray*}%
Thus, 
\begin{equation}
R_{p}^{\left( q\right) }=\alpha \circ R_{\alpha ^{-1}\left( p\right) }\circ
\alpha ^{-1},  \label{Rpqalpha}
\end{equation}%
and hence, (\ref{dphat1}) becomes 
\begin{equation}
\frac{d\hat{p}\left( t\right) }{dt}=\left( R_{\hat{p}\left( t\right)
}\right) _{\ast }\left( \left( \alpha ^{-1}\right) _{\ast }\xi \right) .
\end{equation}%
This shows that $\hat{p}$ is a solution of (\ref{floweq}) with initial
velocity vector $\left( \alpha ^{-1}\right) _{\ast }\xi \in T_{1}\mathbb{L}$%
, and is hence given by $\hat{p}=\exp \left( t\left( \alpha ^{-1}\right)
_{\ast }\xi \right) .$ Comparing with (\ref{ptildesol}) we see that in this
case, 
\begin{equation}
\exp _{q}\left( t\xi \right) =\alpha \left( \exp \left( t\left( \alpha
^{-1}\right) _{\ast }\xi \right) \right) ,  \label{expqtalpha}
\end{equation}%
and hence the solution $p_{\xi ,q}\left( t\right) $ of (\ref{floweq2}) can
be written as 
\begin{equation}
p_{\xi ,q}\left( t\right) =h\left( \exp \left( t\left( \alpha ^{-1}\right)
_{\ast }\xi \right) \right) .  \label{expqtalpha2}
\end{equation}%
We can summarize these findings in the theorem below.

\begin{theorem}
\label{thmLoopflow}Suppose $\mathbb{L}$ is a smooth loop and suppose $q\in 
\mathcal{C}^{R}\left( \mathbb{L}\right) .$ Then, given $\xi \in T_{1}\mathbb{%
L},$ the equation 
\begin{equation}
\left\{ 
\begin{array}{c}
\frac{dp\left( t\right) }{dt}=\left( R_{p\left( t\right) }\right) _{\ast }\xi
\\ 
p\left( 0\right) =q%
\end{array}%
\right.  \label{floweq4}
\end{equation}%
has the solution 
\begin{equation}
p\left( t\right) =\exp _{q}\left( t\xi \right) q
\end{equation}%
for sufficiently small $t$, where 
\begin{equation*}
\exp _{q}\left( t\xi \right) =\alpha \left( \exp \left( t\left( \alpha
^{-1}\right) _{\ast }\xi \right) \right)
\end{equation*}%
where $\alpha $ is a right pseudoautomorphism of $\mathbb{L}$ that has
companion $q$ and $\exp \left( t\xi \right) $ is defined as the solution of (%
\ref{floweq4}) with initial condition $p\left( t\right) =1$. In particular, $%
\xi $ defines a flow $\Phi _{\xi ,t}$, given by 
\begin{equation}
\Phi _{\xi ,t}\left( q\right) =\exp _{q}\left( t\xi \right) q.
\label{flowPhi}
\end{equation}
\end{theorem}

\begin{remark}
The expression (\ref{expqtalpha}) can be made a bit more general. Suppose $%
\mathbb{L}_{1}$ and $\mathbb{L}_{2}$ are two loops and $\alpha :\mathbb{L}%
_{1}\longrightarrow \mathbb{L}_{2}$ is a loop homomorphism. If we suppose $%
\exp _{\left( 1\right) }$ and $\exp _{\left( 2\right) }$ are the exponential
maps on $\mathbb{L}_{1}$ and $\mathbb{L}_{2}$, respectively, then the
following diagram in Figure \ref{loopexp}.
\end{remark}

\begin{center}
\begin{tikzcd}[sep=large] & T_{1}\mathbb{L}_{1} \arrow[r,"\alpha_{*}"]
\arrow[d,"\func{exp}_{(1)}"] & T_{1}\mathbb{L}_{2}
\arrow[d,"\func{exp}_{(2)}"] & \\ & \mathbb{L}_{1} \arrow[r,"\alpha"] &
\mathbb{L}_{2} & 
\end{tikzcd} 
\captionof{figure}{Loop exponential maps.} \label{loopexp}
\end{center}

\begin{remark}
The action of $\Phi _{\xi ,t}$ given by (\ref{flowPhi}) looks like it
depends on $q$, however we easily see that for sufficiently small $t$, $\exp
_{q}\left( t\xi \right) =\exp _{r}\left( t\xi \right) $ whenever $q$ and $r$
are on the same integral curve generated by $\xi $ (equivalently in the same
orbit of $\Phi _{\xi }$). This is consistent with the $1$-parameter subgroup
property $\Phi _{\xi ,t}\left( \Phi _{\xi ,s}\left( q\right) \right) =\Phi
_{\xi ,t+s}\left( q\right) $.

Indeed, consider $r=\exp _{q}\left( s\xi \right) q$ and $\tilde{r}=\exp
_{q}\left( \left( t+s\right) \xi \right) q.$ These are points that lie along
the solution curve of (\ref{floweq4}). On the other hand, consider the
solution of (\ref{floweq4}) at with $p\left( 0\right) =r.$ This is then
given by $\hat{r}=\exp _{r}\left( t\xi \right) r.$ However, clearly by
uniqueness of solutions of ODEs, $\hat{r}=\tilde{r}.$ So now, 
\begin{eqnarray*}
\hat{r} &=&\tilde{r} \\
&=&\exp _{q}\left( \left( t+s\right) \xi \right) q=\left( \exp _{q}\left(
t\xi \right) \circ _{q}\exp _{q}\left( s\xi \right) \right) q \\
&=&\exp _{q}\left( t\xi \right) \left( \exp _{q}\left( s\xi \right) q\right)
\\
&=&\exp _{q}\left( t\xi \right) r
\end{eqnarray*}%
Hence, we conclude that $\exp _{q}\left( t\xi \right) =\exp _{r}\left( t\xi
\right) .$
\end{remark}

\begin{remark}
Suppose $\left( \mathbb{L},\cdot \right) $ power left-alternative, i.e. $%
x^{k}\left( x^{l}q\right) =x^{k+l}q$ for all $x,q\in \mathbb{L}$ and any
integers $k,l$. In particular this also means that $\left( \mathbb{L},\cdot
\right) $ is power-associative and has the left inverse property. In
particular, powers of $x\in \mathbb{L}$ with respect to $\circ _{q}$ are
equal to powers of $x$ with respect to $\cdot $. For any $q\in \mathbb{L}$, $%
\left( \mathbb{L},\circ _{q}\right) $ is then also power left-alternative.
Now the right-hand side of (\ref{floweq3}) can be written as 
\begin{equation}
\left( R_{\tilde{p}\left( t\right) }^{\left( q\right) }\right) _{\ast }\xi
=\left. \frac{d}{ds}\left( r\left( s\right) \circ _{q}\tilde{p}\left(
t\right) \right) \right\vert _{s=0}  \label{floweq3a}
\end{equation}%
where $r\left( s\right) $ is a curve with $r\left( 0\right) =1$ and $%
r^{\prime }\left( 0\right) =\xi $, so we may take $r\left( s\right) =\tilde{p%
}\left( s\right) .$ Suppose there exist integers $n,k$ and a real number $h$%
, such that $t=nh$ and $s=kh$. Then 
\begin{eqnarray*}
\tilde{p}\left( s\right) \circ _{q}\tilde{p}\left( t\right) &=&\tilde{p}%
\left( kh\right) \circ _{q}\tilde{p}\left( nh\right) \\
&=&\left( \tilde{p}\left( h\right) ^{k}\cdot \tilde{p}\left( h\right)
^{n}q\right) /q \\
&=&\tilde{p}\left( h\right) ^{k+n}=\tilde{p}\left( kh\right) \tilde{p}\left(
nh\right) \\
&=&\tilde{p}\left( s\right) \tilde{p}\left( t\right) .
\end{eqnarray*}%
This is independent of $n$ and $k$, and is hence true for any $s,t$. Thus we
find that (\ref{floweq3a}) is equal to the right-hand side of (\ref{floweq}%
), so $\tilde{p}$ actually satisfies the same equation as $p,$ so by
uniqueness of solutions $\tilde{p}=p$. Hence, in this case, $\exp _{q}=\exp $%
. In general however, the exponential map will not be unique and will depend
on the choice of $q.$
\end{remark}

\subsection{Tangent algebra}

\label{secTangent}Suppose $\xi ,\gamma \in T_{1}\mathbb{L}$ and let $X=\rho
\left( \xi \right) $ and $Y=\rho \left( \gamma \right) $ be the
corresponding right fundamental vector fields on $\mathbb{L}$. Then, recall
that the vector field Lie bracket of $X$ and $Y$ is given by 
\begin{equation}
\left[ X,Y\right] _{p}=\left. \frac{d}{dt}\left( \left( \Phi
_{t}^{-1}\right) _{\ast }\left( Y_{\Phi _{t}\left( p\right) }\right) \right)
\right\vert _{t=0},  \label{vecbracket}
\end{equation}%
where $\Phi _{t}=\Phi \left( \xi ,t\right) $ is the flow generated by $X$.
For sufficiently small $t$, we have $\Phi _{t}\left( p\right) =\exp
_{p}\left( t\xi \right) p,$ and thus 
\begin{equation*}
Y_{\Phi _{t}\left( p\right) }=\left( R_{\exp _{p}\left( t\xi \right)
p}\right) _{\ast }\gamma .
\end{equation*}%
Hence 
\begin{equation}
\left( \Phi _{t}^{-1}\right) _{\ast }\left( Y_{\Phi _{t}\left( p\right)
}\right) =\left( L_{\exp _{p}\left( t\xi \right) }^{-1}\circ R_{\exp
_{p}\left( t\xi \right) p}\right) _{\ast }\gamma .  \label{Phinegt}
\end{equation}%
Now right translating back to $T_{1}\mathbb{L}$, we obtain 
\begin{equation}
\left( R_{p}^{-1}\right) _{\ast }\left[ X,Y\right] _{p}=\left. \frac{d}{dt}%
\left( \left( R_{p}^{-1}\circ L_{\exp _{p}\left( t\xi \right) }^{-1}\circ
R_{\exp _{p}\left( t\xi \right) p}\right) _{\ast }\gamma \right) \right\vert
_{t=0}.  \label{Rpbrack0}
\end{equation}%
In general, let $q,x,y\in \mathbb{L},$ then 
\begin{eqnarray*}
\left( R_{p}^{-1}\circ L_{x}^{-1}\circ R_{yp}\right) q &=&\faktor{\left(
x\backslash \left( q\cdot yp\right) \right)} {p} \\
&=&\faktor{\left( x\backslash \left( \left( q\cdot yp\right) /p\cdot
p\right) \right)}{p} \\
&=&x\backslash _{p}\left( q\circ _{p}y\right) \\
&=&\left( \left( L_{x}^{\left( p\right) }\right) ^{-1}\circ R_{y}^{\left(
p\right) }\right) q,
\end{eqnarray*}%
where we have used (\ref{rprodqleft}). Hence (\ref{Rpbrack0}) becomes%
\begin{eqnarray}
\left( R_{p}^{-1}\right) _{\ast }\left[ X,Y\right] _{p} &=&\left. \frac{d}{dt%
}\left( \left( \left( L_{\exp _{p}\left( t\xi \right) }^{\left( p\right)
}\right) ^{-1}\circ R_{\exp _{p}\left( t\xi \right) }^{\left( p\right)
}\right) _{\ast }\gamma \right) \right\vert _{t=0}  \notag \\
&=&\left. \frac{d}{dt}\left( \left( \func{Ad}_{\exp _{p}\left( t\xi \right)
}^{\left( p\right) }\right) _{\ast }^{-1}\gamma \right) \right\vert ;_{t=0} 
\notag \\
&=&-\left. \frac{d}{dt}\left( \left( \func{Ad}_{\exp _{p}\left( t\xi \right)
}^{\left( p\right) }\right) _{\ast }\gamma \right) \right\vert _{t=0}  \notag
\\
&=&-\left. d_{\xi }\left( \func{Ad}^{\left( p\right) }\right) _{\ast
}\right\vert _{1}\left( \gamma \right)  \label{brackdtAd}
\end{eqnarray}%
Here, $\left( \func{Ad}_{x}^{\left( p\right) }\right) _{\ast }$ denotes the
induced adjoint action of $\mathbb{L}$ on $T_{1}\mathbb{L}.$ As remarked
earlier, this is not an action in the sense of group actions. Similarly, as
for Lie groups and Lie algebras, we can also think of $\left( \func{Ad}%
^{\left( p\right) }\right) _{\ast }:\mathbb{L}\longrightarrow \func{End}%
\left( T_{1}\mathbb{L}\right) $, and then (\ref{brackdtAd}) is just the
differential of this map at $1\in \mathbb{L}$ in the direction $\xi \in T_{1}%
\mathbb{L}$. The differential of $\left( \func{Ad}^{\left( p\right) }\right)
_{\ast }$ at an arbitrary point in $\mathbb{L}$ is given in Lemma \ref%
{lemdtAd}. This now allows us to define the tangent adjoint map $\func{ad}%
^{\left( p\right) }$ on $T_{1}\mathbb{L}.$

\begin{definition}
For any $\xi ,\gamma \in T_{1}\mathbb{L},$ the tangent adjoint map $\func{ad}%
_{\xi }^{\left( p\right) }:T_{1}\mathbb{L}\longrightarrow T_{1}\mathbb{L}$
is defined as 
\begin{equation}
\func{ad}_{\xi }^{\left( p\right) }\left( \gamma \right) =\left. d_{\xi
}\left( \func{Ad}^{\left( p\right) }\right) _{\ast }\right\vert _{1}\left(
\gamma \right) =-\left( R_{p}^{-1}\right) _{\ast }\left[ X,Y\right] _{p}.
\label{ladpx}
\end{equation}
\end{definition}

The negative sign in (\ref{ladpx}) is there to be consistent with the
corresponding definitions for Lie groups for right-invariant vector fields.
We then define the $p$-bracket $\left[ \cdot ,\cdot \right] ^{\left(
p\right) }$ on $T_{1}\mathbb{L}$ as 
\begin{equation}
\left[ \xi ,\gamma \right] ^{\left( p\right) }=\func{ad}_{\xi }^{\left(
p\right) }\left( \gamma \right) .  \label{T1Lbrack}
\end{equation}%
From (\ref{ladpx}) it is clear that it's skew-symmetric in $\xi $ and $%
\gamma $. Equivalently, we can say 
\begin{equation}
\left[ \left( R_{p}^{-1}\right) _{\ast }X_{p},\left( R_{p}^{-1}\right)
_{\ast }Y_{p}\right] ^{\left( p\right) }=-\left( R_{p}^{-1}\right) _{\ast }%
\left[ X,Y\right] _{p}.  \label{T1Lbrack2}
\end{equation}

\begin{definition}
The vector space $T_{1}\mathbb{L}$ together with the bracket $\left[ \cdot
,\cdot \right] ^{\left( p\right) }$ is the \emph{tangent algebra }or $%
\mathbb{L}$\emph{-algebra }$\mathfrak{l}^{\left( p\right) }$ of $\left( 
\mathbb{L},\circ _{p}\right) $.
\end{definition}

This is obviously a generalization of a Lie algebra. However, since now
there is a bracket $\left[ \cdot ,\cdot \right] ^{\left( p\right) }$
corresponding to each point $p\in \mathbb{L},$ it does not make sense to try
and express $\left[ \left[ \cdot ,\cdot \right] ^{\left( p\right) },\cdot %
\right] ^{\left( p\right) }$ in terms of Lie brackets of corresponding
vector fields. Hence, the Jacobi identity for $\left[ \cdot ,\cdot \right]
^{\left( p\right) }$ cannot be inferred, as expected. From (\ref{T1Lbrack2}%
), we cannot even infer that the bracket of two right fundamental vector
fields is again a right fundamental vector field. In fact, at each point $p$
it will be the pushforward of the bracket on $T_{1}\mathbb{L}$ with respect
to $p.$ Overall, we can summarize properties of the bracket in the theorem
below.

\begin{theorem}
Let $\xi ,\gamma \in T_{1}\mathbb{L}$ and suppose $X=\rho \left( \xi \right) 
$ and $Y=\rho \left( \gamma \right) $ are the corresponding right
fundamental vector fields on $\mathbb{L}$. Then, for any $p\in \mathbb{L}$, 
\begin{equation}
\left[ \xi ,\gamma \right] ^{\left( p\right) }=\func{ad}_{\xi }^{\left(
p\right) }\left( \gamma \right) =\left. \frac{d}{dt}\left( \left( \func{Ad}%
_{\exp \left( t\xi \right) }^{\left( p\right) }\right) _{\ast }\gamma
\right) \right\vert _{t=0}=-\left( R_{p}^{-1}\right) _{\ast }\left[ X,Y%
\right] _{p},  \label{Rpbrack}
\end{equation}%
and moreover, 
\begin{eqnarray}
\left[ \xi ,\gamma \right] ^{\left( p\right) } &=&\left. \frac{d^{2}}{%
dtd\tau }\left[ \exp \left( t\xi \right) ,\exp \left( \tau \gamma \right) %
\right] ^{\left( \mathbb{L},\circ _{p}\right) }\right\vert _{t,\tau =0} 
\notag \\
&=&\left. \frac{d^{2}}{dtd\tau }\exp \left( t\xi \right) \circ _{p}\exp
\left( \tau \gamma \right) \right\vert _{t,\tau =0}  \label{brack2deriv} \\
&&-\left. \frac{d^{2}}{dtd\tau }\exp \left( \tau \gamma \right) \circ
_{p}\exp \left( t\xi \right) \right\vert _{t,\tau =0}.  \notag
\end{eqnarray}%
Here $\left[ \cdot ,\cdot \right] ^{\left( p\right) }$ is the $\mathbb{L}$%
-algebra bracket on $\mathfrak{l}^{\left( p\right) }$, $\left[ \cdot ,\cdot %
\right] _{p}$ refers to the value of the vector field Lie bracket at $p\in 
\mathbb{L}$, and $\left[ \cdot ,\cdot \right] ^{\left( \mathbb{L},\circ
_{p}\right) }$ is the loop commutator (\ref{loopcomm2}) on $\left( \mathbb{L}%
,\circ _{p}\right) .$
\end{theorem}

\begin{proof}
We have already shown (\ref{Rpbrack}), so let us prove (\ref{brack2deriv}).
Recall from (\ref{loopcomm2}) that 
\begin{equation}
\left[ \exp \left( t\xi \right) ,\exp \left( \tau \gamma \right) \right]
^{\left( \mathbb{L},\circ _{p}\right) }=\func{Ad}_{\exp \left( t\xi \right)
}^{\left( p\right) }\left( \exp \left( \tau \gamma \right) \right) /_{p}\exp
\left( \tau \gamma \right) .  \label{commexp}
\end{equation}%
Differentiating (\ref{commexp}) with respect to $\tau $ and evaluating at $%
\tau =0$ using Lemma \ref{lemQuotient} gives 
\begin{eqnarray}
\left. \frac{d}{d\tau }\left[ \exp \left( t\xi \right) ,\exp \left( \tau
\gamma \right) \right] ^{\left( \mathbb{L},\circ _{p}\right) }\right\vert
_{\tau =0} &=&\left. \frac{d}{d\tau }\func{Ad}_{\exp \left( t\xi \right)
}^{\left( p\right) }\left( \exp \left( \tau \gamma \right) \right)
\right\vert _{\tau =0}  \notag \\
&&-\left. \frac{d}{d\tau }\exp \left( \tau \gamma \right) \right\vert _{\tau
=0}  \notag \\
&=&\left( \func{Ad}_{\exp \left( t\xi \right) }^{\left( p\right) }\right)
_{\ast }\gamma -\tau
\end{eqnarray}%
where we have also used the definition of $\exp _{p}$ via (\ref{floweq3}).
This gives us the first part of (\ref{brack2deriv}). Now, using Lemma \ref%
{lemQuotient} again, we can differentiate $\left( \func{Ad}_{\exp \left(
t\xi \right) }^{\left( p\right) }\right) _{\ast }\gamma $ with respect to $t$
to get the second part:%
\begin{eqnarray*}
\left. \frac{d}{dt}\left( \left( \func{Ad}_{\exp \left( t\xi \right)
}^{\left( p\right) }\right) _{\ast }\gamma \right) \right\vert _{t=0}
&=&\left. \frac{d^{2}}{dtd\tau }\left( \left( \exp \left( t\xi \right) \circ
_{p}\exp \left( \tau \gamma \right) \right) /_{p}\exp \left( t\xi \right)
\right) \right\vert _{t,\tau =0} \\
&=&\left. \frac{d^{2}}{dtd\tau }\left( \exp \left( t\xi \right) \circ
_{p}\exp \left( \tau \gamma \right) \right) \right\vert _{t,\tau =0} \\
&&-\left. \frac{d^{2}}{dtd\tau }\exp \left( \tau \gamma \right) \circ
_{p}\exp \left( t\xi \right) \right\vert _{t,\tau =0}.
\end{eqnarray*}
\end{proof}

\begin{remark}
Applying (\ref{brack2deriv}) to the Moufang loop of unit octonions and the
corresponding $\mathbb{L}$-algebra of imaginary octonions shows that as
expected, the bracket on the $\mathbb{L}$-algebra coincides with the
commutator of imaginary octonions in the algebra of octonions.
\end{remark}

Although $\mathbb{L}$ and $\mathfrak{l}$ are not in general a Lie group and
a Lie algebra, there are analogs of actions of these spaces on one another,
which we summarize below.

Let $s\in \mathbb{\mathring{L}},$ $A\in \mathbb{L}$, and $\xi ,\eta \in 
\mathfrak{l},$ then we have the following:

\begin{enumerate}
\item Action of $\mathbb{L}$ on $\mathbb{\mathring{L}}$: $A\cdot s=As.$

\item Adjoint action of $\left( \mathbb{L},\circ _{s}\right) $ on $\mathbb{L}
$: $A\cdot B=\func{Ad}_{A}^{\left( s\right) }\left( B\right) =\left( A\circ
_{s}B\right) /_{s}A.$

\item Action of $\left( \mathbb{L},\circ _{s}\right) $ on $\mathfrak{l}$: $%
A\cdot \xi =\left( \func{Ad}_{A}^{\left( s\right) }\right) _{\ast }\xi .$

\item Action of $\mathfrak{l}^{\left( s\right) }$ on itself: $\xi \cdot
_{s}\eta =\left[ \xi ,\eta \right] ^{\left( s\right) }.$

\item Action of $\mathfrak{l}$ on $\mathbb{\mathring{L}}$: $\xi \cdot
s=\left( R_{s}\right) _{\ast }\xi =\left. \frac{d}{dt}\exp _{s}\left( t\xi
\right) s\right\vert _{t=0}.$
\end{enumerate}

\begin{remark}
There may be some confusion about notation because we will sometimes
consider the same objects but in different categories. Generally, for the
loop $\mathbb{L}$, the notation \textquotedblleft $\mathbb{L}$%
\textquotedblright\ will denote the underlying set, the underlying smooth
manifold, the loop, and the $G$-set with the partial action of $\Psi
^{R}\left( \mathbb{L}\right) .$ Similarly, $\mathbb{\mathring{L}}$ will
denote the same underlying set, the same underlying smooth manifold, but
will be different as a $G$-set - it has the full action of $\Psi ^{R}\left( 
\mathbb{L}\right) .$ Since $\mathbb{L}$ and $\mathbb{\mathring{L}}$ are
identical as smooth manifolds, they have the same tangent space at $1$.
Generally, we will only refer to $\mathbb{\mathring{L}}$ if we need to
emphasize the group action. For the $\mathbb{L}$-algebra, the notation
\textquotedblleft $\mathfrak{l}$\textquotedblright\ will denote both the
underlying vector space, and the vector space with the algebra structure on $%
T_{1}\mathbb{L}$ induced from the loop $\mathbb{L}.$ For different values of 
$p\in \mathbb{L}$, $\mathfrak{l}^{\left( p\right) }$ is identical to $%
\mathfrak{l}$ as a vector space, but has a different algebra structure. We
will use the notation $\mathfrak{l}^{\left( p\right) }$ to emphasize the
algebra structure.
\end{remark}

\subsection{Structural equation}

\label{sectStruct}Let us now define an analog of the Maurer-Cartan form on
right fundamental vector fields. Given $p\in \mathbb{L}$ and and $\xi \in 
\mathfrak{l},$ define $\theta _{p}$ to be 
\begin{equation}
\theta _{p}\left( \rho \left( \xi \right) _{p}\right) =\left(
R_{p}^{-1}\right) _{\ast }\rho \left( \xi \right) _{p}=\xi .  \label{MCloop}
\end{equation}%
Thus, this is an $\mathfrak{l}$-valued $1$-form. The right fundamental
vector fields still form a global frame for $T\mathbb{L},$ so this is
sufficient to define the $1$-form $\theta .$ Just as the right fundamental
vector field $\rho \left( \xi \right) $ is generally not right-invariant,
neither is $\theta .$ Indeed, let $q\in \mathbb{L}$ and consider $\left(
R_{q}^{-1}\right) ^{\ast }\theta .$ Then, given $X_{p}=\left( R_{p}\right)
_{\ast }\xi \in T_{p}\mathbb{L}$ 
\begin{eqnarray}
\left( \left( R_{q}^{-1}\right) ^{\ast }\theta \right) _{p}\left(
X_{p}\right) &=&\theta _{p/q}\left( \left( R_{q}^{-1}\circ R_{p}\right)
_{\ast }\xi \right)  \notag \\
&=&\left( R_{p/q}^{-1}\circ R_{q}^{-1}\circ R_{p}\right) _{\ast }\xi  \notag
\\
&=&\left( R_{p/q}^{-1}\circ R_{p/q}^{\left( q\right) }\right) _{\ast }\xi
\label{thetarighttr}
\end{eqnarray}%
where same idea as in (\ref{rightvect}) was used.

Now consider $d\theta .$ Generally, for a $1$-form$,$ we have 
\begin{equation}
d\theta \left( X,Y\right) =X\theta \left( Y\right) -Y\theta \left( X\right)
-\theta \left( \left[ X,Y\right] \right) .
\end{equation}%
Suppose $X,$ $Y$ are right fundamental, then from (\ref{T1Lbrack2}), we get 
\begin{equation}
\left( d\theta \right) _{p}\left( X,Y\right) -\left[ \theta \left( X\right)
,\theta \left( Y\right) \right] ^{\left( p\right) }=0.  \label{MCequation1}
\end{equation}%
However, since right fundamental vector fields span the space of vector
fields on $\mathbb{L}$, (\ref{MCequation1}) is true for any vector fields,
and we obtain the following analogue of the Maurer-Cartan equation.

\begin{theorem}
\label{thmMC}Let $p\in \mathbb{L}$ and let $\left[ \cdot ,\cdot \right]
^{\left( p\right) }$ be bracket on $\mathfrak{l}^{\left( p\right) }$. Then, $%
\theta $ satisfies the following equation at $p$: 
\begin{equation}
\left( d\theta \right) _{p}-\frac{1}{2}\left[ \theta ,\theta \right]
^{\left( p\right) }=0,  \label{MCequation2}
\end{equation}%
\qquad where $\left[ \theta ,\theta \right] ^{\left( p\right) }$ is the
bracket of $\mathbb{L}$-algebra-valued $1$-forms such that for any $X,Y\in
T_{p}\mathbb{L}$, $\frac{1}{2}\left[ \theta ,\theta \right] ^{\left(
p\right) }\left( X,Y\right) =\left[ \theta \left( X\right) ,\theta \left(
Y\right) \right] ^{\left( p\right) }.$

Let $q\in \mathbb{L}$ and $\theta ^{\left( q\right) }=\left( R_{q}\right)
^{\ast }\theta ,$ then $\theta ^{\left( q\right) }$ satisfies 
\begin{equation}
\left( d\theta ^{\left( q\right) }\right) _{p}-\frac{1}{2}\left[ \theta
^{\left( q\right) },\theta ^{\left( q\right) }\right] ^{\left( pq\right) }=0,
\label{MCequation3}
\end{equation}%
where $\left[ \cdot ,\cdot \right] ^{\left( pq\right) }$ is the bracket on $%
\mathfrak{l}^{\left( pq\right) }.$
\end{theorem}

\begin{proof}
The first part already follows from (\ref{MCequation1}). For the second
part, by applying $\left( R_{q}\right) ^{\ast }$ to (\ref{MCequation2}) we
easily see that $\theta ^{\left( q\right) }$ satisfies (\ref{MCequation2})
with the translated bracket $\left[ \cdot ,\cdot \right] ^{\left( pq\right)
} $, and hence we get (\ref{MCequation3}).
\end{proof}

\begin{remark}
The $1$-form $\theta ^{\left( q\right) }$ can be seen as translating a
vector in $T_{p}\mathbb{L}$ by $R_{q}$ to $T_{pq}\mathbb{L}$, and then by $%
R_{pq}^{-1}$ back to $\mathfrak{l}.$ However, given the identity $%
xq/pq=x/_{q}p$, we see that $\theta ^{\left( q\right) }$ is just the loop
Maurer-Cartan form in $\left( \mathbb{L},\circ _{q}\right) .$
\end{remark}

The obvious key difference with the Lie group picture here is that the
bracket in (\ref{MCequation2}) non-constant on $\mathbb{L},$ i.e. given a
basis, the structure \textquotedblleft constants\textquotedblright\ would no
longer be constants. In particular, the Jacobi identity is the integrability
condition for the Maurer-Cartan equation on Lie groups, however here we see
that the right-hand side of the Jacobi identity is related to a ternary form
given by the derivative of the bracket. For any $\xi ,\eta ,\gamma \in 
\mathfrak{l}^{\left( p\right) }$, define 
\begin{equation}
\func{Jac}^{\left( p\right) }\left( \xi ,\eta ,\gamma \right) =\left[ \xi ,%
\left[ \eta ,\gamma \right] ^{\left( p\right) }\right] ^{\left( p\right) }+%
\left[ \eta ,\left[ \gamma ,\xi \right] ^{\left( p\right) }\right] ^{\left(
p\right) }+\left[ \gamma ,\left[ \xi ,\eta \right] ^{\left( p\right) }\right]
^{\left( p\right) }.  \label{Jac}
\end{equation}%
We also need the following definition.

\begin{definition}
Define the \emph{bracket function }$b:\mathbb{\mathring{L}}\longrightarrow 
\mathfrak{l}\otimes \Lambda ^{2}\mathfrak{l}^{\ast }$ to be the map that
takes $p\mapsto \left[ \cdot ,\cdot \right] ^{\left( p\right) }\in \mathfrak{%
l}\otimes \Lambda ^{2}\mathfrak{l}^{\ast }$, so that $b\left( \theta ,\theta
\right) $ is an $\mathfrak{l}$-valued $2$-form on $\mathbb{L}$, i.e. $%
b\left( \theta ,\theta \right) \in \Omega ^{2}\left( \mathfrak{l}\right) .$
\end{definition}

Lemma \ref{lemAssoc} below will give the differential of $b$. The proof is
given in Appendix \ref{secAppendix}.

\begin{lemma}
\label{lemAssoc}For fixed $\eta ,\gamma \in \mathfrak{l},$ 
\begin{equation}
\left. db\right\vert _{p}\left( \eta ,\gamma \right) =\left[ \eta ,\gamma
,\theta _{p}\right] ^{\left( p\right) }-\left[ \gamma ,\eta ,\theta _{p}%
\right] ^{\left( p\right) }\text{,}  \label{db1}
\end{equation}%
where $\left[ \cdot ,\cdot ,\cdot \right] ^{\left( p\right) }$ is the $%
\mathbb{L}$\emph{-algebra associator }on $\mathfrak{l}^{\left( p\right) }$
given by 
\begin{eqnarray}
\left[ \eta ,\gamma ,\xi \right] ^{\left( p\right) } &=&\left. \frac{d^{3}}{%
dtd\tau d\tau ^{\prime }}\exp \left( \tau \eta \right) \circ _{p}\left( \exp
\left( \tau ^{\prime }\gamma \right) \circ _{p}\exp \left( t\xi \right)
\right) \right\vert _{t,\tau ,\tau ^{\prime }=0}  \label{Lalgassoc} \\
&&-\left. \frac{d^{3}}{dtd\tau d\tau ^{\prime }}\left( \exp \left( \tau \eta
\right) \circ _{p}\exp \left( \tau ^{\prime }\gamma \right) \right) \circ
_{p}\exp \left( t\xi \right) \right\vert _{t,\tau ,\tau ^{\prime }=0}. 
\notag
\end{eqnarray}%
Moreover, 
\begin{equation}
\left[ \eta ,\gamma ,\xi \right] ^{\left( p\right) }=\left. \frac{d^{3}}{%
dtd\tau d\tau ^{\prime }}\left[ \exp \left( \tau \eta \right) ,\exp \left(
\tau ^{\prime }\gamma \right) ,\exp \left( t\xi \right) \right] ^{\left( 
\mathbb{L},\circ _{p}\right) }\right\vert _{t,\tau ,\tau ^{\prime }=0}
\label{Lalgassoc2}
\end{equation}%
where $\left[ \cdot ,\cdot ,\cdot \right] ^{\left( \mathbb{L},\circ
_{p}\right) }$ is the loop associator on $\left( \mathbb{L},\circ
_{p}\right) $ as defined by (\ref{loopassoc2}).
\end{lemma}

The skew-symmetric combination of associators, as in (\ref{db1}) will
frequently occur later on, so let us define for convenience 
\begin{equation}
a_{p}\left( \eta ,\gamma ,\xi \right) =\left[ \eta ,\gamma ,\xi \right]
^{\left( p\right) }-\left[ \gamma ,\eta ,\xi \right] ^{\left( p\right) },
\label{ap}
\end{equation}%
which we can can call the \emph{left-alternating associator}, so in
particular, (\ref{db1}) becomes 
\begin{equation}
\left. db\right\vert _{p}\left( \eta ,\gamma \right) =a_{p}\left( \eta
,\gamma ,\theta _{p}\right) .  \label{db2}
\end{equation}

The loop Maurer-Cartan equation can be rewritten as 
\begin{equation}
d\theta =\frac{1}{2}b\left( \theta ,\theta \right) ,  \label{MC3}
\end{equation}%
and hence we see that $b\left( \theta ,\theta \right) $ is an exact form, so
in particular, $d\left( b\left( \theta ,\theta \right) \right) =0$. We will
now use this to derive a generalization of the Jacobi identity.

\begin{theorem}
\label{thmJacobi}The maps $a$ and $b$ satisfy the relation 
\begin{equation}
b\left( \theta ,b\left( \theta ,\theta \right) \right) =a\left( \theta
,\theta ,\theta \right) .  \label{Jac3}
\end{equation}%
where wedge products are assumed. Equivalently, if $\xi ,\eta ,\gamma \in 
\mathfrak{l}$ and $p\in \mathbb{L}$, then%
\begin{equation}
\func{Jac}^{\left( p\right) }\left( \xi ,\eta ,\gamma \right) =a_{p}\left(
\xi ,\eta ,\gamma \right) +a_{p}\left( \eta ,\gamma ,\xi \right)
+a_{p}\left( \gamma ,\xi ,\eta \right) .  \label{Jac2}
\end{equation}
\end{theorem}

\begin{proof}
We know that $d\left( b\left( \theta ,\theta \right) \right) =0,$ and thus,
using (\ref{db1}) and (\ref{MC3}), we have 
\begin{eqnarray*}
0 &=&d\left( b\left( \theta ,\theta \right) \right) \\
&=&\left( db\right) \left( \theta ,\theta \right) +b\left( d\theta ,\theta
\right) -b\left( \theta ,d\theta \right) \\
&=&a\left( \theta ,\theta ,\theta \right) -b\left( \theta ,b\left( \theta
,\theta \right) \right) .
\end{eqnarray*}%
So indeed, (\ref{Jac3}) holds. Now let $X,Y,Z$ be vector fields on $\mathbb{L%
}$, such that $X=\rho \left( \xi \right) ,$ $Y=\rho \left( \eta \right) ,$ $%
Z=\rho \left( \gamma \right) $. Then, $a\left( \theta ,\theta ,\theta
\right) _{p}\left( X,Y,Z\right) =2\func{Jac}^{\left( p\right) }\left( \xi
,\eta ,\gamma \right) $ and $\frac{1}{2}b\left( \theta ,b\left( \theta
,\theta \right) \right) _{p}\left( X,Y,Z\right) $ gives the right-hand side
of (\ref{Jac2}).
\end{proof}

\begin{remark}
An algebra $\left( A,\left[ \cdot ,\cdot \right] ,\left[ \cdot ,\cdot ,\cdot %
\right] \right) $ with a skew-symmetric bracket $\left[ \cdot ,\cdot \right] 
$ and ternary multilinear bracket $\left[ \cdot ,\cdot ,\cdot \right] $ that
satisfies (\ref{Jac2}) is known as an \emph{Akivis algebra} \cite%
{Akivis1,ShestakovAkivis1}. If $\left( \mathbb{L},\circ _{p}\right) $ is
left-alternative, we find from (\ref{Lalgassoc}) that for any $\xi ,\eta \in 
\mathfrak{l},$ $\left[ \xi ,\xi ,\eta \right] ^{\left( p\right) }=0$, that
is, the $\mathbb{L}$-algebra associator on $\mathfrak{l}^{\left( p\right) }$
is skew-symmetric in the first two entries, and thus $a_{p}=2\left[ \cdot
,\cdot ,\cdot \right] ^{\left( p\right) }.$ If the algebra is alternative,
then $\func{Jac}^{\left( p\right) }\left( \xi ,\eta ,\gamma \right) =6\left[
\xi ,\eta ,\gamma \right] ^{\left( p\right) }.$ It is known \cite%
{HofmannStrambach}, that conversely, to an alternative Akivis algebra, there
corresponds a unique, up to local isomorphism, local analytic alternative
loop. If $\left( \mathbb{L},\circ _{p}\right) $ is a left Bol loop (so that
it is left-alternative) then the corresponding algebra on $\mathfrak{l}%
^{\left( p\right) }$ will be a \emph{Bol algebra}, where $\left[ \cdot
,\cdot \right] ^{\left( p\right) }$ and $\left[ \cdot ,\cdot ,\cdot \right]
^{\left( p\right) }$ satisfy additional identities \cite%
{Akivis1,OnishchikVinberg,SabininMikheev1985}. If $\left( \mathbb{L},\circ
_{p}\right) $ is a Moufang loop (so in particular it is alternative), then
the associator is totally skew-symmetric and the algebra on $\mathfrak{l}%
^{\left( p\right) }$ is then a \emph{Malcev\ algebra}. It then satisfies in
addition the following identity \cite{Kuzmin1971,Malcev1955}:%
\begin{equation}
\left[ \xi ,\eta ,\left[ \xi ,\gamma \right] ^{\left( p\right) }\right]
^{\left( p\right) }=\left[ \left[ \xi ,\eta ,\gamma \right] ^{\left(
p\right) },\xi \right] ^{\left( p\right) }.  \label{MalcevId}
\end{equation}%
Moreover, all non-Lie simple Malcev algebras have been classified \cite%
{Kuzmin1968b} - these are either the imaginary octonions over the real
number, imaginary octonions over the complex numbers, or split octonions
over the real numbers.
\end{remark}

We generally will not distinguish the notation between loop associators and $%
\mathbb{L}$-algebra associators. It should be clear from the context which
is being used. Moreover, it will be convenient to define mixed associators
between elements of $\mathbb{L}$ and $\mathfrak{l}$. For example, an $\left( 
\mathbb{L},\mathbb{L},\mathfrak{l}\right) $-associator is defined for any $%
p,q\in \mathbb{L}$ and $\xi \in \mathfrak{l}$ as 
\begin{equation}
\left[ p,q,\xi \right] ^{\left( s\right) }=\left( L_{p}^{\left( s\right)
}\circ L_{q}^{\left( s\right) }\right) _{\ast }\xi -\left( L_{p\circ
_{s}q}^{\left( s\right) }\right) _{\ast }\xi \in T_{p\circ _{s}q}\mathbb{L}
\label{pqxiassoc}
\end{equation}%
and an $\left( \mathbb{L},\mathfrak{l},\mathfrak{l}\right) $-associator is
defined for an $p\in \mathbb{L}$ and $\eta ,\xi \in \mathfrak{l}$ as 
\begin{eqnarray}
\left[ p,\eta ,\xi \right] ^{\left( s\right) } &=&\left. \frac{d}{dtd\tau }%
\left( p\circ _{s}\left( \exp \left( t\eta \right) \circ _{s}\exp \left(
\tau \xi \right) \right) \right) \right\vert _{t=0}  \notag \\
&&-\left. \frac{d}{dtd\tau }\left( \left( p\circ _{s}\exp \left( t\eta
\right) \right) \circ _{s}\exp \left( \tau \xi \right) \right) \right\vert
_{t=0},  \label{etapxiassoc}
\end{eqnarray}%
where we see that $\left[ p,\eta ,\xi \right] ^{\left( s\right) }\in T_{p}%
\mathbb{L}.$ Similarly, for other combinations.

Let us now consider the action of loop homomophisms on $\mathbb{L}$-algebras.

\begin{lemma}
\label{lemAlgHom}Suppose $\mathbb{L}_{1}$ and $\mathbb{L}_{2}$ are two
smooth loops with tangent algebras at identity $\mathfrak{l}_{1}\ $and $%
\mathfrak{l}_{2}$, respectively. Let $\alpha :\mathbb{L}_{1}\longrightarrow 
\mathbb{L}_{2}$ be a smooth loop homomorphism. Then, $\alpha _{\ast }:$ $%
\mathfrak{l}_{1}\longrightarrow \mathfrak{l}_{2}$ is an $\mathbb{L}$-algebra
homomorphism, i.e., for any $\xi ,\gamma \in $ $\mathfrak{l}_{1}$, 
\begin{equation}
\alpha _{\ast }\left[ \xi ,\gamma \right] ^{\left( 1\right) }=\left[ \alpha
_{\ast }\xi ,\alpha _{\ast }\gamma \right] ^{\left( 2\right) },
\label{algebrahom}
\end{equation}%
where $\left[ \cdot ,\cdot \right] ^{\left( 1\right) }$ and $\left[ \cdot
,\cdot \right] ^{\left( 2\right) }$ are the corresponding brackets on $%
\mathfrak{l}_{1}\ $and $\mathfrak{l}_{2}$, respectively. Moreover, $\alpha
_{\ast }$ is an associator homomorphism, i.e., for any $\xi ,\gamma ,\eta
\in $ $\mathfrak{l}_{1}$, 
\begin{equation}
\alpha _{\ast }\left[ \xi ,\gamma ,\eta \right] ^{\left( 1\right) }=\left[
\alpha _{\ast }\xi ,\alpha _{\ast }\gamma ,\alpha _{\ast }\eta \right]
^{\left( 2\right) }  \label{akivishom}
\end{equation}%
where $\left[ \cdot ,\cdot ,\cdot \right] ^{\left( 1\right) }$ and $\left[
\cdot ,\cdot ,\cdot \right] ^{\left( 2\right) }$ are the corresponding
ternary brackets on $\mathfrak{l}_{1}\ $and $\mathfrak{l}_{2}$, respectively.
\end{lemma}

\begin{proof}
Suppose $\exp _{\left( 1\right) }:$ $\mathfrak{l}_{1}\longrightarrow \mathbb{%
L}_{1}$ and $\exp _{\left( 2\right) }:\mathfrak{l}_{2}\longrightarrow 
\mathbb{L}_{2}$ are the corresponding exponential maps. Let $\xi ,\gamma \in 
$ $\mathfrak{l}_{1}$. We know from (\ref{loopexp}) that 
\begin{equation}
\alpha \left( \exp _{\left( 1\right) }\xi \right) =\exp _{\left( 2\right)
}\left( \alpha _{\ast }\xi \right) .  \label{homexp}
\end{equation}%
From (\ref{brack2deriv}), we have 
\begin{equation*}
\left[ \xi ,\gamma \right] ^{\left( 1\right) }=\left. \frac{d^{2}}{dtd\tau }%
\exp _{\left( 1\right) }\left( t\xi \right) \exp _{\left( 1\right) }\left(
\tau \gamma \right) \right\vert _{t,\tau =0}-\left. \frac{d^{2}}{dtd\tau }%
\exp _{\left( 1\right) }\left( \tau \gamma \right) \exp _{\left( 1\right)
}\left( t\xi \right) \right\vert _{t,\tau =0},
\end{equation*}%
Applying $\alpha _{\ast }$ to $\left[ \xi ,\gamma \right] ^{\left( 1\right)
} $, we find 
\begin{eqnarray*}
\alpha _{\ast }\left[ \xi ,\gamma \right] ^{\left( 1\right) } &=&\left. 
\frac{d^{2}}{dtd\tau }\alpha \left( \exp _{\left( 1\right) }\left( t\xi
\right) \exp _{\left( 1\right) }\left( \tau \gamma \right) \right)
\right\vert _{t,\tau =0} \\
&&-\left. \frac{d^{2}}{dtd\tau }\alpha \left( \exp _{\left( 1\right) }\left(
\tau \gamma \right) \exp _{\left( 1\right) }\left( t\xi \right) \right)
\right\vert _{t,\tau =0}.
\end{eqnarray*}%
However, since $\alpha $ is a loop homomorphism, and using (\ref{homexp}),
we have, 
\begin{eqnarray*}
\alpha _{\ast }\left[ \xi ,\gamma \right] ^{\left( 1\right) } &=&\left. 
\frac{d^{2}}{dtd\tau }\exp _{\left( 2\right) }\left( t\alpha _{\ast }\xi
\right) \exp _{\left( 1\right) }\left( \tau \alpha _{\ast }\gamma \right)
\right\vert _{t,\tau =0} \\
&&-\left. \frac{d^{2}}{dtd\tau }\exp _{\left( 1\right) }\left( \tau \alpha
_{\ast }\gamma \right) \exp _{\left( 1\right) }\left( t\alpha _{\ast }\xi
\right) \right\vert _{t,\tau =0} \\
&=&\left[ \alpha _{\ast }\xi ,\alpha _{\ast }\gamma \right] ^{\left(
2\right) }.
\end{eqnarray*}%
Similarly, using the definition (\ref{Lalgassoc}) for the $\mathbb{L}$%
-algebra associator, we obtain (\ref{akivishom}).
\end{proof}

In particular, if $\left( \alpha ,p\right) \in \Psi ^{R}\left( \mathbb{L}%
\right) $, then $\alpha $ induces an $\mathbb{L}$-algebra isomorphism $%
\alpha _{\ast }:\left( \mathfrak{l,}\left[ \cdot ,\cdot \right] \right)
\longrightarrow \left( \mathfrak{l,}\left[ \cdot ,\cdot \right] ^{\left(
p\right) }\right) $. This shows that as long as $p$ is a companion of some
smooth right pseudoautomorphism, the corresponding algebras are isomorphic.
More generally, we have the following.

\begin{corollary}
\label{corLoppalghom}Suppose $h=\left( \alpha ,p\right) \in \Psi ^{R}\left( 
\mathbb{L}\right) $, and $q\in \mathbb{\mathring{L}}$, then, for any $\xi
,\eta ,\gamma \in \mathfrak{l}$,%
\begin{subequations}%
\label{loopalghom} 
\begin{eqnarray}
\alpha _{\ast }\left[ \xi ,\eta \right] ^{\left( q\right) } &=&\left[ \alpha
_{\ast }\xi ,\alpha _{\ast }\eta \right] ^{h\left( q\right) }
\label{loopalghom1} \\
\alpha _{\ast }\left[ \xi ,\eta ,\gamma \right] ^{\left( q\right) } &=&\left[
\alpha _{\ast }\xi ,\alpha _{\ast }\eta ,\alpha _{\ast }\gamma \right]
^{h\left( q\right) }.  \label{loopalghom2}
\end{eqnarray}%
\end{subequations}%
\end{corollary}

\begin{proof}
Since $h=\left( \alpha ,p\right) $ is right pseudo-automorphism of $\mathbb{L%
},$ by Lemma \ref{lemPseudoHom}, it induces a loop homomorphism $\alpha
:\left( \mathbb{L},q\right) \longrightarrow \left( \mathbb{L},h\left(
q\right) \right) ,$ and thus by Lemma \ref{lemAlgHom}, $\alpha _{\ast }:%
\mathfrak{l}^{\left( q\right) }\longrightarrow \mathfrak{l}^{\left( h\left(
q\right) \right) }$ is a loop algebra homomorphism. Thus (\ref{loopalghom})
follows.
\end{proof}

\begin{remark}
In general, Akivis algebras are not fully defined by the binary and ternary
brackets, as shown in \cite{ShestakovUmirbaev}. Indeed, for a fuller
picture, a more complicated structure of a \emph{Sabinin algebra }is needed 
\cite{SabininBook}.
\end{remark}

Generally, we see that $\Psi ^{R}\left( \mathbb{L}\right) $ acts on $%
\mathfrak{l}$ via pushforwards of the action on $\mathbb{L}$, i.e. for $h\in
\Psi ^{R}\left( \mathbb{L}\right) $ and $\xi \in \mathfrak{l}$, we have $%
h\cdot \xi =\left( h^{\prime }\right) _{\ast }\xi $.

The expressions (\ref{loopalghom}) show that the maps $b\in C^{\infty
}\left( \mathbb{\mathring{L}},\Lambda ^{2}\mathfrak{l}^{\ast }\otimes 
\mathfrak{l}\right) $ and $a\in C^{\infty }\left( \mathbb{\mathring{L}}%
,\left( \otimes ^{3}\mathfrak{l}^{\ast }\right) \otimes \mathfrak{l}\right) $
that correspond to the brackets are equivariant maps with respect to the
action of $\Psi ^{R}\left( \mathbb{L}\right) .$ Now suppose $s\in \mathbb{%
\mathring{L}},$ and denote $b_{s}=b\left( s\right) \in \Lambda ^{2}\mathfrak{%
l}^{\ast }\otimes \mathfrak{l}$. Then the equivariance of $b$ means that the
stabilizer $\func{Stab}_{\Psi ^{R}\left( \mathbb{L}\right) }\left(
b_{s}\right) $ in $\Psi ^{R}\left( \mathbb{L}\right) $ of $b_{s}$ is
equivalent to the the set of all $h\in \Psi ^{R}\left( \mathbb{L}\right) $
for which $b_{h\left( s\right) }=b_{s}.$ In particular, $\func{Stab}_{\Psi
^{R}\left( \mathbb{L}\right) }\left( b_{s}\right) $ is a Lie subgroup of $%
\Psi ^{R}\left( \mathbb{L}\right) $, and clearly $\func{Aut}\left( \mathbb{L}%
,\circ _{s}\right) =\func{Stab}_{\Psi ^{R}\left( \mathbb{L}\right) }\left(
s\right) \subset $ $\func{Stab}_{\Psi ^{R}\left( \mathbb{L}\right) }\left(
b_{s}\right) .$ Moreover, note that if $h=\left( \gamma ,C\right) \in \func{%
Aut}\left( \mathbb{L},\circ _{s}\right) \times \mathcal{N}^{R}\left( \mathbb{%
L},\circ _{s}\right) $, then we still have $b_{h\left( s\right) }=b_{s}.$
So, we can say that the corresponding subgroup $\iota _{1}\left( \func{Aut}%
\left( \mathbb{L},\circ _{s}\right) \right) \ltimes \iota _{2}\left( 
\mathcal{N}^{R}\left( \mathbb{L},\circ _{s}\right) \right) \subset \Psi
^{R}\left( \mathbb{L}\right) $ is contained in $\func{Stab}_{\Psi ^{R}\left( 
\mathbb{L}\right) }\left( b_{s}\right) .$ Hence, as long as $\mathcal{N}%
^{R}\left( \mathbb{L},\circ _{s}\right) $ is non-trivial, $\func{Stab}_{\Psi
^{R}\left( \mathbb{L}\right) }\left( b_{s}\right) $ is strictly greater than 
$\func{Aut}\left( \mathbb{L},\circ _{s}\right) .$ Similarly for $a$.

Let us now also consider how the bracket $\left[ \cdot ,\cdot \right] $ is
transformed by $\left( \func{Ad}_{p}^{\left( s\right) }\right) _{\ast }.$

\begin{theorem}
Suppose $s\in \mathbb{\mathring{L}}$ $,\ p\in \mathbb{L}$, and $\xi ,\eta
,\gamma \in \mathfrak{l}.$ Then 
\begin{eqnarray}
\left( \func{Ad}_{p}^{\left( s\right) }\right) _{\ast }\left[ \xi ,\eta %
\right] ^{\left( s\right) } &=&\left[ \left( \func{Ad}_{p}^{\left( s\right)
}\right) _{\ast }\xi ,\left( \func{Ad}_{p}^{\left( s\right) }\right) _{\ast
}\eta \right] ^{\left( ps\right) }  \label{Adbrack1} \\
&&-\left( R_{p}^{\left( s\right) }\right) _{\ast }^{-1}\left[ \left( \func{Ad%
}_{p}^{\left( s\right) }\right) _{\ast }\xi ,p,\eta \right] ^{\left(
s\right) }+\left( R_{p}^{\left( s\right) }\right) _{\ast }^{-1}\left[ \left( 
\func{Ad}_{p}^{\left( s\right) }\right) _{\ast }\eta ,p,\xi \right] ^{\left(
s\right) }  \notag \\
&&+\left( R_{p}^{\left( s\right) }\right) _{\ast }^{-1}\left[ p,\xi ,\eta %
\right] ^{\left( s\right) }-\left( R_{p}^{\left( s\right) }\right) _{\ast
}^{-1}\left[ p,\eta ,\xi \right] ^{\left( s\right) }.  \notag
\end{eqnarray}%
The bracket $\left[ \cdot ,\cdot \right] ^{\left( ps\right) }$ is related to 
$\left[ \cdot ,\cdot \right] ^{\left( s\right) }$ via the expression 
\begin{equation}
\left[ \xi ,\eta \right] ^{\left( ps\right) }=\left[ \xi ,\eta \right]
^{\left( s\right) }+\left( R_{p}^{\left( s\right) }\right) _{\ast
}^{-1}a_{s}\left( \xi ,\eta ,p\right) .  \label{Adbrack1a}
\end{equation}
\end{theorem}

\begin{proof}
Consider%
\begin{eqnarray}
\left( \func{Ad}_{p}^{\left( s\right) }\right) _{\ast }\left[ \xi ,\eta %
\right] ^{\left( s\right) } &=&\left. \frac{d}{dtd\tau }\left( p\circ
_{s}\left( \exp \left( t\xi \right) \circ _{s}\exp \left( \tau \eta \right)
\right) \right) /_{s}p\right\vert _{t,\tau =0}  \notag \\
&&-\left. \frac{d}{dtd\tau }\left( p\circ _{s}\left( \exp \left( t\eta
\right) \circ _{s}\exp \left( \tau \xi \right) \right) \right)
/_{s}p\right\vert _{t,\tau =0}.  \label{Adbrack}
\end{eqnarray}%
For brevity and clarity, let us suppress the derivatives and exponentials,
then using mixed associators such as (\ref{etapxiassoc}), we can write 
\begin{eqnarray*}
\left( p\circ _{s}\left( \xi \circ _{s}\eta \right) \right) /_{s}p &=&\left(
\left( p\circ _{s}\xi \right) \circ _{s}\eta \right) /_{s}f+\left[ p,\xi
,\eta \right] ^{\left( s\right) }/_{s}p \\
&=&\left( \left( \left( p\circ _{s}\xi \right) /_{s}p\circ _{s}p\right)
\circ _{s}\eta \right) /_{s}p+\left[ p,\xi ,\eta \right] ^{\left( s\right)
}/_{s}p \\
&=&\left( \func{Ad}_{p}^{\left( s\right) }\xi \circ _{s}\left( p\circ
_{s}\eta \right) \right) /_{s}p-\left[ \func{Ad}_{p}^{\left( s\right) }\xi
,p,\eta \right] ^{\left( s\right) }/_{s}p \\
&&+\left[ p,\xi ,\eta \right] ^{\left( s\right) }/_{s}p.
\end{eqnarray*}%
Applying (\ref{xrprod}), we get 
\begin{equation}
\left( p\circ _{s}\left( \xi \circ _{s}\eta \right) \right) /_{s}p=\func{Ad}%
_{p}^{\left( s\right) }\xi \circ _{ps}\func{Ad}_{p}^{\left( s\right) }\eta -%
\left[ \func{Ad}_{p}^{\left( s\right) }\xi ,p,\eta \right] ^{\left( s\right)
}/_{s}p+\left[ p,\xi ,\eta \right] ^{\left( s\right) }/_{s}p.
\end{equation}%
Subtracting the same expression with $\xi $ and $\eta $ reversed, (\ref%
{Adbrack}) becomes 
\begin{eqnarray}
\left( \func{Ad}_{p}^{\left( s\right) }\right) _{\ast }\left[ \xi ,\eta %
\right] ^{\left( s\right) } &=&\left[ \left( \func{Ad}_{p}^{\left( s\right)
}\right) _{\ast }\xi ,\left( \func{Ad}_{p}^{\left( s\right) }\right) _{\ast
}\eta \right] ^{\left( ps\right) } \\
&&-\left( R_{p}^{\left( s\right) }\right) _{\ast }^{-1}\left[ \left( \func{Ad%
}_{p}^{\left( s\right) }\right) _{\ast }\xi ,p,\eta \right] ^{\left(
s\right) }+\left( R_{p}^{\left( s\right) }\right) _{\ast }^{-1}\left[ \left( 
\func{Ad}_{p}^{\left( s\right) }\right) _{\ast }\eta ,p,\xi \right] ^{\left(
s\right) }  \notag \\
&&+\left( R_{p}^{\left( s\right) }\right) _{\ast }^{-1}\left[ p,\xi ,\eta %
\right] ^{\left( s\right) }-\left( R_{p}^{\left( s\right) }\right) _{\ast
}^{-1}\left[ p,\eta ,\xi \right] ^{\left( s\right) }.  \notag
\end{eqnarray}%
To obtain (\ref{Adbrack1a}), using (\ref{brack2deriv}), we can write 
\begin{eqnarray}
\left[ \xi ,\eta \right] ^{\left( ps\right) } &=&\left. \frac{d^{2}}{dtd\tau 
}\exp \left( t\xi \right) \circ _{ps}\exp \left( \tau \eta \right)
\right\vert _{t,\tau =0}  \label{brackps} \\
&&-\left. \frac{d^{2}}{dtd\tau }\exp \left( \tau \xi \right) \circ _{ps}\exp
\left( t\eta \right) \right\vert _{t,\tau =0}.  \notag
\end{eqnarray}%
However, from (\ref{xrprod}),%
\begin{equation*}
\exp \left( t\xi \right) \circ _{ps}\exp \left( \tau \eta \right) =\left(
\exp \left( t\xi \right) \circ _{s}\left( \exp \left( \tau \eta \right)
\circ _{s}p\right) \right) /_{s}p,
\end{equation*}%
thus 
\begin{eqnarray*}
\left. \frac{d^{2}}{dtd\tau }\exp \left( t\xi \right) \circ _{ps}\exp \left(
\tau \eta \right) \right\vert _{t,\tau =0} &=&\left( R_{p}^{\left( s\right)
}\right) _{\ast }^{-1}\left. \frac{d^{2}}{dtd\tau }\exp \left( t\xi \right)
\circ _{s}\left( \exp \left( \tau \eta \right) \circ _{s}p\right)
\right\vert _{t,\tau =0} \\
&=&\left( R_{p}^{\left( s\right) }\right) _{\ast }^{-1}\left[ \xi ,\eta ,p%
\right] ^{\left( s\right) } \\
&&+\left. \frac{d^{2}}{dtd\tau }\exp \left( t\xi \right) \circ _{s}\exp
\left( \tau \eta \right) \right\vert _{t,\tau =0}
\end{eqnarray*}%
and similarly for the second term in (\ref{brackps}). Hence, we obtain (\ref%
{Adbrack1a}).
\end{proof}

From (\ref{Adbrack1a}) and noting that for any $h\in \Psi ^{R}\left( \mathbb{%
L}\right) $, $h\left( s\right) =h\left( s\right) /s\cdot s,$ we find that $%
\left[ \cdot ,\cdot \right] ^{\left( s\right) }=\left[ \cdot ,\cdot \right]
^{\left( h\left( s\right) \right) }$ if and only if 
\begin{equation}
a_{s}\left( \xi ,\eta ,\faktor{h\left( s\right)}{s}\right) ^{\left( s\right)
}=0  \label{stabbrackcond}
\end{equation}%
for any $\xi ,\eta \in \mathfrak{l}.$ From (\ref{PsAutoriso}) recall that $%
h\left( s\right) /s$ is the companion that corresponds to $h$ in $\left( 
\mathbb{L},\circ _{s}\right) .$

Also, note that from (\ref{Adbrack1a}), we have 
\begin{equation}
\left[ \theta ,\theta \right] ^{\left( p\right) }=\left[ \theta ,\theta %
\right] ^{\left( 1\right) }+\left( R_{p}\right) _{\ast }^{-1}a_{1}\left(
\theta ,\theta ,p\right) ,  \label{brackthetas}
\end{equation}%
so the left-alternating associator with $p$ is the obstruction for the
brackets $\left[ \cdot ,\cdot \right] ^{\left( p\right) }$ and $\left[ \cdot
,\cdot \right] ^{\left( 1\right) }$ to be equal. Moreover, the structural
equation (\ref{MCequation2}) can be rewritten as 
\begin{equation}
d\theta -\frac{1}{2}\left[ \theta ,\theta \right] ^{\left( 1\right) }=\frac{1%
}{2}\left( R_{p}\right) _{\ast }^{-1}a_{1}\left( \theta ,\theta ,p\right) .
\end{equation}%
This makes the dependence on the associator more explicit.

Using the associator on $\mathfrak{l}^{\left( p\right) }$ we can define the
right nucleus $\mathcal{N}^{R}\left( \mathfrak{l}^{\left( p\right) }\right) $
of $\mathfrak{l}^{\left( p\right) }.$

\begin{definition}
Let $p\in \mathbb{\mathring{L}}$, then, the right nucleus $\mathcal{N}%
^{R}\left( \mathfrak{l}^{\left( p\right) }\right) $ is defined as 
\begin{equation}
\mathcal{N}^{R}\left( \mathfrak{l}^{\left( p\right) }\right) =\left\{ \xi
\in \mathfrak{l}:a_{p}\left( \eta ,\gamma ,\xi \right) =0\ \text{for all }%
\eta ,\gamma \in \mathfrak{l}\right\} .  \label{NRl}
\end{equation}
\end{definition}

It may seem that a more natural definition of $\mathcal{N}^{R}\left( 
\mathfrak{l}^{\left( p\right) }\right) $ would be to be the set of all $\xi
\in \mathfrak{l}$ such that $\left[ \eta ,\gamma ,\xi \right] ^{\left(
p\right) }=0$ for any $\eta ,\gamma \in \mathfrak{l}.\ $However, the
advantage of (\ref{NRl}) is that, as we will see, it will always be a Lie
subalgebra of $\mathfrak{l}^{\left( p\right) }.$ For a left-alternative
algebra, the skew-symmetrization in (\ref{NRl}) would be unnecessary of
course.

\begin{theorem}
The right nucleus $\mathcal{N}^{R}\left( \mathfrak{l}^{\left( p\right)
}\right) $ is a Lie subalgebra of $\mathfrak{l}^{\left( p\right) }.$
\end{theorem}

\begin{proof}
We first need to show that $\mathcal{N}^{R}\left( \mathfrak{l}^{\left(
p\right) }\right) $ is closed under $\left[ \cdot ,\cdot \right] ^{\left(
p\right) }.$ Indeed, taking the exterior derivative of (\ref{db2}), for
vector fields $X,Y$ on $\mathbb{L}$ we have 
\begin{eqnarray*}
0 &=&\left( d^{2}b\left( \beta ,\gamma \right) \right) \left( X,Y\right)
=X\left( d_{Y}b\left( \beta ,\gamma \right) \right) -Y\left( d_{X}b\left(
\beta ,\gamma \right) \right) -d_{\left[ X,Y\right] }b\left( \beta ,\gamma
\right) \\
&=&X\left( a\left( \beta ,\gamma ,\theta \left( Y\right) \right) \right)
-Y\left( a\left( \beta ,\gamma ,\theta \left( X\right) \right) \right)
-a\left( \beta ,\gamma ,\theta \left( \left[ X,Y\right] \right) \right) .
\end{eqnarray*}%
Suppose now $\xi ,\eta \in \mathfrak{l}^{\left( p\right) }$ and let $X=\rho
\left( \xi \right) ,$ $Y=\rho \left( \eta \right) $ be the corresponding
right fundamental vector fields, then using (\ref{T1Lbrack}), we have 
\begin{equation}
a\left( \beta ,\gamma ,b\left( \xi ,\eta \right) \right) =-X\left( a\left(
\beta ,\gamma ,\eta \right) \right) +Y\left( a\left( \beta ,\gamma ,\xi
\right) \right)  \label{d2b}
\end{equation}%
Suppose now $\xi ,\eta \in \mathcal{N}^{R}\left( \mathfrak{l}^{\left(
p\right) }\right) $. Then, the right-hand side of (\ref{d2b}) vanishes, and
at $p\in \mathbb{L}$, 
\begin{equation}
a_{p}\left( \beta ,\gamma ,\left[ \xi ,\eta \right] ^{\left( p\right)
}\right) =0,  \label{d2b2}
\end{equation}%
and thus $\left[ \xi ,\eta \right] ^{\left( p\right) }\in \mathcal{N}%
^{R}\left( \mathfrak{l}^{\left( p\right) }\right) .$

To conclude that $\mathcal{N}^{R}\left( \mathfrak{l}^{\left( p\right)
}\right) $ is a Lie subalgebra, we also need to verify that Lie algebra
Jacobi identity holds. That is, for any $\xi ,\eta ,\gamma \in \mathcal{N}%
^{R}\left( \mathfrak{l}^{\left( p\right) }\right) $, $\func{Jac}^{\left(
p\right) }\left( \xi ,\eta ,\gamma \right) =0$. Indeed, from the Akivis
identity (\ref{Jac2}), 
\begin{equation}
\func{Jac}^{\left( p\right) }\left( \xi ,\eta ,\gamma \right) =a_{p}\left(
\xi ,\eta ,\gamma \right) +a_{p}\left( \eta ,\gamma ,\xi \right)
+a_{p}\left( \gamma ,\xi ,\eta \right) =0,
\end{equation}%
by definition of $\mathcal{N}^{R}\left( \mathfrak{l}^{\left( p\right)
}\right) .$
\end{proof}

For any smooth loop, consider the loop right nucleus $\mathcal{N}^{R}\left( 
\mathbb{L},\circ _{p}\right) $ as a submanifold of $\mathbb{L}.$ Then, 
\begin{equation}
T_{1}\mathcal{N}^{R}\left( \mathbb{L},\circ _{p}\right) =\left\{ \xi \in 
\mathfrak{l}:\left[ q,r,\xi \right] ^{\left( p\right) }=0\ \text{for all }%
q,r\in \mathbb{L}\right\} ,  \label{T1N}
\end{equation}%
where here we are using the mixed associator as defined by (\ref{pqxiassoc}%
). Then, (\ref{Lalgassoc2}) implies that $T_{1}\mathcal{N}^{R}\left( \mathbb{%
L},\circ _{p}\right) \subset \mathcal{N}^{R}\left( \mathfrak{l}^{\left(
p\right) }\right) .$ It is unclear what are the conditions for the converse,
and hence equality, of the two spaces.

Recall from (\ref{CRNucl}) that $A\in \mathcal{N}^{R}\left( \mathbb{L}%
\right) $\ if and only if $\func{Ad}_{p}\left( A\right) \in \mathcal{N}%
^{R}\left( \mathbb{L},\circ _{p}\right) $, so in particular, $\eta \in T_{1}%
\mathcal{N}^{R}\left( \mathbb{L}\right) $ if and only if $\left( \func{Ad}%
_{p}\right) _{\ast }\eta \in T_{1}\mathcal{N}^{R}\left( \mathbb{L},\circ
_{p}\right) .$ In (\ref{Adbrack1}) we then see that for $\eta ,\gamma \in
T_{1}\mathcal{N}^{R}\left( \mathbb{L}\right) $, the associators vanish, and
we get 
\begin{equation}
\left( \func{Ad}_{p}\right) _{\ast }\left[ \eta ,\gamma \right] =\left[
\left( \func{Ad}_{p}\right) _{\ast }\eta ,\left( \func{Ad}_{p}\right) _{\ast
}\gamma \right] ^{\left( p\right) }.  \label{AdNucl}
\end{equation}%
Hence, for each $p\in \mathbb{\mathring{L}},$ $T_{1}\mathcal{N}^{R}\left( 
\mathbb{L}\right) \cong T_{1}\mathcal{N}^{R}\left( \mathbb{L},\circ
_{p}\right) $ as Lie algebras.

\begin{example}
Consider the Moufang loop of unit octonions $U\mathbb{O}.$ Then, $T_{1}U%
\mathbb{O}\cong \func{Im}\mathbb{O}$ - the space of imaginary octonions,
with the bracket given by the commutator on $\func{Im}\mathbb{O}$: for any $%
\xi ,\eta \in \func{Im}\mathbb{O}$, $\left[ \xi ,\eta \right] =\xi \eta
-\eta \xi .$ We also know that $\mathcal{N}\left( U\mathbb{O}\right) \cong 
\mathbb{Z}_{2}$ and $\mathcal{N}\left( \func{Im}\mathbb{O}\right) =\left\{
0\right\} .$ On the other hand, taking a direct product $G\times U\mathbb{O}$
with any Lie group $G$ will give a non-trivial nucleus.
\end{example}

Let $s\in \mathbb{\mathring{L}}.$ Suppose the Lie algebras of $\Psi
^{R}\left( \mathbb{L}\right) $ and $\func{Aut}\left( \mathbb{L},\circ
_{s}\right) $ are $\mathfrak{p}$ and $\mathfrak{h}_{s}$, respectively. In
particular, $\mathfrak{h}_{s}$ is a Lie subalgebra of $\mathfrak{p}$. Define 
$\mathfrak{q}_{s}=T_{1}\mathcal{C}^{R}\left( \mathbb{L},\circ _{s}\right) ,$
then since $\mathcal{C}^{R}\left( \mathbb{L},\circ _{s}\right) \subset 
\mathbb{L}$, so $\mathfrak{q}_{s}\mathfrak{\subset l}^{\left( s\right) }%
\mathfrak{\cong }T_{1}\mathbb{L}.$ On the other hand, $\mathcal{C}^{R}\left( 
\mathbb{L},\circ _{s}\right) \cong 
\faktor{\Psi ^{R}\left( \mathbb{L}\right)}{\func{Aut}\left( 
\mathbb{L},\circ _{s}\right)}$, and the tangent space at the coset $%
1=\left\lfloor \func{Aut}\left( \mathbb{L},\circ _{s}\right) \right\rfloor $
is $\mathfrak{p/h}_{s}.$ Hence, we see that $\mathfrak{q}_{s}\mathfrak{\cong
p/h}_{s}$, at least as vector spaces. The groups $\Psi ^{R}\left( \mathbb{L}%
\right) $ and $\func{Aut}\left( \mathbb{L},\circ _{s}\right) $ act on $%
\mathfrak{p}$ and $\mathfrak{h}_{s}$ via their respective adjoint actions
and hence $\func{Aut}\left( \mathbb{L},\circ _{s}\right) $ acts on $%
\mathfrak{q}_{s}$ via a restriction of the adjoint action of $\Psi
^{R}\left( \mathbb{L}\right) .$ Now note that given $h=\left( \alpha
,A\right) \in \Psi ^{R}\left( \mathbb{L}\right) $ and $\beta \in \func{Aut}%
\left( \mathbb{L},\circ _{s}\right) $, the conjugation of $h$ by $\beta $ is
given by 
\begin{equation*}
\left( \beta ,1\right) \left( \alpha ,A\right) \left( \beta ^{-1},1\right)
=\left( \beta \circ \alpha \circ \beta ^{-1},\beta \left( A\right) \right)
\end{equation*}%
and hence the corresponding action on the companion $A$ is via standard
action of $\beta $ on $\mathbb{L}.$ The differentials of these actions give
the corresponding actions on the tangent spaces. We thus see that the
adjoint action of $\func{Aut}\left( \mathbb{L},\circ _{s}\right) $ on $%
\mathfrak{p/h}_{s}$ is equivalent to the standard tangent action of $\func{%
Aut}\left( \mathbb{L},\circ _{s}\right) $ on $\mathfrak{q}_{s}.$ Hence, $%
\mathfrak{q}_{s}$ and $\mathfrak{p/h}_{s}$ are isomorphic as linear
representations of $\func{Aut}\left( \mathbb{L},\circ _{s}\right) .$ We can
make the isomorphism from $\mathfrak{p/h}_{s}$ to $\mathfrak{q}_{s}$ more
explicit in the following way.

\begin{definition}
Define the map $\varphi :\mathbb{\mathring{L}}\longrightarrow $ $\mathfrak{l}%
\otimes \mathfrak{p}^{\ast }$ such that for each $s\in \mathbb{\mathring{L}}$
and $\gamma \in \mathfrak{p}$, 
\begin{equation}
\varphi _{s}\left( \gamma \right) =\left. \frac{d}{dt}\faktor{\left( \exp
\left( t\gamma \right) \left( s\right) \right)}{s}\right\vert _{t=0}\in 
\mathfrak{l.}  \label{phis}
\end{equation}
\end{definition}

Thus, for each $s\in \mathbb{\mathring{L}},$ $\varphi _{s}$ gives a map from 
$\mathfrak{p}$ to $\mathfrak{l}^{\left( s\right) }.$

\begin{theorem}
\label{lemGammahatsurj}The map $\varphi $ as in (\ref{phis}) is equivariant
with respect to corresponding actions of $\Psi ^{R}\left( \mathbb{L}\right)
, $ in particular for $h\in \Psi ^{R}\left( \mathbb{L}\right) ,$ $s\in 
\mathbb{\mathring{L}}$, $\gamma \in \mathfrak{p},$ we have%
\begin{equation}
\varphi _{h\left( s\right) }\left( \left( \func{Ad}_{h}\right) _{\ast
}\gamma \right) =\left( h^{\prime }\right) _{\ast }\varphi _{s}\left( \gamma
\right) .  \label{phihs}
\end{equation}%
Moreover, the image of $\varphi _{s}$ is $\mathfrak{q}_{s}$ and the kernel
is $\mathfrak{h}_{s}$, and hence, 
\begin{equation}
\mathfrak{p\cong h}_{s}\oplus \mathfrak{q}_{s}.  \label{pdecomp}
\end{equation}
\end{theorem}

\begin{proof}
Consider $h\in \Psi ^{R}\left( \mathbb{L}\right) $. Then, using (\ref%
{PsAutquot2a}), we have 
\begin{eqnarray*}
\varphi _{h\left( s\right) }\left( \gamma \right) &=&\left. \frac{d}{dt}%
\faktor{\left[ \exp \left( t\gamma \right) \left( h\left( s\right) \right)
\right]}{ h\left( s\right)} \right\vert _{t=0} \\
&=&\left. \frac{d}{dt}h^{\prime }\left[ \faktor{\func{Ad}_{h^{-1}}\left(
\exp \left( t\gamma \right) \right) \left( s\right)}{s}\right] \right\vert
_{t=0} \\
&=&\left( h^{\prime }\right) _{\ast }\left. \frac{d}{dt}\faktor{\exp \left(
t\left( \func{Ad}_{h^{-1}}\right) _{\ast }\gamma \right) \left( s\right)} {s}%
\right\vert _{t=0}.
\end{eqnarray*}%
Since $\Psi ^{R}\left( \mathbb{L}\right) $ acts on $\mathfrak{l}$ via $%
\left( h^{\prime }\right) _{\ast }$ and on $\mathfrak{p}$ via $\left( \func{%
Ad}_{h}\right) _{\ast }$ we see that $\varphi $ is equivariant.

Since $\func{Aut}\left( \mathbb{L},\circ _{s}\right) $ is a Lie subgroup of $%
\Psi ^{R}\left( \mathbb{L}\right) ,$ the projection map $\pi :\Psi
^{R}\left( \mathbb{L}\right) \longrightarrow 
\faktor{\Psi ^{R}\left(
\mathbb{L}\right)}{\func{Aut}\left( \mathbb{L},\circ _{s}\right)}\cong 
\mathcal{C}^{R}\left( \mathbb{L},\circ _{s}\right) $ is a smooth submersion
given by $\pi \left( h\right) =h\left( s\right) /s$ for each $h\in \Psi
^{R}\left( \mathbb{L}\right) .$ Thus, $\left. \pi _{\ast }\right\vert _{%
\func{id}}:\mathfrak{p}\longrightarrow \mathfrak{q}_{s}$ is surjective.
However, since $\exp $ is a surjective map from $\mathfrak{p}$ to a
neighborhood of $\func{id}\in \Psi ^{R}\left( \mathbb{L}\right) $, we find
that $\left. \pi _{\ast }\right\vert _{\func{id}}\left( \gamma \right)
=\varphi _{s}\left( \gamma \right) .$ So indeed, the image of the map $%
\varphi _{s}$ is $\mathfrak{q}_{s}.$ Clearly the kernel is $\mathfrak{h}%
_{s}. $ Then, (\ref{pdecomp}) follows immediately.
\end{proof}

Theorem \ref{lemGammahatsurj} implies that $\varphi :\mathbb{\mathring{L}}%
\longrightarrow $ $\mathfrak{l}\otimes \mathfrak{p}^{\ast }$ is equivariant
with respect to the action of $\Psi ^{R}\left( \mathbb{L}\right) ,$ and
similarly as for $b$, we can define $\func{Stab}_{\Psi ^{R}\left( \mathbb{L}%
\right) }\left( \varphi _{s}\right) =\left\{ h\in \Psi ^{R}\left( \mathbb{L}%
\right) :\varphi _{h\left( s\right) }=\varphi _{s}\right\} .$ This is then a
Lie subgroup of $\Psi ^{R}\left( \mathbb{L}\right) ,$ and $\func{Aut}\left( 
\mathbb{L},\circ _{s}\right) \subset $ $\func{Stab}_{\Psi ^{R}\left( \mathbb{%
L}\right) }\left( \varphi ^{\left( s\right) }\right) .$ Suppose $h=\left(
\alpha ,A\right) \in \func{Stab}_{\Psi ^{R}\left( \mathbb{L}\right) }\left(
\varphi ^{\left( s\right) }\right) $, then 
\begin{equation*}
\varphi _{s}\left( \gamma \right) =\varphi _{h\left( s\right) }\left( \gamma
\right) =\left. \frac{d}{dt}\left[ \exp \left( t\gamma \right) \left( \alpha
\left( s\right) A\right) \right] /\left( \alpha \left( s\right) A\right)
\right\vert _{t=0}
\end{equation*}

We can also see the effect on $\varphi $ of left multiplication of $s$ by
elements of $\mathbb{L}$.

\begin{lemma}
Suppose $A\in \mathbb{L}$ and $s\in \mathbb{\mathring{L}}$, then for any $%
\gamma \in \mathfrak{p},$%
\begin{equation}
\varphi _{As}\left( \gamma \right) =\left( R_{A}^{\left( s\right) }\right)
_{\ast }^{-1}\left( \gamma ^{\prime }\cdot A\right) +\left( \func{Ad}%
_{A}^{\left( s\right) }\right) _{\ast }\varphi _{s}\left( \gamma \right) ,
\label{phiAs}
\end{equation}%
where $\gamma ^{\prime }\cdot A=\left. \frac{d}{dt}\left( \exp t\gamma
\right) ^{\prime }\left( A\right) \right\vert _{t=0}$ represents the
infinitesimal action of $\mathfrak{p}$ on $\mathbb{L}.$
\end{lemma}

\begin{proof}
This follows from a direct computation:%
\begin{eqnarray*}
\varphi _{As}\left( \gamma \right) &=&\left. \frac{d}{dt}\exp \left( t\gamma
\right) \left( As\right) /As\right\vert _{t=0} \\
&=&\left. \frac{d}{dt}\left[ \exp \left( t\gamma \right) ^{\prime }\left(
A\right) \exp \left( t\gamma \right) \left( s\right) \right] /As\right\vert
_{t=0} \\
&=&\left. \frac{d}{dt}\left[ A\exp \left( t\gamma \right) \left( s\right) %
\right] /As\right\vert _{t=0}+\left. \frac{d}{dt}\left( \left[ \exp \left(
t\gamma \right) ^{\prime }\left( A\right) \right] s\right) /As\right\vert
_{t=0} \\
&=&\left( \func{Ad}_{A}^{\left( s\right) }\right) _{\ast }\varphi _{s}\left(
\gamma \right) +\left( R_{A}^{\left( s\right) }\right) _{\ast }^{-1}\left(
\gamma ^{\prime }\cdot A\right) ,
\end{eqnarray*}%
where we have used (\ref{rprodqright}).
\end{proof}

\begin{example}
\label{exOct}If $\mathbb{L}\ $is the loop of unit octonions, then we know $%
\mathfrak{p\cong so}\left( 7\right) \cong \Lambda ^{2}\left( \mathbb{R}%
^{7}\right) ^{\ast }$ and $\mathfrak{l\cong }\mathbb{R}^{7}$ , so $\varphi
_{s}$ can be regarded as an element of $\mathbb{R}^{7}\otimes $ $\Lambda ^{2}%
\mathbb{R}^{7},$ and this is precisely a dualized version of the $G_{2}$%
-invariant $3$-form $\varphi .$ The kernel is isomorphic to $\mathfrak{g}%
_{2}.$
\end{example}

\begin{example}
\label{exCx2}Suppose $\mathbb{L=}U\mathbb{C\cong }S^{1}$ - the unit complex
numbers, so that $\mathfrak{l\cong }\mathbb{R}.$ From Example \ref%
{ExNormedDiv}, we may take $\Psi _{n}^{R}\left( U\mathbb{C}\right) =U\left(
n\right) ,$ with a trivial partial action on $U\mathbb{C}.$ The
corresponding Lie algebra is $\mathfrak{p}_{n}\cong \mathfrak{u}\left(
n\right) \cong \mathfrak{su}\left( n\right) \oplus i\mathbb{R}.$ The map $%
\varphi _{s}:\mathfrak{p}_{n}\longrightarrow i\mathbb{R}\ $is then just the
projection $\mathfrak{su}\left( n\right) \oplus i\mathbb{R}\longrightarrow i%
\mathbb{R}$ (i.e. trace). It is independent of $s$. The kernel is $\mathfrak{%
su}\left( n\right) .$ Suppose $V$ is a $n$-dimensional real vector space,
and $V\otimes \mathbb{C}=V^{1,0}\oplus V^{0,1}$. Then, the group $U\left(
n\right) $ acts via unitary transformations on the complex vector space $%
V^{1,0},$ and correspondingly $\mathfrak{u}\left( n\right) \cong V^{1,1}$
(i.e. the space of $\left( 1,1\right) $-forms). Then, we see that $\varphi
_{s}$ is just the dualized version of a Hermitian form on $V\otimes \mathbb{C%
}.$
\end{example}

\begin{example}
\label{exQuat2}Suppose $\mathbb{L=}U\mathbb{H\cong }S^{3}$ - the unit
quaternions, so that $\mathfrak{l\cong }\mathfrak{sp}\left( 1\right) .$ From
Example \ref{ExNormedDiv}, we may take $\Psi _{n}^{R}\left( U\mathbb{H}%
\right) =Sp\left( n\right) Sp\left( 1\right) ,$ with $n\geq 2$, with a
trivial partial action on $U\mathbb{H}.$ The corresponding Lie algebra is $%
\mathfrak{p}_{n}\cong \mathfrak{sp}\left( n\right) \oplus \mathfrak{sp}%
\left( 1\right) .$ The map $\varphi _{s}:\mathfrak{p}_{n}\longrightarrow 
\mathfrak{sp}\left( 1\right) \ $is then given by $\left( a,\xi \right)
\mapsto \left( \func{Ad}_{s}\right) _{\ast }\xi .$ The kernel is then $%
\mathfrak{sp}\left( n\right) .$ Suppose $Sp\left( n\right) Sp\left( 1\right) 
$ acts on a $4n$-dimensional real vector space \ $V$, $\mathfrak{sp}\left(
n\right) \oplus \mathfrak{sp}\left( 1\right) \subset \Lambda ^{2}V^{\ast }$.
Given that $\mathfrak{sp}\left( 1\right) \cong \func{Im}\mathbb{H},$ we can
then write $\varphi _{s}=i\omega _{1}^{\ast }+j\omega _{2}^{\ast }+k\omega
_{3}^{\ast },$ where the $\omega _{i}^{\ast }$ are dualized versions of the
3 linearly independent Hermitian forms that space the $\mathfrak{sp}\left(
1\right) $ subspace of $\Lambda ^{2}V^{\ast }$ \cite{SalamonBook}.
\end{example}

\begin{remark}
The above examples clearly show that one interpretation of the $G_{2}$
structure $3$-form $\varphi $ is as $\func{Im}\mathbb{O}$-valued $2$-form. A
complex Hermitian form is then an $\func{Im}\mathbb{C}$-valued $2$-form, and
a quaternionic Hermitian form is an $\func{Im}\mathbb{H}$-valued $2$-form.
\end{remark}

Now let us summarize the actions of different spaces on one another. For a
fixed $\gamma $, define the map $\hat{\gamma}:\mathbb{\mathring{L}}%
\longrightarrow \mathfrak{l\ }$given by $s\mapsto \hat{\gamma}^{\left(
s\right) }=\varphi _{s}\left( \gamma \right) .$

\begin{theorem}
Suppose $\mathbb{L}$ is a smooth loop with tangent algebra $\mathfrak{l}$
and suppose $\Psi ^{R}\left( \mathbb{L}\right) $ is a Lie group with Lie
algebra $\mathfrak{p}.$ Let $A\in \mathbb{L},$ $s\in \mathbb{\mathring{L}}$, 
$\xi \in \mathfrak{l}$, and $\gamma \in \mathfrak{p}.$ Then, denoting by $%
\cdot $ the relevant action, we have the following:

\begin{enumerate}
\item Infinitesimal action of $\mathfrak{p}$ on $\mathbb{\mathring{L}}$: 
\begin{equation}
\gamma \cdot s=\left. \frac{d}{dt}\exp \left( t\gamma \right) \left(
s\right) \right\vert _{t=0}=\left( R_{s}\right) _{\ast }\hat{\gamma}^{\left(
s\right) }\in T_{s}\mathbb{L}  \label{infplring}
\end{equation}

\item Infinitesimal action of $\mathfrak{p}$ on $\mathbb{L}$, for any $s\in 
\mathbb{\mathring{L}}$: 
\begin{equation}
\gamma \cdot A=\left. \frac{d}{dt}\exp \left( t\gamma \right) ^{\prime
}\left( A\right) \right\vert _{t=0}=\left( R_{A}^{\left( s\right) }\right)
_{\ast }\hat{\gamma}^{\left( As\right) }-\left( L_{A}^{\left( s\right)
}\right) _{\ast }\hat{\gamma}^{\left( s\right) }\in T_{A}\mathbb{L}.
\label{infpl}
\end{equation}%
In particular, if $s=1$, 
\begin{equation}
\gamma \cdot A=\left( R_{A}\right) _{\ast }\hat{\gamma}^{\left( A\right)
}-\left( L_{A}\right) _{\ast }\hat{\gamma}^{\left( 1\right) }.
\label{infpl2}
\end{equation}

\item Action of $\mathfrak{p}$ on $\mathfrak{l\ }$for any $s\in \mathbb{%
\mathring{L}}$:%
\begin{eqnarray}
\gamma \cdot \xi &=&\left. \frac{d}{dt}\left( \exp \left( t\gamma \right)
^{\prime }\right) _{\ast }\left( \xi \right) \right\vert _{t=0}  \notag \\
&=&\left. d\hat{\gamma}\right\vert _{s}\left( \rho _{s}\left( \xi \right)
\right) +\left[ \hat{\gamma}^{\left( s\right) },\xi \right] ^{\left(
s\right) }.  \label{actpl}
\end{eqnarray}%
In particular, for $s=1$, we have 
\begin{equation}
\gamma \cdot \xi =\left. d\hat{\gamma}\right\vert _{1}\left( \xi \right) +%
\left[ \hat{\gamma}^{\left( 1\right) },\xi \right] .  \label{actpl2}
\end{equation}
\end{enumerate}
\end{theorem}

\begin{proof}
Let $A,B\in \mathbb{L},$ $s\in \mathbb{\mathring{L}}$, $\xi ,\eta \in 
\mathfrak{l}$, $h\in \Psi ^{R}\left( \mathbb{L}\right) $, and $\gamma \in 
\mathfrak{p}.$ Then we have the following.

\begin{enumerate}
\item The infinitesimal action of a Lie algebra is a standard definition.

\item Consider now the action of $\mathfrak{p}$ on $\mathbb{L}.$ Suppose now 
$\gamma \in \mathfrak{p}$ and $A\in \mathbb{L}$ 
\begin{equation}
\gamma ^{\prime }\cdot A=\left. \frac{d}{dt}\left( \exp \left( t\gamma
\right) ^{\prime }\right) \left( A\right) \right\vert _{t=0}.
\label{gammaprime}
\end{equation}%
Suppose $h\in \Psi ^{R}\left( \mathbb{L},\circ _{s}\right) $, then by (\ref%
{PsAutoriso}), the action of $h$ on $A\in \mathbb{L}$ is%
\begin{equation*}
h\left( A\right) =h^{\prime }\left( A\right) \circ _{s}\left( %
\faktor{h\left( s\right)}{s}\right)
\end{equation*}%
Thus, the partial action $h^{\prime }\left( A\right) $ is given by 
\begin{equation}
h^{\prime }\left( A\right) =\left( \faktor{h\left( As\right)}{s}\right)
/_{s}\left(\faktor{ h\left( s\right)} {s}\right) .  \label{hprimes}
\end{equation}%
Moreover, 
\begin{equation}
\faktor{h\left( As\right)}{s}=\left(\faktor{h\left( As\right)}{As}\right)
\circ _{s}A.  \label{hprimes2}
\end{equation}%
Hence, substituting into (\ref{gammaprime}), we have 
\begin{eqnarray}
\gamma ^{\prime }\cdot A &=&\left. \frac{d}{dt}\left( \faktor{\exp \left(
t\gamma \left( As\right) \right)}{As}\circ _{s}A\right) /_{s} \left( %
\faktor{\exp \left( t\gamma \right) \left( s\right)}{s}\right) \right\vert
_{t=0}  \notag \\
&=&\left. \frac{d}{dt}\left( \faktor{\exp \left( t\gamma \left( As\right)
\right)}{As}\circ _{s}A\right) \right\vert _{t=0}-\left. \frac{d}{dt}A\circ
_{s}\left( \faktor{\exp \left( t\gamma \right) \left( s\right)}{s}\right)
\right\vert _{t=0}  \notag \\
&=&\left( R_{A}^{\left( s\right) }\right) _{\ast }\hat{\gamma}^{\left(
As\right) }-\left( L_{A}^{\left( s\right) }\right) _{\ast }\hat{\gamma}%
^{\left( s\right) }.
\end{eqnarray}%
Setting $s=1$ immediately gives (\ref{infpl2}).

\item Suppose now $\gamma \in \mathfrak{p}$ and $\xi \in \mathfrak{l}$, then
we have 
\begin{eqnarray}
\gamma \cdot \xi &=&\left. \frac{d}{dt}\left( \exp \left( t\gamma \right)
^{\prime }\right) _{\ast }\left( \xi \right) \right\vert _{t=0}  \notag \\
&=&\left. \frac{d^{2}}{dtd\tau }\exp \left( t\gamma \right) ^{\prime }\left(
\exp _{s}\tau \xi \right) \right\vert _{t,\tau =0}.  \label{pactl1}
\end{eqnarray}%
Let $\Xi =\exp _{s}\tau \xi \in \mathbb{L}$, then using (\ref{hprimes}) and (%
\ref{hprimes2}), we can write 
\begin{eqnarray*}
\exp \left( t\gamma \right) ^{\prime }\left( \exp _{s}\tau \xi \right)
&=&\exp \left( t\gamma \right) ^{\prime }\left( \Xi \right) \\
&&\left( \exp \left( t\gamma \right) \left( \Xi s/\Xi s\circ _{s}\Xi \right)
\right) /_{s}\left( \faktor{\exp \left( t\gamma \right) \left( s\right)}{s}%
\right) .
\end{eqnarray*}%
Using this, (\ref{pactl1}) becomes 
\begin{eqnarray}
\gamma ^{\prime }\cdot \xi &=&\left. \frac{d^{2}}{dtd\tau }\faktor{\left(
\exp \left( t\gamma \right) \left( \left( \exp _{s}\tau \xi \right) s\right)
\right)} {\left( \left( \exp _{s}\tau \xi \right) s\circ _{s}\exp _{s}\tau
\xi \right)}\right\vert _{t,\tau =0}  \notag \\
&&-\left. \frac{d^{2}}{dtd\tau }\exp _{s}\tau \xi \circ _{s}\left( %
\faktor{\exp \left( t\gamma \right) \left( s\right)}{s}\right) \right\vert
_{t,\tau =0}  \notag \\
&=&\left. \frac{d^{2}}{dtd\tau }\exp \left( t\gamma \right) \left( \left(
\exp _{s}\tau \xi \right) s\right) /\left( \exp _{s}\tau \xi \right)
s\right\vert _{t,\tau =0}+  \notag \\
&&+\left. \frac{d^{2}}{dtd\tau }\left( \faktor{\exp \left( t\gamma \right)
\left( s\right)}{s}\right) \circ _{s}\exp _{s}\tau \xi \right\vert _{t,\tau
=0}  \notag \\
&&-\left. \frac{d^{2}}{dtd\tau }\exp _{s}\tau \xi \circ _{s}\left( %
\faktor{\exp \left( t\gamma \right) \left( s\right)}{s}\right) \right\vert
_{t,\tau =0}
\end{eqnarray}%
However $\hat{\gamma}^{\left( s\right) }=\left. \frac{d}{dt}\exp \left(
t\gamma \right) \left( s\right) /s\right\vert _{t=0}\in \mathfrak{l}$, and
thus 
\begin{eqnarray*}
\left. \frac{d}{d\tau }\left( L_{\exp _{s}\tau \xi }^{\left( s\right)
}\right) _{\ast }\hat{\gamma}^{\left( s\right) }\right\vert _{\tau =0}
&=&\left. \frac{d^{2}}{dtd\tau }\left( \exp _{s}\tau \xi \right) \circ
_{s}\exp _{s}\left( t\hat{\gamma}^{\left( s\right) }\right) \right\vert
_{t,\tau =0} \\
\left. \frac{d}{d\tau }\left( R_{\exp _{s}\tau \xi }^{\left( s\right)
}\right) _{\ast }\hat{\gamma}^{\left( s\right) }\right\vert _{\tau =0}
&=&\left. \frac{d^{2}}{dtd\tau }\exp _{s}\left( t\hat{\gamma}^{\left(
s\right) }\right) \circ _{s}\exp _{s}\tau \xi \right\vert _{t,\tau =0}.
\end{eqnarray*}%
Hence, using the expression (\ref{brack2deriv}) for $\left[ \cdot ,\cdot %
\right] ^{\left( s\right) },$ we get 
\begin{equation}
\gamma ^{\prime }\cdot \xi =\left. \frac{d}{d\tau }\hat{\gamma}^{\left( \exp
_{s}\tau \xi \right) s}\right\vert _{\tau =0}+\left[ \hat{\gamma}^{\left(
s\right) },\xi \right] ^{\left( s\right) }.  \label{gampri1}
\end{equation}%
The first term in (\ref{gampri1}) is then precisely the differential of $%
\hat{\gamma}$ at $s\in \mathbb{L}$ in the direction $\left( R_{s}\right)
_{\ast }\xi .$ Setting $s=1$ we get (\ref{actpl2}).
\end{enumerate}
\end{proof}

\begin{remark}
Since the full action of $\Psi ^{R}\left( \mathbb{L}\right) $ does not
preserve $1,$ the pushforward of the action of some $h\in \Psi ^{R}\left( 
\mathbb{L}\right) $ sends $T_{1}\mathbb{L}$ to $T_{A}\mathbb{L}$, where $%
A=h\left( 1\right) $ is the companion of $\mathbb{L}.$ To actually obtain an
action on $T_{1}\mathbb{L},$ translation back to $1$ is needed. This can be
achieved either by right or left division by $A$. Dividing by $A$ on the
right reduces to the partial action of $\Psi ^{R}\left( \mathbb{L}\right) ,$
i.e. action by $h^{\prime }$. This is how the action of $\mathfrak{p}$ on $%
\mathfrak{l}$ in (\ref{actpl}) is defined. Dividing by $A$ on the left,
gives the map $h^{\prime \prime }=\func{Ad}_{A^{-1}}\circ h^{\prime }$, as
defined in (\ref{nuclearaction}). In that setting, it was defined on the
nucleus, and hence gave an actual group action of $\Psi ^{R}\left( \mathbb{L}%
\right) $, however in a non-associative setting, in general this will not be
a group action.
\end{remark}

Combining some of the above results, we also have the following useful
relationship.

\begin{lemma}
Suppose $\xi \in \mathfrak{p}$ and $\eta ,\gamma \in \mathfrak{l},$ then 
\begin{equation}
\xi \cdot \left[ \eta ,\gamma \right] ^{\left( s\right) }=\left[ \xi \cdot
\eta ,\gamma \right] ^{\left( s\right) }+\left[ \eta ,\xi \cdot \gamma %
\right] ^{\left( s\right) }+a_{s}\left( \eta ,\gamma ,\varphi _{s}\left( \xi
\right) \right) .  \label{xilbrack}
\end{equation}
\end{lemma}

\begin{proof}
Using the definition (\ref{actpl}) of the action of $\mathfrak{p}$ on $%
\mathfrak{l},$ we have 
\begin{eqnarray*}
\xi \cdot \left[ \eta ,\gamma \right] ^{\left( s\right) } &=&\left. \frac{d}{%
dt}\left( \exp \left( t\xi \right) ^{\prime }\right) _{\ast }\left[ \eta
,\gamma \right] ^{\left( s\right) }\right\vert _{t=0} \\
&=&\left. \frac{d}{dt}\left[ \left( \exp \left( t\xi \right) ^{\prime
}\right) _{\ast }\eta ,\left( \exp \left( t\xi \right) ^{\prime }\right)
_{\ast }\gamma \right] ^{\exp \left( t\xi \right) \left( s\right)
}\right\vert _{t=0}
\end{eqnarray*}%
where we have also used (\ref{loopalghom1}). Hence, 
\begin{equation}
\xi \cdot \left[ \eta ,\gamma \right] ^{\left( s\right) }=\left[ \xi \cdot
\eta ,\gamma \right] ^{\left( s\right) }+\left[ \eta ,\xi \cdot \gamma %
\right] ^{\left( s\right) }+\left. \frac{d}{dt}\left[ \eta ,\gamma \right]
^{\exp \left( t\xi \right) \left( s\right) }\right\vert _{t=0}.
\label{xilbrack2}
\end{equation}%
We can rewrite the last term in (\ref{xilbrack2}) as 
\begin{equation*}
\left. \frac{d}{dt}\left[ \eta ,\gamma \right] ^{\exp \left( t\xi \right)
\left( s\right) }\right\vert _{t=0}=\left. \frac{d}{dt}\left[ \eta ,\gamma %
\right] ^{\exp _{s}\left( t\varphi _{s}\left( \xi \right) \right)
s}\right\vert _{t=0}=\left. d_{\rho \left( \hat{\xi}\right) }b\right\vert
_{s}\left( \eta ,\gamma \right)
\end{equation*}%
where $\hat{\xi}=\varphi _{s}\left( \xi \right) $. Then, from (\ref{db1}),
we see that 
\begin{equation}
\left. d_{\rho \left( \hat{\xi}\right) }b\right\vert _{s}\left( \eta ,\gamma
\right) =a_{s}\left( \eta ,\gamma ,\hat{\xi}\right)
\end{equation}%
and overall, we obtain (\ref{xilbrack}).
\end{proof}

Recall that for each $s\in \mathbb{\mathring{L}}$, the bracket function $%
b_{s}\ $is in$\ \Lambda ^{2}\mathfrak{l}^{\ast }\otimes \mathfrak{l}$, which
is a tensor product of $\mathfrak{p}$-modules, so (\ref{xilbrack}) can be
used to define the action of $\xi \in \mathfrak{p}$ on $b_{s}$. Using the
derivation property of Lie algebra representations on tensor products, we
find that for $\eta ,\gamma \in \mathfrak{l},$ 
\begin{eqnarray}
\left( \xi \cdot b_{s}\right) \left( \eta ,\gamma \right) &=&\xi \cdot
\left( b_{s}\left( \eta ,\gamma \right) \right) -b_{s}\left( \xi \cdot \eta
,\gamma \right) -b_{s}\left( \eta ,\xi \cdot \gamma \right)  \notag \\
&=&a_{s}\left( \eta ,\gamma ,\varphi _{s}\left( \xi \right) \right) .
\label{bsact}
\end{eqnarray}

\begin{definition}
Suppose $\mathfrak{g}$ is a Lie algebra with a representation on a vector
space $M$, so that $\left( M,\mathfrak{g}\right) $ is a $\mathfrak{g}$%
-module. Then if $x\in M$, define the \emph{annihilator subalgebra }$\func{%
Ann}_{\mathfrak{g}}\left( x\right) $ in $\mathfrak{g}$ of $x$ as 
\begin{equation}
\func{Ann}_{\mathfrak{g}}\left( x\right) =\left\{ \xi \in \mathfrak{g}:\xi
\cdot x=0\right\} .  \label{anng}
\end{equation}
\end{definition}

From (\ref{bsact}), we see that 
\begin{equation}
\func{Ann}_{\mathfrak{p}}\left( b_{s}\right) =\left\{ \xi \in \mathfrak{p}%
:a_{s}\left( \eta ,\gamma ,\varphi _{s}\left( \xi \right) \right) =0\ \text{%
for all }\eta ,\gamma \in \mathfrak{l}\right\} .  \label{annpbs}
\end{equation}%
The definition (\ref{annpbs}) is simply that $\xi \in $ $\func{Ann}_{%
\mathfrak{p}}\left( b_{s}\right) \ $if and only if $\varphi _{s}\left( \xi
\right) \in \mathcal{N}^{R}\left( \mathfrak{l}^{\left( s\right) }\right) ,$
so that $\func{Ann}_{\mathfrak{p}}\left( b_{s}\right) =\varphi
_{s}^{-1}\left( \mathcal{N}^{R}\left( \mathfrak{l}^{\left( s\right) }\right)
\right) .$ This is the Lie algebra that corresponds to the Lie group $\func{%
Stab}_{\Psi ^{R}\left( \mathbb{L}\right) }\left( b_{s}\right) .$ Indeed, the
condition (\ref{annpbs}) is precisely the infinitesimal version of (\ref%
{stabbrackcond}). If $\mathbb{L}$ is a $G$-loop, so that $\varphi _{s}\left( 
\mathfrak{p}\right) =\mathfrak{l}^{\left( s\right) },$ then $\varphi
_{s}\left( \func{Ann}_{\mathfrak{p}}\left( b_{s}\right) \right) =\mathcal{N}%
^{R}\left( \mathfrak{l}^{\left( s\right) }\right) .$ Hence, in this case, $%
\func{Ann}_{\mathfrak{p}}\left( b_{s}\right) \cong \mathfrak{h}_{s}\oplus 
\mathcal{N}^{R}\left( \mathfrak{l}^{\left( s\right) }\right) $.

Using the definition (\ref{phis}) of $\varphi _{s}$, let us consider the
action of $\mathfrak{p}$ on $\varphi _{s}.$

\begin{lemma}
\label{lempactl}Suppose $\xi ,\eta \in \mathfrak{p}$, then for any $s\in 
\mathbb{L}$, we have 
\begin{equation}
\xi \cdot \varphi _{s}\left( \eta \right) -\eta \cdot \varphi _{s}\left( \xi
\right) =\varphi _{s}\left( \left[ \xi ,\eta \right] _{\mathfrak{p}}\right) +%
\left[ \varphi _{s}\left( \xi \right) ,\varphi _{s}\left( \eta \right) %
\right] ^{\left( s\right) },  \label{xiphi}
\end{equation}%
where $\cdot $ means the action of $\mathfrak{p}$ on $\mathfrak{l}.$
\end{lemma}

\begin{proof}
Using (\ref{actpl}) and the definition (\ref{phis}) of $\varphi _{s}$, we
have 
\begin{eqnarray}
\xi \cdot \varphi _{s}\left( \eta \right) &=&\left. \frac{d^{2}}{dtd\tau }%
\exp \left( t\xi \right) ^{\prime }\left( \faktor{\exp \left( \tau \eta
\right) \left( s\right)}{s}\right) \right\vert _{t,\tau =0}  \notag \\
&=&\left. \frac{d^{2}}{dtd\tau }\faktor{\exp \left( t\xi \right) \left( \exp
\left( \tau \eta \right) \left( s\right) \right)}{\exp \left( t\xi \right)
\left( s\right)} \right\vert _{t,\tau =0}  \notag \\
&=&\left. \frac{d^{2}}{dtd\tau }\exp \left( t\xi \right) \left( \exp \left(
\tau \eta \right) \left( s\right) \right) /s\right\vert _{t,\tau =0}  \notag
\\
&&-\left. \frac{d^{2}}{dtd\tau }\left( \faktor{\exp \left( \tau \eta \right)
\left( s\right)}{s}\cdot \exp \left( t\xi \right) \left( s\right) \right)
/s\right\vert _{t,\tau =0}  \notag \\
&=&\left. \frac{d^{2}}{dtd\tau }\left( \exp \left( t\xi \right) \exp \left(
\tau \eta \right) \right) \left( s\right) /s\right\vert _{t,\tau =0}
\label{xiphi1} \\
&&-\left. \frac{d^{2}}{dtd\tau }\faktor{\exp \left( \tau \eta \right) \left(
s\right)}{s}\circ _{s}\faktor{\exp \left( t\xi \right) \left( s\right)} {s}%
\right\vert _{t,\tau =0},  \notag
\end{eqnarray}%
where we have used (\ref{PsAutquot2a}) and Lemma \ref{lemQuotient}. Now
subtracting the same expression but with $\xi $ and $\eta $ switched around,
we obtain (\ref{xiphi}).
\end{proof}

\begin{remark}
In terms of the Chevalley-Eilenberg complex of $\mathfrak{p}$ with values in 
$\mathfrak{l},$ the relation (\ref{xiphi}) shows that if we regard $\varphi
_{s}\in C^{1}\left( \mathfrak{p};\mathfrak{l}\right) $, i.e. a $1$-form on $%
\mathfrak{p}$ with values in $\mathfrak{l}$, then the Chevalley-Eilenberg
differential $d_{CE}$ of $\varphi _{s}$ is given by 
\begin{equation}
\left( d_{CE}\varphi _{s}\right) \left( \xi ,\eta \right) =\left[ \varphi
_{s}\left( \xi \right) ,\varphi _{s}\left( \eta \right) \right] ^{\left(
s\right) }  \label{dCEphis}
\end{equation}%
for any $\xi ,\eta \in \mathfrak{p}.$ It is interesting that, at least on $%
\mathfrak{q}_{s},$ the bracket $\left[ \cdot ,\cdot \right] ^{\left(
s\right) }$ corresponds to an exact $2$-cochain.
\end{remark}

Similarly, from (\ref{xiphi}), we then see that the action of $\xi \in 
\mathfrak{p}$ on $\varphi _{s}$ as an $\mathfrak{p}^{\ast }\otimes \mathfrak{%
l}$-valued map. Indeed, given $\xi ,\eta \in \mathfrak{p},$ we have 
\begin{eqnarray}
\left( \xi \cdot \varphi _{s}\right) \left( \eta \right) &=&\xi \cdot
\varphi _{s}\left( \eta \right) -\varphi _{s}\left( \left[ \xi ,\eta \right]
_{\mathfrak{p}}\right)  \notag \\
&=&\eta \cdot \varphi _{s}\left( \xi \right) -\left[ \varphi _{s}\left( \eta
\right) ,\varphi _{s}\left( \xi \right) \right] ^{\left( s\right) }
\label{xiphi3}
\end{eqnarray}%
where we have first used the fact that $\mathfrak{p}$ acts on itself via the
adjoint representation and then (\ref{xiphi}) in the second line.

Let us now consider $\func{Ann}_{\mathfrak{p}}\left( \varphi _{s}\right) .$
From (\ref{xiphi3}), we see that we have two equivalent characterizations of 
$\func{Ann}_{\mathfrak{p}}\left( \varphi _{s}\right) .$ In particular, $\xi
\in \func{Ann}_{\mathfrak{p}}\left( \varphi _{s}\right) $ if and only if 
\begin{equation}
\xi \cdot \hat{\eta}=\varphi _{s}\left( \left[ \xi ,\eta \right] _{\mathfrak{%
p}}\right)  \label{xiphieta}
\end{equation}%
or equivalently, for $\xi \not\in \mathfrak{h}_{s},$ if and only if, 
\begin{equation}
\eta \cdot \hat{\xi}=\left[ \hat{\eta},\hat{\xi}\right] ^{\left( s\right) },
\label{etaphixi}
\end{equation}%
for any $\eta \in \mathfrak{p.}$ Here we are again setting $\hat{\xi}%
=\varphi _{s}\left( \xi \right) $ and $\hat{\eta}=\varphi _{s}\left( \eta
\right) .$ In particular, (\ref{xiphieta}) shows that $\mathfrak{q}_{s}$ is
a representation of $\func{Ann}_{\mathfrak{p}}\left( \varphi _{s}\right) .$
Suppose now, $\xi _{1},\xi _{2}\in \func{Ann}_{\mathfrak{p}}\left( \varphi
_{s}\right) ,$ then using (\ref{xiphieta}) and (\ref{etaphixi}), we find
that 
\begin{equation}
\varphi _{s}\left( \left[ \xi _{1},\xi _{2}\right] _{\mathfrak{p}}\right)
=\xi _{1}\cdot \hat{\xi}_{2}=\left[ \hat{\xi}_{1},\hat{\xi}_{2}\right]
^{\left( s\right) }.  \label{annpphi}
\end{equation}%
Therefore, $\varphi _{s}\left( \func{Ann}_{\mathfrak{p}}\left( \varphi
_{s}\right) \right) $ is a Lie subalgebra of $\mathfrak{l}^{\left( s\right)
} $ with $\varphi _{s}$ being a Lie algebra homomorphism. The kernel $%
\mathfrak{h}_{s}=\ker \varphi _{s}$ is then of course an ideal of $\func{Ann}%
_{\mathfrak{p}}\left( \varphi _{s}\right) .$ Thus, the quotient $\func{Ann}_{%
\mathfrak{p}}\left( \varphi _{s}\right) /\mathfrak{h}_{s}$ is again a Lie
algebra, and hence $\func{Ann}_{\mathfrak{p}}\left( \varphi _{s}\right) $ is
a trivial Lie algebra extension of $\mathfrak{h}_{s}$. Moreover, note that
the Lie algebra $\func{Ann}_{\mathfrak{p}}\left( \varphi _{s}\right) $
corresponds to the Lie group $\func{Stab}_{\Psi ^{R}\left( \mathbb{L}\right)
}\left( \varphi _{s}\right) $, and thus if $\func{Aut}\left( \mathbb{L}%
,\circ _{s}\right) $ and $\func{Stab}_{\Psi ^{R}\left( \mathbb{L}\right)
}\left( \varphi _{s}\right) $ are both connected, then we see that $\func{Aut%
}\left( \mathbb{L},\circ _{s}\right) $ is a normal subgroup of $\func{Stab}%
_{\Psi ^{R}\left( \mathbb{L}\right) }\left( \varphi _{s}\right) .$

In the special case when $\mathbb{L}$ is a $G$-loop, we get a nice property
of $\func{Ann}_{\mathfrak{p}}\left( \varphi _{s}\right) .$

\begin{theorem}
Suppose $\mathbb{L}$ is a $G$-loop, then $\func{Ann}_{\mathfrak{p}}\left(
\varphi _{s}\right) \subset \func{Ann}_{\mathfrak{p}}\left( b_{s}\right) .$
\end{theorem}

\begin{proof}
Suppose $\xi \in \func{Ann}_{\mathfrak{p}}\left( \varphi _{s}\right) $ and
let $\eta ,\gamma \in \mathfrak{p}$. Consider 
\begin{eqnarray*}
\left[ \gamma ,\eta \right] _{\mathfrak{p}}\cdot \hat{\xi} &=&\gamma \cdot
\left( \eta \cdot \hat{\xi}\right) -\eta \cdot \left( \gamma \cdot \hat{\xi}%
\right) \\
&=&\gamma \cdot \left[ \hat{\eta},\hat{\xi}\right] ^{\left( s\right) }-\eta
\cdot \left[ \hat{\gamma},\hat{\xi}\right] ^{\left( s\right) } \\
&=&\left[ \gamma \cdot \hat{\eta},\hat{\xi}\right] ^{\left( s\right) }+\left[
\hat{\eta},\gamma \cdot \hat{\xi}\right] ^{\left( s\right) }+a_{s}\left( 
\hat{\eta},\hat{\xi},\hat{\gamma}\right) \\
&&-\left[ \eta \cdot \hat{\gamma},\hat{\xi}\right] ^{\left( s\right) }-\left[
\hat{\gamma},\eta \cdot \hat{\xi}\right] ^{\left( s\right) }-a_{s}\left( 
\hat{\gamma},\hat{\xi},\hat{\eta}\right) ^{\left( s\right) } \\
&=&\left[ \varphi _{s}\left( \left[ \gamma ,\eta \right] _{\mathfrak{p}%
}\right) ,\hat{\xi}\right] ^{\left( s\right) }+\left[ \left[ \hat{\gamma},%
\hat{\eta}\right] ^{\left( s\right) },\hat{\xi}\right] +\left[ \hat{\eta},%
\left[ \hat{\gamma},\hat{\xi}\right] ^{\left( s\right) }\right] ^{\left(
s\right) } \\
&&-\left[ \hat{\gamma},\left[ \hat{\eta},\hat{\xi}\right] ^{\left( s\right) }%
\right] ^{\left( s\right) } \\
&&+a_{s}\left( \hat{\eta},\hat{\xi},\hat{\gamma}\right) -a_{s}\left( \hat{%
\gamma},\hat{\xi},\hat{\eta}\right) \\
&=&\left[ \gamma ,\eta \right] _{\mathfrak{p}}\cdot \hat{\xi}-a_{s}\left( 
\hat{\gamma},\hat{\xi},\hat{\eta}\right)
\end{eqnarray*}%
where we have used (\ref{etaphixi}), (\ref{xilbrack}), (\ref{xiphi}), and
the Akivis identity (\ref{Jac2}). We hence find that 
\begin{equation}
a_{s}\left( \hat{\gamma},\hat{\xi},\hat{\eta}\right) =0.  \label{annphicond}
\end{equation}%
We know that if $\mathbb{L}$ is a $G$-loop, then $\mathfrak{l}^{\left(
s\right) }=\varphi _{s}\left( \mathfrak{p}\right) ,$ and thus the condition (%
\ref{annphicond}) is the same as (\ref{annpbs}), that is $\xi \in \func{Ann}%
_{\mathfrak{p}}\left( b_{s}\right) .$
\end{proof}

\begin{remark}
Overall, if $\mathbb{L}$ is a $G$-loop, we have the following inclusions of
Lie algebras 
\begin{equation}
\ker \varphi _{s}=\mathfrak{h}_{s}\underset{\text{ideal}}{\subset }\func{Ann}%
_{\mathfrak{p}}\left( \varphi _{s}\right) \subset \func{Ann}_{\mathfrak{p}%
}\left( b_{s}\right) \cong \mathfrak{h}_{s}\oplus \mathcal{N}^{R}\left( 
\mathfrak{l}^{\left( s\right) }\right) \subset \mathfrak{p.}  \label{lieseq}
\end{equation}%
If we look at the octonion case, with $\mathbb{L=}U\mathbb{O},$ then $%
\mathfrak{p=so}\left( 7\right) $, $\mathfrak{h}_{s}\cong \mathfrak{g}_{2}.$
Moreover, in this case, $\mathcal{N}^{R}\left( \mathfrak{l}\right) =\left\{
0\right\} $, so we must have $\mathfrak{h}_{s}=\func{Ann}_{\mathfrak{p}%
}\left( \varphi _{s}\right) =\func{Ann}_{\mathfrak{p}}\left( b_{s}\right) .$
This also makes sense because in this case, $\varphi _{s}$ and $b_{s}$ are
essentially the same objects, and moreover, almost uniquely determine $s$
(up to $\pm 1$). At the other extreme, if $\mathbb{L}$ is associative, so
that $\mathcal{N}^{R}\left( \mathfrak{l}\right) =\mathfrak{l},$ then $\func{%
Ann}_{\mathfrak{p}}\left( b_{s}\right) =\mathfrak{p,}$ but $\func{Ann}_{%
\mathfrak{p}}\left( \varphi _{s}\right) $ does not have to equal $\func{Ann}%
_{\mathfrak{p}}\left( b_{s}\right) .$
\end{remark}

\begin{example}
\label{ExNormedDiv2}Using the setup from Examples \ref{ExNormedDiv}, \ref%
{exCx2}, and \ref{exQuat2}, if $\mathbb{L=}U\mathbb{C}$ with $\Psi
_{n}^{R}\left( U\mathbb{C}\right) =U\left( n\right) $ or $\mathbb{L=}U%
\mathbb{H}$ with $\Psi _{n}^{R}\left( U\mathbb{H}\right) =Sp\left( n\right)
Sp\left( 1\right) $, the since the partial action of $\Psi _{n}^{R}$ in each
case here is trivial, from (\ref{pactl1}), we see that the action of each
Lie algebra $\mathfrak{p}_{n}$ on $\mathfrak{l}$ is trivial. In the complex
case, $\mathfrak{l\cong }\mathbb{R},$ and is thus abelian. Hence, from (\ref%
{xiphi3}), we see that in this case $\xi \cdot \varphi _{s}=0$ for each $\xi
\in \mathfrak{p}_{n}.$ This makes because in Example \ref{exCx2} we noted
that $\varphi _{s}$ does not depend on $s$ in the complex case. In the
quaternion case, (\ref{xiphi3}) shows that if $\xi ,\eta \in \mathfrak{sp}%
\left( n\right) \oplus \mathfrak{sp}\left( 1\right) =\mathfrak{p}_{n}$, then 
\begin{eqnarray}
\left( \xi \cdot \varphi _{s}\right) \left( \eta \right) &=&-\varphi
_{s}\left( \left[ \xi ,\eta \right] _{\mathfrak{p}_{n}}\right)  \notag \\
&=&-\left[ \xi _{1},\eta _{1}\right] _{\func{Im}\mathbb{H}}
\label{quatbrack}
\end{eqnarray}%
where $\xi _{1},\eta _{1}$ are the $\mathfrak{sp}\left( 1\right) $
components of $\xi $ and $\eta ,$ and $\left[ \cdot ,\cdot \right] _{\func{Im%
}\mathbb{H}}$ is the bracket on $\func{Im}\mathbb{H}$ (and equivalently on $%
\mathfrak{sp}\left( 1\right) $). In particular, $\func{Ann}_{\mathfrak{p}%
_{n}}\left( \varphi _{s}\right) =\mathfrak{sp}\left( n\right) .$
\end{example}

Note that, while it is known that any simple (i.e. has no nontrivial proper
normal subloops) Moufang loop is a $G$-loop, it is not known whether there
are simple Bol loops that are not $G$-loops \cite{NagyLoop}. On the other
hand, there is an example of a Bol loop that is a $G$-loop but is not a
Moufang loop \cite{Robinson68}. That particular example is constructed from
an alternative division ring, but if that is taken to be $\mathbb{O},$ we
obtain a smooth loop.

\subsection{Killing form}

\label{sectKilling}Similarly as for Lie groups, we may define a Killing form 
$K^{\left( s\right) }$ on $\mathfrak{l}^{\left( s\right) }$. For $\xi ,\eta
\in \mathfrak{l}$, we have 
\begin{equation}
K^{\left( s\right) }\left( \xi ,\eta \right) =\func{Tr}\left( \func{ad}_{\xi
}^{\left( s\right) }\circ \func{ad}_{\eta }^{\left( s\right) }\right) ,
\label{Killing}
\end{equation}%
where $\circ $ is just composition of linear maps on $\mathfrak{l}$ and $%
\func{ad}_{\xi }^{\left( s\right) }\left( \cdot \right) =\left[ \xi ,\cdot %
\right] ^{\left( s\right) },$ as in (\ref{ladpx}). Clearly $K^{\left(
s\right) }$ is a symmetric bilinear form on $\mathfrak{l.}$ Given the form $%
K^{\left( s\right) }$ on $\mathfrak{l}$, we can extend it to a
\textquotedblleft right-invariant\textquotedblright\ form $\left\langle
{}\right\rangle ^{\left( s\right) }$ on $\mathbb{L}$ via right translation,
so that for vector fields $X,Y$ on $\mathbb{L}$, 
\begin{equation}
\left\langle X,Y\right\rangle _{\mathbb{L}}^{\left( s\right) }=K^{\left(
s\right) }\left( \theta \left( X\right) ,\theta \left( Y\right) \right) .
\label{Killing2}
\end{equation}

\begin{theorem}
\label{thmKillingprop}The bilinear form $K^{\left( s\right) }$ (\ref{Killing}%
) on $\mathfrak{l}$ has the following properties.

\begin{enumerate}
\item Let $h\in \Psi ^{R}\left( \mathbb{L}\right) $, then for any $\xi ,\eta
\in \mathfrak{l},$ 
\begin{equation}
K^{\left( h\left( s\right) \right) }\left( h_{\ast }^{\prime }\xi ,h_{\ast
}^{\prime }\eta \right) =K^{\left( s\right) }\left( \xi ,\eta \right) .
\label{Kpsi}
\end{equation}

\item Suppose also $\gamma \in \mathfrak{l,}$ then 
\begin{eqnarray}
K^{\left( s\right) }\left( \func{ad}_{\gamma }^{\left( s\right) }\eta ,\xi
\right) &=&-K^{\left( s\right) }\left( \eta ,\func{ad}_{\gamma }^{\left(
s\right) }\xi \right) +\func{Tr}\left( \func{Jac}_{\xi ,\gamma }^{\left(
s\right) }\circ \func{ad}_{\eta }^{\left( s\right) }\right)  \notag \\
&&+\func{Tr}\left( \func{Jac}_{\eta ,\gamma }^{\left( s\right) }\circ \func{%
ad}_{\xi }^{\left( s\right) }\right) ,  \label{Kad}
\end{eqnarray}%
where $\func{Jac}_{\gamma ,\xi }^{\left( s\right) }:\mathfrak{l}%
\longrightarrow \mathfrak{l}$ is given by $\func{Jac}_{\eta ,\gamma
}^{\left( s\right) }\left( \xi \right) =\func{Jac}^{\left( s\right) }\left(
\xi ,\eta ,\gamma \right) .$

\item Let $\alpha \in \mathfrak{p},$ then 
\begin{eqnarray}
K^{\left( s\right) }\left( \alpha \cdot \xi ,\eta \right) &=&-K^{\left(
s\right) }\left( \xi ,\alpha \cdot \eta \right) +\func{Tr}\left( a_{\eta ,%
\hat{\alpha}}^{\left( s\right) }\circ \func{ad}_{\xi }^{\left( s\right)
}\right)  \label{Klie} \\
&&+\func{Tr}\left( a_{\xi ,\hat{\alpha}}^{\left( s\right) }\circ \func{ad}%
_{\eta }^{\left( s\right) }\right) ,  \notag
\end{eqnarray}%
where $a_{\xi ,\eta }^{\left( s\right) }:\mathfrak{l}\longrightarrow 
\mathfrak{l}$ is given by $a_{\xi ,\eta }^{\left( s\right) }\left( \gamma
\right) =\left[ \gamma ,\xi ,\eta \right] ^{\left( s\right) }-\left[ \xi
,\gamma ,\eta \right] ^{\left( s\right) }$ and $\hat{\alpha}=\varphi
_{s}\left( \alpha \right) $.
\end{enumerate}
\end{theorem}

The proof of Theorem (\ref{thmKillingprop}) is given in Appendix \ref%
{secAppendix}.

\begin{remark}
If $\left( \mathbb{L},\circ _{s}\right) $ is an alternative loop, we know
that $\func{Jac}_{\eta ,\gamma }^{\left( s\right) }=3a^{\left( s\right) },$
so in that in case, $K^{\left( s\right) }$ is invariant with respect to both 
$\func{ad}^{\left( s\right) }$ and the action of $\mathfrak{p}\ $if and only
if 
\begin{equation}
\func{Tr}\left( a_{\eta ,\hat{\alpha}}^{\left( s\right) }\circ \func{ad}%
_{\xi }^{\left( s\right) }\right) +\func{Tr}\left( a_{\xi ,\hat{\alpha}%
}^{\left( s\right) }\circ \func{ad}_{\eta }^{\left( s\right) }\right) =0.
\end{equation}%
Indeed, in \cite{SagleMalcev}, it is shown that for a Malcev algebra, the
Killing form is $\func{ad}$-invariant. A Malcev algebra is alternative and
hence the Killing form is also $\mathfrak{p}$-invariant in that case.
Moreover, it shown in \cite{LoosMalcev} that for a \emph{semisimple} Malcev
algebra, the Killing form is non-degenerate. Here the definition of
\textquotedblleft semisimple\textquotedblright\ is the same as for Lie
algebras, namely that the maximal solvable ideal is zero. Indeed, given the
algebra of imaginary octonions on $\mathbb{R}^{7},$ it is known that the
corresponding Killing form is negative-definite \cite{BaezOcto}. Moreover,
since in this case, the pseudoautomorphism group is $SO\left( 7\right) ,$ so
(\ref{Kpsi}) actually shows that $K^{h\left( s\right) }=K^{s}$ for every $h$%
, and thus is independent of $s$. General criteria for a loop algebra to
admit an invariant definite (or even just non-degenerate) Killing form do
not seem to appear in the literature, and could be the subject of further
study. At least for well-behaved loops, such as Malcev loops, it is likely
that there is significant similarity to Lie groups.
\end{remark}

Suppose now $K^{\left( s\right) }$ is nondegenerate and both $\func{ad}%
^{\left( s\right) }$- and $\mathfrak{p}$-invariant, and moreover suppose $%
\mathfrak{p}$ is semisimple itself, so that it has a nondegenerate,
invariant Killing form $K_{\mathfrak{p}}.$ We will use $\left\langle
{}\right\rangle ^{\left( s\right) }$ and $\left\langle {}\right\rangle _{%
\mathfrak{p}}$ to denote the inner products using $K^{\left( s\right) }$ and 
$K_{\mathfrak{p},}$ respectively. Then, given the map $\varphi _{s}:%
\mathfrak{p}\longrightarrow \mathfrak{l}^{\left( s\right) }$, we can define
its adjoint with respect to these two bilinear maps.

\begin{definition}
Define the map $\varphi _{s}^{t}:\mathfrak{l}^{\left( s\right)
}\longrightarrow \mathfrak{p}$ such that for any $\xi \in \mathfrak{l}%
^{\left( s\right) }$\ and $\eta \in \mathfrak{p}$, 
\begin{equation}
\left\langle \varphi _{s}^{t}\left( \xi \right) ,\eta \right\rangle _{%
\mathfrak{p}}=\left\langle \xi ,\varphi _{s}\left( \eta \right)
\right\rangle ^{\left( s\right) }.  \label{phiadj}
\end{equation}
\end{definition}

Since $\mathfrak{h}_{s}\cong \ker \varphi _{s}$, we then clearly have $%
\mathfrak{p\cong h}_{s}\oplus \func{Im}\varphi _{s}^{t}$, so that $\mathfrak{%
h}_{s}^{\perp }=\func{Im}\varphi _{s}^{t}.$ On the other hand, we also have $%
\mathfrak{l}^{\left( s\right) }\cong \ker \varphi _{s}^{t}\oplus \mathfrak{q}%
_{s}$, since $\mathfrak{q}_{s}=\func{Im}\varphi _{s}.$ Define the
corresponding projections $\pi _{\mathfrak{h}_{s}},\pi _{\mathfrak{h}%
_{s}^{\perp }}$ and $\pi _{\mathfrak{q}_{s}},\pi _{\mathfrak{q}_{s}^{\perp
}.}$We then have the following properties.

\begin{lemma}
\label{lemphisphist}Suppose $\mathfrak{q}_{s}$ is an irreducible
representation of $\mathfrak{h}\ $and suppose the base field of $\mathfrak{p}
$ is $\mathbb{F=R}$ or $\mathbb{C}.$ Then, there exists a $\lambda _{s}\in $ 
$\mathbb{F}$ such that 
\begin{equation}
\varphi _{s}\varphi _{s}^{t}=\lambda _{s}\pi _{\mathfrak{q}^{\left( s\right)
}}\ \ \text{and }\varphi _{s}^{t}\varphi _{s}=\lambda _{s}\pi _{\mathfrak{h}%
_{s}^{\perp }}.  \label{phistphis}
\end{equation}%
Moreover, for any $h\in \Psi ^{R}\left( \mathbb{L}\right) $, $\lambda
_{s}=\lambda _{h\left( s\right) }$.
\end{lemma}

\begin{proof}
Let $\gamma ,\eta \in \mathfrak{p}$ and $\xi \in \mathfrak{l}^{\left(
s\right) }$, then using (\ref{xiphi3}), 
\begin{eqnarray}
\left\langle \left( \gamma \cdot \varphi _{s}^{t}\right) \left( \xi \right)
,\eta \right\rangle _{\mathfrak{p}} &=&\left\langle \left[ \gamma ,\varphi
_{s}^{t}\left( \xi \right) \right] _{\mathfrak{p}},\eta \right\rangle _{%
\mathfrak{p}}-\left\langle \varphi _{s}^{t}\left( \gamma \cdot \xi \right)
,\eta \right\rangle _{\mathfrak{p}}  \notag \\
&=&-\left\langle \varphi _{s}^{t}\left( \xi \right) ,\left[ \gamma ,\eta %
\right] _{\mathfrak{p}}\right\rangle -\left\langle \gamma \cdot \xi ,\varphi
_{s}\left( \eta \right) \right\rangle ^{\left( s\right) }  \notag \\
&=&\left\langle \xi ,\gamma \cdot \varphi _{s}\left( \eta \right) -\varphi
_{s}\left( \left[ \gamma ,\eta \right] _{\mathfrak{p}}\right) \right\rangle
^{\left( s\right) }  \notag \\
&=&\left\langle \xi ,\left( \gamma \cdot \varphi _{s}\right) \left( \eta
\right) \right\rangle ^{\left( s\right) },  \label{gammaphit}
\end{eqnarray}%
so in particular, $\func{Ann}_{\mathfrak{p}}\left( \varphi _{s}\right) =%
\func{Ann}_{\mathfrak{p}}\left( \varphi _{s}^{t}\right) .$ Thus, the map $%
\varphi _{s}\varphi _{s}^{t}:\mathfrak{l}^{\left( s\right) }\longrightarrow 
\mathfrak{l}^{\left( s\right) }$ is an equivariant map of representations of
the Lie subalgebra $\func{Ann}_{\mathfrak{p}}\left( \varphi _{s}\right)
\subset \mathfrak{p}\ $and is moreover self-adjoint with respect to $%
\left\langle {}\right\rangle ^{\left( s\right) }.$ We can also restrict this
map to $\mathfrak{q}_{s},$ which is also a representation of $\func{Ann}_{%
\mathfrak{p}}\left( \varphi _{s}\right) $, and in particular of $\mathfrak{h}%
_{s}.$ Hence, if $\mathfrak{q}_{s}$ is an irreducible representation of $%
\mathfrak{h}_{s},$ since $\varphi _{s}\varphi _{s}^{t}$ is diagonalizable
(in general, if $\mathbb{C}$ is the base field, or because it symmetric if
the base field is $\mathbb{R}$), by Schur's Lemma, there exists some number $%
\lambda _{s}\neq 0$ such that 
\begin{equation}
\left. \varphi _{s}\varphi _{s}^{t}\right\vert _{\mathfrak{q}^{\left(
s\right) }}=\lambda _{s}\func{id}_{\mathfrak{q}^{\left( s\right) }}.
\label{phistid}
\end{equation}%
Applying $\varphi _{s}^{t}$ to (\ref{phistid}), we also obtain. 
\begin{equation}
\left. \varphi _{s}^{t}\varphi _{s}\right\vert _{\mathfrak{h}_{s}^{\perp
}}=\lambda _{s}\func{id}_{\mathfrak{h}_{s}^{\perp }}.
\end{equation}%
Since $\varphi _{s}^{t}$ and $\varphi _{s}$ vanish on $\mathfrak{q}%
_{s}^{\perp }$ and $\mathfrak{h}_{s}$, respectively, we obtain (\ref%
{phistphis}).

Let $h\in \Psi ^{R}\left( \mathbb{L}\right) $, then from (\ref{phihs}),
recall that 
\begin{equation}
\varphi _{h\left( s\right) }=\left( h^{\prime }\right) _{\ast }\circ \varphi
_{s}\circ \left( \func{Ad}_{h}^{-1}\right) _{\ast }.
\end{equation}%
It is then easy to see using (\ref{Kpsi}) and the invariance of the Killing
form on $\mathfrak{p}$ that 
\begin{equation}
\varphi _{h\left( s\right) }^{t}=\left( \func{Ad}_{h}\right) _{\ast }\circ
\varphi _{s}^{t}\circ \left( h^{\prime }\right) _{\ast }^{-1}.
\label{phiths}
\end{equation}%
In particular, we see that 
\begin{equation*}
\left( h^{\prime }\right) _{\ast }\mathfrak{q}_{s}=\mathfrak{q}_{h\left(
s\right) }\ \text{and }\left( \func{Ad}_{h}\right) _{\ast }\mathfrak{h}%
_{s}^{\perp }=\mathfrak{h}_{s}.
\end{equation*}%
Hence, 
\begin{eqnarray*}
\left. \varphi _{h\left( s\right) }\varphi _{h\left( s\right)
}^{t}\right\vert _{\mathfrak{q}_{h\left( s\right) }} &=&\left. \left(
h^{\prime }\right) _{\ast }\circ \varphi _{s}\varphi _{s}^{t}\circ \left(
h^{\prime }\right) _{\ast }^{-1}\right\vert _{\mathfrak{q}_{h\left( s\right)
}} \\
&=&\lambda _{s}\func{id}_{\mathfrak{q}_{h\left( s\right) }}
\end{eqnarray*}%
and so indeed, $\lambda _{s}=\lambda _{h\left( s\right) }.$
\end{proof}

\begin{example}
In the case of octonions, suppose we set $\varphi _{s}\left( \eta \right)
_{a}=k\varphi _{abc}\eta ^{bc}$ where $\eta \in \mathfrak{so}\left( 7\right)
\cong \Lambda ^{2}\left( \mathbb{R}^{7}\right) ^{\ast }$, $\varphi $ is the
defining $3$-form on $\mathbb{R}^{7},$ and $k\in \mathbb{R}$ is some
constant. Then, $\varphi _{s}^{t}\left( \gamma \right) _{ab}=k\varphi
_{abc}\gamma ^{c}$ where $\gamma \in \mathbb{R}^{7}\cong \func{Im}\mathbb{O}%
. $ Now, $\mathbb{R}^{7}$ is an irreducible representation of $\mathfrak{g}%
_{2} $, so the hypothesis of Lemma \ref{lemphisphist} is satisfied. In this
case, $\lambda _{s}=6k^{2}$ due to the contraction identities for $\varphi $ 
\cite{GrigorianG2Torsion1,karigiannis-2005-57}.
\end{example}

Consider the action of $\varphi _{s}^{t}\left( \mathfrak{l}^{\left( s\right)
}\right) \subset \mathfrak{p}$ on $\mathfrak{q}_{s}.$ Let $\xi ,\eta \in 
\mathfrak{q}_{s},$ then from (\ref{xiphi}),%
\begin{equation}
\varphi _{s}^{t}\left( \xi \right) \cdot \varphi _{s}\varphi _{s}^{t}\left(
\eta \right) -\varphi _{s}^{t}\left( \eta \right) \cdot \varphi _{s}\varphi
_{s}^{t}\left( \xi \right) =\varphi _{s}\left( \left[ \varphi _{s}^{t}\left(
\xi \right) ,\varphi _{s}^{t}\left( \eta \right) \right] _{\mathfrak{p}%
}\right) +\left[ \varphi _{s}\varphi _{s}^{t}\left( \xi \right) ,\varphi
_{s}\varphi _{s}^{t}\left( \eta \right) \right] ^{\left( s\right) },
\end{equation}%
and thus, 
\begin{equation}
\varphi _{s}^{t}\left( \xi \right) \cdot \eta -\varphi _{s}^{t}\left( \eta
\right) \cdot \xi =\frac{1}{\lambda _{s}}\varphi _{s}\left( \left[ \varphi
_{s}^{t}\left( \xi \right) ,\varphi _{s}^{t}\left( \eta \right) \right] _{%
\mathfrak{p}}\right) +\lambda _{s}\left[ \xi ,\eta \right] ^{\left( s\right)
}.  \label{phistxi}
\end{equation}%
We now show that $\varphi _{s}^{t}\left( \xi \right) \cdot \eta $ is
skew-symmetric when restricted to $\mathfrak{q}_{s}$ and then projected back
to $\mathfrak{q}_{s}.$

\begin{lemma}
\label{lemPhibrack}Suppose $\mathbb{L}$ is a loop and $s\in \mathbb{L}$,
such that the Killing form is non-degenerate and $\func{ad}^{\left( s\right)
}$- and $\mathfrak{p}$-invariant. Then, for any $\xi ,\eta \in \mathfrak{q}%
_{s}$, 
\begin{equation}
\pi _{\mathfrak{q}_{s}}\left( \varphi _{s}^{t}\left( \xi \right) \cdot \eta
\right) =-\pi _{\mathfrak{q}_{s}}\left( \varphi _{s}^{t}\left( \eta \right)
\cdot \xi \right) .  \label{piqsphit}
\end{equation}
\end{lemma}

\begin{proof}
Suppose $\xi ,\eta \in \mathfrak{q}_{s},$ then using the $\func{ad}^{\left(
s\right) }$- and $\mathfrak{p}$-invariance of the Killing form on $\mathfrak{%
l}^{\left( s\right) }$ and (\ref{phistxi}) we have 
\begin{eqnarray*}
\left\langle \varphi _{s}^{t}\left( \eta \right) \cdot \eta ,\xi
\right\rangle ^{\left( s\right) } &=&-\left\langle \eta ,\varphi
_{s}^{t}\left( \eta \right) \cdot \xi \right\rangle ^{\left( s\right) } \\
&=&-\left\langle \eta ,\varphi _{s}^{t}\left( \xi \right) \cdot \eta -\frac{1%
}{\lambda _{s}}\varphi _{s}\left( \left[ \varphi _{s}^{t}\left( \xi \right)
,\varphi _{s}^{t}\left( \eta \right) \right] _{\mathfrak{p}}\right) -\lambda
_{s}\left[ \xi ,\eta \right] ^{\left( s\right) }\right\rangle ^{\left(
s\right) } \\
&=&-\left\langle \eta ,\varphi _{s}^{t}\left( \xi \right) \cdot \eta
\right\rangle ^{\left( s\right) }+\frac{1}{\lambda _{s}}\left\langle \varphi
_{s}^{t}\left( \eta \right) ,\left[ \varphi _{s}^{t}\left( \xi \right)
,\varphi _{s}^{t}\left( \eta \right) \right] _{\mathfrak{p}}\right\rangle \\
&&-\lambda _{s}\left\langle \left[ \eta ,\eta \right] ^{\left( s\right)
},\xi \right\rangle ^{\left( s\right) } \\
&=&-\left\langle \eta ,\varphi _{s}^{t}\left( \xi \right) \cdot \eta
\right\rangle ^{\left( s\right) }=\left\langle \varphi _{s}^{t}\left( \xi
\right) \cdot \eta ,\eta \right\rangle ^{\left( s\right) } \\
&=&0.
\end{eqnarray*}%
Thus, we see that $\pi _{\mathfrak{q}_{s}}\left( \varphi _{s}^{t}\left( \eta
\right) \cdot \eta \right) =0,$ and hence (\ref{piqsphit}) holds.
\end{proof}

Taking the $\pi _{\mathfrak{q}_{s}}$ projection of (\ref{phistxi}) gives 
\begin{equation}
\pi _{\mathfrak{q}_{s}}\left( \varphi _{s}^{t}\left( \xi \right) \cdot \eta
\right) =\frac{1}{2\lambda _{s}}\varphi _{s}\left( \left[ \varphi
_{s}^{t}\left( \xi \right) ,\varphi _{s}^{t}\left( \eta \right) \right] _{%
\mathfrak{p}}+\lambda _{s}\varphi _{s}^{t}\left( \left[ \xi ,\eta \right]
^{\left( s\right) }\right) \right) .  \label{piqsact}
\end{equation}%
The relation (\ref{piqsact}) suggests that we can define a new bracket $%
\left[ \cdot ,\cdot \right] _{\varphi _{s}}$ on $\mathfrak{l}^{\left(
s\right) }$ using $\varphi _{s}$.

\begin{definition}
Suppose $\mathbb{L}$ satisfies the assumptions of Lemma \ref{lemPhibrack}.
Then, for $\xi ,\eta \in \mathfrak{l}^{\left( s\right) }$, define 
\begin{equation}
\left[ \xi ,\eta \right] _{\varphi _{s}}=\varphi _{s}\left( \left[ \varphi
_{s}^{t}\left( \xi \right) ,\varphi _{s}^{t}\left( \eta \right) \right] _{%
\mathfrak{p}}\right) .  \label{phisbrack}
\end{equation}
\end{definition}

This bracket restricts to $\mathfrak{q}_{s}$ and vanishes on $\mathfrak{q}%
_{s}^{\perp }$, so that $\mathfrak{q}_{s}^{\perp }$ is an abelian ideal with
respect to it. We can rewrite (\ref{piqsact}) as 
\begin{equation}
\pi _{\mathfrak{q}_{s}}\left( \varphi _{s}^{t}\left( \xi \right) \cdot \eta
\right) =\frac{1}{2\lambda _{s}}\left[ \xi ,\eta \right] _{\varphi _{s}}+%
\frac{\lambda _{s}}{2}\pi _{\mathfrak{q}_{s}}\left( \left[ \xi ,\eta \right]
^{\left( s\right) }\right) .  \label{piqsact1}
\end{equation}

\begin{example}
In the case of octonions, if, as before, we set $\varphi _{s}\left( \eta
\right) _{a}=k\varphi _{abc}\eta ^{bc}$ and $\left( \left[ \xi ,\gamma %
\right] ^{\left( s\right) }\right) _{a}=2\varphi _{abc}\xi ^{b}\gamma ^{c}$,
we find that $\left[ \cdot ,\cdot \right] _{\varphi _{s}}=3k^{3}\left[ \cdot
,\cdot \right] ^{\left( s\right) }.$ Then, recalling that $\lambda
_{s}=6k^{2}$, (\ref{piqsact1}) shows that in this case 
\begin{equation*}
\varphi _{s}^{t}\left( \xi \right) \cdot \gamma =\left( \frac{k}{4}%
+3k^{2}\right) \left[ \xi ,\gamma \right] ^{\left( s\right) },
\end{equation*}%
and to be consistent with the standard action of $\mathfrak{so}\left(
7\right) $ on $\mathbb{R}^{7}$, we must have 
\begin{equation*}
k\varphi _{abc}\xi ^{c}\gamma ^{b}=\left( \frac{k}{2}+6k^{2}\right) \varphi
_{abc}\xi ^{b}\gamma ^{c},
\end{equation*}%
which means that $6k^{2}+\frac{3}{2}k=0$ and therefore, $k=-\frac{1}{4}.$
This also implies that $\lambda _{s}=\frac{3}{8}$ in this case.
\end{example}

\begin{example}
If $\mathbb{L}$ is a Lie group, and $\Psi ^{R}\left( \mathbb{L}\right) $ is
the full group of pseudoautomorphism pairs, then $\mathfrak{p\cong aut}%
\left( \mathbb{L}\right) \oplus \mathfrak{l}$, where $\mathfrak{aut}\left( 
\mathbb{L}\right) $ is the Lie algebra of $\func{Aut}\left( \mathbb{L}%
\right) $ and $\mathfrak{l}$ is the Lie algebra of $\mathbb{L}.$ In this
case, $\varphi _{s}^{t}\varphi _{s}$ is just the projection to $\mathfrak{%
l\subset p},$ and thus $\lambda _{s}=1$ and $\left[ \cdot ,\cdot \right]
_{\varphi _{s}}=\left[ \cdot ,\cdot \right] ^{\left( s\right) }.$ Then (\ref%
{piqsact1}) just shows that $\mathfrak{l}$ acts on itself via the adjoint
representation.
\end{example}

\begin{remark}
Both of the above examples have the two brackets $\left[ \cdot ,\cdot \right]
_{\varphi _{s}}$and $\left[ \cdot ,\cdot \right] ^{\left( s\right) }$
proportional to one another. This is really means that $\mathfrak{l}^{\left(
s\right) }$ and $\mathfrak{h}_{s}^{\perp }$ have equivalent $\mathbb{L}$%
-algebra structures with $\varphi _{s}$ and $\varphi _{s}^{t}$ (up to a
constant factor) being the corresponding isomorphisms. It is not clear if
this is always the case.
\end{remark}

The bracket $\left[ \cdot ,\cdot \right] _{\varphi _{s}}$ has some
reasonable properties.

\begin{lemma}
\label{lemPhibrack2}Under the assumptions of Lemma \ref{lemPhibrack}, the
bracket $\left[ \cdot ,\cdot \right] _{\varphi _{s}}$ satisfies the
following properties. Let $\xi ,\eta ,\gamma \in \mathfrak{l}$, then

\begin{enumerate}
\item $\left\langle \left[ \xi ,\eta \right] _{\varphi _{s}},\gamma
\right\rangle ^{\left( s\right) }=-\left\langle \eta ,\left[ \xi ,\gamma %
\right] _{\varphi _{s}}\right\rangle ^{\left( s\right) }.$

\item For any $h\in \Psi ^{R}\left( \mathbb{L}\right) $, $\left[ \xi ,\eta %
\right] _{\varphi _{h\left( s\right) }}=\left( h^{\prime }\right) _{\ast }%
\left[ \left( h^{\prime }\right) _{\ast }^{-1}\xi ,\left( h^{\prime }\right)
_{\ast }^{-1}\eta \right] _{\varphi _{s}}.$
\end{enumerate}
\end{lemma}

\begin{proof}
The first property follows directly from the definition (\ref{phisbrack})
and the $\func{ad}$-invariance of the Killing form on $\mathfrak{p}.$
Indeed, 
\begin{eqnarray*}
\left\langle \left[ \xi ,\eta \right] _{\varphi _{s}},\gamma \right\rangle
^{\left( s\right) } &=&\left\langle \varphi _{s}\left( \left[ \varphi
_{s}^{t}\left( \xi \right) ,\varphi _{s}^{t}\left( \eta \right) \right] _{%
\mathfrak{p}}\right) ,\gamma \right\rangle ^{\left( s\right) } \\
&=&\left\langle \left[ \varphi _{s}^{t}\left( \xi \right) ,\varphi
_{s}^{t}\left( \eta \right) \right] _{\mathfrak{p}},\varphi _{s}^{t}\left(
\gamma \right) \right\rangle ^{\left( s\right) } \\
&=&-\left\langle \varphi _{s}^{t}\left( \eta \right) ,\left[ \varphi
_{s}^{t}\left( \xi \right) ,\varphi _{s}^{t}\left( \gamma \right) \right] _{%
\mathfrak{p}}\right\rangle ^{\left( s\right) } \\
&=&-\left\langle \eta ,\left[ \xi ,\gamma \right] _{\varphi
_{s}}\right\rangle ^{\left( s\right) }.
\end{eqnarray*}%
Now let $h\in \Psi ^{R}\left( \mathbb{L}\right) $, and then since $\left( 
\func{Ad}_{h}\right) _{\ast }$ is a Lie algebra automorphism of $\mathfrak{p}
$, we have 
\begin{eqnarray}
\left[ \xi ,\eta \right] _{\varphi _{h\left( s\right) }} &=&\varphi
_{h\left( s\right) }\left( \left[ \varphi _{h\left( s\right) }^{t}\left( \xi
\right) ,\varphi _{h\left( s\right) }^{t}\left( \eta \right) \right] _{%
\mathfrak{p}}\right)  \notag \\
&=&\left( h^{\prime }\right) _{\ast }\circ \varphi _{s}\circ \left( \func{Ad}%
_{h}^{-1}\right) _{\ast }\left( \left[ \left( \func{Ad}_{h}\right) _{\ast
}\left( \varphi _{s}^{t}\left( \left( h^{\prime }\right) _{\ast }^{-1}\left(
\xi \right) \right) \right) ,\left( \func{Ad}_{h}\right) _{\ast }\left(
\varphi _{s}^{t}\left( \left( h^{\prime }\right) _{\ast }^{-1}\left( \eta
\right) \right) \right) \right] _{\mathfrak{p}}\right)  \notag \\
&=&\left( h^{\prime }\right) _{\ast }\circ \varphi _{s}\left( \left[ \varphi
_{s}^{t}\left( \left( h^{\prime }\right) _{\ast }^{-1}\left( \xi \right)
\right) ,\varphi _{s}^{t}\left( \left( h^{\prime }\right) _{\ast
}^{-1}\left( \eta \right) \right) \right] _{\mathfrak{p}}\right)  \notag \\
&=&\left( h^{\prime }\right) _{\ast }\left[ \left( h^{\prime }\right) _{\ast
}^{-1}\xi ,\left( h^{\prime }\right) _{\ast }^{-1}\eta \right] _{\varphi
_{s}}.  \label{phibrackequi}
\end{eqnarray}%
Therefore, $\left[ \cdot ,\cdot \right] _{\varphi _{s}}$ is equivariant with
respect to transformations of $s$.
\end{proof}

\subsection{Darboux derivative}

\label{sectDarboux}Let $M$ be a smooth manifold and suppose $%
s:M\longrightarrow \mathbb{L}$ is a smooth map. The map $s$ can be used to
define a product on $\mathbb{L}$-valued maps from $M$ and a corresponding
bracket on $\mathfrak{l}$-valued maps. Indeed, let $A,B:M\longrightarrow 
\mathbb{L}$ and $\xi ,\eta :M\longrightarrow \mathfrak{l}$ be smooth maps,
then at each $x\in M$, define 
\begin{subequations}%
\label{maniproducts} 
\begin{eqnarray}
\left. A\circ _{s}B\right\vert _{x} &=&A_{x}\circ _{s_{x}}B_{x}\in \mathbb{L}
\\
\left. A/_{s}B\right\vert _{x} &=&A_{x}/_{s_{x}}B_{x}\in \mathbb{L} \\
\left. A\backslash _{s}B\right\vert _{x} &=&A_{x}\backslash _{s}B_{x}\in 
\mathbb{L} \\
\left. \left[ \xi ,\eta \right] ^{\left( s\right) }\right\vert _{x} &=&\left[
\xi _{x},\eta _{x}\right] ^{\left( s_{x}\right) }\in \mathfrak{l.}
\end{eqnarray}%
\end{subequations}%
In particular, the bracket $\left[ \cdot ,\cdot \right] ^{\left( s\right) }$
defines the map $b_{s}:M\longrightarrow \Lambda ^{2}\mathfrak{l}^{\ast
}\otimes \mathfrak{l.}$ We also have the corresponding associator $\left[
\cdot ,\cdot ,\cdot \right] ^{\left( s\right) }$ and the left-alternating
associator map $a_{s}:M\longrightarrow \Lambda ^{2}\mathfrak{l}^{\ast
}\otimes \mathfrak{l}^{\ast }\otimes \mathfrak{l}.$ Similarly, define the
map $\varphi _{s}:M\longrightarrow \mathfrak{p}^{\ast }\otimes \mathfrak{l}.$

Then, similarly as for maps to Lie groups, we may define the (right) \emph{%
Darboux derivative} $\theta _{s}$ of $s,$ which is an $\mathfrak{l}$-valued $%
1$-form on $M$ given by $s^{\ast }\theta $ \cite{SharpeBook}. In particular,
at every $x\in M$, 
\begin{equation}
\left. \left( \theta _{s}\right) \right\vert _{x}=\left( R_{s\left( x\right)
}^{-1}\right) _{\ast }\left. ds\right\vert _{x}.  \label{Darbouxf}
\end{equation}%
It is then clear that $\theta _{s}$, being a pullback of $\theta $,
satisfies the loop Maurer-Cartan structural equation (\ref{MCequation1}). In
particular, for any vectors $X,Y\in T_{x}M$, 
\begin{equation}
d\theta _{s}\left( X,Y\right) -\left[ \theta _{s}\left( X\right) ,\theta
_{s}\left( Y\right) \right] ^{\left( s\right) }=0.  \label{DarbouxMC}
\end{equation}

We can then calculate the derivatives of these maps. For clarity, we will
somewhat abuse notation, we will suppress the pushforwards of right
multiplication and their inverses (i.e. quotients) on $T\mathbb{L}$, so that
if $X\in T_{q}\mathbb{L}$, then we will write $X\circ _{s}A$ for $\left(
R_{A}^{\left( s\right) }\right) _{\ast }X.$

\begin{theorem}
\label{thmmaniDeriv}Let $M$ be a smooth manifold and let $x\in M$. Suppose $%
A,B,s\in C^{\infty }\left( M,\mathbb{L}\right) ,$ then 
\begin{equation}
d\left( A\circ _{s}B\right) =\left( dA\right) \circ _{s}B+A\circ _{s}\left(
dB\right) +\left[ A,B,\theta _{s}\right] ^{\left( s\right) }  \label{dAsB1}
\end{equation}%
and 
\begin{subequations}%
\label{dquots} 
\begin{eqnarray}
d\left( A/_{s}B\right)  &=&dA/_{s}B-\left( A/_{s}B\circ _{s}dB\right) /_{s}B-
\left[ A/_{s}B,B,\theta _{s}\right] ^{\left( s\right) }/_{s}B  \label{drquot}
\\
d\left( B\backslash _{s}A\right)  &=&B\backslash _{s}dA-B\backslash
_{s}\left( dB\circ _{s}\left( B\backslash _{s}A\right) \right) 
\label{dlquot} \\
&&-B\backslash _{s}\left[ B,B\backslash _{s}A,\theta _{s}\right] ^{\left(
s\right) }.  \notag
\end{eqnarray}%
\end{subequations}%
Suppose now $\xi ,\eta \in C^{\infty }\left( M,\mathfrak{l}\right) $, then 
\begin{equation}
d\left[ \xi ,\eta \right] ^{\left( s\right) }=\left[ d\xi ,\eta \right]
^{\left( s\right) }+\left[ \xi ,d\eta \right] ^{\left( s\right)
}+a_{s}\left( \xi ,\eta ,\theta _{s}\right) .  \label{dbrack}
\end{equation}

The $\mathfrak{l}\otimes \mathfrak{p}^{\ast }$-valued map $\varphi
_{s}:M\longrightarrow $ $\mathfrak{l}\otimes \mathfrak{p}^{\ast }$ satisfies 
\begin{equation}
d\varphi _{s}=\func{id}_{\mathfrak{p}}\cdot \theta _{s}-\left[ \varphi
_{s},\theta _{s}\right] ^{\left( s\right) },  \label{dphis0}
\end{equation}%
where $\func{id}_{\mathfrak{p}}\ $is the identity map of $\mathfrak{p}$ and $%
\cdot $ denotes the action of the Lie algebra $\mathfrak{p}$ on $\mathfrak{l}
$ given by (\ref{pactl1})
\end{theorem}

\begin{proof}
Let $V\in T_{x}M$ and let $x\left( t\right) $ be a curve on $M$ with $%
x\left( 0\right) =x$ and $\dot{x}\left( 0\right) =V.$ To show (\ref{dAsB1}),
first note that 
\begin{equation}
\left. d\left( A\circ _{s}B\right) \right\vert _{x}\left( V\right) =\left. 
\frac{d}{dt}\left( A_{x\left( t\right) }\circ _{s_{x\left( t\right)
}}B_{x\left( t\right) }\right) \right\vert _{t=0}.
\end{equation}%
However,%
\begin{eqnarray}
\left. \frac{d}{dt}\left( A_{x\left( t\right) }\circ _{s_{x\left( t\right)
}}B_{x\left( t\right) }\right) \right\vert _{t=0} &=&\left. \frac{d}{dt}%
\left( A_{x\left( t\right) }\circ _{s_{x}}B_{x}\right) \right\vert
_{t=0}+\left. \frac{d}{dt}\left( A_{x}\circ _{s_{x}}B_{x\left( t\right)
}\right) \right\vert _{t=0}  \notag \\
&&+\left. \frac{d}{dt}\left( A_{x}\circ _{s_{x\left( t\right) }}B_{x}\right)
\right\vert _{t=0}  \notag \\
&=&\left( R_{B_{x}}^{\left( s_{x}\right) }\right) _{\ast }\left.
dA\right\vert _{x}\left( V\right) +\left( L_{A_{x}}^{\left( s_{x}\right)
}\right) _{\ast }\left. dB\right\vert _{x}\left( V\right)  \label{dABprod0}
\\
&&+\left. \frac{d}{dt}\left( A_{x}\circ _{s_{x\left( t\right) }}B_{x}\right)
\right\vert _{t=0}  \notag
\end{eqnarray}%
and then, using Lemma \ref{lemQuotient}, 
\begin{eqnarray}
\left. \frac{d}{dt}\left( A_{x}\circ _{s_{x\left( t\right) }}B_{x}\right)
\right\vert _{t=0} &=&\left. \frac{d}{dt}\left( \left( A_{x}\cdot
B_{x}s_{x\left( t\right) }\right) /s_{x\left( t\right) }\right) \right\vert
_{t=0}  \notag \\
&=&\left. \frac{d}{dt}\left( \left( A_{x}\cdot B_{x}s_{x\left( t\right)
}\right) /s_{x}\right) \right\vert _{t=0}  \label{dABprod} \\
&&+\left. \frac{d}{dt}\left( \left( A_{x}\cdot B_{x}s_{x}\right) /s_{x}\cdot
s_{x\left( t\right) }\right) /s_{x}\right\vert _{t=0}.  \notag
\end{eqnarray}%
Looking at each term in (\ref{dABprod}), we have 
\begin{eqnarray*}
\left( A_{x}\cdot B_{x}s_{x\left( t\right) }\right) /s_{x} &=&\left(
A_{x}\cdot B_{x}\left( s_{x\left( t\right) }/s_{x}\cdot s_{x}\right) \right)
/s_{x} \\
&=&A_{x}\circ _{s_{x}}\left( B_{x}\circ _{s_{x}}\left( s_{x\left( t\right)
}/s_{x}\right) \right)
\end{eqnarray*}%
and%
\begin{equation*}
\left( \left( A_{x}\cdot B_{x}s_{x}\right) /s_{x}\cdot s_{x\left( t\right)
}\right) /s_{x}=\left( A_{x}\circ _{s_{x}}B_{x}\right) \circ _{s_{x}}\left(
s_{x\left( t\right) }/s_{x}\right) .
\end{equation*}%
Overall (\ref{dABprod0}) becomes, 
\begin{equation}
\left. \frac{d}{dt}\left( A_{x}\circ _{s_{x\left( t\right) }}B_{x}\right)
\right\vert _{t=0}=\left( \left( L_{A_{x}}^{\left( s_{x}\right) }\circ
L_{B_{x}}^{\left( s_{x}\right) }\right) _{\ast }-\left( L_{A_{x}\circ
_{s_{x}}B_{x}}^{\left( s_{x}\right) }\right) _{\ast }\right) \left(
R_{s_{x}}^{-1}\right) _{\ast }\left. ds\right\vert _{x}\left( V_{x}\right)
\end{equation}%
and hence we get (\ref{dAsB1}) using the definitions of $\theta _{s}$ and
the mixed associator (\ref{pxiqsol}).

Let us now show (\ref{dquots}). From Lemma \ref{lemQuotient}$,$ we find 
\begin{subequations}%
\label{dquots copy(1)} 
\begin{eqnarray}
d\left( A/B\right) &=&\left( dA\right) /B-\left( A/B\cdot dB\right) /B
\label{dquots1} \\
d\left( B\backslash A\right) &=&B\backslash \left( dA\right) -B\backslash
\left( dB\cdot B\backslash A\right) .  \label{dquots2}
\end{eqnarray}%
\end{subequations}%
Now if we instead have the quotient defined by $s$, using (\ref{rprodqright}%
), we have a modification: 
\begin{eqnarray}
d\left( A/_{s}B\right) &=&d\left( As/Bs\right) =d\left( As\right) /\left(
Bs\right) -\left( A/_{s}B\cdot d\left( Bs\right) \right) /\left( Bs\right) 
\notag \\
&=&dA/_{s}B+A\left( ds\right) /\left( Bs\right) -\left( A/_{s}B\cdot \left(
dB\right) s\right) /\left( Bs\right)  \notag \\
&&-\left( A/_{s}B\cdot B\left( ds\right) \right) /\left( Bs\right)  \notag \\
&=&dA/_{s}B-\left( A/_{s}B\circ _{s}dB\right) /_{s}B+\left( A\circ
_{s}\theta _{s}\right) /_{s}B  \notag \\
&&-\left( A/_{s}B\circ _{s}\left( B\circ _{s}\theta _{s}\right) \right)
/_{s}B  \notag \\
&=&dA/_{s}B-\left( A/_{s}B\circ _{s}dB\right) /_{s}B-\left[ A/_{s}B,B,\theta
_{s}\right] ^{\left( s\right) }/_{s}B.
\end{eqnarray}%
Similarly, for the left quotient, using (\ref{rprodqleft}), we have 
\begin{eqnarray}
d\left( B\backslash _{s}A\right) &=&d\left( \left( B\backslash As\right)
/s\right)  \notag \\
&=&d\left( B\backslash As\right) /s-\left( \left( \left( B\backslash
As\right) /s\right) \cdot ds\right) /s  \notag \\
&=&\left( B\backslash d\left( As\right) \right) /s-\left( B\backslash \left(
dB\cdot B\backslash As\right) \right) /s-\left( B\backslash _{s}A\right)
\circ _{s}\theta _{s}  \notag \\
&=&B\backslash _{s}dA+\left( B\backslash \left( A\left( ds\right) \right)
\right) /s-B\backslash _{s}\left( \left( dB\cdot B\backslash As\right)
/s\right)  \notag \\
&&-\left( B\backslash _{s}A\right) \circ _{s}\theta _{s}  \notag \\
&=&B\backslash _{s}dA-B\backslash _{s}\left( dB\circ _{s}\left( B\backslash
_{s}A\right) \right) +B\backslash _{s}\left( A\circ _{s}\theta _{s}\right) 
\notag \\
&&-\left( B\backslash _{s}A\right) \circ _{s}\theta _{s}
\end{eqnarray}%
However, using the mixed associator (\ref{pxiqsol}), 
\begin{eqnarray}
A\circ _{s}\theta _{s} &=&\left( B\circ _{s}\left( B\backslash _{s}A\right)
\right) \circ _{s}\theta _{s}  \notag \\
&=&B\circ _{s}\left( \left( B\backslash _{s}A\right) \circ _{s}\theta
_{s}\right) -\left[ B,B\backslash _{s}A,\theta _{s}\right] ^{\left( s\right)
},
\end{eqnarray}%
and thus, 
\begin{equation*}
d\left( B\backslash _{s}A\right) =B\backslash _{s}dA-B\backslash _{s}\left(
dB\circ _{s}\left( B\backslash _{s}A\right) \right) -B\backslash _{s}\left[
B,B\backslash _{s}A,\theta _{s}\right] ^{\left( s\right) }.
\end{equation*}

To show (\ref{dbrack}), note that 
\begin{eqnarray*}
\left. d\left( \left[ \xi ,\eta \right] ^{\left( s\right) }\right)
\right\vert _{x}\left( V\right) &=&\left. \frac{d}{dt}\left[ \xi _{x\left(
t\right) },\eta _{x\left( t\right) }\right] ^{\left( s_{x\left( t\right)
}\right) }\right\vert _{t=0} \\
&=&\left[ \left. d\xi \right\vert _{x}\left( V\right) ,\eta _{x}\right]
^{\left( s_{x}\right) }+\left[ \xi _{x},\left. d\eta \right\vert _{x}\right]
^{\left( s_{x}\right) } \\
&&+\left. \frac{d}{dt}\left[ \xi _{x},\eta _{x}\right] ^{\left( s_{x\left(
t\right) }\right) }\right\vert _{t=0}
\end{eqnarray*}%
However, using (\ref{db1}), the last term becomes 
\begin{equation*}
\left. \frac{d}{dt}\left[ \xi _{x},\eta _{x}\right] ^{\left( s_{x\left(
t\right) }\right) }\right\vert _{t=0}=a_{s_{x}}\left( \xi _{x},\eta
_{x},\left. \theta _{s}\right\vert _{x}\right)
\end{equation*}%
and hence we obtain (\ref{dbrack}).

Let us now show (\ref{dphis0}). From (\ref{actpl}), given $\gamma \in 
\mathfrak{p}$, setting $\hat{\gamma}\left( r\right) =\varphi _{r}\left(
\gamma \right) $ for each $r\in \mathbb{L}$, we have 
\begin{equation}
\left. d\hat{\gamma}\right\vert _{r}\left( \rho _{r}\left( \xi \right)
\right) =\gamma \cdot \xi -\left[ \hat{\gamma}\left( r\right) ,\xi \right]
^{\left( r\right) }
\end{equation}%
for some $\xi \in \mathfrak{l}.$ Now for at each $x\in M$ we have 
\begin{eqnarray}
\left. d\left( \varphi _{s}\left( \gamma \right) \right) \right\vert
_{x}\left( V\right) &=&\left. d\hat{\gamma}\right\vert _{s_{x}}\circ \left.
ds\right\vert _{x}\left( V\right)  \notag \\
&=&\left. d\hat{\gamma}\right\vert _{s_{x}}\left( \rho _{s_{x}}\left( \theta
_{s}\left( V\right) \right) \right)  \notag \\
&=&\gamma \cdot \theta _{s}\left( V\right) -\left[ \varphi _{s_{x}}\left(
\gamma \right) ,\theta _{s}\left( V\right) \right] ^{\left( s_{x}\right) }.
\end{eqnarray}%
Therefore, $d\varphi _{s}$ is given by 
\begin{equation}
d\varphi _{s}\left( \gamma \right) =\gamma \cdot \theta _{s}-\left[ \varphi
_{s}\left( \gamma \right) ,\theta _{s}\right] ^{\left( s\right) }.
\label{dphis0a}
\end{equation}
\end{proof}

\begin{remark}
Suppose $A$ and $B$ are now smooth maps from $M$ to $\mathbb{L}$. In the
case when $\mathbb{L}$ has the right inverse property, i.e. $A/B=AB^{-1}$
for any $A,B\in \mathbb{L}$, (\ref{dquots1}) becomes 
\begin{equation}
d\left( AB^{-1}\right) =\left( dA\right) B^{-1}-\left( AB^{-1}\cdot
dB\right) B^{-1}\text{.}  \label{drquot2}
\end{equation}%
However, from $d\left( BB^{-1}\right) =0$, we find that $d\left(
B^{-1}\right) =-B^{-1}\left( dB\cdot B^{-1}\right) $, and then expanding $%
d\left( AB^{-1}\right) $ using the product rule, and comparing with (\ref%
{drquot2}), we find 
\begin{equation}
\left( AB^{-1}\cdot dB\right) B^{-1}=A\left( B^{-1}\left( dB\cdot
B^{-1}\right) \right) ,  \label{drquot3}
\end{equation}%
which is an infinitesimal version of the right Bol identity (\ref{rightBol}%
). In particular, 
\begin{equation}
\left( B^{-1}\cdot dB\right) B^{-1}=B^{-1}\left( dB\cdot B^{-1}\right) .
\end{equation}%
Similarly, using (\ref{dlquot}), the left inverse property then implies an
infinitesimal left Bol identity.
\end{remark}

At each point $x\in M$, the map $s$ defines a stabilizer subgroup $\func{Stab%
}\left( s_{x}\right) =\func{Aut}\left( \mathbb{L},\circ _{s}\right) \subset
\Psi ^{R}\left( \mathbb{L}\right) $ $\ $with the corresponding Lie algebra $%
\mathfrak{h}_{s_{x}}.$ Similarly, we also have the orbit of $s_{x}$ given by 
$\mathcal{C}^{R}\left( \mathbb{L},\circ _{s_{x}}\right) \cong 
\faktor{\Psi ^{R}\left( \mathbb{L}\right)} {\func{Aut}\left( \mathbb{L},\circ
_{s_{x}}\right)}$, and the corresponding tangent space $\mathfrak{q}%
_{s_{x}}\cong \mathfrak{p/h}_{s_{x}}.$ Suppose $\left. \theta
_{s}\right\vert _{x}\in \mathfrak{q}_{s_{x}}$ for each $x\in M$. This of
course always holds if $\mathbb{L}$ is a $G$-loop, in which case $\mathfrak{q%
}_{s_{x}}=\mathfrak{l}^{\left( s_{x}\right) }.$ In this case, there exists a 
$\mathfrak{p}$-valued $1$-form $\Theta $ on $M$ such that $\theta
_{s}=\varphi _{s}\left( \Theta \right) .$ We can then characterize $\Theta $
in the following way.

\begin{theorem}
\label{thmThetaPhi}Suppose there exists $\Theta \in \Omega ^{1}\left( M,%
\mathfrak{p}\right) $ such that $\theta _{s}=\varphi _{s}\left( \Theta
\right) $. Then, for each $x\in M$, $\left. d\Theta -\frac{1}{2}\left[
\Theta ,\Theta \right] _{\mathfrak{p}}\right\vert _{x}\in \mathfrak{h}%
_{s_{x}}$, where $\left[ \cdot ,\cdot \right] _{\mathfrak{p}}$ is the Lie
bracket on $\mathfrak{p.}$
\end{theorem}

\begin{proof}
Consider $d\theta _{s}$ in this case. Using (\ref{dphis0a}), we have 
\begin{eqnarray}
d\theta _{s} &=&d\left( \varphi _{s}\left( \Theta \right) \right) =\left(
d\varphi _{s}\right) \left( \Theta \right) +\varphi _{s}\left( d\Theta
\right)  \notag \\
&=&-\Theta \cdot \theta _{s}+\left[ \varphi _{s}\left( \Theta \right)
,\theta _{s}\right] ^{\left( s\right) }.  \label{dthetasq}
\end{eqnarray}%
Note that the signs are switched in (\ref{dthetasq}) because we also have an
implied wedge product of $1$-forms. Overall, we have 
\begin{equation}
d\left( \varphi _{s}\left( \Theta \right) \right) =\varphi _{s}\left(
d\Theta \right) -\Theta \cdot \varphi _{s}\left( \Theta \right) +\left[
\varphi _{s}\left( \Theta \right) ,\varphi _{s}\left( \Theta \right) \right]
^{\left( s\right) },  \label{dphistheta}
\end{equation}%
however since $\theta _{s}=\varphi _{s}\left( \Theta \right) $, it satisfies
the Maurer-Cartan structural equation (\ref{DarbouxMC}), so we also have 
\begin{equation}
d\left( \varphi _{s}\left( \Theta \right) \right) =\frac{1}{2}\left[ \varphi
_{s}\left( \Theta \right) ,\varphi _{s}\left( \Theta \right) \right] .
\label{dphistheta2}
\end{equation}%
Equating (\ref{dphistheta}) and (\ref{dphistheta2}) , we find 
\begin{equation}
\varphi _{s}\left( d\Theta \right) =\Theta \cdot \varphi _{s}\left( \Theta
\right) -\frac{1}{2}\left[ \varphi _{s}\left( \Theta \right) ,\varphi
_{s}\left( \Theta \right) \right] ^{\left( s\right) }.  \label{dphistheta3}
\end{equation}%
However, from (\ref{xiphi}), we find that 
\begin{equation}
\Theta \cdot \varphi _{s}\left( \Theta \right) -\frac{1}{2}\left[ \varphi
_{s}\left( \Theta \right) ,\varphi _{s}\left( \Theta \right) \right] =\frac{1%
}{2}\varphi _{s}\left( \left[ \Theta ,\Theta \right] _{\mathfrak{p}}\right) .
\end{equation}%
Thus, we see that 
\begin{equation}
\varphi _{s}\left( d\Theta -\frac{1}{2}\left[ \Theta ,\Theta \right] _{%
\mathfrak{p}}\right) =0.  \label{dphistheta4}
\end{equation}
\end{proof}

\begin{remark}
In general, we can think of $d-\Theta $ as a connection on the trivial Lie
algebra bundle $M\times \mathfrak{p}$ with curvature contained in $\mathfrak{%
h}_{s\left( x\right) }$ for each $x\in M$. In general the spaces $\mathfrak{h%
}_{s\left( x\right) }$ need not be all of the same dimension, and thus may
this may not give a vector subbundle. On the other hand, if $\mathbb{L}$ is
a $G$-loop$,$ then we do get a subbundle.
\end{remark}

Now consider how $\theta _{s}$ behaves under the action of $\Psi ^{R}\left( 
\mathbb{L}\right) .$

\begin{lemma}
Suppose $h:M\longrightarrow \Psi ^{R}\left( \mathbb{L}\right) $ is a smooth
map, then 
\begin{equation}
\theta _{h\left( s\right) }=\left( h^{\prime }\right) _{\ast }\left( \varphi
_{s}\left( \theta _{h}^{\left( \mathfrak{p}\right) }\right) +\theta
_{s}\right) ,  \label{thetahf}
\end{equation}%
where $\theta _{h}^{\left( \mathfrak{p}\right) }=$ $h^{\ast }\theta ^{\left( 
\mathfrak{p}\right) }$ is the pullback of the left-invariant Maurer-Cartan
form $\theta ^{\left( \mathfrak{p}\right) }$ on $\Psi ^{R}\left( \mathbb{L}%
\right) .$
\end{lemma}

\begin{proof}
Suppose $h:M\longrightarrow \Psi ^{R}\left( \mathbb{L}\right) $ is a smooth
map, then consider $\theta _{h\left( s\right) }$. We then have 
\begin{eqnarray*}
\left. \left( \theta _{h\left( s\right) }\right) \right\vert _{x} &=&\left(
R_{h\left( s\left( x\right) \right) }^{-1}\right) _{\ast }\left. d\left(
h\left( s\right) \right) \right\vert _{x} \\
&=&\left( R_{h\left( s\left( x\right) \right) }^{-1}\right) _{\ast }\left.
\left( \left( dh\right) \left( s\right) +h\left( ds\right) \right)
\right\vert _{x}.
\end{eqnarray*}%
Consider each term. Using simplified notation, we have 
\begin{eqnarray*}
\left( dh\right) \left( s\right) /h\left( s\right) &=&\left( h^{\prime
}\right) _{\ast }\left( \left( h^{-1}dh\right) \left( s\right) /s\right) \\
\left( R_{h\left( s\left( x\right) \right) }^{-1}\right) _{\ast }\left.
\left( h\left( ds\right) \right) \right\vert _{x} &=&\left( h^{\prime
}\right) _{\ast }\left( \theta _{s}\right) .
\end{eqnarray*}%
Thus, 
\begin{equation*}
\left( R_{h\left( s\left( x\right) \right) }^{-1}\right) _{\ast }\left.
\left( dh\right) \left( s\right) \right\vert _{x}=\left( h\left( x\right)
^{\prime }\right) _{\ast }\varphi _{s\left( x\right) }\left( \left. \theta
_{h}^{\left( \mathfrak{p}\right) }\right\vert _{x}\right) ,
\end{equation*}%
and hence we get (\ref{thetahf}).
\end{proof}

If we have another smooth map $f:M\longrightarrow \mathbb{L}$, using right
multiplication with respect to $\circ _{s\left( x\right) }$, we can define a
modified Darboux derivative $\theta _{f}^{\left( s\right) }$ with respect to 
$s$:%
\begin{equation}
\left. \left( \theta _{f}^{\left( s\right) }\right) \right\vert _{x}=\left(
R_{f\left( x\right) }^{\left( s\left( x\right) \right) }\right) _{\ast
}^{-1}\left. df\right\vert _{x}.
\end{equation}%
Note that this is now no longer necessarily a pullback of $\theta $ and
hence may not satisfy the Maurer-Cartan equation. Adopting simplified
notation, we have the following: 
\begin{eqnarray}
d\left( fs\right) /fs &=&\left( df\cdot s+f\cdot ds\right) /fs  \notag \\
&=&df/_{s}f+\func{Ad}_{f}^{\left( s\right) }\theta _{s}  \label{thetafs}
\end{eqnarray}%
Hence, 
\begin{equation}
\theta _{f}^{\left( s\right) }=\theta _{fs}-\left( \func{Ad}_{f}^{\left(
s\right) }\right) _{\ast }\theta _{s}.  \label{thetafs2}
\end{equation}

\begin{lemma}
Suppose $f,s\in C^{\infty }\left( M,\mathbb{L}\right) ,$ then 
\begin{equation}
d\theta _{f}^{\left( s\right) }=\frac{1}{2}\left[ \theta _{f}^{\left(
s\right) },\theta _{f}^{\left( s\right) }\right] ^{\left( fs\right) }-\left(
R_{f}^{\left( s\right) }\right) _{\ast }^{-1}\left[ \theta _{f}^{\left(
s\right) },f,\theta _{s}\right] ^{\left( s\right) }.  \label{dthetafs}
\end{equation}
\end{lemma}

\begin{proof}
Applying the exterior derivative to (\ref{thetafs2}) and then the structural
equation for $\theta _{fs}$, we have%
\begin{equation}
d\theta _{f}^{\left( s\right) }=\frac{1}{2}\left[ \theta _{fs},\theta _{fs}%
\right] ^{\left( fs\right) }-d\left( \left( \func{Ad}_{f}^{\left( s\right)
}\right) _{\ast }\theta _{s}\right) .  \label{dthetafs2}
\end{equation}%
From Lemma \ref{lemdtAd}, we can see that for $\xi \in \mathfrak{l}$, 
\begin{eqnarray}
d\left( \func{Ad}_{f}^{\left( s\right) }\right) _{\ast }\xi &=&\left[ \theta
_{f}^{\left( s\right) },\left( \func{Ad}_{f}^{\left( s\right) }\right)
_{\ast }\xi \right] ^{\left( fs\right) }-\left( R_{f}^{\left( s\right)
}\right) _{\ast }^{-1}\left[ \theta _{f}^{\left( s\right) },f,\xi \right]
^{\left( s\right) }  \notag \\
&&+\left( R_{f}^{\left( s\right) }\right) _{\ast }^{-1}\left[ f,\xi ,\theta
_{s}\right] ^{\left( s\right) }  \label{dAdfs1} \\
&&-\left( R_{f}^{\left( s\right) }\right) _{\ast }^{-1}\left[ \left( \func{Ad%
}_{f}^{\left( s\right) }\right) _{\ast }\xi ,f,\theta _{s}\right] ^{\left(
s\right) },  \notag
\end{eqnarray}%
and hence 
\begin{eqnarray}
d\left( \func{Ad}_{f}^{\left( s\right) }\right) _{\ast }\wedge \theta _{s}
&=&\left[ \theta _{f}^{\left( s\right) },\left( \func{Ad}_{f}^{\left(
s\right) }\right) _{\ast }\theta _{s}\right] ^{\left( fs\right) }-\left(
R_{f}^{\left( s\right) }\right) _{\ast }^{-1}\left[ \theta _{f}^{\left(
s\right) },f,\theta _{s}\right] ^{\left( s\right) }  \notag \\
&&-\left( R_{f}^{\left( s\right) }\right) _{\ast }^{-1}\left[ f,\theta
_{s},\theta _{s}\right] ^{\left( s\right) }  \label{dAdfs2} \\
&&+\left( R_{f}^{\left( s\right) }\right) _{\ast }^{-1}\left[ \left( \func{Ad%
}_{f}^{\left( s\right) }\right) _{\ast }\theta _{s},f,\theta _{s}\right]
^{\left( s\right) },  \notag
\end{eqnarray}%
where wedge products are implied. Now, using the structural equation and (%
\ref{Adbrack1}), we find%
\begin{eqnarray}
\left( \func{Ad}_{f}^{\left( s\right) }\right) _{\ast }d\theta _{s} &=&\frac{%
1}{2}\left( \func{Ad}_{f}^{\left( s\right) }\right) _{\ast }\left[ \theta
_{s},\theta _{s}\right] ^{\left( s\right) }  \notag \\
&=&\frac{1}{2}\left[ \left( \func{Ad}_{f}^{\left( s\right) }\right) _{\ast
}\theta _{s},\left( \func{Ad}f\right) _{\ast }\theta _{s}\right] ^{\left(
fs\right) }  \notag \\
&&-\left( R_{f}^{\left( s\right) }\right) _{\ast }^{-1}\left[ \left( \func{Ad%
}_{f}^{\left( s\right) }\right) _{\ast }\theta _{s},f,\theta _{s}\right]
^{\left( s\right) }  \notag \\
&&+\left( R_{f}^{\left( s\right) }\right) _{\ast }^{-1}\left[ f,\theta
_{s},\theta _{s}\right] ^{\left( s\right) }.  \label{Adbracktheta}
\end{eqnarray}%
Combining (\ref{dAdfs2}) and (\ref{Adbracktheta}), we see that 
\begin{eqnarray}
d\left( \left( \func{Ad}_{f}^{\left( s\right) }\right) _{\ast }\theta
_{s}\right) &=&d\left( \func{Ad}_{f}^{\left( s\right) }\right) _{\ast
}\wedge \theta _{s}+\left( \func{Ad}_{f}^{\left( s\right) }\right) _{\ast
}d\theta _{s}  \notag \\
&=&\left[ \theta _{f}^{\left( s\right) },\left( \func{Ad}f\right) _{\ast
}\theta _{s}\right] ^{\left( fs\right) }+\frac{1}{2}\left[ \left( \func{Ad}%
_{f}^{\left( s\right) }\right) _{\ast }\theta _{s},\left( \func{Ad}f\right)
_{\ast }\theta _{s}\right] ^{\left( fs\right) }  \notag \\
&&-\left( R_{f}^{\left( s\right) }\right) _{\ast }^{-1}\left[ \theta
_{f}^{\left( s\right) },f,\theta _{s}\right] ^{\left( s\right) }  \notag \\
&=&\frac{1}{2}\left[ \theta _{fs},\theta _{fs}\right] ^{\left( fs\right) }-%
\frac{1}{2}\left[ \theta _{f}^{\left( s\right) },\theta _{f}^{\left(
s\right) }\right] ^{\left( fs\right) }  \label{dadtheta} \\
&&-\left( R_{f}^{\left( s\right) }\right) _{\ast }^{-1}\left[ \theta
_{f}^{\left( s\right) },f,\theta _{s}\right] ^{\left( s\right) }.  \notag
\end{eqnarray}%
Thus, overall, substituting (\ref{dadtheta}) into (\ref{dthetafs2}), we
obtain (\ref{dthetafs}).
\end{proof}

For Lie groups, $\theta _{f}$ determines $f$ up to right translation by a
constant element, however in the non-associative case this is not
necessarily true.

\begin{lemma}
\label{lemThetauniq}Let $M$ be a connected manifold and suppose $%
A,B:M\longrightarrow \mathbb{L}$ be smooth maps. Then, $A=BC$\ for some
constant $C\in \mathbb{L}$ if and only if 
\begin{equation}
\theta _{A}=\theta _{B}^{\left( B\backslash A\right) }.  \label{thetaAB}
\end{equation}
\end{lemma}

\begin{proof}
From (\ref{thetafs2}), 
\begin{equation*}
\theta _{A}-\theta _{B}^{\left( B\backslash A\right) }=\left( \func{Ad}%
_{B}^{\left( B\backslash A\right) }\right) _{\ast }\theta _{B\backslash A},
\end{equation*}%
and thus, $B\backslash A$ is constant if and only if (\ref{thetaAB}) holds.
\end{proof}

In particular, if $B\backslash A\in \mathcal{N}^{R}\left( \mathbb{L}\right) $%
, then $\theta _{B}^{\left( B\backslash A\right) }=\theta _{B}$, and hence $%
\theta _{A}=\theta _{B}$. If $\mathbb{L}$ is associative, then of course $%
\theta _{B}^{\left( A\right) }=\theta _{B}$ for any $A,B$, and we get the
standard result \cite{SharpeBook}.

We can also get a version of the structural equation integration theorem. In
particular, the question is whether an $\mathfrak{l}$-valued $1$-form that
satisfies the structural equation is the Darboux derivative of some $\mathbb{%
L}$-valued function.

\begin{lemma}
\label{lemAlphastruct}Suppose $M$ is a smooth manifold and $\mathbb{L}$ a
smooth loop. Let $s\in C^{\infty }\left( M,\mathbb{L}\right) $ and $\alpha
\in \Omega ^{1}\left( M,\mathfrak{l}\right) $ satisfy the structural
equation 
\begin{equation}
d\alpha -\frac{1}{2}\left[ \alpha ,\alpha \right] ^{\left( s\right) }=0,
\label{alphastructeq}
\end{equation}%
then 
\begin{equation}
\left[ \alpha ,\alpha ,\alpha -\theta _{s}\right] ^{\left( s\right) }=0,
\label{alphastruct2}
\end{equation}%
where wedge products are implied.
\end{lemma}

\begin{proof}
Applying $d$ to (\ref{alphastructeq}) we have 
\begin{eqnarray*}
0 &=&d\left[ \alpha ,\alpha \right] ^{\left( s\right) } \\
&=&\left[ d\alpha ,\alpha \right] ^{\left( s\right) }-\left[ \alpha ,d\alpha %
\right] ^{\left( s\right) }+\left[ \alpha ,\alpha ,\theta _{s}\right]
^{\left( s\right) } \\
&=&\left[ \left[ \alpha ,\alpha \right] ,\alpha \right] +\left[ \alpha
,\alpha ,\theta _{s}\right] ^{\left( s\right) } \\
&=&-\left[ \alpha ,\alpha ,\alpha \right] ^{\left( s\right) }+\left[ \alpha
,\alpha ,\theta _{s}\right] ^{\left( s\right) },
\end{eqnarray*}%
where we have used (\ref{db1}) and in the last line an analog of (\ref{Jac3}%
).
\end{proof}

\begin{theorem}
\label{thmLoopCartan}Suppose $M$ be a connected and simply-connected smooth
manifold and $\mathbb{L}$ a smooth loop. Let $s\in C^{\infty }\left( M,%
\mathbb{L}\right) $ and $\alpha \in \Omega ^{1}\left( M,\mathfrak{l}\right) $
is such that 
\begin{equation}
d\alpha -\frac{1}{2}\left[ \alpha ,\alpha \right] ^{\left( s\right) }=0,
\label{alphastruct}
\end{equation}%
and 
\begin{equation}
\left( \func{Ad}_{s}^{-1}\right) _{\ast }\left( \alpha -\theta _{s}\right)
\in \Omega ^{1}\left( M,T_{1}\mathcal{N}^{R}\left( \mathbb{L}\right) \right)
.  \label{cartanhyp}
\end{equation}%
Then, there exists a function $f\in C^{\infty }\left( M,\mathcal{N}%
^{R}\left( \mathbb{L}\right) \right) $ such that $\alpha =\theta _{sf}.$
Moreover, $f$ is unique up to right multiplication by a constant element of $%
\mathcal{N}^{R}\left( \mathbb{L}\right) .$
\end{theorem}

\begin{proof}
Modifying the standard technique \cite{SharpeBook,WarnerBook}, let $%
N=M\times \mathcal{N}^{R}\left( \mathbb{L}\right) \subset M\times \mathbb{L}%
. $ Define the projection map $\pi _{M}:N\longrightarrow M$ and the map 
\begin{eqnarray*}
L_{s} &:&N\longrightarrow \mathbb{L} \\
\left( x,p\right) &\mapsto &s\left( x\right) p
\end{eqnarray*}%
Given the Maurer-Cartan form $\theta $ on $\mathbb{L}$ and $\alpha \in
\Omega ^{1}\left( M,\mathfrak{l}\right) $, define $\beta \in \Omega
^{1}\left( N,\mathfrak{l}\right) $ by%
\begin{equation}
\beta =\pi _{M}^{\ast }\alpha -\left( L_{s}\right) ^{\ast }\theta .
\label{beta}
\end{equation}%
Then, at each point $\left( x,p\right) \in N$, define $\mathcal{D}_{\left(
x,p\right) }=\left. \ker \beta \right\vert _{\left( x,p\right) }$. We can
then see that this is a distribution on $N$ of rank $\dim M$. Let $\left(
v,w\right) \in T_{\left( x,p\right) }N$, where we consider $w\in $ $T_{p}%
\mathcal{N}^{R}\left( \mathbb{L}\right) \subset T_{p}\mathbb{L}.$ Then, 
\begin{equation}
\beta _{\left( x,p\right) }\left( v,w\right) =\alpha _{x}\left( v\right)
-\theta _{s\left( x\right) p}\left( \left( L_{s}\right) _{\ast }\left(
v,w\right) \right) .  \label{betavw}
\end{equation}%
Now, let $x\left( t\right) $ be a curve on $M$ with $x\left( 0\right) =x$
and $\dot{x}\left( 0\right) =v,$ and $p\left( t\right) $ a curve in $%
\mathcal{N}^{R}\left( \mathbb{L}\right) \subset \mathbb{L}$ with $p\left(
0\right) =p$ and $\dot{p}\left( 0\right) =w$. Then, using the fact that $p$
is in the right nucleus,%
\begin{eqnarray*}
\theta _{s\left( x\right) p}\left( \left( L_{s}\right) _{\ast }\left(
v,w\right) \right) &=&\left. \frac{d}{dt}\faktor{\left( s\left( x\left(
t\right) \right) p\left( t\right) \right)} {\left( s\left( x\right) p\right)}
\right\vert _{t=0} \\
&=&\left. \frac{d}{dt}\faktor{s\left( x\left( t\right) \right)} {s\left(
x\right)} \right\vert _{t=0}+\left. \frac{d}{dt}\faktor{\left( s\left(
x\right) \left( \faktor{p\left( t\right)}{p}\right) \right)} {s\left(
x\right)} \right\vert _{t=0} \\
&=&\left. \theta _{s}\left( v\right) \right\vert _{x}+\left( \func{Ad}%
_{s\left( x\right) }\right) _{\ast }w.
\end{eqnarray*}%
So overall, 
\begin{equation}
\beta _{\left( x,p\right) }\left( v,w\right) =\left( \alpha -\theta
_{s}\right) _{x}\left( v\right) -\left( \func{Ad}_{s\left( x\right) }\right)
_{\ast }w.
\end{equation}%
Hence, $\left( v,w\right) \in \mathcal{D}_{\left( x,p\right) }$ if and only
if $\left( \alpha -\theta _{s}\right) _{x}\left( v\right) =\left( \func{Ad}%
_{s\left( x\right) }\right) _{\ast }w.$ Now, consider $\left. \left( \pi
_{M}\right) _{\ast }\right\vert _{\left( x,p\right) }:\mathcal{D}_{\left(
x,p\right) }\longrightarrow T_{x}M$. Suppose $\left. \left( \pi _{M}\right)
_{\ast }\right\vert _{\left( x,p\right) }\left( v,w\right) =0.$ Then, $v=0$,
and since $\left( \alpha -\theta _{s}\right) _{x}\left( v\right) =\left( 
\func{Ad}_{s\left( x\right) }\right) _{\ast }w,$ we have $w=0$. Thus $\left.
\left( \pi _{M}\right) _{\ast }\right\vert _{\left( x,p\right) }$ is
injective on $\mathcal{D}_{\left( x,p\right) }.$ On the other hand, it is
also clearly surjective, since if given $v\in T_{x}M$, then $\left( v,\left( 
\func{Ad}_{s\left( x\right) }^{-1}\right) _{\ast }\left( \left( \alpha
-\theta _{s}\right) _{x}\left( v\right) \right) \right) \in \mathcal{D}%
_{\left( x,p\right) }.$ Overall, $\left. \left( \pi _{M}\right) _{\ast
}\right\vert _{\left( x,p\right) }$ is a bijection from $\mathcal{D}_{\left(
x,p\right) }$ to $T_{x}M$, so in particular, $\dim \mathcal{D}_{\left(
x,p\right) }=\dim M$ and thus $\mathcal{D}$ is a distribution of rank $\dim
M.$

Now let us show that $\mathcal{D}$ is involutive. We have 
\begin{eqnarray}
\left. d\beta \right\vert _{\left( x,p\right) } &=&\left. \pi _{M}^{\ast
}d\alpha \right\vert _{\left( x,p\right) }-\left. \left( L_{s}\right) ^{\ast
}d\theta \right\vert _{\left( x,p\right) }  \notag \\
&=&\frac{1}{2}\left. \pi _{M}^{\ast }\left[ \alpha ,\alpha \right] ^{\left(
s\right) }\right\vert _{\left( x,p\right) }-\frac{1}{2}\left. \left(
L_{s}\right) ^{\ast }\left[ \theta ,\theta \right] \right\vert _{\left(
x,p\right) }  \notag \\
&=&\frac{1}{2}\left[ \left. \pi _{M}^{\ast }\alpha \right\vert _{\left(
x,p\right) },\left. \pi _{M}^{\ast }\alpha \right\vert _{\left( x,p\right) }%
\right] ^{s\left( x\right) }  \label{dbeta} \\
&&-\frac{1}{2}\left[ \left. \left( L_{s}\right) ^{\ast }\theta \right\vert
_{\left( x,p\right) },\left. \left( L_{s}\right) ^{\ast }\theta \right\vert
_{\left( x,p\right) }\right] ^{s\left( x\right) p}.  \notag
\end{eqnarray}%
Note however that because $p\in \mathcal{N}^{R}\left( \mathbb{L}\right) ,$
we have $\left[ \cdot ,\cdot \right] ^{s\left( x\right) }=\left[ \cdot
,\cdot \right] ^{s\left( x\right) p}.$ So overall, using (\ref{beta}), we
get 
\begin{equation*}
\left. d\beta \right\vert _{\left( x,p\right) }=\frac{1}{2}\left[ \left.
\beta \right\vert _{\left( x,p\right) },\left. \beta \right\vert _{\left(
x,p\right) }\right] ^{s\left( x\right) }+\left[ \left. \beta \right\vert
_{\left( x,p\right) },\left. \left( L_{s}\right) ^{\ast }\theta \right\vert
_{\left( x,p\right) }\right] ^{s\left( x\right) }.
\end{equation*}%
Thus, $d\beta =0$ whenever $\beta =0$, and hence $\mathcal{D=}\ker \beta $
is involutive, and by the Frobenius Theorem, $\mathcal{D}$ is integrable.
Let $\mathcal{L}$ be a leaf through the point $\left( x,p\right) \in N.$
Then, $\pi _{M}$ induced a local diffeomorphism from a neighborhood to $%
\left( x,p\right) $ to some neighborhood of $x\in M.$ Then, let $%
F:U\longrightarrow \mathcal{L}$ be the inverse map, such that $F\left(
y\right) =\left( y,f\left( y\right) \right) $ for some $f:U\longrightarrow 
\mathcal{N}^{R}\left( \mathbb{L}\right) .$ By definition, $F^{\ast }\beta =0$%
, so 
\begin{eqnarray*}
0 &=&F^{\ast }\beta \\
&=&F^{\ast }\left( \pi _{M}^{\ast }\alpha -\left( L_{s}\right) ^{\ast
}\theta \right) \\
&=&\alpha -\left( L_{s}\circ f\right) ^{\ast }\theta
\end{eqnarray*}%
Hence, on $U$, $\alpha =\theta _{sf}$.

It is obvious that the distribution $\mathcal{D}$ is right-invariant with
respect to $\mathcal{N}^{R}\left( \mathbb{L}\right) $, then proceeding in
the same way as for Lie groups, we find that in fact that when $M$ is
connected and simply-connected, the function $f$ extends to the whole
manifold.

Now suppose $f,g\in C^{\infty }\left( M,\mathcal{N}^{R}\left( \mathbb{L}%
\right) \right) $ such that $\theta _{sf}=\theta _{sg}.$ Then using (\ref%
{thetafs}), but with roles of $s$ and $f$ reversed, we find 
\begin{equation*}
\theta _{sf}=\theta _{s}+\left( \func{Ad}_{s}\right) _{\ast }\theta _{f},
\end{equation*}%
and similarly for $g$. Hence, we see that $\theta _{f}=\theta _{g}.$ Using
Lemma \ref{lemThetauniq} for Lie groups, we find that $f=gC$ for some
constant $C\in \mathcal{N}^{R}\left( \mathbb{L}\right) .$
\end{proof}

\begin{remark}
In the case when $\mathbb{L}$ is a group, Theorem \ref{thmLoopCartan}
reduces to the well-known analogous result for groups since the function $s$
can be taken to be arbitrary. In particular, the hypothesis (\ref{cartanhyp}%
) is automatically satisfied in that case. On the other hand, for the loop
of unit octonions, this theorem becomes trivial. In this case, $\mathcal{N}%
^{R}\left( \mathbb{L}\right) \cong \mathbb{Z}_{2},$ so the hypothesis (\ref%
{cartanhyp}) immediately implies that $\alpha =\theta _{s},$ even without
using the equation (\ref{alphastruct}). However, under certain additional
assumptions about $\alpha $ and $s$, (\ref{alphastruct}) may actually imply (%
\ref{cartanhyp}). Generally, (\ref{cartanhyp}) is stronger than (\ref%
{alphastruct2}), which we know holds for any $\alpha \in \Omega ^{1}\left( M,%
\mathfrak{l}\right) $ that satisfies (\ref{alphastruct}). To bridge the gap
between (\ref{alphastruct2}) and (\ref{cartanhyp}), additional properties of 
$\mathbb{L}$ and $\alpha $ are needed.
\end{remark}

\begin{corollary}
\label{corLoopCartan}Suppose $M$ be a connected and simply-connected smooth
manifold and $\mathbb{L}$ a smooth loop such that $\dim \left( \mathcal{N}%
^{R}\left( \mathbb{L}\right) \right) =\dim \left( \mathcal{N}^{R}\left( 
\mathfrak{l}\right) \right) .$ Also suppose that $s\in C^{\infty }\left( M,%
\mathbb{L}\right) $ and $\alpha \in \Omega ^{1}\left( M,\mathfrak{l}\right) $
are such that

\begin{enumerate}
\item $d\alpha -\frac{1}{2}\left[ \alpha ,\alpha \right] ^{\left( s\right)
}=0,$

\item $\left. \alpha \right\vert _{x}:T_{x}M\longrightarrow \mathfrak{l}$ is
surjective for every $x\in M,$

\item $T_{x}M\cong \ker \left. \alpha \right\vert _{x}+\ker \left( \left.
\theta _{s}\right\vert _{x}-\left. \alpha \right\vert _{x}\right) $ for
every $x\in M$,

\item $s_{x}\in \mathcal{C}^{R}\left( \mathbb{L}\right) $ for every $x\in M$.
\end{enumerate}

Then, there exists a function $f\in C^{\infty }\left( M,\mathcal{N}%
^{R}\left( \mathbb{L}\right) \right) $ such that $\alpha =\theta _{sf}$ with 
$f$ unique up to right multiplication by a constant element of $\mathcal{N}%
^{R}\left( \mathbb{L}\right) .$
\end{corollary}

\begin{proof}
Since $\alpha $ satisfies (\ref{alphastruct}), from Lemma \ref%
{lemAlphastruct} we know that it also satisfies (\ref{alphastruct2}).
Suppose $X,Y,Z\in T_{x}M$, such that $Z\in \ker \left. \alpha \right\vert
_{x}$. Then, from (\ref{alphastruct2}) we obtain 
\begin{equation}
\left[ \alpha \left( X\right) ,\alpha \left( Y\right) ,\left( \alpha -\theta
_{s_{x}}\right) Z\right] ^{\left( s_{x}\right) }-\left[ \alpha \left(
Y\right) ,\alpha \left( X\right) ,\left( \alpha -\theta _{s_{x}}\right) Z%
\right] ^{\left( s_{x}\right) }=0.  \label{alphastructassoc}
\end{equation}%
However, since $T_{x}M\cong \ker \left. \alpha \right\vert _{x}+\ker \left(
\left. \theta _{s}\right\vert _{x}-\left. \alpha \right\vert _{x}\right) $,
this is true for any $Z\in T_{x}M.$ Since$\left. \alpha \right\vert _{x}$ is
surjective, we hence find that for any $\xi ,\eta \in \mathfrak{l},$ 
\begin{equation}
\left[ \xi ,\eta ,\left( \alpha -\theta _{s_{x}}\right) Z\right] ^{\left(
s_{x}\right) }-\left[ \eta ,\xi ,\left( \alpha -\theta _{s_{x}}\right) Z%
\right] ^{\left( s_{x}\right) }=0.  \label{alphastructassoc2}
\end{equation}%
Now, since $s_{x}\in \mathcal{C}^{R}\left( \mathbb{L}\right) ,$ it is the
right companion of some $h\in \Psi ^{R}\left( \mathbb{L}\right) $, thus
applying $\left( h^{\prime }\right) _{\ast }^{-1}$ to (\ref%
{alphastructassoc2}), and using (\ref{loopalghom2}), we find that for any $%
\xi ,\eta \in \mathfrak{l},$ 
\begin{equation*}
\left[ \xi ,\eta ,\left( h^{\prime }\right) _{\ast }^{-1}\left( \left(
\alpha -\theta _{s_{x}}\right) Z\right) \right] ^{\left( 1\right) }-\left[
\eta ,\xi ,\left( h^{\prime }\right) _{\ast }^{-1}\left( \left( \alpha
-\theta _{s_{x}}\right) Z\right) \right] ^{\left( 1\right) }=0.
\end{equation*}%
Thus, we see that for any $Z\in T_{x}M$, $\left( h^{\prime }\right) _{\ast
}^{-1}\left( \left( \alpha -\theta _{s_{x}}\right) Z\right) \in \mathcal{N}%
^{R}\left( \mathfrak{l}\right) .$ We know that $\mathcal{\ }T_{1}\mathcal{N}%
^{R}\left( \mathbb{L}\right) \subset \mathcal{N}^{R}\left( \mathfrak{l}%
\right) $, however by hypothesis, their dimensions are equal, so in fact, $%
T_{1}\mathcal{N}^{R}\left( \mathbb{L}\right) =\mathcal{N}^{R}\left( 
\mathfrak{l}\right) .$ Thus, $\left( h^{\prime }\right) _{\ast }^{-1}\left(
\left( \alpha -\theta _{s_{x}}\right) Z\right) \in T_{1}\mathcal{N}%
^{R}\left( \mathbb{L}\right) $ and hence, from (\ref{nuclearaction}), $%
\left( \func{Ad}_{s\left( x\right) }^{-1}\right) _{\ast }\left( \alpha
-\theta _{s_{x}}\right) \in \Omega ^{1}\left( M,T_{1}\mathcal{N}^{R}\left( 
\mathbb{L}\right) \right) .$ This fulfils the hypothesis (\ref{cartanhyp})
for Theorem \ref{thmLoopCartan}, and thus there exists a function $f\in
C^{\infty }\left( M,\mathcal{N}^{R}\left( \mathbb{L}\right) \right) $ such
that $\alpha =\theta _{sf}.$
\end{proof}

\begin{remark}
Since $\alpha $ is assumed to be surjective in Corollary \ref{corLoopCartan}
and $\alpha =\theta _{sf},$ we see that $sf:M\longrightarrow \mathbb{L}$ is
a smooth submersion.
\end{remark}

\section{Loop bundles}

\setcounter{equation}{0}\label{sectBundle}Let $\mathbb{L}$ be a smooth loop
with the $\mathbb{L}$-algebra $\mathfrak{l},$ and let us define for brevity $%
\Psi ^{R}\left( \mathbb{L}\right) =\Psi $, $\func{Aut}\left( \mathbb{L}%
\right) =H$, and $\func{PsAut}^{R}\left( \mathbb{L}\right) =G\supset H$, and 
$\mathcal{N}^{R}\left( \mathbb{L}\right) =\mathcal{N}$. Suppose $\Psi ,H,G,%
\mathcal{N}$ are Lie groups. Recall that we also have $\Psi /\mathcal{N}%
\cong G.$

Let $M$ be a smooth, finite-dimensional manifold with a $\Psi $-principal
bundle $\mathcal{P}.$ Then we will define several associated bundles. In
general, recall that there is a one-to-one correspondence between
equivariant maps from a principal bundle and sections of associated bundles.
More precisely, suppose we have a manifold $S$ with a left action $l:\Psi
\times S\longrightarrow S.$ Consider the associated bundle $E=\mathcal{P}%
\times _{\Psi }S.$ Suppose we have a section $\tilde{f}:M\longrightarrow E$,
then this defines a unique equivariant map $f$ $:\mathcal{P}\longrightarrow
S,$ that is, such that for any $h\in \Psi $, 
\begin{equation}
f_{ph}=l_{h^{-1}}\left( f_{p}\right) .  \label{equimap}
\end{equation}%
Conversely, any equivariant map $f:\mathcal{P}\longrightarrow S$ defines a
section $\left( \func{id},f\right) :\mathcal{P}\longrightarrow \mathcal{P}%
\times S$, and then via the quotient map $q:\mathcal{P}\times
S\longrightarrow \mathcal{P}\times _{\Psi }S=E$, it defines a section $%
\tilde{f}:M\longrightarrow E.$ In particular, for each $x\in M$, $\tilde{f}%
\left( x\right) =$ $\left\lfloor p,f_{p}\right\rfloor _{\Psi }$ where $p\in
\pi ^{-1}\left( x\right) \subset \mathcal{P}$ and $\left\lfloor \cdot ,\cdot
\right\rfloor _{\Psi }$ is the equivalence class with respect to the action
of $\Psi :$%
\begin{equation}
\left( p,f_{p}\right) \sim \left( ph,l_{h^{-1}}\left( f_{p}\right) \right)
=\left( ph,f_{ph}\right) \ \ \text{for any }h\in \Psi .
\end{equation}%
For our purposes we will have the following associated bundles. Let $h\in
\Psi $ and, as before, denote by $h^{\prime }$ the partial action of $h$.

\begin{equation}
\begin{tabular}{l|l|l}
\textbf{Bundle} & \textbf{Equivariant map} & \textbf{Equivariance property}
\\ \hline
$\mathcal{P}$ & $k:\mathcal{P}\longrightarrow \Psi $ & $k_{ph}=h^{-1}k_{p}$
\\ 
$\mathcal{Q}=\mathcal{P}\times _{\Psi ^{\prime }}\mathbb{L}$ & $q:\mathcal{P}%
\longrightarrow \mathbb{L}$ & $q_{ph}=\left( h^{\prime }\right) ^{-1}q_{p}$
\\ 
$\mathcal{\mathring{Q}}=\mathcal{P\times }_{\Psi }\mathbb{\mathring{L}}$ & $%
r:\mathcal{P}\longrightarrow \mathbb{\mathring{L}}$ & $r_{ph}=h^{-1}\left(
r_{p}\right) $ \\ 
$\mathcal{N}\cong \mathcal{P}\times _{\Psi }\left( \Psi /H\right) $ & $s:%
\mathcal{P}\longrightarrow \Psi /H\cong \mathcal{C\subset }\mathbb{\mathring{%
L}}$ & $s_{ph}=h^{-1}\left( s_{p}\right) $ \\ 
$\mathcal{A}=\mathcal{P\times }_{\Psi _{\ast }^{\prime }}\mathfrak{l}$ & $%
\eta :\mathcal{P}\longrightarrow \mathfrak{l}$ & $\eta _{ph}=\left(
h^{\prime }\right) _{\ast }^{-1}\eta _{p}$ \\ 
$\mathfrak{p}_{\mathcal{P}}=\mathcal{P\times }_{\left( \func{Ad}_{\xi
}\right) _{\ast }}\mathfrak{p}$ & $\xi :\mathcal{P}\longrightarrow \mathfrak{%
p}$ & $\xi _{ph}=\left( \func{Ad}_{h}^{-1}\right) _{\ast }\xi _{p}$ \\ 
$\mathcal{G=P}\times _{\Psi ^{\prime }}G$ & $\gamma :\mathcal{P}%
\longrightarrow G$ & $\gamma _{ph}=\left( h^{\prime }\right) ^{-1}\gamma
_{p} $ \\ 
$\func{Ad}\left( \mathcal{P}\right) =\mathcal{P}\times _{\func{Ad}_{\Psi
}}\Psi $ & $u:\mathcal{P}\longrightarrow \Psi $ & $u_{ph}=h^{-1}u_{p}h$%
\end{tabular}
\label{equimap2}
\end{equation}

The bundle $\mathcal{Q}$ is the loop bundle with respect to the partial
action $\Psi ^{\prime }$ and the bundle $\mathcal{\mathring{Q}}$ is the loop
bundle with respect to the full action of $\Psi $. The bundle $\mathcal{N}$
has fibers isomorphic to $\Psi /H\cong \mathcal{C},$ which is the set of
right companions $\mathcal{C}^{R}\left( \mathbb{L}\right) \subset \mathbb{%
\mathring{L}}$. Thus it is a subbundle of $\mathcal{\mathring{Q}}.$
Equivalently, $\mathcal{N}=\mathcal{P}/H$ is the orbit space of the right $H$%
-action on $\mathcal{P}$. Recall that the structure group of $\mathcal{P}$
reduces to $H$ if and only if the bundle $N$ has a global section. If this
is the case, then we can reduce the bundle $\mathcal{P}$ to a principal $H$%
-bundle $\mathcal{H}$ over $M$, and then since $H\subset G$, lift to a
principal $G$-bundle $\mathcal{G}$. Also, let $\mathcal{Q}=\mathcal{P}\times
_{\Psi ^{\prime }}\mathbb{L}$ be the bundle associated to $\mathcal{P}$ with
fiber $\mathbb{L}$, where $\Psi ^{\prime }$ denotes that the left action on $%
\mathbb{L}$ is via the partial action of $\Psi $.

We also have some associated vector bundles - namely the vector bundle $%
\mathcal{A}$ with fibers isomorphic to the $\mathbb{L}$-algebra $\mathfrak{l}%
\ $with the tangent partial action of $\Psi $ and the vector bundle $%
\mathfrak{p}_{\mathcal{P}}$ with fibers isomorphic to the Lie algebra $%
\mathfrak{p}$, with the adjoint action of $\Psi $.

\begin{example}
Suppose $\mathbb{L}=U\mathbb{O}$ - the Moufang loop of unit octonions. In
this case, $\Psi =Spin\left( 7\right) $, $H=G_{2}$, $G=SO\left( 7\right) $, $%
\mathcal{N}=\mathbb{Z}_{2}$, and then we have the well-known relations%
\begin{eqnarray*}
SO\left( 7\right) &\cong &Spin\left( 7\right) /\mathbb{Z}_{2} \\
Spin\left( 7\right) /G_{2} &\cong &U\mathbb{O\cong }S^{7} \\
SO\left( 7\right) /G_{2} &\cong &S^{7}/\mathbb{Z}_{2}.
\end{eqnarray*}%
Then, if an orientable $7$-manifold has spin structure, we have a principal $%
Spin\left( 7\right) $-bundle $\mathcal{P}$ over $M$ and the corresponding $%
Spin\left( 7\right) /G_{2}$-bundle always has a smooth section, hence
allowing the $Spin\left( 7\right) $-bundle to reduce to a $G_{2}$-principal
bundle, which in turn lifts to an $SO\left( 7\right) $-bundle. The
associated bundle $\mathcal{Q}$ in this case transforms under $SO\left(
7\right) $, and is precisely the unit subbundle of the octonion bundle $%
\mathbb{R}\oplus TM$ defined in \cite{GrigorianOctobundle}. The bundle $%
\mathcal{\mathring{Q}}$ then transforms under $Spin\left( 7\right) $ and
corresponds to the bundle of unit spinors. The associated vector bundle $%
\mathcal{A}$ in this case has fibers isomorphic to $\func{Im}\mathbb{O}\cong 
\mathbb{R}^{7}$, and then the bundle itself is isomorphic to the tangent
bundle $TM.$
\end{example}

Let $s:$ $\mathcal{P}\longrightarrow \mathbb{\mathring{L}}$ be an
equivariant map. In particular, the equivalence class $\left\lfloor
p,s_{p}\right\rfloor _{\Psi }$ defines a section of the bundle $\mathcal{%
\mathring{Q}}.$ We will refer to $s$ as \emph{the defining map} (or \emph{%
section}). It should be noted that such a map may not always exist globally.
If $\mathbb{L}$ is a $G$-loop\textbf{,} then $\mathcal{\mathring{Q}\cong N}$
and hence existence of a global section of $\mathcal{\mathring{Q}}$ is
equivalent to the reduction of the structure group of $\mathcal{P}.$ There
may be topological obstructions for this.

\begin{example}
\label{exCx3}As in Example \ref{ExNormedDiv}, let $\mathbb{L=}U\mathbb{%
C\cong }U\left( 1\right) $ - the unit complex numbers, and $\Psi =U\left(
n\right) $, $H=G=SU\left( n\right) .$ Then in this setting, $\mathcal{P}$ is
a principal $U\left( n\right) $-bundle over $M$ and $\mathcal{Q}$ is a
circle bundle. Existence of a section of $\mathcal{Q}$ is equivalent to the
reduction of the structure group of $\mathcal{P}$ to $SU\left( n\right) .$
The obstruction for this is the first Chern class of $\mathcal{Q}$ \cite%
{MilnorStasheff}. In the quaternionic case, the structure group reduction
from $Sp\left( n\right) Sp\left( 1\right) $ to $Sp\left( n\right) $ is less
well understood \cite{BoyerGalicki}.
\end{example}

Given equivariant maps $q,r:\mathcal{P}\longrightarrow \mathbb{L}$, we can
define an equivariant product using $s$, such that for any $p\in \mathcal{P}$%
,%
\begin{equation}
\left. q\circ _{s}r\right\vert _{p}=q_{p}\circ _{s_{p}}r_{p}.
\label{equiprod}
\end{equation}%
Indeed, using (\ref{PsiActcircr}), 
\begin{eqnarray}
\left. q\circ _{s}r\right\vert _{ph} &=&q_{ph}\circ _{s_{ph}}r_{ph}  \notag
\\
&=&\left( h^{\prime }\right) ^{-1}q_{p}\circ _{h^{-1}\left( s_{p}\right)
}\left( h^{\prime }\right) ^{-1}r_{p}  \notag \\
&=&\left( h^{\prime }\right) ^{-1}\left( \left. q\circ _{s}r\right\vert
_{p}\right) .  \label{equiprod2}
\end{eqnarray}%
In particular, this allows to define a fiberwise product on sections of $%
\mathcal{Q}$. Similarly, we define equivariant left and right quotients, and
thus well-defined fiberwise quotients of sections of $\mathcal{Q}.$

\begin{remark}
The map $s$ is required to define an equivariant product of two $\mathbb{L}$%
-valued maps. In the $G_{2}$-structure case, as discussed above, sections of 
$\mathcal{\mathring{Q}}$ correspond to unit spinors, and each unit spinor
defines a $G_{2}$-structure, and hence a product on the corresponding
octonion bundle \cite{GrigorianOctobundle}. On the other hand, a product of
an equivariant $\mathbb{L}$-valued map and an equivariant $\mathbb{\mathring{%
L}}$-valued map will be always equivariant, using (\ref{PsAutprod}). In the $%
G_{2}$-structure case, this corresponds to the Clifford product of a unit
octonion, interpreted as an element of $\mathbb{R}\oplus T_{x}M$ at each
point, and a unit spinor. The result is then again a unit spinor. This does
not require any additional structure beyond the spinor bundle.
\end{remark}

Given equivariant maps $\xi ,\eta :\mathcal{P}\longrightarrow \mathfrak{l}$,
we can define an equivariant bracket using $s$. For any $p\in \mathcal{P}$:%
\begin{equation}
\left. \left[ \xi ,\eta \right] ^{\left( s\right) }\right\vert _{p}=\left[
\xi _{p},\eta _{p}\right] ^{\left( s_{p}\right) }.  \label{equibracket}
\end{equation}%
Here the equivariance follows from (\ref{loopalghom}). Using (\ref{phihs})
we then also have an equivariant map $\varphi _{s}$ from equivariant $%
\mathfrak{p}$-valued maps to equivariant $\mathfrak{l}$-valued maps:%
\begin{equation}
\left. \varphi _{s}\left( \gamma \right) \right\vert _{p}=\varphi
_{s_{p}}\left( \gamma _{p}\right) .  \label{equiphis}
\end{equation}%
Other related objects such as the Killing form $K^{\left( s\right) }$ and
the adjoint $\varphi _{s}^{t}$ to $\varphi _{s}$ are then similarly also
equivariant.

Overall, we can condense the above discussion into the following definition
and theorem.

\begin{definition}
\label{defLoopStructure}A \emph{loop bundle structure} over a smooth
manifold $M$ is a quadruple $\left( \mathbb{L},\Psi ,\mathcal{P},s\right) $
where

\begin{enumerate}
\item $\mathbb{L}$ is a finite-dimensional smooth loop with a smoothly
acting group of right pseudoautomorphism pairs $\Psi $.

\item $\mathcal{P}$ is a principal $\Psi $-bundle over $M$.

\item $s:\mathcal{P}\longrightarrow \mathbb{\mathring{L}}$ is a smooth
equivariant map.
\end{enumerate}
\end{definition}

\begin{theorem}
Given a loop bundle structure $\left( \mathbb{L},\Psi ,\mathcal{P},s\right) $
over a manifold $M,$ and associated bundles $\mathcal{Q}=\mathcal{P}\times
_{\Psi ^{\prime }}\mathbb{L},$ $\mathcal{\mathring{Q}}=\mathcal{P\times }%
_{\Psi }\mathbb{\mathring{L}},$ $\mathcal{A}=\mathcal{P\times }_{\Psi _{\ast
}^{\prime }}\mathfrak{l},$ and $\mathfrak{p}_{\mathcal{P}}=\mathcal{P\times }%
_{\left( \func{Ad}_{\xi }\right) _{\ast }}\mathfrak{p},$ where $\mathfrak{l}$
is the $\mathbb{L}$-algebra of $\mathbb{L}$ and $\mathfrak{p}$ the Lie
algebra of $\Psi ,$

\begin{enumerate}
\item $s$ determines a smooth section $\sigma \in \Gamma \left( \mathcal{%
\mathring{Q}}\right) .$

\item For any $A,B\in \Gamma \left( \mathcal{Q}\right) $, $\sigma $ defines
a fiberwise product $A\circ _{\sigma }B,$ via (\ref{equiprod}).

\item For any $X,Y\in \Gamma \left( \mathcal{A}\right) ,$ $\sigma $ defines
a fiberwise bracket $\left[ X,Y\right] ^{\left( \sigma \right) }$, via (\ref%
{equibracket}).

\item $\sigma $ defines a fiberwise map $\varphi _{\sigma }:\Gamma \left( 
\mathfrak{p}_{\mathcal{P}}\right) \longrightarrow \Gamma \left( \mathcal{A}%
\right) $, via (\ref{equiphis}).
\end{enumerate}
\end{theorem}

\subsection{Connections and Torsion}

\label{sectTorsion}Suppose the principal $\Psi $-bundle $\mathcal{P}$ has a
principal Ehresmann connection given by the decomposition 
\begin{equation}
T\mathcal{P}=\mathcal{HP}\oplus \mathcal{VP}  \label{HPVP}
\end{equation}%
with $\mathcal{H}_{ph}\mathcal{P}=\left( R_{h}\right) _{\ast }\mathcal{H}_{p}%
\mathcal{P}$ for any $p\in \mathcal{P}$ and $h\in \Psi $ and $\mathcal{V}%
\mathcal{P}=\ker d\pi $, where $\pi :\mathcal{P}\longrightarrow M$ is the
bundle projection map. Let the projection 
\begin{equation*}
v:T\mathcal{P}\longrightarrow \mathcal{VP}
\end{equation*}
be the Ehresmann connection $1$-form. Similarly, define the projection $%
\func{proj}_{\mathcal{H}}:T\mathcal{P}\longrightarrow \mathcal{HP}.$

Let $\mathfrak{p}$ be the Lie algebra of $\Psi $. Then, as it is well-known,
we have an isomorphism 
\begin{eqnarray}
\sigma &:&\mathcal{P}\times \mathfrak{p}\longrightarrow \mathcal{VP}  \notag
\\
\left( p,\xi \right) &\mapsto &\left. \frac{d}{dt}\left( p\exp \left( t\xi
\right) \right) \right\vert _{t=0}.  \label{mapsigma}
\end{eqnarray}%
For any $\xi \in \mathfrak{p},$ this defines a vertical vector field $\sigma
\left( \xi \right) $ on $\mathcal{P}$. Given the Ehresmann connection 1-form 
$v$, define the $\mathfrak{p}$-valued connection $1$-form $\omega $ via 
\begin{equation*}
\left( \pi ,\omega \right) =\sigma ^{-1}\circ v:T\mathcal{P}\longrightarrow 
\mathcal{P}\times \mathfrak{p}
\end{equation*}%
and recall that for any $h\in \Psi $, 
\begin{equation*}
\left( R_{h}\right) ^{\ast }\omega =\func{Ad}_{h^{-1}}\circ \omega .
\end{equation*}

As previously, suppose $S$ is a manifold with a left action $l$ of $\Psi .$
Given an equivariant map $f:\mathcal{P}\longrightarrow S$, define 
\begin{equation}
d^{\mathcal{H}}f:=f_{\ast }\circ \func{proj}_{\mathcal{H}}:T\mathcal{P}%
\longrightarrow \mathcal{HP}\longrightarrow TS.  \label{dHftilde}
\end{equation}%
This is then a horizontal map since it vanishes on any vertical vectors.
Equivalently, for any $X_{p}\in T_{p}\mathcal{P},$ if $\gamma \left(
t\right) $ is a curve on $\mathcal{P}$ with $\gamma \left( 0\right) =0$ and $%
\dot{\gamma}\left( 0\right) =\func{proj}_{\mathcal{H}}X_{p}\in \mathcal{H}%
_{p}\mathcal{P}$, then 
\begin{equation}
\left. d^{\mathcal{H}}f\right\vert _{p}\left( X_{p}\right) =\left. \frac{d}{%
dt}f\left( \gamma \left( t\right) \right) \right\vert _{t=0}.
\end{equation}%
The map $d^{\mathcal{H}}f$ is moreover still equivariant. The group $\Psi $
acts on $T\mathcal{P}$ via pushforwards of the right action of $\Psi $ on $%
\mathcal{P}.$ Let $h\in \Psi $, so that $r_{h}:\mathcal{P}\longrightarrow 
\mathcal{P}$ gives the right action of $\Psi $ on $\mathcal{P},$ and the
corresponding action of $\Psi $ on $T\mathcal{P}$ is $\left( r_{h}\right)
_{\ast }:T\mathcal{P}\longrightarrow T\mathcal{P}.$ Note that the
corresponding action of $\Psi $ on $TS$ is then $\left( l_{h^{-1}}\right)
_{\ast }:TS\longrightarrow TS.$ Now,%
\begin{eqnarray*}
d^{\mathcal{H}}f\circ \left( r_{h}\right) _{\ast } &=&f_{\ast }\circ \func{%
proj}_{\mathcal{H}}\circ \left( r_{h}\right) _{\ast }=f_{\ast }\circ \left(
r_{h}\right) _{\ast }\circ \func{proj}_{\mathcal{H}} \\
&=&\left( f\circ r_{h}\right) _{\ast }\circ \func{proj}_{\mathcal{H}}=\left(
l_{h^{-1}}\circ f\right) _{\ast }\circ \func{proj}_{\mathcal{H}} \\
&=&\left( l_{h^{-1}}\right) _{\ast }\circ d^{\mathcal{H}}f
\end{eqnarray*}%
where we have used the equivariance of both $f$ and $\func{proj}_{\mathcal{H}%
}.$ So indeed, $d^{\mathcal{H}}f$ is equivariant. Now consider the quotient
map $q^{\prime }:\mathcal{P}\times TS\longrightarrow \mathcal{P\times }%
_{\Psi }TS,$ where $\Psi $ acts via $r_{h}$ on $\mathcal{P}$ and $\left(
l_{h^{-1}}\right) _{\ast }$ on $TS.$ This is a partial differential of the
map $q:\mathcal{P}\times S\longrightarrow E.$ Since $d^{\mathcal{H}}f$ is
horizontal, it vanishes on the kernel of $\pi _{\ast }:T\mathcal{P}%
\longrightarrow TM$. Given $\tilde{f}\ $, the section of the associated
bundle $\mathcal{P}\times _{\Psi }S$ that corresponds to $f$, we can use $d^{%
\mathcal{H}}f$ to define the unique map%
\begin{equation}
d^{\mathcal{H}}\tilde{f}:TM\longrightarrow \mathcal{P\times }_{\Psi }TS
\label{dHf}
\end{equation}%
such that 
\begin{equation*}
d^{\mathcal{H}}\tilde{f}\circ \pi _{\ast }=\left( \pi _{T\mathcal{P}},d^{%
\mathcal{H}}f\right) \circ q^{\prime }
\end{equation*}%
where $\pi _{T\mathcal{P}}:T\mathcal{P}\longrightarrow \mathcal{P}$ is the
bundle projection for $T\mathcal{P}.$ Moreover, $d^{\mathcal{H}}\tilde{f}$
covers the identity map on $M,$ and hence is a section of the fiber product $%
TM\times _{M}\left( \mathcal{P\times }_{\Psi }TS\right) .$ This construction
is summarized in the commutative diagram in Figure \ref{tikCovDer}.

\begin{center}
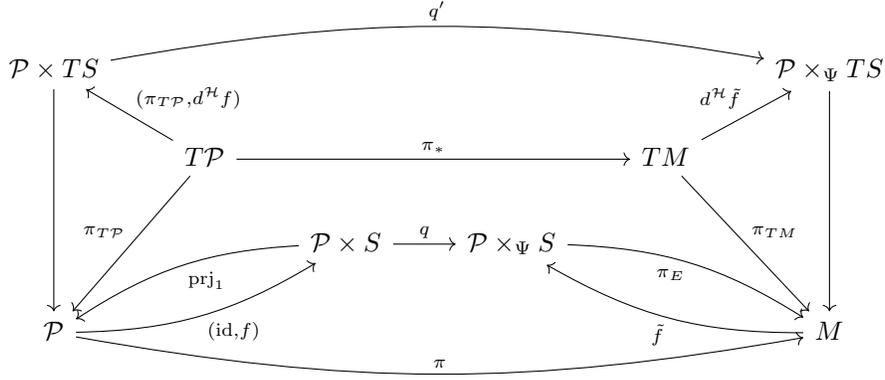

\begin{tikzcd}
\mathcal{P} \times TS \arrow[rrrrr,"q'",bend left=10] \arrow[ddd] & & & & & \mathcal{P} \times_{\Psi} TS \arrow[ddd]\\
 &T\mathcal{P} \arrow[ul,"{(\pi_{T\mathcal{P}},d^{\mathcal{H}} f)}",swap] \arrow[ddl,"\pi_{T\mathcal{P}}",swap] \arrow[rrr,"\pi_{*}"]& & &TM \arrow[ddr,"\pi_{TM}"] \arrow[ur,"d^{\mathcal{H}} \tilde{f}"] & \\
 & &\mathcal{P} \times S \arrow[r,"q"] \arrow[dll,"\func{prj}_{1}",bend right=15]&  \mathcal{P} \times_{\Psi} S \arrow[drr,"\pi_{E}",bend left=15,swap]  & & \\
\mathcal{P} \arrow[rrrrr,"\pi",bend right=10] \arrow[urr,"({\func{id},f)}",bend right=15,swap]& & & & & M \arrow[ull,"\tilde{f}",bend left=15] \\
\end{tikzcd}
\captionof{figure}{Covariant derivative on maps and sections.} \label%
{tikCovDer}
\end{center}

Of course, if $S$ is a vector space, then this reduces to the usual
definition of the exterior covariant derivative of a vector bundle-valued
function and $d^{\mathcal{H}}f$ is a vector-bundle-valued $1$-form.

Given the above correspondence between equivariant maps from $\mathcal{P}$
and sections of associated bundles, for convenience, we will work with
equivariant maps rather than sections. This will allow us to use the
properties of $\mathbb{L}$ from the previous section more directly.

Given a $\mathfrak{p}$-valued connection $1$-form $\omega $ on \QTR{cal}{P}$%
, $ we can concretely work out $d^{\mathcal{H}}f.$ Suppose $X\in \Gamma
\left( T\mathcal{P}\right) $ is a vector field on $\mathcal{P}$, then using
the definition (\ref{dHftilde}), we have 
\begin{eqnarray*}
\left( d^{\mathcal{H}}f\right) \left( X\right) &=&df\left( \func{proj}_{%
\mathcal{H}}\left( X\right) \right) \\
&=&df\left( X-v\left( X\right) \right) \\
&=&df\left( X\right) -df\left( \sigma \left( \pi _{TP}\left( X\right)
,\omega \left( X\right) \right) \right)
\end{eqnarray*}%
where from (\ref{mapsigma}), for $p\in \mathcal{P},$ 
\begin{equation*}
\sigma \left( \pi _{TP}\left( X\right) ,\omega \left( X\right) \right)
_{p}=\left. \frac{d}{dt}\left( p\exp \left( t\omega \left( X_{p}\right)
\right) \right) \right\vert _{t=0}.
\end{equation*}%
Now, let $\gamma \left( t\right) =\exp \left( t\omega \left( X_{p}\right)
\right) \in \Psi $, and note that $\gamma \left( t\right) ^{-1}=\gamma
\left( -t\right) $, so that 
\begin{eqnarray}
\left. df\left( \sigma \left( \pi _{TP}\left( X\right) ,\omega \left(
X\right) \right) \right) \right\vert _{p} &=&\left. \frac{d}{dt}\left(
f\left( p\gamma \left( t\right) \right) \right) \right\vert _{t=0}  \notag \\
&=&-\left. \frac{d}{dt}\left( \exp \left( t\omega \left( X_{p}\right)
\right) f\left( p\right) \right) \right\vert _{t=0}  \notag \\
&=&-\omega \left( X_{p}\right) \cdot f\left( p\right)  \label{omegasharp}
\end{eqnarray}%
where we have used the equivariance of $f$ and where, $\omega \left(
X_{p}\right) \cdot f\left( p\right) \in T_{f\left( p\right) }S$ denotes the
infinitesimal action of $\omega \left( X_{p}\right) \in \mathfrak{p}$ on $S$.

\begin{lemma}
Let $s$ be a $\Psi $-equivariant $S$-valued function on $\mathcal{P}\ $and
let $\omega $ be a $\mathfrak{p}$-valued connection $1$-form on $\mathcal{P}%
, $ then the covariant differential $d^{\mathcal{H}}s:T\mathcal{P}%
\longrightarrow TS$ is given by 
\begin{equation}
d^{\mathcal{H}}s=ds+\omega \cdot s  \label{dHftilde2}
\end{equation}%
where $\omega \cdot s:T_{p}\mathcal{P}\longrightarrow T_{s\left( p\right) }S$
for each $p\in \mathcal{P}\ $gives the infinitesimal action of $\omega $ on $%
S$.
\end{lemma}

Now, more concretely, given a principal connection $\omega $ on $\mathcal{P}%
, $ consider the induced covariant derivatives on equivariant $\mathbb{L}$-
and $\mathbb{\mathring{L}}$-valued maps. To avoid confusion, denote $d^{%
\mathcal{H}}$ acting on $\mathbb{L}$-valued maps by $D$ and by $\mathring{D}$
when it is acting on $\mathbb{\mathring{L}}$-valued maps. Similarly,
consider equivariant $\mathfrak{l}$-valued maps from $\mathcal{P}.$ Given $%
\xi :\mathcal{P}\longrightarrow \mathfrak{l}$ such that $\xi _{ph}=\left(
h^{-1}\right) _{\ast }^{\prime }\left( \xi \right) ,$ define the covariant
derivative $d^{\mathcal{H}}\xi $ via (\ref{dHftilde2}), so overall, given $%
X\in \Gamma \left( T\mathcal{P}\right) ,$ 
\begin{equation}
d_{X}^{\mathcal{H}}\xi =d_{X}\xi +\omega \left( X\right) \cdot \xi
\label{dHxi}
\end{equation}%
where $\omega \left( X\right) \cdot \xi $ refers to the linear
representation of the Lie algebra $\mathfrak{p}$ on $\mathfrak{l}$ given by (%
\ref{pactl1}).

We have the following useful relation between $D$ and $\mathring{D}.$

\begin{lemma}
\label{lemProdAs}Suppose $A:\mathcal{P}\longrightarrow \mathbb{L}$ and $s:%
\mathcal{P}\longrightarrow \mathbb{\mathring{L}}$ are equivariant, and let $%
p\in \mathcal{P}.$ Then, 
\begin{equation}
\left. \mathring{D}\left( As\right) \right\vert _{p}=\left( R_{s_{p}}\right)
_{\ast }\left. DA\right\vert _{p}+\left( L_{A_{p}}\right) _{\ast }\left. 
\mathring{D}s\right\vert _{p}.  \label{DAs}
\end{equation}%
Note that $\left. \mathring{D}\left( As\right) \right\vert _{p}:T_{p}%
\mathcal{P}\longrightarrow T_{As}\mathbb{\mathring{L}}.$
\end{lemma}

\begin{proof}
Let $X_{p}\in T_{p}\mathcal{P}$ and let $p\left( t\right) $ be a curve on $%
\mathcal{P}$ with $p\left( 0\right) =p$ and $\dot{p}\left( 0\right) =\func{%
proj}_{\mathcal{H}}\left( X_{p}\right) \in \mathcal{H}_{p}\mathcal{P}.$
Consider 
\begin{equation}
\left. \mathring{D}\left( As\right) \right\vert _{p}\left( X_{p}\right)
=\left. \frac{d}{dt}\left( A_{p\left( t\right) }s_{p\left( t\right) }\right)
\right\vert _{t=0}
\end{equation}%
However, 
\begin{eqnarray}
\left. \frac{d}{dt}\left( A_{p\left( t\right) }s_{p\left( t\right) }\right)
\right\vert _{t=0} &=&\left. \frac{d}{dt}\left( A_{p\left( t\right)
}s_{p}\right) \right\vert _{t=0}+\left. \frac{d}{dt}\left( A_{p}s_{p\left(
t\right) }\right) \right\vert _{t=0}  \notag \\
&=&\left( R_{s_{p}}\right) _{\ast }\left( DA\right) _{p}\left( X_{p}\right)
+\left( L_{A_{p}}\right) _{\ast }\left( \mathring{D}s\right) _{p}\left(
X_{p}\right)
\end{eqnarray}%
and thus (\ref{DAs}) holds.
\end{proof}

Suppose now $\left( \mathbb{L},\Psi ,\mathcal{P},s\right) $ is a loop bundle
structure, as in Definition \ref{defLoopStructure}, so that $s\ $is an $%
\mathbb{\mathring{L}}$-valued equivariant map. Then we have the following
important definition.

\begin{definition}
\label{defTors}The $\emph{torsion}$ $T^{\left( s,\omega \right) }$ of $%
\left( \mathbb{L},\Psi ,\mathcal{P},s\right) $ with respect to $\omega $ is
a horizontal $\mathfrak{l}$-valued $1$-form on $\mathcal{P}$ given by 
\begin{equation}
T^{\left( s,\omega \right) }=\theta _{s}\circ \func{proj}_{\mathcal{H}}
\end{equation}%
where $\theta _{s}$ is the Darboux derivative of $s$. Equivalently, at $p\in 
\mathcal{P}$, we have%
\begin{equation}
\left. T^{\left( s,\omega \right) }\right\vert _{p}=\left(
R_{s_{p}}^{-1}\right) _{\ast }\left. \mathring{D}s\right\vert _{p}.
\end{equation}
\end{definition}

Thus, $T^{\left( s,\omega \right) }$ is the horizontal component of $\theta
_{s}.$ We also easily see that it is $\Psi $-equivariant. Using the
equivariance of $s$ and $\mathring{D}s$, we have for $h\in \Psi $, 
\begin{equation}
T_{ph}^{\left( s,\omega \right) }=\left( h_{\ast }^{\prime }\right)
^{-1}T_{p}^{\left( s,\omega \right) }.  \label{Tequi1}
\end{equation}%
Thus, $T^{\left( s,\omega \right) }$ is a \emph{basic} (i.e. horizontal and
equivariant) $\mathfrak{l}$-valued $1$-form on $\mathcal{P}$, and thus
defines a $1$-form on $M$ with values in the associated vector bundle $%
\mathcal{A=P\times }_{\Psi _{\ast }^{\prime }}\mathfrak{l}.$ We also have
the following key property of $T^{\left( s,\omega \right) }.$

\begin{theorem}
\label{thmTprop}Suppose $T^{\left( s,\omega \right) }$ is as given in
Definition \ref{defTors} and also let $\hat{\omega}^{\left( s\right) }\in
\Omega ^{1}\left( \mathcal{P},\mathfrak{l}\right) $ be given by 
\begin{equation}
\hat{\omega}^{\left( s\right) }=\varphi _{s}\left( \omega \right) .
\label{omegahat}
\end{equation}%
Then, 
\begin{equation}
\theta _{s}=T^{\left( s,\omega \right) }-\hat{\omega}^{\left( s\right) }.
\label{stheta}
\end{equation}%
In particular, $T^{\left( s,\omega \right) }$ and the quantity $-\hat{\omega}%
^{\left( s\right) }$ are respectively the horizontal and vertical components
of $\theta _{s}$.
\end{theorem}

\begin{proof}
Let $p\in \mathcal{P}.$ Then, from (\ref{dHftilde2}) we have 
\begin{eqnarray}
\left( R_{s_{p}}^{-1}\right) _{\ast }\left. \mathring{D}s\right\vert _{p}
&=&\left( R_{s_{p}}^{-1}\right) _{\ast }\left. ds\right\vert _{p}+\left(
R_{s_{p}}^{-1}\right) _{\ast }\left( \omega \cdot s_{p}\right)  \notag \\
&=&\left. \theta _{s}\right\vert _{p}+\left. \frac{d}{dt}\faktor{\left( \exp
\left( t\omega _{p}\right) \left( s_{p}\right) \right)}{ s_{p}}\right\vert
_{t=0}  \notag \\
&=&\left. \theta _{s}\right\vert _{p}+\varphi _{s_{p}}\left( \omega
_{p}\right)  \label{Torsion1f2}
\end{eqnarray}%
where we have used the definition (\ref{phis}) of $\varphi _{s}$. Hence we
get (\ref{stheta}).
\end{proof}

Suppose $p\left( t\right) $ is a curve on $\mathcal{P}$ with $p\left(
0\right) =p$ and with a horizontal initial velocity vector $\dot{p}\left(
0\right) =X_{p}^{\mathcal{H}}.$ Then, by definition, 
\begin{equation}
\left. \frac{d}{dt}s_{p\left( t\right) }\right\vert _{t=0}=\left. \mathring{D%
}_{X}s\right\vert _{p}=\left( R_{s_{p}}\right) _{\ast }\left.
T_{X_{p}}^{\left( s,\omega \right) }\right\vert _{p},  \label{dts}
\end{equation}%
where $T_{X}^{\left( s,\omega \right) }=T^{\left( s,\omega \right) }\left(
X\right) \in \mathfrak{l}.$ This observation will come in useful later on.

\begin{remark}
If $s_{p}\in \mathcal{C\cong \Psi }/H$ for all $p\in \mathcal{P},$ then as
we know, the structure group of $\mathcal{P}$ is reduced to $H.$ Moreover,
the reduced holonomy group of $\omega $ is contained in $H$ if and only if
there exists such a map $s$ with $d^{\mathcal{H}}s=0.$ This is equivalent to 
$T^{\left( s,\omega \right) }=0$, so this is the motivation for calling this
quantity the torsion. If $s$ is not necessarily in $\mathcal{C},$ then we
can still say something about the holonomy of $\omega $ in the case $d^{%
\mathcal{H}}s=0$. Let $p\in \mathcal{P}$ and suppose $\Gamma \left( t\right) 
$ is the horizontal lift with respect to the connection $\omega $ of some
closed curve based at $\pi \left( p\right) .$ Then, the endpoint of $\Gamma $
is $\Gamma \left( 1\right) =ph$ for some $h\in \Psi .$ The set of all such $%
h\in \Psi $ form the holonomy group $\func{Hol}_{p}\left( \omega \right) $
of $\omega $ at $p$ \cite{KobayashiNomizu1}. Now if we have an equivariant
map $s:\mathcal{P}\longrightarrow \mathbb{L}$, then $s\circ \Gamma $ is a
curve on $\mathbb{L}$ with $s\left( \Gamma \left( 1\right) \right)
=s_{ph}=h^{-1}s_{p}.$ However, $\frac{d}{dt}\left( s\circ \Gamma \left(
t\right) \right) =\left( d^{\mathcal{H}}s\right) _{s\circ \Gamma \left(
t\right) }\dot{\Gamma}\left( t\right) $ since the velocity vectors of $%
\Gamma \left( t\right) $ are horizontal. Thus, if $d^{\mathcal{H}}s=0$
everywhere, then the curve $s\circ \Gamma \left( t\right) $ is constant, and
hence $h^{-1}s_{p}=s_{p}.$ By (\ref{AutLr}), this means that $h\in \func{Aut}%
\left( \mathbb{L},\circ _{s_{p}}\right) .$ This is true for any horizontal
lift $\Gamma $, hence we see that $\func{Hol}_{p}\left( \omega \right)
\subset \func{Aut}\left( \mathbb{L},\circ _{s_{p}}\right) .$
\end{remark}

The torsion also enters expressions for covariant derivatives of the loop
product, loop algebra bracket, as well as the map $\varphi _{s}$.

\begin{theorem}
\label{lemProd}Suppose $A,B:\mathcal{P}\longrightarrow \mathbb{L},$ and $s:%
\mathcal{P}\longrightarrow \mathbb{\mathring{L}}$ are equivariant, and let $%
p\in \mathcal{P}.$ Then, 
\begin{eqnarray}
\left. D\left( A\circ _{s}B\right) \right\vert _{p} &=&\left(
R_{B_{p}}^{\left( s_{p}\right) }\right) _{\ast }\left. DA\right\vert
_{p}+\left( L_{A_{p}}^{\left( s_{p}\right) }\right) _{\ast }\left.
DB\right\vert _{p}  \label{DAsB} \\
&&+\left[ A_{p},B_{p},\left. T^{\left( s,\omega \right) }\right\vert _{p}%
\right] ^{\left( s_{p}\right) }.  \notag
\end{eqnarray}%
If $\xi ,\eta :\mathcal{P}\longrightarrow \mathfrak{l}$ are equivariant,
then 
\begin{equation}
d^{\mathcal{H}}\left[ \xi ,\eta \right] ^{\left( s\right) }=\left[ d^{%
\mathcal{H}}\xi ,\eta \right] ^{\left( s\right) }+\left[ \xi ,d^{\mathcal{H}%
}\eta \right] ^{\left( s\right) }+\left[ \xi ,\eta ,T^{\left( s,\omega
\right) }\right] ^{\left( s\right) }-\left[ \eta ,\xi ,T^{\left( s,\omega
\right) }\right] ^{\left( s\right) }.  \label{dHbrack}
\end{equation}

The $\mathfrak{l}\otimes \mathfrak{p}^{\ast }$-valued map $\varphi _{s}:%
\mathcal{P}\longrightarrow $ $\mathfrak{l}\otimes \mathfrak{p}^{\ast }$
satisfies 
\begin{equation}
d^{\mathcal{H}}\varphi _{s}=\func{id}_{\mathfrak{p}}\cdot T^{\left( s,\omega
\right) }-\left[ \varphi _{s},T^{\left( s,\omega \right) }\right] ^{\left(
s\right) },  \label{dhphis}
\end{equation}%
where $\func{id}_{\mathfrak{p}}\ $is the identity map of $\mathfrak{p}$ and $%
\cdot $ denotes the action of the Lie algebra $\mathfrak{p}$ on $\mathfrak{l}
$ given by (\ref{pactl1}).
\end{theorem}

\begin{proof}
Let $X_{p}\in T_{p}\mathcal{P}$ and let $p\left( t\right) $ be a curve on $%
\mathcal{P}$ with $p\left( 0\right) =p$ and $\dot{p}\left( 0\right) =\func{%
proj}_{\mathcal{H}}\left( X_{p}\right) \in \mathcal{H}_{p}\mathcal{P}.$ To
show (\ref{DAsB}), first note that 
\begin{equation}
\left. D\left( A\circ _{s}B\right) \right\vert _{p}\left( X_{p}\right)
=\left. \frac{d}{dt}\left( A_{p\left( t\right) }\circ _{s_{p\left( t\right)
}}B_{p\left( t\right) }\right) \right\vert _{t=0}.  \label{dHprod}
\end{equation}%
However,%
\begin{eqnarray}
\left. \frac{d}{dt}\left( A_{p\left( t\right) }\circ _{s_{p\left( t\right)
}}B_{p\left( t\right) }\right) \right\vert _{t=0} &=&\left. \frac{d}{dt}%
\left( A_{p\left( t\right) }\circ _{s_{p}}B_{p}\right) \right\vert
_{t=0}+\left. \frac{d}{dt}\left( A_{p}\circ _{s_{p}}B_{p\left( t\right)
}\right) \right\vert _{t=0}  \notag \\
&&+\left. \frac{d}{dt}\left( A_{p}\circ _{s_{p\left( t\right) }}B_{p}\right)
\right\vert _{t=0}  \notag \\
&=&\left( R_{B_{p}}^{\left( s_{p}\right) }\right) _{\ast }\left.
DA\right\vert _{p}\left( X_{p}\right) +\left( L_{A_{p}}^{\left( s_{p}\right)
}\right) _{\ast }\left. DB\right\vert _{p}\left( X_{p}\right)  \notag \\
&&+\left. \frac{d}{dt}\left( A_{p}\circ _{s_{p\left( t\right) }}B_{p}\right)
\right\vert _{t=0}  \label{dHprod2}
\end{eqnarray}%
and then, using Lemma \ref{lemQuotient}, 
\begin{eqnarray}
\left. \frac{d}{dt}\left( A_{p}\circ _{s_{p\left( t\right) }}B_{p}\right)
\right\vert _{t=0} &=&\left. \frac{d}{dt}\left( \left( A_{p}\cdot
B_{p}s_{p\left( t\right) }\right) /s_{p\left( t\right) }\right) \right\vert
_{t=0}  \notag \\
&=&\left. \frac{d}{dt}\left( \left( A_{p}\cdot B_{p}s_{p\left( t\right)
}\right) /s_{p}\right) \right\vert _{t=0}  \label{dHprod2a} \\
&&+\left. \frac{d}{dt}\left( \left( A_{p}\cdot B_{p}s_{p}\right) /s_{p}\cdot
s_{p\left( t\right) }\right) /s_{p}\right\vert _{t=0}.  \notag
\end{eqnarray}%
Looking at each term in (\ref{dHprod2a}), we have 
\begin{eqnarray*}
\left( A_{p}\cdot B_{p}s_{p\left( t\right) }\right) /s_{p} &=&\left(
A_{p}\cdot B_{p}\left( \faktor{s_{p\left( t\right) }}{s_{p}}\cdot
s_{p}\right) \right) /s_{p} \\
&=&A_{p}\circ _{s_{p}}\left( B_{p}\circ _{s_{p}}\left( \faktor{s_{p\left(
t\right) }}{s_{p}}\right) \right)
\end{eqnarray*}%
and%
\begin{equation*}
\left( \left( A_{p}\cdot B_{p}s_{p}\right) /s_{p}\cdot s_{p\left( t\right)
}\right) /s_{p}=\left( A_{p}\circ _{s_{p}}B_{p}\right) \circ _{s_{p}}\left(%
\faktor{ s_{p\left( t\right) }}{s_{p}}\right) .
\end{equation*}%
Overall (\ref{dHprod2}) becomes, 
\begin{equation}
\left. \frac{d}{dt}\left( A_{p}\circ _{s_{p\left( t\right) }}B_{p}\right)
\right\vert _{t=0}=\left( \left( L_{A_{p}}^{\left( s_{p}\right) }\circ
L_{B_{p}}^{\left( s_{p}\right) }\right) _{\ast }-\left( L_{A_{p}\circ
_{s_{p}}B_{p}}^{\left( s_{p}\right) }\right) _{\ast }\right) \left(
R_{s_{p}}^{-1}\right) _{\ast }\left. \mathring{D}s\right\vert _{p}\left(
X_{p}\right)  \label{dHprod3}
\end{equation}%
and hence we get (\ref{DAsB}) using the definitions of $T^{\left( s,\omega
\right) }$ and the mixed associator (\ref{pxiqsol}).

To show (\ref{dHbrack}), note that 
\begin{eqnarray*}
\left. d_{X}^{\mathcal{H}}\left( \left[ \xi ,\eta \right] ^{\left( s\right)
}\right) \right\vert _{p} &=&\left. \frac{d}{dt}\left[ \xi _{p\left(
t\right) },\eta _{p\left( t\right) }\right] ^{\left( s_{p\left( t\right)
}\right) }\right\vert _{t=0} \\
&=&\left[ \left. d_{X}^{\mathcal{H}}\xi \right\vert _{p},\eta _{p}\right]
^{\left( s_{p}\right) }+\left[ \xi _{p},\left. d_{X}^{\mathcal{H}}\eta
\right\vert _{p}\right] ^{\left( s_{p}\right) } \\
&&+\left. \frac{d}{dt}\left[ \xi _{p},\eta _{p}\right] ^{\left( s_{p\left(
t\right) }\right) }\right\vert _{t=0}.
\end{eqnarray*}%
However, using (\ref{db1}) and (\ref{dts}), the last term becomes 
\begin{equation*}
\left. \frac{d}{dt}\left[ \xi _{p},\eta _{p}\right] ^{\left( s_{p\left(
t\right) }\right) }\right\vert _{t=0}=\left[ \xi _{p},\eta _{p},\left.
T_{X}^{\left( s,\omega \right) }\right\vert _{p}\right] ^{\left(
s_{p}\right) }-\left[ \eta _{p},\xi _{p},\left. T_{X}^{\left( s,\omega
\right) }\right\vert _{p}\right] ^{\left( s_{p}\right) }
\end{equation*}%
and hence we obtain (\ref{dHbrack}).

Let us now show (\ref{dhphis}). From (\ref{actpl}), given $\gamma \in 
\mathfrak{p}$, setting $\hat{\gamma}\left( r\right) =\varphi _{r}\left(
\gamma \right) $ for each $r\in \mathbb{L}$, we have 
\begin{equation}
\left. d\hat{\gamma}\right\vert _{r}\left( \rho _{r}\left( \xi \right)
\right) =\gamma \cdot \xi -\left[ \hat{\gamma}\left( r\right) ,\xi \right]
^{\left( r\right) }
\end{equation}%
for some $\xi \in \mathfrak{l}.$ Now for a map $s:\mathcal{P}\longrightarrow 
\mathbb{L}\ $and some vector field $X$ on $\mathcal{P},$ we have at each $%
p\in \mathcal{P}$ 
\begin{eqnarray}
\left. d\left( \varphi _{s}\left( \gamma \right) \right) \right\vert
_{p}\left( X\right) &=&\left. d\hat{\gamma}\right\vert _{s_{p}}\circ \left.
ds\right\vert _{p}\left( X\right)  \notag \\
&=&\left. d\hat{\gamma}\right\vert _{s_{p}}\left( \rho _{s_{p}}\left( \theta
_{s}\left( X_{p}\right) \right) \right)  \notag \\
&=&\gamma \cdot \theta _{s}\left( X_{p}\right) -\left[ \varphi
_{s_{p}}\left( \gamma \right) ,\theta _{s}\left( X_{p}\right) \right]
^{\left( s_{p}\right) }.
\end{eqnarray}%
Therefore, $d\varphi _{s}$ is given by 
\begin{equation}
d\varphi _{s}\left( \gamma \right) =\gamma \cdot \theta _{s}-\left[ \varphi
_{s}\left( \gamma \right) ,\theta _{s}\right] ^{\left( s\right) }.
\label{dphis}
\end{equation}%
To obtain $d^{\mathcal{H}}\varphi _{s}$ we take the horizontal component,
and hence using (\ref{stheta}), we just replace $\theta _{s}$ in (\ref{dphis}%
) by $T^{\left( s,\omega \right) },$ which gives (\ref{dhphis}).
\end{proof}

\begin{remark}
If $\mathbb{L}$ is associative, i.e. is a group, then certainly $A\circ
_{s}B=AB$ and this is then an equivariant section, if $A$ and $B$ are such.
In (\ref{DAsB}) the second term on the right vanishes, and thus $D$
satisfies the product rule with respect to multiplication on $\mathbb{L}.$
\end{remark}

We can rewrite (\ref{DAs}) as 
\begin{eqnarray}
\mathring{D}\left( As\right) &=&\left( DA\right) s+A\left( \left( \mathring{D%
}s\right) /s\cdot s\right)  \notag \\
&=&\left( DA\right) s+\left( A\circ _{s}T^{\left( s,\omega \right) }\right)
s.  \label{DAs2}
\end{eqnarray}%
Using this, we can then define an adapted covariant derivative $D^{\left(
s\right) }$ on equivariant $\mathbb{L}$-valued maps, given by 
\begin{equation}
\left. D^{\left( s\right) }A\right\vert _{p}=\left( R_{s_{p}}^{-1}\right)
_{\ast }\left. \mathring{D}\left( As\right) \right\vert _{p}=\left.
DA\right\vert _{p}+\left( L_{A_{p}}^{\left( s_{p}\right) }\right) _{\ast
}T_{p}^{\left( s,\omega \right) }  \label{Dsderiv}
\end{equation}%
with respect to which, 
\begin{equation}
\left. D^{\left( s\right) }\left( A\circ _{s}B\right) \right\vert
_{p}=\left( R_{B_{p}}^{\left( s_{p}\right) }\right) _{\ast }\left.
DA\right\vert _{p}+\left( L_{A_{p}}^{\left( s_{p}\right) }\right) _{\ast
}\left. D^{\left( s\right) }B\right\vert _{p}.  \label{Dsderivprod}
\end{equation}%
This is the precise analog of the octonion covariant derivative from \cite%
{GrigorianOctobundle}. The derivative $D^{\left( s\right) }$ essentially
converts an $\mathbb{L}$-valued map into an $\mathbb{\mathring{L}}$-valued
one using $s$ and then differentiates it using $\mathring{D}$ before
converting back to $\mathbb{L}.$ In particular, if we take $A=1$, 
\begin{equation}
D^{\left( s\right) }1=T^{\left( s,\omega \right) }.  \label{Ds1}
\end{equation}

\begin{remark}
Up to the sign of $T$, (\ref{DAsB}) and (\ref{Dsderiv}) are precisely the
expressions obtained in \cite{GrigorianOctobundle} for the covariant
derivative with respect to the Levi-Civita connection of the product on the
octonion bundle over a $7$-manifold. In that case, $T$ is precisely the
torsion of the $G_{2}$-structure that defines the octonion bundle. This
provides additional motivation for calling this quantity the torsion of $s$
and $\omega .$ In the case of $G_{2}$-structures, usually one takes the
torsion with respect to the preferred Levi-Civita connection, however in
this more general setting, we don't have a preferred connection, thus $%
T^{\left( s,\omega \right) }$ should also be taken to depend on the
connection.
\end{remark}

\begin{corollary}
Suppose $\mathbb{L}$ is an alternative loop, so that the associator is
skew-symmetric. Suppose $\xi ,\eta \longrightarrow \mathfrak{l}$ and $s:%
\mathcal{P}\longrightarrow \mathbb{\mathring{L}}$ are equivariant. Then,
defining a modified exterior derivative $d^{\left( s\right) }$ on
equivariant maps from $\mathcal{P}$ to $\mathfrak{l}$ via%
\begin{equation}
d^{\left( s\right) }\xi =d^{\mathcal{H}}\xi +\frac{1}{3}\left[ \xi
,T^{\left( s\right) }\right] ^{\left( s\right) },  \label{dsbrack}
\end{equation}%
it satisfies 
\begin{equation}
d^{\left( s\right) }\left[ \xi ,\eta \right] ^{\left( s\right) }=\left[
d^{\left( s\right) }\xi ,\eta \right] ^{\left( s\right) }+\left[ \xi
,d^{\left( s\right) }\eta \right] ^{\left( s\right) }.
\end{equation}
\end{corollary}

\begin{proof}
If $\mathbb{L}$ is alternative, then the loop Jacobi identity (\ref{Jac2})
becomes 
\begin{equation}
\left[ \xi ,\left[ \eta ,\gamma \right] ^{\left( s\right) }\right] ^{\left(
s\right) }+\left[ \eta ,\left[ \gamma ,\xi \right] ^{\left( s\right) }\right]
^{\left( s\right) }+\left[ \gamma ,\left[ \xi ,\eta \right] ^{\left(
s\right) }\right] ^{\left( s\right) }=6\left[ \xi ,\eta ,\gamma \right]
^{\left( s\right) }.  \label{Jacalt}
\end{equation}%
On the other hand, (\ref{dHbrack}) becomes 
\begin{equation}
d^{\mathcal{H}}\left[ \xi ,\eta \right] ^{\left( s\right) }=\left[ d^{%
\mathcal{H}}\xi ,\eta \right] ^{\left( s\right) }+\left[ \xi ,d^{\mathcal{H}%
}\eta \right] ^{\left( s\right) }+2\left[ \xi ,\eta ,T^{\left( s\right) }%
\right] ^{\left( s\right) }.  \label{dHbrackalt}
\end{equation}%
Thus, using both (\ref{Jacalt}) and (\ref{dHbrackalt}), we obtain 
\begin{eqnarray*}
d^{\left( s\right) }\left[ \xi ,\eta \right] ^{\left( s\right) } &=&d^{%
\mathcal{H}}\left[ \xi ,\eta \right] ^{\left( s\right) }+\frac{1}{3}\left[ %
\left[ \xi ,\eta \right] ^{\left( s\right) },T^{\left( s\right) }\right]
^{\left( s\right) } \\
&=&\left[ d^{\left( s\right) }\xi ,\eta \right] ^{\left( s\right) }+\left[
\xi ,d^{\left( s\right) }\eta \right] ^{\left( s\right) } \\
&&-\frac{1}{3}\left[ \left[ \xi ,T^{\left( s\right) }\right] ^{\left(
s\right) },\eta \right] ^{\left( s\right) }-\frac{1}{3}\left[ \xi ,\left[
\eta ,T^{\left( s\right) }\right] ^{\left( s\right) }\right] ^{\left(
s\right) } \\
&&+\frac{1}{3}\left[ \left[ \xi ,\eta \right] ^{\left( s\right) },T^{\left(
s\right) }\right] ^{\left( s\right) }+2\left[ \xi ,\eta ,T^{\left( s\right) }%
\right] ^{\left( s\right) } \\
&=&\left[ d^{\left( s\right) }\xi ,\eta \right] ^{\left( s\right) }+\left[
\xi ,d^{\left( s\right) }\eta \right] ^{\left( s\right) }.
\end{eqnarray*}
\end{proof}

\begin{remark}
In the case of $G_{2}$-structures and octonions, the derivative (\ref%
{dsbrack}) exactly replicates the modified covariant derivative that
preserves the $G_{2}$-structure that was introduced in \cite{DGKisoflow}.
\end{remark}

\begin{example}
The map $\varphi _{s}$ is equivariant on $\mathcal{P}$ and hence defines a
section of the associated bundle $\mathcal{A}\otimes \func{ad}\left( 
\mathcal{P}\right) ^{\ast }$ over $M.$ If $\mathbb{L}$ is the loop of unit
octonions and $\mathfrak{l\cong }\func{Im}\mathbb{O},$ and we have a $G_{2}$%
-structure on $M,$ then $\varphi _{s}$ corresponds to a section of $%
TM\otimes \Lambda ^{2}TM,$ which up to a constant factor is a multiple of
the corresponding $G_{2}$-structure $3$-form $\varphi $ with indices raised
using the associated metric. The torsion $T$ of $\varphi $ with respect to
the Levi-Civita connection on $TM$ is then a section of $TM\otimes T^{\ast
}M.$ Noting that $\mathfrak{so}\left( 7\right) $ acts on $\mathbb{R}^{7}$ by
matrix multiplication, if we set $\left( \varphi _{s}\right) _{\ }^{abc}=-%
\frac{1}{4}\varphi ^{abc}$ in local coordinates, then (\ref{dhphis})
precisely recovers the well-known formula for $\nabla \varphi $ in terms of $%
T.$ Indeed, suppose $\xi \in \Gamma \left( \Lambda ^{2}T^{\ast }M\right) $,
then in a local basis $\left\{ e_{a}\right\} $, for some fixed vector field $%
X$, we have 
\begin{eqnarray*}
\left( \nabla _{X}\varphi _{s}\right) \left( \xi \right) &=&\xi \cdot T_{X}- 
\left[ \varphi _{s}\left( \xi \right) ,T_{X}\right] ^{\left( s\right) } \\
&=&\left( \xi _{\ b}^{a}T_{X}^{b}+\frac{1}{2}\varphi _{\ bc}^{a}\varphi
^{bde}\xi _{de}T_{X}^{c}\right) e_{a} \\
&=&\left( \xi _{\ b}^{a}T_{X}^{b}-\frac{1}{2}\left( \psi _{\ c}^{a\
de}+g^{ad}g_{c}^{\ e}-g^{ae}g_{c}^{\ d}\right) \xi _{de}T_{X}^{c}\right)
e_{a} \\
&=&\frac{1}{2}T_{X}^{c}\psi _{c\ }^{\ ade}\xi _{de}e_{a},
\end{eqnarray*}%
where $\psi =\ast \varphi $. Hence, indeed, 
\begin{equation}
\nabla _{X}\varphi =-2T_{X}\lrcorner \psi ,  \label{nablaXphi}
\end{equation}%
which is exactly as in \cite{GrigorianOctobundle}, taking into account that
the torsion here differs by a sign from \cite{GrigorianOctobundle}. Here we
also used the convention that $\left[ X,Y\right] =2X\lrcorner Y\lrcorner
\varphi $ and also contraction identities for $\varphi $ \cite%
{GrigorianG2Torsion1,karigiannis-2005-57}. This is also consistent with the
expression (\ref{dHbrack}) for the covariant derivative of the bracket.
Indeed, in the case of an alternative loop, (\ref{dHbrackalt}) shows that
the covariant derivative of the bracket function $b_{s}$ is given by 
\begin{equation}
d^{\mathcal{H}}b_{s}=2\left[ \cdot ,\cdot ,T^{\left( s\,,\omega \right) }%
\right] ^{\left( s\right) }.  \label{dHbrack1a}
\end{equation}%
Taking $b_{s}=2\varphi $ and $\left[ \cdot ,\cdot ,\cdot \right] ^{\left(
s\right) }$ given by $\left( \left[ X,Y,Z\right] ^{\left( s\right) }\right)
^{a}=2\psi _{\ bcd}^{a}X^{b}Y^{c}Z^{d},$ as in \cite{GrigorianOctobundle},
we again recover (\ref{nablaXphi}).
\end{example}

\begin{example}
Suppose $\mathcal{P}$ is a principal $U\left( n\right) $-bundle and $\mathbb{%
L}\cong U\left( 1\right) $, the unit complex numbers, as in Example \ref%
{exCx2}. Then, (\ref{dhphis}) shows that $d^{\mathcal{H}}\varphi _{s}=0$. If 
$V\ $is an $n$-dimensional complex vector space with the standard action of $%
U\left( n\right) $ on it and $\mathcal{V=P\times }_{U\left( n\right) }V$ is
the associated vector bundle to $\mathcal{P}$ with fiber $V$, then $\varphi
_{s}$ defines a K\"{a}hler form on $\mathcal{V}.$
\end{example}

\begin{example}
Suppose $\mathcal{P}$ is a principal $Sp\left( n\right) Sp\left( 1\right) $%
-bundle and $\mathbb{L}\cong Sp\left( 1\right) ,$ the unit quaternions, as
in Example \ref{exQuat2}. Then, (\ref{dhphis}) shows that $d^{\mathcal{H}%
}\varphi _{s}=-\left[ \varphi _{s},T^{\left( s\,,\omega \right) }\right] _{%
\func{Im}\mathbb{H}}.$ If $V\ $is an $n$-dimensional quaternionic vector
space with the standard action of $Sp\left( n\right) Sp\left( 1\right) $ on
it and $\mathcal{V=P\times }_{Sp\left( n\right) Sp\left( 1\right) }V$ is the
associated vector bundle to $\mathcal{P}$ with fiber $V$, then $\varphi _{s}$
defines a 2-form on $\mathcal{V}\ $with values in $\func{Im}\mathbb{H}$
(since the bundle $\mathcal{A}\ $is trivial). So this gives rise to $3$
linearly independent $2$-forms $\omega _{1},\omega _{2},\omega _{3}.$ If $%
T^{\left( s,\omega \right) }=0$, then this reduces to a HyperK\"{a}hler
structure on $\mathcal{V}.$ It is an interesting question whether the case $%
T^{\left( s,\omega \right) }\neq 0$ is related to \textquotedblleft HyperK%
\"{a}hler with torsion\textquotedblright\ geometry \cite%
{GrantcharovPoonHKT,VerbitskyHKT}.
\end{example}

\subsection{Curvature}

\label{sectCurv}Recall that the curvature $F\in \Omega ^{2}\left( \mathcal{P}%
,\mathfrak{p}\right) $ of the connection $\omega $ on $\mathcal{P}$ is given
by 
\begin{equation}
F^{\left( \omega \right) }=d^{\mathcal{H}}\omega =d\omega \circ \func{proj}_{%
\mathcal{H}},  \label{curvom}
\end{equation}%
so that, for $X,Y\in \Gamma \left( T\mathcal{P}\right) $, 
\begin{equation}
F^{\left( \omega \right) }\left( X,Y\right) =d\omega \left( X^{\mathcal{H}%
},Y^{\mathcal{H}}\right) =-\omega \left( \left[ X^{\mathcal{H}},Y^{\mathcal{H%
}}\right] \right) ,  \label{curvom2}
\end{equation}%
where $X^{\mathcal{H}},Y^{\mathcal{H}}$ are the projections of $X,Y$ to $%
\mathcal{HP}.$

Similarly as $\hat{\omega}$, define $\hat{F}^{\left( s,\omega \right) }\in
\Omega ^{2}\left( \mathcal{P},\mathfrak{l}\right) $ to be the projection of
the curvature $F^{\left( \omega \right) }$ to $\mathfrak{l}$ with respect to 
$s$, such that for any $X_{p},Y_{p}\in T_{p}\mathcal{P},$ 
\begin{eqnarray}
\hat{F}^{\left( s,\omega \right) }\left( X_{p},Y_{p}\right) &=&\varphi
_{s}\left( F^{\left( \omega \right) }\right) \left( X_{p},Y_{p}\right) 
\notag \\
&=&\left. \frac{d}{dt}\faktor{\left( \exp \left( tF^{\left( \omega \right)
}\left( X_{p},Y_{p}\right) \right) \left( s_{p}\right) \right)}{ s_{p}}%
\right\vert _{t=0}.  \label{Fhat}
\end{eqnarray}

We easily see that 
\begin{equation}
d^{\mathcal{H}}\hat{\omega}^{\left( s\right) }=\hat{F}^{\left( s,\omega
\right) }.  \label{dHom}
\end{equation}%
Indeed, 
\begin{equation*}
d^{\mathcal{H}}\hat{\omega}^{\left( s\right) }=d^{\mathcal{H}}\left( \varphi
_{s}\left( \omega \right) \right) =d^{\mathcal{H}}\varphi _{s}\wedge \left(
\omega \circ \func{proj}_{\mathcal{H}}\right) +\varphi _{s}\left( d^{%
\mathcal{H}}\omega \right) =\hat{F}^{\left( s,\omega \right) },
\end{equation*}%
where we have used the fact that $\omega $ is vertical.

We then have the following structure equations

\begin{theorem}
\label{thmFTstruct}$\hat{F}^{\left( s,\omega \right) }$ and $T^{\left(
s,\omega \right) }$ satisfy the following structure equation 
\begin{equation}
\hat{F}^{\left( s,\omega \right) }=d^{\mathcal{H}}T^{\left( s,\omega \right)
}-\frac{1}{2}\left[ T^{\left( s,\omega \right) },T^{\left( s,\omega \right) }%
\right] ^{\left( s\right) },  \label{dHT}
\end{equation}%
where a wedge product between the $1$-forms $T^{\left( s,\omega \right) }$
is implied. Equivalently, (\ref{dHT}) can be written as 
\begin{equation}
d\hat{\omega}^{\left( s\right) }+\frac{1}{2}\left[ \hat{\omega}^{\left(
s\right) },\hat{\omega}^{\left( s\right) }\right] ^{\left( s\right) }=\hat{F}%
^{\left( s,\omega \right) }-d^{\mathcal{H}}\varphi _{s}\wedge \omega ,
\label{dwstruct}
\end{equation}%
where $\left( d^{\mathcal{H}}\varphi _{s}\wedge \omega \right) \left(
X,Y\right) =\left( d_{X}^{\mathcal{H}}\varphi _{s}\right) \left( \omega
\left( Y\right) \right) -\left( d_{Y}^{\mathcal{H}}\varphi _{s}\right)
\left( \omega \left( X\right) \right) $ for any vector fields $X$ and $Y$ on 
$\mathcal{P}.$
\end{theorem}

\begin{proof}
Using (\ref{stheta}), we have 
\begin{eqnarray}
d^{\mathcal{H}}T^{\left( s,\omega \right) } &=&dT^{\left( s,\omega \right)
}\circ \func{proj}_{\mathcal{H}}  \notag \\
&=&\left( d\theta _{s}+d\hat{\omega}^{\left( s\right) }\right) \circ \func{%
proj}_{\mathcal{H}}.  \label{Torsion1f3}
\end{eqnarray}%
Now consider the first term. Let $X_{p},Y_{p}\in T_{p}\mathcal{P}$, then 
\begin{eqnarray}
d\theta _{s}\left( X_{p}^{\mathcal{H}},Y_{p}^{\mathcal{H}}\right) &=&\left(
d\theta \right) _{s_{p}}\left( s_{\ast }X_{p}^{\mathcal{H}},s_{\ast }Y_{p}^{%
\mathcal{H}}\right)  \notag \\
&=&\left( d\theta \right) _{s_{p}}\left( \mathring{D}_{X_{p}}s,\mathring{D}%
_{Y_{P}}s\right) \\
&=&\left[ \theta \left( \mathring{D}_{X_{p}}s\right) ,\theta \left( 
\mathring{D}_{Y_{P}}s\right) \right] ^{\left( s_{p}\right) }  \notag \\
&=&\left[ T^{\left( s,\omega \right) }\left( X_{p}\right) ,T^{\left(
s,\omega \right) }\left( Y_{p}\right) \right] ^{\left( s_{p}\right) },
\label{Torsion1f3a}
\end{eqnarray}%
where we have used the Maurer-Cartan structural equation for loops (\ref%
{MCequation1}). Using (\ref{dHom}) for the second term, overall, we obtain (%
\ref{dHT}).

From the Maurer-Cartan equation (\ref{MCequation1}), 
\begin{equation*}
d\theta _{s}-\frac{1}{2}\left[ \theta _{s},\theta _{s}\right] ^{\left(
s\right) }=0.
\end{equation*}%
We also have from (\ref{stheta}) 
\begin{equation*}
\left[ \theta _{s},\theta _{s}\right] ^{\left( s\right) }=\left[ T^{\left(
s,\omega \right) },T^{\left( s,\omega \right) }\right] ^{\left( s\right) }-2%
\left[ \hat{\omega}^{\left( s\right) },T^{\left( s,\omega \right) }\right]
^{\left( s\right) }+\left[ \hat{\omega}^{\left( s\right) },\hat{\omega}%
^{\left( s\right) }\right] ^{\left( s\right) }.
\end{equation*}%
Hence%
\begin{equation*}
d\theta _{s}=dT^{\left( s,\omega \right) }-d\hat{\omega}^{\left( s\right) }=%
\frac{1}{2}\left[ T^{\left( s,\omega \right) },T^{\left( s,\omega \right) }%
\right] ^{\left( s\right) }-\left[ \hat{\omega}^{\left( s\right) },T^{\left(
s,\omega \right) }\right] ^{\left( s\right) }+\frac{1}{2}\left[ \hat{\omega}%
^{\left( s\right) },\hat{\omega}^{\left( s\right) }\right] ^{\left( s\right)
}.
\end{equation*}%
Noting that 
\begin{equation*}
dT^{\left( s,\omega \right) }=d^{\mathcal{H}}T^{\left( s,\omega \right)
}-\omega \dot{\wedge}T^{\left( s,\omega \right) }
\end{equation*}%
we find 
\begin{eqnarray*}
d\hat{\omega}^{\left( s\right) }+\frac{1}{2}\left[ \hat{\omega}^{\left(
s\right) },\hat{\omega}^{\left( s\right) }\right] ^{\left( s\right) } &=&d^{%
\mathcal{H}}T^{\left( s,\omega \right) }-\omega \dot{\wedge}T^{\left(
s,\omega \right) } \\
&&-\frac{1}{2}\left[ T^{\left( s,\omega \right) },T^{\left( s,\omega \right)
}\right] ^{\left( s\right) }+\left[ \hat{\omega}^{\left( s\right)
},T^{\left( s,\omega \right) }\right] ^{\left( s\right) }
\end{eqnarray*}%
and then using (\ref{dHT}) and (\ref{dhphis}) we obtain (\ref{dwstruct}).
\end{proof}

\begin{corollary}[Bianchi identity]
The quantity $\hat{F}^{\left( s,\omega \right) }$ satisfies the equation 
\begin{eqnarray}
d^{\mathcal{H}}\hat{F}^{\left( s,\omega \right) } &=&d^{\mathcal{H}}\varphi
_{s}\wedge F  \notag \\
&=&F\dot{\wedge}T^{\left( s,\omega \right) }-\left[ \hat{F}^{\left( s,\omega
\right) },T^{\left( s,\omega \right) }\right] ^{\left( s\right) }
\label{Bianchi}
\end{eqnarray}%
where $\cdot $ denotes the representation of $\mathfrak{p}$ on $\mathfrak{l}$
\end{corollary}

\begin{proof}
Using the definition (\ref{Fhat}) of $\hat{F}^{\left( s,\omega \right) }$,
we have 
\begin{equation*}
d^{\mathcal{H}}\hat{F}^{\left( s,\omega \right) }=d^{\mathcal{H}}\left(
\varphi _{s}\left( F\right) \right) =d^{\mathcal{H}}\varphi _{s}\wedge
F+\varphi _{s}\left( d^{\mathcal{H}}F\right) ,
\end{equation*}%
however using the standard Bianchi identity, $d^{\mathcal{H}}F=0,$ and (\ref%
{dhphis}), we obtain (\ref{Bianchi}).
\end{proof}

\begin{example}
The equation (\ref{dHT}) is the precise analog of what is known as the
\textquotedblleft $G_{2}$-structure Bianchi identity\textquotedblright\ \cite%
{GrigorianOctobundle,karigiannis-2007} (not to be confused with the Bianchi
identity (\ref{Bianchi})). In that case, $\hat{F}$ corresponds precisely to
the quantity $\frac{1}{4}\pi _{7}\func{Riem}$, which is the projection of
the endomorphism part of $\func{Riem}$ to the $7$-dimensional representation
of $G_{2}.$ In local coordinates, it is given by $\frac{1}{4}\func{Riem}%
_{abcd}\varphi ^{cde}$.
\end{example}

\begin{example}
\label{exCx4}In the complex case, with $\mathbb{L=}U\mathbb{C}$ and $%
\mathcal{P}$ a principal $U\left( n\right) $-bundle, (\ref{dHT}) shows that $%
\hat{F}^{\left( s,\omega \right) }=dT^{\left( s,\omega \right) }.$ Here $d^{%
\mathcal{H}}=d$ on $\mathfrak{l}$-valued forms because the action of $%
\mathfrak{p}_{n}$ on $\mathfrak{l}$ is trivial (as in Example \ref{exCx2}).
If $s$ is a global section, then this shows that $\hat{F}$ is an exact $2$%
-form - and so the class $\left[ \hat{F}\right] =0$. This is consistent with
a vanishing first Chern class which is a necessary condition for existence
of a global $s$. On the other hand, if we suppose that $s$ is only a local
section, so that $T^{\left( s,\omega \right) }$ is a local $1$-form$,$ then
we only get that $\hat{F}^{\left( s,\omega \right) }$ is closed, so in this
case it may define a non-trivial first Chern class. If $\mathcal{P}$ is the
unitary frame bundle over a complex manifold, it defines a K\"{a}hler
metric, and then $\hat{F}$ precisely corresponds to the Ricci curvature, so
that the Ricci-flat condition for reduction to a Calabi-Yau manifold is $%
\hat{F}=0.$
\end{example}

The equation (\ref{dwstruct}) is interesting because this is an analog of
the structure equation for the connection $1$-form $\omega $ on $\mathcal{P}%
. $ However, in the case of $\omega $, the quantity $d\omega -\frac{1}{2}%
\left[ \omega ,\omega \right] $ is horizontal. However, for $\hat{\omega}%
^{\left( s\right) }$, $\hat{F}^{\left( s,\omega \right) }$ gives the
horizontal component, while the remaining terms give mixed vertical and
horizontal components. The fully vertical components vanish. We also see
that $\hat{\omega}^{\left( s\right) }$ satisfies the loop Maurer-Cartan
equation if and only if $\hat{F}^{\left( s,\omega \right) }=0$ and $d^{%
\mathcal{H}}\varphi _{s}=0.$ In the $G_{2}$ case, $\nabla \varphi =0$ of
course is equivalent to $T=0$ and hence implies $\frac{1}{4}\pi _{7}\func{%
Riem}=0.$ More generally, this may not need to be the case.

\begin{lemma}
\label{lemTcond}Suppose $\mathbb{L}$ is a left-alternative loop and suppose $%
-\hat{\omega}^{\left( s\right) }$ satisfies the Maurer-Cartan equation 
\begin{equation}
d\hat{\omega}^{\left( s\right) }+\frac{1}{2}\left[ \hat{\omega}^{\left(
s\right) },\hat{\omega}^{\left( s\right) }\right] ^{\left( s\right) }=0,
\label{omegahatMC}
\end{equation}%
then for any $\alpha ,\beta \in \mathfrak{q}^{\left( s_{p}\right) }\cong
T_{1}\mathcal{C}^{R}\left( \mathbb{L},\circ _{s_{p}}\right) $, 
\begin{equation}
\left[ \alpha ,\beta ,T_{p}^{\left( s,\omega \right) }\right] ^{\left(
s_{p}\right) }=0\text{.}  \label{Trestrict}
\end{equation}
\end{lemma}

\begin{proof}
Taking the exterior derivative of (\ref{omegahatMC}) and applying (\ref%
{alphastructeq}), we find $\hat{\omega}^{\left( s\right) }$ satisfies 
\begin{equation}
0=\left[ \hat{\omega}^{\left( s\right) },\hat{\omega}^{\left( s\right)
},\theta _{s}+\hat{\omega}^{\left( s\right) }\right] ^{\left( s\right) }=%
\left[ \hat{\omega}^{\left( s\right) },\hat{\omega}^{\left( s\right)
},T^{\left( s,\omega \right) }\right] ^{\left( s\right) }.
\label{omegahatassoc}
\end{equation}%
Since $\mathbb{L}$ is left-alternative, we know that the $\mathbb{L}$%
-algebra associator is skew in the first two entries, so if given vector
fields $X,Y,Z$ on $\mathcal{P}$, we have 
\begin{eqnarray}
0 &=&\left[ \hat{\omega}^{\left( s\right) }\left( X\right) ,\hat{\omega}%
^{\left( s\right) }\left( Y\right) ,T^{\left( s,\omega \right) }\left(
Z\right) \right] ^{\left( s\right) }+\left[ \hat{\omega}^{\left( s\right)
}\left( Y\right) ,\hat{\omega}^{\left( s\right) }\left( Z\right) ,T^{\left(
s,\omega \right) }\left( X\right) \right] ^{\left( s\right) }  \notag \\
&&+\left[ \hat{\omega}^{\left( s\right) }\left( Z\right) ,\hat{\omega}%
^{\left( s\right) }\left( X\right) ,T^{\left( s,\omega \right) }\left(
Y\right) \right] ^{\left( s\right) }.  \label{omhatassoc2}
\end{eqnarray}%
Let $\xi \in \mathfrak{p\ }$and let $X=\sigma \left( \xi \right) $ be a
vertical vector field on $\mathcal{P}$, then 
\begin{equation*}
\hat{\omega}^{\left( s\right) }\left( X\right) =\varphi _{s}\left( \omega
\left( X\right) \right) =\varphi _{s}\left( \xi \right) .
\end{equation*}%
In (\ref{omhatassoc2}), we take $X=\sigma \left( \xi \right) $ and $Y=\sigma
\left( \eta \right) $ to be vertical vector fields and $Z=Z^{h}$ a
horizontal vector field. Then since $\hat{\omega}^{\left( s\right) }$ is
vertical and $T^{\left( s,\omega \right) }$ is horizontal, we find that for
any $\xi ,\eta \in \mathfrak{p},$ 
\begin{equation*}
\left[ \varphi _{s}\left( \xi \right) ,\varphi _{s}\left( \eta \right)
,T^{\left( s,\omega \right) }\left( Z\right) \right] ^{\left( s\right) }=0.
\end{equation*}%
We know that for each $p\in \mathcal{P}$, the map $\varphi _{s_{p}}$ is
surjective onto $\mathfrak{q}^{\left( s_{p}\right) }\subset \mathfrak{l}%
^{\left( s_{p}\right) }$ and thus (\ref{Trestrict}) holds.
\end{proof}

\begin{theorem}
\label{thmTNucl}Suppose $\mathcal{P}$ is connected and simply-connected and $%
\mathbb{L}$ a smooth loop such that

\begin{enumerate}
\item $\mathfrak{l}$ is a left-alternative algebra (i.e. the associator on $%
\mathfrak{l}$ is skew-symmetric in the first two entries),

\item $\dim \left( \mathcal{N}^{R}\left( \mathbb{L}\right) \right) =\dim
\left( \mathcal{N}^{R}\left( \mathfrak{l}\right) \right) .$
\end{enumerate}

Moreover, suppose $s_{p}\in \mathcal{C}^{R}\left( \mathbb{L}\right) $ for
every $p\in \mathcal{P}$, then $\hat{\omega}^{\left( s\right) }$ satisfies
the Maurer-Cartan equation (\ref{omegahatMC}) if and only if there exists a
map $f:\mathcal{P}\longrightarrow \mathcal{N}^{R}\left( \mathbb{L}\right) $
such that 
\begin{equation}
T^{\left( s,\omega \right) }=-\left( \func{Ad}_{s}\right) _{\ast }\theta
_{f}.  \label{Tsthetaf}
\end{equation}
\end{theorem}

\begin{proof}
Since $s$ has values in $\mathcal{C}^{R}\left( \mathbb{L}\right) ,$ using
Lemma \ref{lemTcond}, we see that the conditions of Corollary \ref%
{corLoopCartan} are satisfied, and hence there exists a map $f:\mathcal{P}%
\longrightarrow \mathcal{N}^{R}\left( \mathbb{L}\right) $ such that 
\begin{eqnarray*}
-\hat{\omega}^{\left( s\right) } &=&\theta _{sf} \\
&=&\theta _{s}+\left( \func{Ad}_{s}\right) _{\ast }\theta _{f}.
\end{eqnarray*}%
From (\ref{stheta}), 
\begin{equation*}
T^{\left( s,\omega \right) }=\theta _{s}+\hat{\omega}^{\left( s\right)
}=-\left( \func{Ad}_{s}\right) _{\ast }\theta _{f}.
\end{equation*}%
Conversely, suppose (\ref{Tsthetaf}) holds for some right nucleus-valued map 
$f$. Then, clearly $\hat{\omega}^{\left( s\right) }=-\theta _{sf}$, and thus 
$-\hat{\omega}^{\left( s\right) }$ satisfies (\ref{omegahatMC}).
\end{proof}

\begin{remark}
Theorem \ref{thmTNucl} shows that if $\mathbb{L}$ has a sufficiently large\
nucleus, then $\hat{F}^{\left( s,\omega \right) }=0$ and $d^{\mathcal{H}%
}\varphi _{s}=0$ do not necessarily imply that $T^{\left( s,\omega \right)
}=0$. In the case of unit octonions, the nucleus is just $\left\{ \pm
1\right\} $, so any nucleus-valued map is constant on connected components,
hence in this case if $\hat{\omega}^{\left( s\right) }$ satisfies (\ref%
{omegahatMC}), then $T^{\left( s,\omega \right) }=0.$
\end{remark}

\subsection{Deformations}

\label{sectDeform}The torsion of a loop structure is determined by the
equivariant $\mathbb{\mathring{L}}$-valued map $s$ and the connection $%
\omega $ on $\mathcal{P}.$ There are several possible deformations of $s$
and $\omega $. In particular, $s$ may be deformed by the action of $\Psi $
or by left multiplication action of $\mathbb{L}.$ The connection $\omega $
may be deformed by the affine action of $\Omega _{basic}^{1}\left( \mathcal{P%
},\mathfrak{p}\right) $ or by gauge transformations in $\Psi .$ Moreover, of
course, these deformations may be combined or considered infinitesimally.
Since $T^{\left( s,\omega \right) }$ is the horizontal part of $\theta _{s}$%
, when considering deformations of $s$ it is sufficient to consider what
happens to $\theta _{s}$ and then taking the horizontal component.

Recall that the space of connections on $\mathcal{P}$ is an affine space
modelled on equivariant horizontal (i.e. basic) $\mathfrak{p}$-valued $1$%
-forms on $\mathcal{P}.$ Thus, any connection $\tilde{\omega}=\omega +A$ for
some basic $\mathfrak{p}$-valued $1$-form $A$. Then, 
\begin{equation}
T^{\left( s,\tilde{\omega}\right) }=\theta _{s}+\varphi _{s}\left( \tilde{%
\omega}\right) =T^{\left( s,\omega \right) }+\hat{A}  \label{wtild}
\end{equation}%
where $\hat{A}=\varphi _{s}\left( A\right) $. Thus, we can set $T^{\left( s,%
\tilde{\omega}\right) }=0$ by choosing $A$ such that $\hat{A}=-T^{\left(
s,\omega \right) }$ if and only if for each $p\in P\,,\ T_{p}^{\left(
s,\omega \right) }\in $ $\mathfrak{q}^{\left( s_{p}\right) }=\varphi
_{s_{p}}\left( \mathfrak{p}\right) $. Since $\hat{\omega}$ is always in the
image of $\varphi _{s}$, we conclude there exists a connection $\tilde{\omega%
}$ for which $T^{\left( s,\tilde{\omega}\right) }=0$ if and only if $\left.
\theta _{s}\right\vert _{p}$ $\in \mathfrak{q}^{\left( s_{p}\right) }$ for
each $p$. In that case, $\theta _{s}=-\varphi _{s}\left( \tilde{\omega}%
\right) .$ From Theorem \ref{thmThetaPhi}, we then see that $\tilde{\omega}$
has curvature with values in $\mathfrak{h}_{s}.$

Recall that if $\phi :\mathcal{P}\longrightarrow \mathcal{P}$ is a gauge
transformation, then there exists an $\func{Ad}_{\Psi }$-equivariant map $u:%
\mathcal{P}\longrightarrow \Psi $ such that for each $p\in \mathcal{P}$, $%
\phi \left( p\right) =pu_{p}$. Each such map then corresponds to a section
of the associated bundle $\func{Ad}\left( \mathcal{P}\right) .$ The
gauge-transformed connection $1$-form is then $\omega ^{\phi }=u^{\ast
}\omega $, where 
\begin{equation}
u^{\ast }\omega =\left( \func{Ad}_{u^{-1}}\right) _{\ast }\omega +u^{\ast
}\theta _{\Psi }  \label{omgauge}
\end{equation}%
where $\theta _{\Psi }$ is the \emph{left}-invariant Maurer-Cartan form on $%
\Psi $. Then, 
\begin{eqnarray}
d^{u^{\ast }\mathcal{H}}s &=&\left( l_{u}^{-1}\right) _{\ast }d^{\mathcal{H}%
}\left( l_{u}s\right)  \notag \\
&=&d^{\mathcal{H}}s+\left( u^{\ast }\theta _{\Psi }\right) ^{\mathcal{H}%
}\cdot s_{p}  \label{dHphi}
\end{eqnarray}%
where at each $p\in \mathcal{P}.$%
\begin{equation*}
\left. \left( u^{\ast }\theta _{\Psi }\right) ^{\mathcal{H}}\right\vert
_{p}=\left( l_{u_{p}}\right) _{\ast }^{-1}\circ \left( d^{\mathcal{H}%
}u\right) _{p}\mathfrak{.}
\end{equation*}%
Hence, 
\begin{equation}
T^{\left( s,u^{\ast }\omega \right) }=\left( R_{s}^{-1}\right) _{\ast
}d^{u^{\ast }\mathcal{H}}s=T^{\left( s,\omega \right) }+\varphi _{s}\left(
\left( u^{\ast }\theta _{\Psi }\right) ^{\mathcal{H}}\right) .
\label{Tsgauge}
\end{equation}%
Consider the curvature $F^{u^{\ast }\omega }$ of the connection $u^{\ast
}\omega $. It is well-known that it is given by%
\begin{equation}
F^{u^{\ast }\omega }=\left( \func{Ad}_{u^{-1}}\right) _{\ast }F.
\end{equation}%
From Theorem \ref{lemGammahatsurj}, we then have 
\begin{equation}
\hat{F}^{\left( s,u^{\ast }\omega \right) }=\varphi _{s}\left( \left( \func{%
Ad}_{u^{-1}}\right) _{\ast }F\right) =\left( u^{-1}\right) _{\ast }^{\prime }%
\hat{F}^{\left( u\left( s\right) ,\omega \right) }.
\end{equation}%
On the other hand, using (\ref{dHphi}) and (\ref{Dsderiv}) we have 
\begin{eqnarray*}
T^{\left( s,u^{\ast }\omega \right) } &=&\left( R_{s}^{-1}\right) _{\ast
}\left( u^{\ast }\mathring{D}\right) \left( s\right) \\
&=&\left( R_{s}^{-1}\right) _{\ast }\left( u^{-1}\right) _{\ast }\mathring{D}%
\left( u\left( s\right) \right) \\
&=&\left( u^{-1}\right) _{\ast }^{\prime }\left( R_{u\left( s\right)
}^{-1}\right) _{\ast }\mathring{D}\left( u\left( s\right) \right) \\
&=&\left( u^{-1}\right) _{\ast }^{\prime }T^{\left( u\left( s\right) ,\omega
\right) }.
\end{eqnarray*}%
Summarizing, we have the following.

\begin{theorem}
Suppose $s:\mathcal{P}\longrightarrow \mathbb{\mathring{L}}$ and $u:\mathcal{%
P}\longrightarrow \Psi $ are equivariant smooth maps. Then, 
\begin{subequations}
\begin{eqnarray}
T^{\left( s,u^{\ast }\omega \right) } &=&T^{\left( s,\omega \right)
}+\varphi _{s}\left( \left( u^{\ast }\theta _{\Psi }\right) ^{\mathcal{H}%
}\right)  \label{Tsustom} \\
&=&\left( u^{-1}\right) _{\ast }^{\prime }T^{\left( u\left( s\right) ,\omega
\right) }  \notag \\
\hat{F}^{\left( s,u^{\ast }\omega \right) } &=&\left( u^{-1}\right) _{\ast
}^{\prime }\hat{F}^{\left( u\left( s\right) ,\omega \right) }.
\end{eqnarray}%
\end{subequations}%
In particular, 
\begin{equation}
T^{\left( u^{-1}\left( s\right) ,u^{\ast }\omega \right) }=\left( u^{\prime
}\right) _{\ast }^{-1}T^{\left( s,\omega \right) }\ \ \text{and }\hat{F}%
^{\left( u^{-1}\left( s\right) ,u^{\ast }\omega \right) }=\left(
u^{-1}\right) _{\ast }^{\prime }\hat{F}^{\left( s,\omega \right) }.
\label{Tuom}
\end{equation}
\end{theorem}

This shows that both $T$ and $\hat{F}$ transform equivariantly with respect
to a simultaneous transformation of $s$ and $\omega $. In particular, if we
have a Riemannian metric on the base manifold $M$ and a $\Psi $-covariant
metric on $\mathfrak{l},$ then with respect to the induced metric on $%
T^{\ast }\mathcal{P}\otimes \mathfrak{l}$, the quantities $\left\vert
T\right\vert ^{2}$ and $\left\vert F\right\vert ^{2}$ are invariant with
respect to the transformation $\left( s,\omega \right) \mapsto \left(
u^{-1}\left( s\right) ,u^{\ast }\omega \right) .$ In the case of $G_{2}$%
-structure, the key question is regarding the holonomy of the Levi-Civita,
so in this general setting, if we are interested in the holonomy of $\omega $%
, it makes sense to consider individual transformations $s\mapsto As$ for
some equivariant $A\in C^{\infty }\left( \mathcal{P},\mathbb{L}\right) $ and 
$\omega \mapsto u^{\ast }\omega $ because each of these transformations
leaves the holonomy group unchanged. We also see that every transformation $%
s\mapsto u\left( s\right) $ for some equivariant $u\in C^{\infty }\left( 
\mathcal{P},\Psi \right) $ corresponds to a transformation $s\mapsto As,$
where $A=h\left( s\right) /s$. From (\ref{PsAutoriso}), this is precisely
the companion of the corresponding map $u_{s}\in \Psi \left( \mathbb{L}%
,\circ _{s}\right) .$ Moreover, this correspondence is one-to-one if and
only if $\mathbb{L}$ is a $G$-loop. It is easy to see that $A$ is then an
equivariant $\mathbb{L}$-valued map. Thus, considering transformations $%
s\mapsto As$ is more general in some situations.

\begin{theorem}
Suppose $A:\mathcal{P}\longrightarrow \mathbb{L}$ and $s:\mathcal{P}%
\longrightarrow \mathbb{\mathring{L}}$ . Then, 
\begin{subequations}
\begin{eqnarray}
T^{\left( As,\omega \right) } &=&\left( R_{A}^{\left( s\right) }\right)
_{\ast }^{-1}DA+\left( \func{Ad}_{A}^{\left( s\right) }\right) _{\ast
}T^{\left( s,\omega \right) }=\left( R_{A}^{\left( s\right) }\right) _{\ast
}^{-1}D^{\left( s\right) }A  \label{Trom} \\
\hat{F}^{\left( As,\omega \right) } &=&\left( R_{A}^{\left( s\right)
}\right) _{\ast }^{-1}\left( F^{\prime }\cdot A\right) +\left( \func{Ad}%
_{A}^{\left( s\right) }\right) _{\ast }\hat{F}^{\left( s,\omega \right) },
\label{From}
\end{eqnarray}%
\end{subequations}%
where $F^{\prime }\cdot A$ denotes the infinitesimal action of $\mathfrak{p}$
on $\mathbb{L}.$
\end{theorem}

\begin{proof}
Recall from (\ref{thetafs2}), that 
\begin{equation}
\theta _{As}=\theta _{A}^{\left( s\right) }+\left( \func{Ad}_{A}^{\left(
s\right) }\right) _{\ast }\theta _{s}.  \label{thetaAs}
\end{equation}%
Now, $T^{\left( s,\omega \right) }$ is just the horizontal part of $\theta
_{s}$, so taking the horizontal projection in (\ref{thetaAs}), we
immediately get (\ref{Trom}). To obtain (\ref{From}), from (\ref{phiAs}) we
have%
\begin{equation}
\hat{F}^{\left( As,\omega \right) }=\varphi _{As}\left( F\right) =\left(
R_{A}^{\left( s\right) }\right) _{\ast }^{-1}\left( F^{\prime }\cdot
A\right) +\left( \func{Ad}_{A}^{\left( s\right) }\right) _{\ast }\varphi
_{s}\left( F\right) ,
\end{equation}%
and hence we obtain (\ref{From}).
\end{proof}

\begin{remark}
The expression (\ref{Trom}) precisely replicates the formula for the
transformation of torsion of a $G_{2}$-structure within a fixed metric
class, as derived in \cite{GrigorianOctobundle}.
\end{remark}

Now suppose $s_{t}$ is a $1$-parameter family of equivariant $\mathbb{%
\mathring{L}}$-valued maps that satisfy 
\begin{equation}
\frac{\partial s_{t}}{\partial t}=\left( R_{s_{t}}\right) _{\ast }\xi _{t}
\label{Aevol}
\end{equation}%
where $\xi _{t}$ is a $1$-parameter family of $\mathfrak{l}$-valued maps. In
particular, if $\xi \left( t\right) $ is independent of $t$, then $s\left(
t\right) =\exp _{s_{0}}\left( t\xi \right) s_{0}.$ Then let us work out the
evolution of $T^{\left( s\left( t\right) ,\omega \right) }$ and $\hat{F}%
^{\left( s\left( t\right) ,\omega \right) }.$ First consider the evolution
of $\theta _{s\left( t\right) }$ and $\varphi _{s\left( t\right) }$.

\begin{lemma}
Suppose $s\left( t\right) $ satisfies (\ref{Aevol}), then 
\begin{subequations}
\begin{eqnarray}
\frac{\partial \theta _{s\left( t\right) }}{\partial t} &=&d\xi \left(
t\right) -\left[ \theta _{s\left( t\right) },\xi \left( t\right) \right]
^{\left( s\left( t\right) \right) }  \label{dtthetas} \\
\frac{\partial \varphi _{s\left( t\right) }}{\partial t} &=&\func{id}_{%
\mathfrak{p}}\cdot \xi \left( t\right) -\left[ \varphi _{s\left( t\right)
},\xi \left( t\right) \right] ^{\left( s\left( t\right) \right) }.
\label{dtphis}
\end{eqnarray}%
\end{subequations}%
\end{lemma}

\begin{proof}
For $\theta _{s\left( t\right) }$, suppressing pushforwards, we have%
\begin{eqnarray}
\frac{\partial \theta _{s\left( t\right) }}{\partial t} &=&\frac{\partial }{%
\partial t}\left( \left( ds\left( t\right) \right) /s\left( t\right) \right)
\notag \\
&=&\left( d\dot{s}\right) /s-\left( \left( ds\right) /s\cdot \dot{s}\right)
/s  \notag \\
&=&d\left( \xi s\right) /s-\left( \left( ds\right) /s\cdot \left( \xi
s\right) \right) /s  \notag \\
&=&d\xi -\left[ \theta _{s\left( t\right) },\xi \right] ^{\left( s\left(
t\right) \right) }.
\end{eqnarray}%
Similarly, for $\varphi _{s\left( t\right) },$ let $\eta \in \mathfrak{p}$,
then 
\begin{eqnarray}
\frac{\partial \varphi _{s\left( t\right) }\left( \eta \right) }{\partial t}
&=&\frac{\partial }{\partial t}\left( \left. \frac{d}{d\tau }\exp \left(
\tau \eta \right) \left( s\right) /s\right\vert _{\tau =0}\right)  \notag \\
&=&\left. \frac{d}{d\tau }\exp \left( \tau \eta \right) \left( \left( \xi
s\right) /s\right) \right\vert _{\tau =0}-\left. \frac{d}{d\tau }\left( \exp
\left( \tau \eta \right) \left( \left( s\right) /s\right) \cdot \left( \xi
s\right) \right) /s\right\vert _{\tau =0}  \notag \\
&=&\left. \frac{d}{d\tau }\exp \left( \tau \eta \right) ^{\prime }\left( \xi
\right) \right\vert _{\tau =0}+\left. \frac{d}{d\tau }\left( \xi \exp \left(
\tau \eta \right) \left( s\right) \right) /s\right\vert _{\tau =0}  \notag \\
&&-\left. \frac{d}{d\tau }\left( \exp \left( \tau \eta \right) \left( \left(
s\right) /s\right) \cdot \left( \xi s\right) \right) /s\right\vert _{\tau =0}
\notag \\
&=&\eta \cdot \xi \left( t\right) -\left[ \varphi _{s\left( t\right) }\left(
\eta \right) ,\xi \left( t\right) \right] ^{\left( s\left( t\right) \right)
}.
\end{eqnarray}
\end{proof}

To obtain the evolution of $T^{\left( s\left( t\right) ,\omega \right) }$
and $\hat{F}^{\left( s\left( t\right) ,\omega \right) }$, we just take the
horizontal component of (\ref{dtphis}) and substitute $F$ into (\ref{dtphis}%
).

\begin{corollary}
Suppose $s\left( t\right) $ satisfies (\ref{Aevol}), then 
\begin{subequations}%
\label{dtTF} 
\begin{eqnarray}
\frac{\partial T^{\left( s\left( t\right) ,\omega \right) }}{\partial t}
&=&d^{\mathcal{H}}\xi \left( t\right) -\left[ T^{\left( s\left( t\right)
,\omega \right) },\xi \left( t\right) \right] ^{\left( s\left( t\right)
\right) }  \label{dtTF1} \\
\frac{\partial \hat{F}^{\left( s\left( t\right) ,\omega \right) }}{\partial t%
} &=&F\cdot \xi \left( t\right) -\left[ \hat{F}^{\left( s\left( t\right)
,\omega \right) },\xi \left( t\right) \right] ^{\left( s\left( t\right)
\right) }.  \label{dtTF2}
\end{eqnarray}%
\end{subequations}%
\end{corollary}

The expression (\ref{dtTF1}) is the analog of a similar expression for the
evolution of the torsion of a $G_{2}$-structure, as given in \cite%
{GrigorianIsoflow,karigiannis-2007}.

\begin{remark}
Suppose $u_{t}$ is a $1$-parameter family of equivariant $\Psi $-valued maps
that satisfy 
\begin{equation}
\frac{\partial u_{t}}{\partial t}=\left( l_{u_{t}}\right) _{\ast }\gamma _{t}
\end{equation}%
for a $1$-parameter family $\gamma _{t}$ of equivariant $\mathfrak{p}$%
-valued maps. Then, each $u_{t}$ defines a gauge transformation of the
connection $\omega .$ Define 
\begin{equation}
\omega _{t}=u_{t}^{\ast }\omega .
\end{equation}%
Then, it is easy to see that 
\begin{equation}
\frac{\partial \omega _{t}}{\partial t}=d\gamma _{t}+\left[ \omega
_{t},\gamma _{t}\right] _{\mathfrak{p}}=d^{\mathcal{H}_{t}}\gamma _{t},
\label{dtomt}
\end{equation}%
where $d^{\mathcal{H}_{t}}$ is the covariant derivative corresponding to $%
\omega _{t}.$ Similarly, the corresponding curvature $F_{t}$ evolves via the
equation%
\begin{equation}
\frac{\partial F_{t}}{\partial t}=\left[ F_{t},\gamma _{t}\right] _{%
\mathfrak{p}}.  \label{dtFt}
\end{equation}%
Using (\ref{dtomt}) together with (\ref{dtTF1}) gives 
\begin{equation}
\frac{\partial T^{\left( s_{t},\omega _{t}\right) }}{\partial t}=d^{\mathcal{%
H}_{t}}\xi _{t}-\left[ T^{\left( s_{t},\omega _{t}\right) },\xi _{t}\right]
^{\left( s_{t}\right) }+\varphi _{s_{t}}\left( d^{\mathcal{H}_{t}}\gamma
_{t}\right) .  \label{dTstomt}
\end{equation}%
However, 
\begin{eqnarray*}
\varphi _{s_{t}}\left( d^{\mathcal{H}_{t}}\gamma _{t}\right) &=&d^{\mathcal{H%
}_{t}}\hat{\gamma}_{t}^{\left( s_{t}\right) }-\left( d^{\mathcal{H}%
_{t}}\varphi _{s_{t}}\right) \left( \gamma _{t}\right) \\
&=&d^{\mathcal{H}_{t}}\hat{\gamma}_{t}^{\left( s_{t}\right) }-\gamma
_{t}\cdot T^{\left( s_{t},\omega _{t}\right) }-\left[ T^{\left( s_{t},\omega
_{t}\right) },\hat{\gamma}_{t}^{\left( s_{t}\right) }\right] ^{\left(
s_{t}\right) }
\end{eqnarray*}%
and thus (\ref{dTstomt}) becomes 
\begin{equation}
\frac{\partial T^{\left( s_{t},\omega _{t}\right) }}{\partial t}=-\gamma
_{t}\cdot T^{\left( s_{t},\omega _{t}\right) }+d^{\mathcal{H}_{t}}\left( \xi
_{t}+\hat{\gamma}_{t}^{\left( s_{t}\right) }\right) -\left[ T^{\left(
s_{t},\omega _{t}\right) },\xi _{t}+\hat{\gamma}_{t}^{\left( s_{t}\right) }%
\right] ^{\left( s_{t}\right) }.  \label{dTstomt2}
\end{equation}%
For the curvature, using (\ref{dtFt}) together with (\ref{dtTF2}) gives 
\begin{equation}
\frac{\partial \hat{F}^{\left( s_{t},\omega _{t}\right) }}{\partial t}%
=F_{t}\cdot \xi _{t}-\left[ \hat{F}^{\left( s_{t},\omega _{t}\right) },\xi
_{t}\right] ^{\left( s_{t}\right) }+\varphi _{s_{t}}\left( \left[
F_{t},\gamma _{t}\right] _{\mathfrak{p}}\right) .  \label{dFstomt}
\end{equation}%
Using (\ref{xiphi}), we then get%
\begin{equation}
\frac{\partial \hat{F}^{\left( s_{t},\omega _{t}\right) }}{\partial t}%
=-\gamma _{t}\cdot \hat{F}_{t}+F_{t}\cdot \left( \xi _{t}+\hat{\gamma}%
_{t}^{\left( s_{t}\right) }\right) -\left[ \hat{F}^{\left( s_{t},\omega
_{t}\right) },\xi _{t}+\hat{\gamma}_{t}^{\left( s_{t}\right) }\right]
^{\left( s_{t}\right) }.  \label{dFstomt2}
\end{equation}%
Taking $\xi _{t}=-\hat{\gamma}_{t}^{\left( s_{t}\right) }$ in (\ref{dTstomt2}%
) and (\ref{dFstomt2}), we obtain the infinitesimal versions of (\ref{Tuom}).
\end{remark}

\subsection{Variational principles}

\label{sectVar}In general we have seen that the loop bundle structure is
given by $\mathbb{\mathring{L}}$-valued map $s$ as well as a connection $%
\omega $ on $\mathcal{P}.$ We call the pair $\left( s,\omega \right) $ the
configuration of the loop bundle structure. Each point in the configuration
space gives rise to the corresponding torsion $T^{\left( s,\omega \right) }$
and curvature $\hat{F}^{\left( s,\omega \right) }.$Previously we considered $%
T$ and $\hat{F}$ as horizontal equivariant forms on $\mathcal{P}$, but of
course we can equivalently consider them as bundle-valued differential forms
on the base manifold $M$. To be able to define functionals on $M,$ let us
suppose $M$ has a Riemannian metric and moreover, $\mathbb{L}$ has the
following properties:

\begin{enumerate}
\item For each $s\in \mathbb{\mathring{L}}$, the Killing form $K^{\left(
s\right) }$ is nondegenerate and invariant with respect to $\func{ad}%
^{\left( s\right) }$ and the action of $\mathfrak{p}.$

\item $\mathbb{L}$ is a $G$-loop, so that in particular, for each $s\in 
\mathbb{\mathring{L}},$ $\mathfrak{l}^{\left( s\right) }=\mathfrak{q}_{s}.$

\item For each $s\in \mathbb{\mathring{L}}$, the space $\mathfrak{q}_{s}$ is
an irreducible representation of the Lie algebra $\mathfrak{h}_{s}$.
\end{enumerate}

These properties may not be strictly necessary, but they will simplify
arguments. Moreover, these are the properties satisfied by the loop of unit
octonions, which is the key example. The first property means we can defined
the map $\varphi _{s}^{t},$ and then the second and third properties
together make sure that there exists a constant $\lambda $ such that for any 
$s\in \mathbb{\mathring{L}},$ $\varphi _{s}\varphi _{s}^{t}=\lambda \func{id}%
_{\mathfrak{l}}$ and $\varphi _{s}^{t}\varphi _{s}=\lambda \func{id}_{%
\mathfrak{h}_{s}^{\perp }},$ as per Lemma \ref{lemphisphist}. If $\mathfrak{q%
}_{s}$ is a reducible representation, then each irreducible component may
have its own constant. Moreover, the first and second properties together
imply that $K^{\left( s\right) }$ is independent of the choice of $s$, and
when extended as an inner product on sections, it is covariantly constant
with respect to a principal connection on $\mathcal{P}.$

Let $s\in \mathbb{\mathring{L}}$ be fixed. Suppose we have a path of
connections on $\mathcal{P}$ given by $\tilde{\omega}\left( t\right) =\omega
+tA$ for some basic $\mathfrak{p}$-valued $1$-form $A$ and a fixed principal
connection $\omega $. Then, define 
\begin{subequations}
\begin{eqnarray}
T\left( t\right)  &=&T^{\left( s,\tilde{\omega}\left( t\right) \right)
}=\theta _{s}+\varphi _{s}\left( \tilde{\omega}\left( t\right) \right)
=T^{\left( s,\omega \right) }+t\hat{A}. \\
\hat{F}\left( t\right)  &=&\hat{F}^{\left( s,\hat{\omega}\left( t\right)
\right) }=\varphi _{s}\left( F^{\tilde{\omega}\left( t\right) }\right) =\hat{%
F}^{\left( s,\omega \right) }+t\varphi _{s}\left( d^{\mathcal{H}}A\right)  \\
&&+\frac{1}{2}t^{2}\varphi _{s}\left( \left[ A,A\right] _{\mathfrak{p}%
}\right) ,  \notag
\end{eqnarray}%
\end{subequations}%
where $\hat{A}=\varphi _{s}\left( A\right) $. Hence, using (\ref{dhphis}), 
\begin{subequations}
\begin{eqnarray}
\left. \frac{d}{dt}T\left( t\right) \right\vert _{t=0} &=&\hat{A} \\
\left. \frac{d}{dt}\hat{F}\left( t\right) \right\vert _{t=0} &=&\varphi
_{s}\left( d^{\mathcal{H}}A\right) =d^{\mathcal{H}}\hat{A}-\left( d^{%
\mathcal{H}}\varphi _{s}\right) \wedge A  \notag \\
&=&d^{\mathcal{H}}\hat{A}+A\cdot T-\left[ \hat{A},T\right] ^{\left( s\right)
},
\end{eqnarray}%
\end{subequations}%
where for brevity, $T=T\left( 0\right) =T^{\left( s,\omega \right) }$. Note
that if for each $p\in \mathcal{P}$, $A_{p}\in \mathfrak{h}_{s_{p}}$, then
the torsion is unaffected, so these deformations are not relevant for the
loop bundle structure. Hence, let us assume that $A_{p}\in \mathfrak{h}%
_{s_{p}}^{\perp }$ for each $p\in \mathcal{P}.$ Equivalently, this means
that $A\in \varphi _{s}^{t}\left( \mathfrak{l}\right) .$ So now suppose $\xi
\in \Omega _{basic}^{1}\left( \mathcal{P},\mathfrak{l}\right) $ is a basic $%
\mathfrak{l}$-valued $1$-form on $\mathcal{P}\ $such that $A=\frac{1}{%
\lambda }\varphi _{s}^{t}\left( \xi \right) $, and thus, $\hat{A}=\xi .$
Moreover, from (\ref{piqsact}), we see that 
\begin{equation}
A\cdot T=\frac{1}{\lambda }\varphi _{s}^{t}\left( \xi \right) \cdot T=\frac{1%
}{2\lambda ^{2}}\left[ \xi ,T\right] _{\varphi _{s}}+\frac{1}{2}\left[ \xi ,T%
\right] ^{\left( s\right) },
\end{equation}%
where the bracket $\left[ \cdot ,\cdot \right] _{\varphi _{s}}$ on $%
\mathfrak{l}$ is given by 
\begin{equation}
\left[ \xi ,\eta \right] _{\varphi _{s}}=\varphi _{s}\left( \left[ \varphi
_{s}^{t}\left( \xi \right) ,\varphi _{s}^{t}\left( \eta \right) \right] _{%
\mathfrak{p}}\right) ,
\end{equation}%
as defined in (\ref{phisbrack}). Overall, the deformations are now given by 
\begin{subequations}%
\label{TFdeformxi} 
\begin{eqnarray}
\left. \frac{d}{dt}T\left( t\right) \right\vert _{t=0} &=&\xi  \\
\left. \frac{d}{dt}\hat{F}\left( t\right) \right\vert _{t=0} &=&d^{\mathcal{H%
}}\xi +\frac{1}{2\lambda ^{2}}\left[ \xi ,T\right] _{\varphi _{s}}-\frac{1}{2%
}\left[ \xi ,T\right] ^{\left( s\right) }.
\end{eqnarray}%
\end{subequations}%
Suppose now $M$ is a $3$-dimensional compact manifold. For a fixed section $%
s\in \mathcal{\mathring{Q}},$ consider now a functional $\mathcal{F}^{\left(
s\right) }$ on the space of connections on $\mathcal{P}$ modulo $\mathfrak{h}%
_{s},$ given by 
\begin{equation}
\mathcal{F}^{\left( s\right) }\left( \omega \right) =\int_{M}\left\langle T,%
\hat{F}\right\rangle ^{\left( s\right) }-\frac{1}{6\lambda ^{2}}\left\langle
T,\left[ T,T\right] _{\varphi _{s}}\right\rangle ^{\left( s\right) },
\label{Fsfunctional}
\end{equation}%
where wedge products between forms are implicit. From the properties of $T,%
\hat{F},$ $\left[ \cdot ,\cdot \right] _{\varphi _{s}}$, and $\left\langle
{}\right\rangle ^{\left( s\right) }$, we see that that this is invariant
under simultaneous gauge transformation $\left( s,\omega \right) \mapsto
\left( u^{-1}\left( s\right) ,u^{\ast }\omega \right) .$

Now using (\ref{TFdeformxi}) consider deformations of each piece of (\ref%
{Fsfunctional}). For the first term, using (\ref{dHT}), we obtain 
\begin{eqnarray}
\left. \frac{d}{dt}\int_{M}\left\langle T\left( t\right) ,\hat{F}\left(
t\right) \right\rangle ^{\left( s\right) }\right\vert _{t=0}
&=&\int_{M}\left\langle \xi ,\hat{F}\right\rangle ^{\left( s\right) }  \notag
\\
&&+\int_{M}\left\langle T,d^{\mathcal{H}}\xi +\frac{1}{2\lambda ^{2}}\left[
\xi ,T\right] _{\varphi _{s}}-\frac{1}{2}\left[ \xi ,T\right] ^{\left(
s\right) }\right\rangle ^{\left( s\right) }  \notag \\
&=&\int_{M}\left\langle \xi ,\hat{F}+d^{\mathcal{H}}T+\frac{1}{2\lambda ^{2}}%
\left[ T,T\right] _{\varphi _{s}}-\frac{1}{2}\left[ T,T\right] ^{\left(
s\right) }\right\rangle ^{\left( s\right) }  \notag \\
&=&\int_{M}\left\langle \xi ,2\hat{F}+\frac{1}{2\lambda ^{2}}\left[ T,T%
\right] _{\varphi _{s}}\right\rangle ^{\left( s\right) },  \label{dtFs1}
\end{eqnarray}%
For the second term in (\ref{Fsfunctional}), using Lemma \ref{lemPhibrack2},
we obtain 
\begin{equation}
-\frac{1}{6\lambda ^{2}}\left. \frac{d}{dt}\int_{M}\left\langle T,\left[ T,T%
\right] _{\varphi _{s}}\right\rangle ^{\left( s\right) }\right\vert _{t=0}=-%
\frac{1}{2\lambda ^{2}}\int_{M}\left\langle \xi ,\left[ T,T\right] _{\varphi
_{s}}\right\rangle ^{\left( s\right) }.  \label{dtFs2}
\end{equation}%
Combining (\ref{dtFs1}) and (\ref{dtFs2}), we obtain 
\begin{equation}
\left. \frac{d}{dt}\mathcal{F}^{\left( s\right) }\left( \tilde{\omega}\left(
t\right) \right) \right\vert _{t=0}=2\int_{M}\left\langle \xi ,\hat{F}%
\right\rangle ^{\left( s\right) }.  \label{dtFs3}
\end{equation}%
Therefore, we see that the critical points of $\mathcal{F}^{\left( s\right)
} $ are precisely the connections for which $\hat{F}=0$. This gives a
generalization of the standard Chern-Simons functional.

\begin{remark}
The condition $\hat{F}=0$ means that each point, the curvature $F^{\left(
\omega \right) }$ lies in $\mathfrak{h}_{s}.$ This is a restriction on the
Lie algebra part of the curvature. Usually instanton conditions on curvature
give conditions on the $2$-form part. So what we have here is a different
kind of condition to an instanton, and there is term for this, coined by
Spiro Karigiannis - an \emph{extanton}. As we from Example \ref{exCx4}, on a
K\"{a}hler manifold, this just corresponds to the Ricci-flat condition.
\end{remark}

The above construction on $3$-manifolds can be extended to an $n$%
-dimensional manifold $M$ if we have a closed $\left( n-3\right) $%
-dimensional form. In that case, similarly as in \cite{DonaldsonHigherDim},
consider the functional 
\begin{equation}
\mathcal{F}^{\left( s\right) }\left( \omega \right) =\int_{M^{n}}\left(
\left\langle T,\hat{F}\right\rangle ^{\left( s\right) }-\frac{1}{6\lambda
^{2}}\left\langle T,\left[ T,T\right] _{\varphi _{s}}\right\rangle ^{\left(
s\right) }\right) \wedge \psi .
\end{equation}%
In this case, the critical points then satisfy 
\begin{equation}
\hat{F}\wedge \psi =0.  \label{Extantonndim}
\end{equation}%
For example if $M$ is a $7$-dimensional manifold with a \emph{co-closed }$%
G_{2}$-structure, i.e. $\psi =\ast \varphi $ is closed, then (\ref%
{Extantonndim}) shows that as a $2$-form, $\hat{F}$ has a vanishing
component in the $7$-dimensional representation of $G_{2}.$ In contrast,
Donaldson-Thomas connections \cite{DonaldsonHigherDim} satisfy $F\wedge \psi
=0$. If $F=\func{Riem}$, is the Riemann curvature on the frame bundle, then
equation (\ref{Extantonndim}) shows that, in local coordinates, 
\begin{equation}
\func{Riem}_{ijkl}\varphi _{\ \alpha }^{ij}\varphi _{\ \ \beta }^{kl}=0.
\label{Extanton7dim}
\end{equation}%
The quantity on the left-hand side of (\ref{Extanton7dim}), is sometimes
denoted as $\func{Ric}^{\ast }$ \cite%
{CleytonIvanovClosed,CleytonIvanovCurv,GrigorianFlowSurvey}. The traceless
part of $\func{Ric}^{\ast }$ corresponds to a component of the Riemann
curvature that lies in a $27$-dimensional representation of $G_{2}$, with
another $27$-dimensional component given by the traceless Ricci tensor $%
\func{Ric}$.

Now consider the functional (\ref{Fsfunctional}), however now as functional
on sections of $\mathcal{\mathring{Q}}$ for a fixed connection $\omega $, so
that now we vary $s$.%
\begin{equation}
\mathcal{F}^{\left( \omega \right) }\left( s\right) =\int_{M}\left\langle T,%
\hat{F}\right\rangle ^{\left( s\right) }-\frac{1}{6\lambda ^{2}}\left\langle
T,\left[ T,T\right] _{\varphi _{s}}\right\rangle ^{\left( s\right) },
\end{equation}%
Suppose we have 
\begin{subequations}%
\label{sdeforms}%
\begin{eqnarray}
\frac{\partial T^{\left( s\left( t\right) ,\omega \right) }}{\partial t}
&=&d^{\mathcal{H}}\eta \left( t\right) -\left[ T^{\left( s\left( t\right)
,\omega \right) },\eta \left( t\right) \right] ^{\left( s\left( t\right)
\right) } \\
\frac{\partial \hat{F}^{\left( s\left( t\right) ,\omega \right) }}{\partial t%
} &=&F\cdot \eta \left( t\right) -\left[ \hat{F}^{\left( s\left( t\right)
,\omega \right) },\eta \left( t\right) \right] ^{\left( s\left( t\right)
\right) }.
\end{eqnarray}%
\end{subequations}%
for some $\eta \in \Gamma \left( \mathcal{A}\right) .$ Let us now make
additional assumptions:

\begin{enumerate}
\item $\left[ \cdot ,\cdot \right] _{\varphi _{s}}=k\left[ \cdot ,\cdot %
\right] ^{\left( s\right) }$

\item $\mathbb{L}$ is alternative
\end{enumerate}

The last assumption implies in particular, that the associator is
skew-symmetric, and moreover, for any $\alpha ,\beta ,\xi ,\eta \in 
\mathfrak{l}^{\left( s\right) }$, 
\begin{equation}
\left\langle a_{s}\left( \alpha ,\beta ,\xi \right) ,\eta \right\rangle
^{\left( s\right) }=\left\langle \xi ,a_{s}\left( \alpha ,\beta ,\eta
\right) \right\rangle ^{\left( s\right) }.
\end{equation}%
Now, 
\begin{equation}
\mathcal{F}^{\left( \omega \right) }\left( s\right) =\int_{M}\left\langle T,%
\hat{F}\right\rangle ^{\left( s\right) }-\frac{k}{6\lambda ^{2}}\left\langle
T,\left[ T,T\right] ^{\left( s\right) }\right\rangle ^{\left( s\right) },
\end{equation}%
and in this case the derivative of $\mathcal{F}^{\left( \omega \right)
}\left( s\right) $ is 
\begin{eqnarray}
\left. \frac{d}{dt}\mathcal{F}^{\left( \omega \right) }\left( s\left(
t\right) \right) \right\vert _{t=0} &=&\int_{M}\left\langle d^{\mathcal{H}%
}\eta -\left[ T,\eta \right] ^{\left( s\right) },\hat{F}\right\rangle
^{\left( s\right) }+\int_{M}\left\langle T,F\cdot \eta -\left[ \hat{F},\eta %
\right] ^{\left( s\right) }\right\rangle ^{\left( s\right) }  \notag \\
&&-\frac{k}{2\lambda ^{2}}\int_{M}\left\langle d^{\mathcal{H}}\eta -\left[
T,\eta \right] ^{\left( s\right) },\left[ T,T\right] ^{\left( s\right)
}\right\rangle ^{\left( s\right) }  \notag \\
&&-\frac{k}{6\lambda ^{2}}\int_{M}\left\langle T,a_{s}\left( T,T,\eta
\right) \right\rangle ^{\left( s\right) }.  \label{dtFoms}
\end{eqnarray}%
Consider the first two terms in (\ref{dtFoms}). 
\begin{subequations}%
\begin{eqnarray}
\int_{M}\left\langle d^{\mathcal{H}}\eta -\left[ T,\eta \right] ^{\left(
s\right) },\hat{F}\right\rangle ^{\left( s\right) } &=&\int_{M}\left\langle
\eta ,-d^{\mathcal{H}}\hat{F}-\left[ \hat{F},T\right] ^{\left( s\right)
}\right\rangle ^{\left( s\right) } \\
\int_{M}\left\langle T,F\cdot \eta -\left[ \hat{F},\eta \right] ^{\left(
s\right) }\right\rangle ^{\left( s\right) } &=&\int_{M}\left\langle \eta ,%
\left[ \hat{F},T\right] ^{\left( s\right) }-F\cdot T\right\rangle ^{\left(
s\right) }.
\end{eqnarray}%
\end{subequations}%
The third term in (\ref{dtFoms}) becomes 
\begin{eqnarray*}
\int_{M}\left\langle d^{\mathcal{H}}\eta -\left[ T,\eta \right] ^{\left(
s\right) },\left[ T,T\right] ^{\left( s\right) }\right\rangle ^{\left(
s\right) } &=&\int_{M}\left\langle \eta ,-d^{\mathcal{H}}\left[ T,T\right]
^{\left( s\right) }+\left[ T,\left[ T,T\right] ^{\left( s\right) }\right]
^{\left( s\right) }\right\rangle ^{\left( s\right) } \\
&=&\int_{M}\left\langle \eta ,-2\left[ \hat{F},T\right] ^{\left( s\right)
}+a_{s}\left( T,T,T\right) \right\rangle ^{\left( s\right) }.
\end{eqnarray*}%
The last term in (\ref{dtFoms}) is 
\begin{equation*}
\int_{M}\left\langle T,a_{s}\left( T,T,\eta \right) \right\rangle ^{\left(
s\right) }=\int_{M}\left\langle \eta ,a_{s}\left( T,T,T\right) \right\rangle
^{\left( s\right) }.
\end{equation*}%
Overall, since $a_{s}\left( T,T,T\right) =\left[ T,\left[ T,T\right]
^{\left( s\right) }\right] ^{\left( s\right) }$, 
\begin{eqnarray}
\left. \frac{d}{dt}\mathcal{F}^{\left( \omega \right) }\left( s\left(
t\right) \right) \right\vert _{t=0} &=&-\int_{M}\left\langle \eta ,d^{%
\mathcal{H}}\hat{F}+F\cdot T-\frac{k}{\lambda ^{2}}\left[ \hat{F},T\right]
^{\left( s\right) }\right\rangle ^{\left( s\right) } \\
&&-\int_{M}\,\left\langle \eta ,\frac{2k}{3\lambda ^{2}}\left[ T,\left[ T,T%
\right] ^{\left( s\right) }\right] ^{\left( s\right) }\right\rangle ^{\left(
s\right) }.  \notag
\end{eqnarray}

From the Bianchi identity (\ref{Bianchi}), 
\begin{equation*}
F\cdot T=d^{\mathcal{H}}\hat{F}+\left[ \hat{F},T\right] ^{\left( s\right) },
\end{equation*}%
and thus, 
\begin{eqnarray}
\left. \frac{d}{dt}\mathcal{F}^{\left( \omega \right) }\left( s\left(
t\right) \right) \right\vert _{t=0} &=&-\int_{M}\left\langle \eta ,2d^{%
\mathcal{H}}\hat{F}+\left( 1-\frac{k}{\lambda ^{2}}\right) \left[ \hat{F},T%
\right] ^{\left( s\right) }\right\rangle ^{\left( s\right) } \\
&&-\int_{M}\,\left\langle \eta ,\frac{2k}{3\lambda ^{2}}\left[ T,\left[ T,T%
\right] ^{\left( s\right) }\right] ^{\left( s\right) }\right\rangle ^{\left(
s\right) }.
\end{eqnarray}%
Thus, the critical points with respect to deformations of $s$ satisfy%
\begin{equation}
d^{\mathcal{H}}\hat{F}+\left( \frac{1}{2}-\frac{k}{2\lambda ^{2}}\right) %
\left[ \hat{F},T\right] ^{\left( s\right) }+\frac{k}{3\lambda ^{2}}\left[ T,%
\left[ T,T\right] ^{\left( s\right) }\right] ^{\left( s\right) }=0.
\label{SecondCrit}
\end{equation}

\begin{example}
In the case when $\mathbb{L}$ is a Lie group, $a_{s}=0$ and $k=\lambda =1$,
so we just obtain $d^{\mathcal{H}}\hat{F}=0$, which is of course the Bianchi
identity. This shows that we just have a reduction from a $\Psi ^{R}\left( 
\mathbb{L}\right) $ connection to an $\mathbb{L}$-connection. In the case of 
$\mathbb{L}$ being the loop of unit octonions, we know $\lambda =\frac{3}{8}$
and $k=3\lambda ^{3}=\frac{81}{512}$ so (\ref{SecondCrit}) becomes%
\begin{equation}
d^{\mathcal{H}}\hat{F}-\frac{1}{16}\left[ \hat{F},T\right] ^{\left( s\right)
}+\frac{3}{8}\left[ T,\left[ T,T\right] ^{\left( s\right) }\right] ^{\left(
s\right) }=0.
\end{equation}%
The significance of this condition is not immediately clear.
\end{example}

However combining the two variations, we find that critical points over $%
\left( s,\omega \right) $ satisfy 
\begin{equation*}
\left\{ 
\begin{array}{c}
\hat{F}=0 \\ 
\left[ T,T,T\right] ^{\left( s\right) }=0%
\end{array}%
\right. .
\end{equation*}

\begin{remark}
It will be the subject of further work to understand the significance of
this Chern-Simons type functional $\mathcal{F}$. In particular, given the
non-trivial $3$-form $\left[ T,\left[ T,T\right] ^{\left( s\right) }\right]
^{\left( s\right) }$, there may be additional possibilities for similar
higher-dimensional functionals. The functional $\mathcal{F}$ is invariant
under simultaneous gauge transformations of $\left( s,\omega \right) ,$ but
not the individual ones. For the standard Chern-Simons functional in 3
dimensions, the lack of gauge invariance causes it to be multi-valued, with
only the exponentiated action functional becomes truly gauge-invariant. It
will be interesting to see if there are any analogous properties in this
case.
\end{remark}

In the context of $G_{2}$-structures, another functional has been considered
in several papers \cite{Bagaglini2,DGKisoflow,GrigorianOctobundle,
GrigorianIsoflow,SaEarpLoubeau}, namely the $L_{2}$-norm of the torsion,
considered as functional on the space of isometric $G_{2}$-structures, i.e. $%
G_{2}$-structures that correspond to the same metric. In the context of loop
structures we may define a similar functional. Given a compact Riemannian
manifold $\left( M,g\right) $ and a fixed connection $\omega $ on $\mathcal{P%
}$, for any section $s\in \Gamma \left( \mathcal{\mathring{Q}}\right) $ let $%
T^{\left( s\right) }$ be the torsion of $s$ with respect to $\omega .$ Then
define the energy functional on $\Gamma \left( \mathcal{\mathring{Q}}\right) 
$ given by: 
\begin{equation}
\mathcal{E}\left( s\right) =\int_{M}\left\langle T^{\left( s\right) },\ast
T^{\left( s\right) }\right\rangle ^{\left( s\right) },  \label{Efunc}
\end{equation}%
where the wedge product is assumed. With respect to deformations of $s$
given by (\ref{Aevol}) and the corresponding deformation of $T$ given by (%
\ref{sdeforms}) we have 
\begin{eqnarray}
\left. \frac{d}{dt}\mathcal{E}\left( s_{t}\right) \right\vert _{t=0}
&=&2\int_{M}\left\langle d^{\mathcal{H}}\eta -\left[ T^{\left( s\right)
},\eta \right] ^{\left( s\right) },\ast T^{\left( s\right) }\right\rangle
^{\left( s\right) }  \notag \\
&=&-2\int_{M}\left\langle \eta ,d^{\mathcal{H}}\ast T^{\left( s\right) }-%
\left[ T^{\left( s\right) },\ast T^{\left( s\right) }\right] ^{\left(
s\right) }\right\rangle ^{\left( s\right) }  \notag \\
&=&-2\int_{M}\left\langle \eta ,d^{\mathcal{H}}\ast T^{\left( s\right)
}\right\rangle ^{\left( s\right) },  \label{Edeform}
\end{eqnarray}%
where $\left[ T^{\left( s\right) },\ast T^{\left( s\right) }\right] ^{\left(
s\right) }=0$ due to symmetry considerations. Thus the critical points of $%
\mathcal{E}$ satisfy 
\begin{equation}
\left( d^{\mathcal{H}}\right) ^{\ast }T^{\left( s\right) }=0,  \label{divT0}
\end{equation}%
which is precisely the analog of the \textquotedblleft divergence-free
torsion\textquotedblright\ condition in \cite%
{Bagaglini2,DGKisoflow,GrigorianOctobundle, GrigorianIsoflow,SaEarpLoubeau}.
Also, similarly as in \cite{SaEarpLoubeau}, if we assume $\mathcal{P}$ is
compact, the functional $\mathcal{E}$ may be related to the equivariant
Dirichlet energy functional for maps from $\mathcal{P}$ to $\mathbb{%
\mathring{L}}$. Given a metric $\left\langle \cdot ,\cdot \right\rangle
^{\left( s\right) }$ on $\mathfrak{l}$, we may extend it to a metric on all
of $\mathbb{L}$ via right translations: $\left\langle \cdot ,\cdot
\right\rangle _{p}^{\left( s\right) }=\left\langle \left( R_{p}\right)
_{\ast }^{-1}\cdot ,\left( R_{p}\right) _{\ast }^{-1}\cdot \right\rangle
^{\left( s\right) }.$ Then, the Dirichlet energy functional on \emph{%
equivariant} maps from $\mathcal{P}$ to $\mathbb{\mathring{L}}$ is given by 
\begin{equation}
\mathcal{D}\left( s\right) =\int_{\mathcal{P}}\left\vert ds\right\vert
^{2}=\int_{\mathcal{P}}\left\vert \theta _{s}\right\vert ^{2},
\label{dirichletE}
\end{equation}%
where we endow $T\mathcal{P}$ with a metric such that the decomposition $T%
\mathcal{P=HP\oplus VP}$ is orthogonal with respect to it, and moreover such
that it is compatible with the metrics on $M$ and $\Psi $. Then, using (\ref%
{stheta}) 
\begin{equation}
\mathcal{D}\left( s\right) =\int_{\mathcal{P}}\left\vert T^{\left( s\right)
}\right\vert ^{2}+\int_{\mathcal{P}}\left\vert \hat{\omega}^{\left( s\right)
}\right\vert ^{2}
\end{equation}%
Note that given an orthogonal basis $\left\{ X_{i}\right\} $ on $\mathfrak{p}
$, $\left\vert \hat{\omega}^{\left( s\right) }\right\vert ^{2}=\left\vert 
\hat{\omega}^{\left( s\right) }\left( \sigma \left( X_{i}\right) \right)
\right\vert ^{2}=\left\vert \hat{X}_{i}\right\vert ^{2}=\lambda _{s}\dim 
\mathfrak{l}.$ With our previous assumptions, $\lambda _{s}=\lambda $ - does
not depend on $s$, so we have 
\begin{equation*}
\mathcal{D}\left( s\right) =a\mathcal{E}\left( s\right) +b
\end{equation*}%
where $a=\func{Vol}\left( \Psi \right) $ and $b=\lambda \left( \dim \mathbb{L%
}\right) $ $\func{Vol}\left( \mathcal{P}\right) .$ Hence, the critical
points of $\mathcal{E}\left( s\right) $ are precisely the critical points of 
$\mathcal{D}\left( s\right) $ with respect to deformations through
equivariant maps, i.e. equivariant harmonic maps. So indeed, to understand
the properties of these critical points, a rigorous equivariant harmonic map
theory is required, as initiated in \cite{SaEarpLoubeau}.

\section{Concluding remarks}

\setcounter{equation}{0}\label{sectConclusion}Given a smooth loop $\mathbb{L}
$ with tangent algebra $\mathfrak{l}$ and a group $\Psi $ that acts smoothly
on $\mathbb{L}$ via pseudoautomorphism pairs, we have defined the concept of
a loop bundle structure $\left( \mathbb{L},\Psi ,\mathcal{P},s\right) $ for
a principal $\Psi $-bundle and a corresponding equivariant $\mathbb{%
\mathring{L}}$-valued map $s$, that also defines a section of the
corresponding associated bundle. If we moreover have a connection $\omega $
on $\mathcal{P}$, then horizontal component of the Darboux derivative of $s$
defines an $\mathfrak{l}$-valued $1$-form $T^{\left( s,\omega \right) }$,
which we called the torsion. This object $T^{\left( s,\omega \right) }$ then
satisfies a structural equation based on the loop Maurer-Cartan equation and
gives rise to an $\mathfrak{l}$-valued component of the curvature $\hat{F}%
^{\left( s,\omega \right) }.$ Overall, there are several possible directions
to further this non-associative theory.

\begin{enumerate}
\item From a more algebraic perspective it would be interesting to construct
additional examples of smooth loops, in particular those that are not
Moufang and possibly are not even $G$-loops in order to more concretely
study the corresponding bundles in those situations. In fact, it may not
even be necessary to have a full loop structure - it may be sufficient to
just have a right loop structure, so that division is possible only on the
right. Left division was used rarely, and it may be possible to build up a
full theory without needing it. New examples of loops may give rise to new
geometric structures.

\item In Lie theory, the Maurer-Cartan equation plays a central role. As
we've seen there is an analog in smooth loop theory as well. A better
understanding of this equation is needed. The standard Maurer-Cartan
equation is closely related to the concept of integrability, but it is not
clear how to interpret the non-associative version.

\item In defining the loop bundle structure, we generally have assumed that
the map $s$ is globally defined. However, this may place strict topological
restrictions. It may be reasonable to allow $s$ to be defined only locally.
This would give more flexibility, but it would need to be checked carefully
whether other related quantities are well-defined.

\item We have defined a functional of Chern-Simons type in Section \ref%
{sectVar}. There are further properties that need to be investigated. For
example, is it possible to use the associator to define reasonable
functionals on higher-dimensional manifolds? If the section $s$ is defined
only locally, are these functionals well-defined? Finally, do these
functionals have any topological meaning?

\item In $G_{2}$-geometry, significant progress has been made in \cite%
{Bagaglini2,DGKisoflow,GrigorianOctobundle, GrigorianIsoflow,SaEarpLoubeau}
regarding the existence of critical points of the energy functional (\ref%
{Efunc}) via a heat flow approach. However, it is likely that a more direct
approach, similar to Uhlenbeck's existence result for the Coulomb gauge \cite%
{UhlenbeckConnection}, could also be used. This would give existence of a
preferred section $s$ for a given connection or conversely, a preferred
connection in a gauge class for a fixed section $s$.
\end{enumerate}

Overall, the framework presented in this paper may give an impetus to the
development of a larger theory of \textquotedblleft nonassociative
geometry\textquotedblright .

\appendix 

\section{Appendix}

\setcounter{equation}{0}\label{secAppendix}

\begin{lemma}
\label{lemQuotient}Suppose $A\left( t\right) $ and $B\left( t\right) $ are
smooth curves in $\mathbb{L}$ with $A\left( 0\right) =A_{0}$ and $B\left(
0\right) =B_{0}$, then 
\begin{subequations}%
\begin{eqnarray}
\left. \frac{d}{dt}A\left( t\right) /B\left( t\right) \right\vert _{t=0}
&=&\left. \frac{d}{dt}A\left( t\right) /B_{0}\right\vert _{t=0}-\left. \frac{%
d}{dt}\left( A_{0}/B_{0}\cdot B\left( t\right) \right) /B_{0}\right\vert
_{t=0}  \label{ddtrighquot} \\
\left. \frac{d}{dt}B\left( t\right) \backslash A\left( t\right) \right\vert
_{t=0} &=&\left. \frac{d}{dt}B_{0}\backslash A\left( t\right) \right\vert
_{t=0}-\left. \frac{d}{dt}B_{0}\backslash \left( B\left( t\right) \cdot
B_{0}\backslash A_{0}\right) \right\vert _{t=0}  \label{ddtleftquot}
\end{eqnarray}%
\end{subequations}%
\end{lemma}

\begin{proof}
First note that 
\begin{eqnarray*}
\left. \frac{d}{dt}A\left( t\right) \right\vert _{t=0} &=&\left. \frac{d}{dt}%
\left( A\left( t\right) /B\left( t\right) \cdot B\left( t\right) \right)
\right\vert _{t=0} \\
&=&\left. \frac{d}{dt}\left( A\left( t\right) /B\left( t\right) \right)
\cdot B_{0}\right\vert _{t=0}+\left. \frac{d}{dt}\left( A_{0}/B_{0}\cdot
B\left( t\right) \right) \right\vert _{t=0} \\
&=&\left( R_{B_{0}}\right) _{\ast }\left. \frac{d}{dt}A\left( t\right)
/B\left( t\right) \right\vert _{t=0}+\left. \frac{d}{dt}\left(
A_{0}/B_{0}\cdot B\left( t\right) \right) \right\vert _{t=0}
\end{eqnarray*}%
Hence, applying $\left( R_{B_{0}}^{-1}\right) _{\ast }$ to both sides, we
obtain (\ref{ddtrighquot}). Similarly, 
\begin{eqnarray*}
\left. \frac{d}{dt}A\left( t\right) \right\vert _{t=0} &=&\left. \frac{d}{dt}%
\left( B\left( t\right) \cdot B\left( t\right) \backslash A\left( t\right)
\right) \right\vert _{t=0} \\
&=&\left( L_{B_{0}}\right) _{\ast }\left. \frac{d}{dt}\left( B\left(
t\right) \backslash A\left( t\right) \right) \right\vert _{t=0}+\left. \frac{%
d}{dt}\left( B\left( t\right) \cdot B_{0}\backslash A_{0}\right) \right\vert
_{t=0}
\end{eqnarray*}%
and applying $\left( L_{B_{0}}^{-1}\right) _{\ast }$ to both sides gives (%
\ref{ddtleftquot}).
\end{proof}

\begin{lemma}[Lemma \protect\ref{lemAssoc}]
\label{lemAssoc2}For fixed $\eta ,\gamma \in \mathfrak{l},$ 
\begin{equation}
\left. db\right\vert _{p}\left( \eta ,\gamma \right) =\left[ \eta ,\gamma
,\theta _{p}\right] ^{\left( p\right) }-\left[ \gamma ,\eta ,\theta _{p}%
\right] ^{\left( p\right) }\text{,}
\end{equation}%
where $\left[ \cdot ,\cdot ,\cdot \right] ^{\left( p\right) }$ is the $%
\mathbb{L}$\emph{-algebra associator }on $\mathfrak{l}^{\left( p\right) }$
given by 
\begin{eqnarray}
\left[ \eta ,\gamma ,\xi \right] ^{\left( p\right) } &=&\left. \frac{d^{3}}{%
dtd\tau d\tau ^{\prime }}\exp \left( \tau \eta \right) \circ _{p}\left( \exp
\left( \tau ^{\prime }\gamma \right) \circ _{p}\exp \left( t\xi \right)
\right) \right\vert _{t,\tau ,\tau ^{\prime }=0} \\
&&-\left. \frac{d^{3}}{dtd\tau d\tau ^{\prime }}\left( \exp \left( \tau \eta
\right) \circ _{p}\exp \left( \tau ^{\prime }\gamma \right) \right) \circ
_{p}\exp \left( t\xi \right) \right\vert _{t,\tau ,\tau ^{\prime }=0}. 
\notag
\end{eqnarray}%
Moreover, 
\begin{equation}
\left[ \eta ,\gamma ,\xi \right] ^{\left( p\right) }=\left. \frac{d^{3}}{%
dtd\tau d\tau ^{\prime }}\left[ \exp \left( \tau \eta \right) ,\exp \left(
\tau ^{\prime }\gamma \right) ,\exp \left( t\xi \right) \right] ^{\left( 
\mathbb{L},\circ _{p}\right) }\right\vert _{t,\tau ,\tau ^{\prime }=0}
\end{equation}%
where $\left[ \cdot ,\cdot ,\cdot \right] ^{\left( \mathbb{L},\circ
_{p}\right) }$ is the loop associator on $\left( \mathbb{L},\circ
_{p}\right) $ as defined by (\ref{loopassoc2}).
\end{lemma}

\begin{proof}
Let $X=\rho \left( \xi \right) $ and $x\left( t\right) =\exp _{p}\left( t\xi
\right) p$, then consider 
\begin{eqnarray}
X\left( b\left( \eta ,\gamma \right) \right) _{p} &=&\left. \frac{d}{dt}%
\left( \left[ \eta ,\gamma \right] ^{x\left( t\right) }\right) \right\vert
_{t=0}  \notag \\
&=&\left. \frac{d^{3}}{dtd\tau d\tau ^{\prime }}\left( \exp \left( \tau \eta
\right) \circ _{x\left( t\right) }\exp \left( \tau ^{\prime }\gamma \right)
\right) \right\vert _{t,\tau ,\tau ^{\prime }=0}  \label{Xb} \\
&&-\left. \frac{d^{3}}{dtd\tau d\tau ^{\prime }}\left( \exp \left( \tau
^{\prime }\gamma \right) \circ _{x\left( t\right) }\exp \left( \tau \eta
\right) \right) \right\vert _{t,\tau ,\tau ^{\prime }=0}  \notag
\end{eqnarray}%
where we have used (\ref{brack2deriv}). Then, 
\begin{equation}
\exp \left( \tau \eta \right) \circ _{x\left( t\right) }\exp \left( \tau
^{\prime }\gamma \right) =\faktor{\left( \exp \left( \tau \eta \right)
\left( \exp \left( \tau ^{\prime }\gamma \right) x\left( t\right) \right)
\right)} {x\left( t\right)} .
\end{equation}%
For brevity let us $\Xi $ for $\exp $, and $d_{0}^{3}$ for $\left. \frac{%
d^{3}}{dtd\tau d\tau ^{\prime }}\right\vert _{t,\tau ,\tau ^{\prime }=0},$
so that, using Lemma \ref{lemQuotient}, we get 
\begin{eqnarray}
d_{0}^{3}\left( \Xi \left( \tau \eta \right) \circ _{x\left( t\right) }\Xi
\left( \tau ^{\prime }\gamma \right) \right) &=&d_{0}^{3}\left(\faktor{ \Xi
\left( \tau \eta \right) \left( \Xi \left( \tau ^{\prime }\gamma \right)
x\left( t\right) \right)}{p} \right)  \label{Xb1} \\
&&-d_{0}^{3}\left(\faktor{ \left( \faktor{\Xi \left( \tau \eta \right)
\left( \Xi \left( \tau ^{\prime }\gamma \right) p\right)} {p} \right) \cdot
x\left( t\right) } {p}\right)  \notag
\end{eqnarray}%
However, 
\begin{eqnarray}
\left( \Xi \left( \tau \eta \right) \left( \Xi \left( \tau ^{\prime }\gamma
\right) x\left( t\right) \right) \right) /p &=&\left( \Xi \left( \tau \eta
\right) \left( \Xi \left( \tau ^{\prime }\gamma \right) \left( %
\faktor{x\left( t\right)}{p}\cdot p\right) \right) \right) /p  \notag \\
&=&\left( \Xi \left( \tau \eta \right) \left( \left( \Xi \left( \tau
^{\prime }\gamma \right) \circ _{p}\left( \faktor{x\left( t\right)} {p}%
\right) \right) p\right) \right) /p  \notag \\
&=&\Xi \left( \tau \eta \right) \circ _{p}\left( \Xi \left( \tau ^{\prime
}\gamma \right) \circ _{p}\Xi _{p}\left( t\xi \right) \right)  \label{Xb2a}
\end{eqnarray}%
and similarly, 
\begin{equation}
\left( \faktor{\Xi \left( \tau \eta \right) \left( \Xi \left( \tau ^{\prime
}\gamma \right) p\right)} {p}\cdot x\left( t\right) \right) /p=\left( \Xi
\left( \tau \eta \right) \circ _{p}\Xi \left( \tau ^{\prime }\gamma \right)
\right) \circ _{p}\Xi _{p}\left( t\xi \right) .
\end{equation}%
The derivatives of $\Xi _{p}\left( t\xi \right) $ and $\Xi \left( t\xi
\right) $ with respect to $t$ at $t=0$ are equal, thus, from (\ref{Xb1}), we
find 
\begin{equation}
\left. \frac{d^{3}}{dtd\tau d\tau ^{\prime }}\left( \Xi \left( \tau \eta
\right) \circ _{x\left( t\right) }\Xi \left( \tau ^{\prime }\gamma \right)
\right) \right\vert _{t,\tau ,\tau ^{\prime }=0}=\left[ \eta ,\gamma ,\xi %
\right] ^{\left( p\right) }  \label{Xb4}
\end{equation}%
and hence, from (\ref{Xb}),%
\begin{equation}
X\left( b\left( \eta ,\gamma \right) \right) _{p}=\left[ \eta ,\gamma ,\xi %
\right] ^{\left( p\right) }-\left[ \gamma ,\eta ,\xi \right] ^{\left(
p\right) }.  \label{Xb5}
\end{equation}%
For the last part, using (\ref{loopassoc2}) and Lemma \ref{lemQuotient}, we
get 
\begin{eqnarray*}
d_{0}^{3}\left[ \Xi \left( \tau \eta \right) ,\Xi \left( \tau ^{\prime
}\gamma \right) ,\Xi \left( t\xi \right) \right] ^{\left( \mathbb{L},\circ
_{p}\right) } &=&d_{0}^{3}\left( \Xi \left( \tau \eta \right) \circ _{\Xi
_{p}\left( t\xi \right) p}\Xi \left( \tau ^{\prime }\gamma \right) \right)
/_{p}\left( \Xi \left( \tau \eta \right) \circ _{p}\Xi \left( \tau ^{\prime
}\gamma \right) \right) \\
&=&d_{0}^{3}\left( \Xi \left( \tau \eta \right) \circ _{\Xi _{p}\left( t\xi
\right) p}\Xi \left( \tau ^{\prime }\gamma \right) \right) /_{p}\Xi \left(
\tau ^{\prime }\gamma \right) \\
&&-d_{0}^{3}\left( \Xi \left( \tau \eta \right) \circ _{p}\Xi \left( \tau
^{\prime }\gamma \right) \right) /\Xi \left( \tau \eta \right) \\
&=&d_{0}^{3}\left( \Xi \left( \tau \eta \right) \circ _{\Xi _{p}\left( t\xi
\right) p}\Xi \left( \tau ^{\prime }\gamma \right) \right)
\end{eqnarray*}%
and hence from (\ref{Xb4}) we see that indeed (\ref{Lalgassoc2}) holds.
\end{proof}

\begin{lemma}
\label{lemdtAd}Suppose $s\left( t\right) $ and $f\left( t\right) $ are
smooth curves in $\mathbb{L}$ with $s\left( 0\right) =s$, $f\left( 0\right)
=f$, $\dot{s}\left( 0\right) =\dot{s},$ $\dot{f}\left( 0\right) =\dot{f}.$
Also, let $\xi \in \mathfrak{l}$, then 
\begin{eqnarray}
\left. \frac{d}{dt}\left( \func{Ad}_{f\left( t\right) }^{\left( s\left(
t\right) \right) }\right) _{\ast }\xi \right\vert _{t,\tau =0} &=&\left[
\left( R_{f}^{\left( s\right) }\right) _{\ast }^{-1}\dot{f},\left( \func{Ad}%
_{f}^{\left( s\right) }\right) _{\ast }\xi \right] ^{\left( fs\right) }
\label{dAdfs} \\
&&-\left( R_{f}^{\left( s\right) }\right) _{\ast }^{-1}\left[ \left(
R_{f}^{\left( s\right) }\right) _{\ast }^{-1}\dot{f},f,\xi \right] ^{\left(
s\right) }  \notag \\
&&+\left( R_{f}^{\left( s\right) }\right) _{\ast }^{-1}\left[ f,\xi ,\left(
R_{s}\right) _{\ast }^{-1}\dot{s}\right] ^{\left( s\right) }  \notag \\
&&-\left( R_{f}^{\left( s\right) }\right) _{\ast }^{-1}\left[ \left( \func{Ad%
}_{f}^{\left( s\right) }\right) _{\ast }\xi ,f,\left( R_{s}\right) _{\ast
}^{-1}\dot{s}\right] ^{\left( s\right) }.  \notag
\end{eqnarray}
\end{lemma}

\begin{proof}
Let $\xi \in \mathfrak{l},$ and consider $s_{t}=s\left( t\right) ,$ $%
f_{t}=f\left( t\right) $, then, for brevity suppressing pushforwards, we
have 
\begin{eqnarray}
\left. \frac{d}{dt}\left( \func{Ad}_{f_{t}}^{\left( s_{t}\right) }\right)
_{\ast }\xi \right\vert _{t=0} &=&\left. \frac{d}{dt}\left( f_{t}\circ
_{s_{t}}\xi \right) /_{s_{t}}f_{t}\right\vert _{t=0}  \notag \\
&=&\left. \frac{d}{dt}\left( f\circ _{s_{t}}\xi \right)
/_{s_{t}}f\right\vert _{t=0}+\left. \frac{d}{dt}\left( f_{t}\circ _{s}\xi
\right) /_{s}f_{t}\right\vert _{t=0}  \notag \\
&=&\left. \frac{d}{dt}\left( f\cdot \xi s_{t}\right) /\left( fs_{t}\right)
\right\vert _{t=0}+\left. \frac{d}{dt}\left( f_{t}\circ _{s}\xi \right)
/_{s}f\right\vert _{t=0}  \notag \\
&&-\left. \frac{d}{dt}\left( \left( f\circ _{s}\xi \right) /_{s}f\circ
_{s}f_{t}\right) /_{s}f\right\vert _{t=0}  \notag \\
&=&\left. \frac{d}{dt}\left( f\cdot \xi s_{t}\right) /\left( fs\right)
\right\vert _{t=0}-\left. \frac{d}{dt}\left( \left( f\cdot \xi s\right)
/\left( fs\right) \cdot fs_{t}\right) /fs\right\vert _{t=0}  \notag \\
&&+\left. \frac{d}{dt}\left( f_{t}\circ _{s}\xi \right) /_{s}f\right\vert
_{t=0}-\left. \frac{d}{dt}\left( \func{Ad}_{f}^{\left( s\right) }\xi \circ
_{s}f_{t}\right) /_{s}f\right\vert _{t=0}  \label{ddtAd1}
\end{eqnarray}%
Now consider the first two terms (suppressing the derivatives for clarity):%
\begin{eqnarray*}
\left( f\cdot \xi s_{t}\right) /\left( fs\right) &=&\left( f\circ _{s}\left(
\xi \circ _{s}\left( s_{t}/s\right) \right) \right) /_{s}f \\
\left( \left( f\cdot \xi s\right) /\left( fs\right) \cdot fs_{t}\right) /fs
&=&\left( \left( \left( f\circ _{s}\xi \right) /_{s}f\right) \circ
_{s}\left( f\circ _{s}s_{t}/s\right) \right) /_{s}f \\
&=&\left( \left( f\circ _{s}\xi \right) \circ _{s}s_{t}/s\right) /_{s}f+ 
\left[ \func{Ad}_{f}^{\left( s\right) }\xi ,f,s_{t}/s\right] ^{\left(
s\right) }/_{s}f
\end{eqnarray*}%
Thus, 
\begin{eqnarray}
\left( f\cdot \xi s_{t}\right) /\left( fs\right) -\left( \left( f\cdot \xi
s\right) /\left( fs\right) \cdot fs_{t}\right) /fs &=&\left[ f,\xi ,s_{t}/s%
\right] ^{\left( s\right) }/_{s}f  \label{ddtAd2} \\
&&-\left[ \func{Ad}_{f}^{\left( s\right) }\xi ,f,s_{t}/s\right] ^{\left(
s\right) }/_{s}f.  \notag
\end{eqnarray}
\end{proof}

The next two terms in (\ref{ddtAd1}) become%
\begin{eqnarray*}
\left( f_{t}\circ _{s}\xi \right) /_{s}f &=&\left( \left( f_{t}/_{s}f\circ
_{s}f\right) \circ _{s}\xi \right) /_{s}f \\
&=&\left( f_{t}/_{s}f\circ _{s}\left( f\circ _{s}\xi \right) \right) /_{s}f- 
\left[ f_{t}/_{s}f,f,\xi \right] ^{\left( s\right) }/_{s}f \\
&=&\left( f_{t}/_{s}f\right) \circ _{fs}\func{Ad}_{f}^{\left( s\right) }\xi -%
\left[ f_{t}/_{s}f,f,\xi \right] ^{\left( s\right) }/_{s}f \\
\left( \func{Ad}_{f}^{\left( s\right) }\xi \circ _{s}f_{t}\right) /_{s}f &=&%
\func{Ad}_{f}^{\left( s\right) }\xi \circ _{fs}\left( f_{t}/_{s}f\right)
\end{eqnarray*}%
Thus, 
\begin{equation}
\left( f_{t}\circ _{s}\xi \right) /_{s}f-\left( \func{Ad}_{f}^{\left(
s\right) }\xi \circ _{s}f_{t}\right) /_{s}f=\left[ f_{t}/_{s}f,\func{Ad}%
_{f}^{\left( s\right) }\xi \right] ^{\left( fs\right) }-\left[
f_{t}/_{s}f,f,\xi \right] ^{\left( s\right) }/_{s}f  \label{ddtAd3}
\end{equation}%
Overall, combining (\ref{ddtAd2}) and (\ref{ddtAd3}) and now using proper
notation, we obtain (\ref{dAdfs}).

\begin{theorem}[Theorem \protect\ref{thmKillingprop}]
The bilinear form $K^{\left( s\right) }$ (\ref{Killing}) on $\mathfrak{l}$
has the following properties.

\begin{enumerate}
\item Let $h\in \Psi ^{R}\left( \mathbb{L}\right) $, then for any $\xi ,\eta
\in \mathfrak{l},$ 
\begin{equation}
K^{\left( h\left( s\right) \right) }\left( h_{\ast }^{\prime }\xi ,h_{\ast
}^{\prime }\eta \right) =K^{\left( s\right) }\left( \xi ,\eta \right) .
\label{Kpsi1}
\end{equation}

\item Suppose also $\gamma \in \mathfrak{l,}$ then 
\begin{eqnarray}
K^{\left( s\right) }\left( \func{ad}_{\gamma }^{\left( s\right) }\eta ,\xi
\right) &=&-K^{\left( s\right) }\left( \eta ,\func{ad}_{\gamma }^{\left(
s\right) }\xi \right) +\func{Tr}\left( \func{Jac}_{\xi ,\gamma }^{\left(
s\right) }\circ \func{ad}_{\eta }^{\left( s\right) }\right)  \notag \\
&&+\func{Tr}\left( \func{Jac}_{\eta ,\gamma }^{\left( s\right) }\circ \func{%
ad}_{\xi }^{\left( s\right) }\right) ,  \label{Kad1}
\end{eqnarray}%
where $\func{Jac}_{\gamma ,\xi }^{\left( s\right) }:\mathfrak{l}%
\longrightarrow \mathfrak{l}$ is given by $\func{Jac}_{\eta ,\gamma
}^{\left( s\right) }\left( \xi \right) =\func{Jac}^{\left( s\right) }\left(
\xi ,\eta ,\gamma \right) .$

\item Let $\alpha \in \mathfrak{p},$ then 
\begin{eqnarray}
K^{\left( s\right) }\left( \alpha \cdot \xi ,\eta \right) &=&-K^{\left(
s\right) }\left( \xi ,\alpha \cdot \eta \right) +\func{Tr}\left( a_{\eta ,%
\hat{\alpha}}^{\left( s\right) }\circ \func{ad}_{\xi }^{\left( s\right)
}\right)  \label{Klie1} \\
&&+\func{Tr}\left( a_{\xi ,\hat{\alpha}}^{\left( s\right) }\circ \func{ad}%
_{\eta }^{\left( s\right) }\right) ,  \notag
\end{eqnarray}%
where $a_{\xi ,\eta }^{\left( s\right) }:\mathfrak{l}\longrightarrow 
\mathfrak{l}$ is given by $a_{\xi ,\eta }^{\left( s\right) }\left( \gamma
\right) =\left[ \gamma ,\xi ,\eta \right] ^{\left( s\right) }-\left[ \xi
,\gamma ,\eta \right] ^{\left( s\right) }$ and $\hat{\alpha}=\varphi
_{s}\left( \alpha \right) $.
\end{enumerate}
\end{theorem}

\begin{proof}
\begin{enumerate}
\item Let $h\in \Psi $, and then using the cyclic property of trace as well
as (\ref{loopalghom}), we have 
\begin{eqnarray*}
K^{\left( h\left( s\right) \right) }\left( h_{\ast }^{\prime }\xi ,h_{\ast
}^{\prime }\eta \right) &=&\func{Tr}\left( \func{ad}_{h_{\ast }^{\prime }\xi
}^{\left( h\left( s\right) \right) }\circ \func{ad}_{h_{\ast }^{\prime }\eta
}^{\left( h\left( s\right) \right) }\right) \\
&=&\func{Tr}\left( \left[ h_{\ast }^{\prime }\xi ,\left[ h_{\ast }^{\prime
}\eta ,\cdot \right] ^{\left( h\left( s\right) \right) }\right] ^{\left(
h\left( s\right) \right) }\right) \\
&=&\func{Tr}\left( \left[ h_{\ast }^{\prime }\xi ,\left[ h_{\ast }^{\prime
}\eta ,h_{\ast }^{\prime }\left( h_{\ast }^{\prime }\right) ^{-1}\cdot %
\right] ^{\left( h\left( s\right) \right) }\right] ^{\left( s\right) }\right)
\\
&=&\func{Tr}\left( \left[ h_{\ast }^{\prime }\xi ,h_{\ast }^{\prime }\left[
\eta ,\left( h_{\ast }^{\prime }\right) ^{-1}\cdot \right] ^{\left( s\right)
}\right] ^{\left( s\right) }\right) \\
&=&\func{Tr}\left( h_{\ast }^{\prime }\left[ \xi ,\left[ \eta ,\left(
h_{\ast }^{\prime }\right) ^{-1}\cdot \right] ^{\left( s\right) }\right]
^{\left( s\right) }\right) \\
&=&\func{Tr}\left( h_{\ast }^{\prime }\circ \left( \func{ad}_{\xi }^{\left(
s\right) }\circ \func{ad}_{\eta }^{\left( s\right) }\right) \circ \left(
h_{\ast }^{\prime }\right) ^{-1}\right) \\
&=&\func{Tr}\left( \func{ad}_{\xi }^{\left( s\right) }\circ \func{ad}_{\eta
}^{\left( s\right) }\right) \\
&=&K^{\left( s\right) }\left( \xi ,\eta \right) .
\end{eqnarray*}

\item From (\ref{Jac}), we see that 
\begin{eqnarray}
\func{ad}_{\left[ \eta ,\gamma \right] ^{\left( s\right) }}^{\left( s\right)
} &=&-\left[ \cdot ,\left[ \eta ,\gamma \right] ^{\left( s\right) }\right]
^{\left( s\right) }  \notag \\
&=&\left[ \eta ,\left[ \gamma ,\cdot \right] ^{\left( s\right) }\right]
^{\left( s\right) }-\left[ \gamma ,\left[ \eta ,\cdot \right] ^{\left(
s\right) }\right] ^{\left( s\right) }-\func{Jac}_{\eta ,\gamma }^{\left(
s\right) }  \notag \\
&=&\func{ad}_{\eta }^{\left( s\right) }\circ \func{ad}_{\gamma }^{\left(
s\right) }-\func{ad}_{\gamma }^{\left( s\right) }\circ \func{ad}_{\eta
}^{\left( s\right) }-\func{Jac}_{\eta ,\gamma }^{\left( s\right) }
\end{eqnarray}%
Hence, 
\begin{eqnarray*}
\func{ad}_{\left[ \eta ,\gamma \right] ^{\left( s\right) }}^{\left( s\right)
}\circ \func{ad}_{\xi }^{\left( s\right) } &=&\func{ad}_{\eta }^{\left(
s\right) }\circ \func{ad}_{\gamma }^{\left( s\right) }\circ \func{ad}_{\xi
}^{\left( s\right) }-\func{ad}_{\gamma }^{\left( s\right) }\circ \func{ad}%
_{\eta }^{\left( s\right) }\circ \func{ad}_{\xi }^{\left( s\right) } \\
&&-\func{Jac}_{\eta ,\gamma }^{\left( s\right) }\circ \func{ad}_{\xi
}^{\left( s\right) }
\end{eqnarray*}%
and so using the cycling symmetry of trace, we have%
\begin{eqnarray*}
K^{\left( s\right) }\left( \left[ \eta ,\gamma \right] ^{\left( s\right)
},\xi \right) &=&\func{Tr}\left( \func{ad}_{\eta }^{\left( s\right) }\circ
\left( \func{ad}_{\gamma }^{\left( s\right) }\circ \func{ad}_{\xi }^{\left(
s\right) }-\func{ad}_{\xi }^{\left( s\right) }\circ \func{ad}_{\gamma
}^{\left( s\right) }\right) \right) \\
&&-\func{Tr}\left( \func{Jac}_{\eta ,\gamma }^{\left( s\right) }\circ \func{%
ad}_{\xi }^{\left( s\right) }\right) \\
&=&\func{Tr}\left( \func{ad}_{\eta }^{\left( s\right) }\circ \func{ad}_{%
\left[ \gamma ,\xi \right] ^{\left( s\right) }}^{\left( s\right) }\right) +%
\func{Tr}\left( \func{ad}_{\eta }^{\left( s\right) }\circ \func{Jac}_{\gamma
,\xi }^{\left( s\right) }\right) \\
&&-\func{Tr}\left( \func{Jac}_{\eta ,\gamma }^{\left( s\right) }\circ \func{%
ad}_{\xi }^{\left( s\right) }\right) \\
&=&K^{\left( s\right) }\left( \eta ,\left[ \gamma ,\xi \right] ^{\left(
s\right) }\right) +\func{Tr}\left( \func{Jac}_{\gamma ,\xi }^{\left(
s\right) }\circ \func{ad}_{\eta }^{\left( s\right) }\right) \\
&&+\func{Tr}\left( \func{Jac}_{\gamma ,\eta }^{\left( s\right) }\circ \func{%
ad}_{\xi }^{\left( s\right) }\right) .
\end{eqnarray*}%
This then gives (\ref{Kad1}).

\item Now let $\alpha \in \mathfrak{p}$ and consider 
\begin{equation*}
K^{\left( s\right) }\left( \alpha \cdot \xi ,\eta \right) =\func{Tr}\left( 
\func{ad}_{\alpha \cdot \xi }^{\left( s\right) }\circ \func{ad}_{\eta
}^{\left( s\right) }\right) .
\end{equation*}%
Denote by $l_{\alpha }:\mathfrak{l}\longrightarrow \mathfrak{l}$ the left
action of $\mathfrak{p}$ on $\mathfrak{l.}$ From (\ref{xilbrack}), we then
have 
\begin{equation}
\func{ad}_{\alpha \cdot \xi }^{\left( s\right) }=l_{\alpha }\circ \func{ad}%
_{\xi }^{\left( s\right) }-\func{ad}_{\xi }^{\left( s\right) }\circ
l_{\alpha }+a_{\xi ,\hat{\alpha}}^{\left( s\right) }
\end{equation}%
So now, 
\begin{eqnarray*}
K^{\left( s\right) }\left( \alpha \cdot \xi ,\eta \right) &=&\func{Tr}\left(
l_{\alpha }\circ \func{ad}_{\xi }^{\left( s\right) }\circ \func{ad}_{\eta
}^{\left( s\right) }-\func{ad}_{\xi }^{\left( s\right) }\circ l_{\alpha
}\circ \func{ad}_{\eta }^{\left( s\right) }\right) \\
&&+\func{Tr}\left( a_{\xi ,\hat{\alpha}}^{\left( s\right) }\circ \func{ad}%
_{\eta }^{\left( s\right) }\right) \\
&=&\func{Tr}\left( \func{ad}_{\xi }^{\left( s\right) }\circ \left( \func{ad}%
_{\eta }^{\left( s\right) }\circ l_{\alpha }-l_{\alpha }\circ \func{ad}%
_{\eta }^{\left( s\right) }\right) \right) +\func{Tr}\left( a_{\xi ,\hat{%
\alpha}}^{\left( s\right) }\circ \func{ad}_{\eta }^{\left( s\right) }\right)
\\
&=&-\func{Tr}\left( \func{ad}_{\xi }^{\left( s\right) }\circ \func{ad}%
_{\alpha \cdot \eta }^{\left( s\right) }\right) +\func{Tr}\left( \func{ad}%
_{\xi }^{\left( s\right) }\circ a_{\eta ,\hat{\alpha}}^{\left( s\right)
}\right) +\func{Tr}\left( a_{\xi ,\hat{\alpha}}^{\left( s\right) }\circ 
\func{ad}_{\eta }^{\left( s\right) }\right) \\
&=&-K^{\left( s\right) }\left( \xi ,\alpha \cdot \eta \right) +\func{Tr}%
\left( a_{\eta ,\hat{\alpha}}^{\left( s\right) }\circ \func{ad}_{\xi
}^{\left( s\right) }\right) +\func{Tr}\left( a_{\xi ,\hat{\alpha}}^{\left(
s\right) }\circ \func{ad}_{\eta }^{\left( s\right) }\right) .
\end{eqnarray*}
\end{enumerate}
\end{proof}

\bibliographystyle{s:/tex/BibTeX/bst/habbrv}
\bibliography{s:/tex/input/refs2}

\end{document}